\documentclass[12pt]{amsart}

\usepackage{frontmatter}

\title[Dimer FPUT Nanopterons with Exponentially Small, Nonvanishing Ripples]{Mass and Spring Dimer Fermi--Pasta--Ulam--Tsingou Nanopterons with Exponentially Small, Nonvanishing Ripples}

\author{Timothy E. Faver}
\address{Department of Mathematics, Kennesaw State University, 850 Polytechnic Lane, Marietta, GA 30060 USA, {\tt{tfaver1@kennesaw.edu}}}

\author{Hermen Jan Hupkes}
\address{Mathematical Institute, Universiteit Leiden, P.O. Box 9512, 2300 RA Leiden, The Netherlands, {\tt{hhupkes@math.leidenuniv.nl}}}



\date{\today}

\begin{document}

\begin{abstract}
We study traveling waves in mass and spring dimer Fermi--Pasta--Ulam--Tsingou (FPUT) lattices in the long wave limit. 
Such lattices are known to possess nanopteron traveling waves in relative displacement coordinates. 
These nanopteron profiles consist of the superposition of an exponentially localized ``core,'' which is close to a KdV solitary wave, and a periodic ``ripple,'' whose amplitude is small beyond all algebraic orders of the long wave parameter, although a zero amplitude is not precluded. 
Here we deploy techniques of spatial dynamics, inspired by results of Iooss and Kirchg\"{a}ssner, Iooss and James, and Venney and Zimmer, to construct mass and spring dimer nanopterons whose ripples are both exponentially small and also nonvanishing. 
We first obtain ``growing front'' traveling waves in the original position coordinates and then pass to relative displacement. 
To study position, we recast its traveling wave problem as a first-order equation on an infinite-dimensional Banach space; then we develop hypotheses that, when met, allow us to reduce such a first-order problem to one solved by Lombardi. A key part of our analysis is then the passage back from the reduced problem to the original one. 
Our hypotheses free us from working strictly with lattices but are easily checked for FPUT mass and spring dimers. 
We also give a detailed exposition and reinterpretation of Lombardi's methods, to illustrate how our hypotheses work in concert with his techniques, and we provide a dialogue with prior methods of constructing FPUT nanopterons, to expose similarities and differences with the present approach.
\end{abstract}

\maketitle

\section{Introduction}

\subsection{The polyatomic FPUT lattice}
A polyatomic, or polymer, Fermi--Pasta--Ulam--Tsingou (FPUT) lattice \cite{fput-original, dauxois, gmwz, pankov,vainchtein-survey} is an infinite, one-dimensional chain of particles, connected on the left and right to their nearest neighbors by springs, with motion constrained to a horizontal axis.
We index the particles and springs by integers $j \in \Z$; the $j$th spring connects the $j$th particle on the left to the $(j+1)$st particle on the right.
Let $m_j$ be the mass of the $j$th particle, let $\V_j$ be the potential of the $j$th spring, and let $\ell_j$ be the equilibrium length of the $j$th spring, so that the $j$th spring exerts the force $\V_j'(r-\ell_j)$ when stretched a distance $r$ along this horizontal axis. 

If $y_j$ denotes the position of the $j$th particle along the axis, relative to some fixed origin point, then Newton's second law gives the lattice's equations of motion:
\begin{equation}\label{eqn: original equations of motion with ell-j}
m_j\ddot{y}_j
= \V_j'(y_{j+1}-y_j-\ell_j) - \V_{j-1}'(y_j-y_{j-1}-\ell_{j-1}).
\end{equation}
We may assume that the equilibrium lengths are all 0 by making the change of variables
\begin{equation}\label{eqn: equilibrium length cov}
y_j = \begin{cases}
u_j + \medsum_{k=0}^{j-1} \ell_k, & j \ge 1 \\
u_0, & j = 0 \\
u_j-\medsum_{k=j}^{-1} \ell_k, & j \le -1
\end{cases}
\end{equation}
to find
\begin{equation}\label{eqn: rel disp y and u}
y_{j+1}-y_j - \ell_j
= u_{j+1}-u_j.
\end{equation}
Then $y_j$ satisfies \eqref{eqn: original equations of motion with ell-j} if and only if $u_j$ satisfies
\begin{equation}\label{eqn: original equations of motion}
m_j\ddot{u}_j 
= \V_j'(r_j) - \V_{j-1}'(r_{j-1}),
\end{equation}
where 
\begin{equation}\label{eqn: rel disp coords}
r_j 
:= u_{j+1} - u_j
\end{equation}
are the classical relative displacement coordinates.
In terms of relative displacement, the equations of motion are
\begin{equation}\label{eqn: rel disp eqns}
\ddot{r}_j
= \frac{1}{m_{j+1}}\V_{j+1}'(r_{j+1}) - \left(\frac{1}{m_j}+\frac{1}{m_{j+1}}\right)\V_j'(r_j) + \frac{1}{m_j}\V_{j-1}'(r_{j-1}).
\end{equation}

While most studies of lattice dynamics concentrate on relative displacement, we will construct solutions both in position coordinates \eqref{eqn: original equations of motion} and relative displacement \eqref{eqn: rel disp eqns}.
We state these main results in Theorems \ref{thm: main theorem mass dimer} and \ref{thm: main theorem spring dimer}, discuss the prior history of our problems of interest in Section \ref{sec: prior nanopteron results}, present our primary motivations for the current research in Section \ref{sec: guiding questions}, and give an overview of our methods in Section \ref{sec: spatial dynamics method}.

\begin{remark}
What we call ``position'' coordinates $u_j$ are really ``displacement from equilibrium'' coordinates, due to the change of variables \eqref{eqn: equilibrium length cov}.
Nonetheless, we will retain the ``position'' terminology so as not to overwork the noun ``displacement.''
\end{remark}

We now impose further structure on our material data by assuming that the masses of the particles and the potentials of the springs repeat with some finite periodicity. 
That is, for some integer $N \ge 1$ we have
\begin{equation}\label{eqn: N-periodic data}
m_{j+N} = m_j,
\qquad
\V_{j+N} = \V_j,
\quadword{and}
\ell_{j+N} = \ell_j, \ j \in \Z,
\end{equation}
in which case we might say that the lattice is $N$-polyatomic.
These material values will be fixed throughout our analysis; other studies permit the masses or springs to vary in certain limits, which we briefly discuss in Section \ref{sec: prior nanopteron results}.
If $N = 1$, then the lattice is called monatomic, and if $N = 2$, then the lattice is a (general) dimer.

From now on we will focus on dimers, and we will eventually specialize to two species of dimer: the mass dimer (or diatomic lattice), in which $\V_j = \V$ for some single potential function $\V$ and $m_1 \ne m_2$ are distinct, and the spring dimer, in which $m_j = m$ for some single mass value $m > 0$ and $\V_1 \ne \V_2$ are distinct.
We sketch the ``general'' dimer in Figure \ref{fig: general dimer}.

\begin{figure}
\[
\begin{tikzpicture}

\def\arith#1#2#3{
#3*#1+#3*#2
};

\def\squaremass#1#2#3{
\draw[fill=blue, opacity=.4, draw opacity=1,thick]
(#1,#2) rectangle (#1+2*#2,-#2);
\node at (#1+#2,0){$\boldsymbol{#3}$};
};

\def\circlemass#1#2#3{
\draw[fill=yellow,thick] (#1,0)node{$\boldsymbol{#3}$} circle(#2);
};

\def\spring#1#2#3#4{
\draw[line width = 1.5pt] (#1,0)
--(#1 + #2,0)
--(#1 + #2 + #3, #4)
--(#1 + #2 + 3*#3,-#4)
--(#1 + #2 + 5*#3,#4)
--(#1 + #2 + 7*#3,-#4)
--(#1 + #2 + 9*#3,#4)
--(#1 + #2 + 10*#3,0)
--(#1 + 2*#2 + 10*#3,0);
};

\def\d{.15}; 
\def\mh{1}; 
\def\sh{.25}; 
\def\swl{.15}; 
\def\r{.35}; 
\def \sl#1{\arith{2*\d}{10*\swl}{#1}}; 


\def\h{.45}; 
\def\M{.88}; 
\def\N{.3009}; 
\def\E{.2}; 


\def\coil#1{ 
{\N+\M*\t+\E*sin(4*\t*pi r)+#1},
{\h*cos(4*\t*pi r)}
}


\def\NLSPRING#1{
\draw[line width=1.5pt,domain={-.125:1.125},smooth,variable=\t,samples=100]
(#1,0)--plot(\coil{#1+\d})
--(#1+\sl{1},0);
}

\def\labeldown#1#2{
\draw[densely dotted,thick] (#1,-\mh/2-\d)--(#1,-\mh/2-\d-.75)node[below]{#2}
};

\def\brace#1#2{
\draw[decoration={brace, amplitude = 10pt,mirror},decorate,thick] 
(#1,-\mh/2-\d)--node[midway,below=7pt]{#2}(#1+\sl{1}+\mh/2+\r,-\mh/2-\d)
}


\fill[yellow] (-\r,\r) arc(90:-90:\r)--cycle;
\draw[thick] (-\r,\r) arc(90:-90:\r);


\NLSPRING{0};


\squaremass{\sl{1}}{\mh/2}{m};


\spring{\sl{1}+\mh}{\d}{\swl}{\sh};


\circlemass{\sl{2}+\mh+\r}{\r}{1};


\NLSPRING{\sl{2}+\mh+2*\r}:


\squaremass{\sl{3}+\mh+2*\r}{\mh/2}{m};


\spring{\sl{3}+2*\mh+2*\r}{\d}{\swl}{\sh};


\circlemass{\sl{4}+2*\mh+3*\r}{\r}{1};


\NLSPRING{\sl{4}+2*\mh+4*\r};


\fill[blue,opacity=.4]
(\sl{5}+2*\mh+4*\r,\mh/2) rectangle (\sl{5}+2*\mh+4*\r+\mh/2,-\mh/2);

\draw[thick] (\sl{5}+2*\mh+4*\r+\mh/2,\mh/2)
--(\sl{5}+2*\mh+4*\r,\mh/2)
--(\sl{5}+2*\mh+4*\r,-\mh/2)
--(\sl{5}+2*\mh+4*\r+\mh/2,-\mh/2);


\labeldown{\sl{1}+\mh/2}{$u_{j-1}$};
\labeldown{\sl{2}+\mh+\r}{$u_j$};
\labeldown{\sl{3}+3*\mh/2+2*\r}{$u_{j+1}$};
\labeldown{\sl{4}+2*\mh + 3*\r}{$u_{j+2}$};

\brace{\sl{1}+\mh/2}{$r_{j-1}$};
\brace{\sl{2}+\mh+\r}{$r_j$};
\brace{\sl{3}+3*\mh/2+2*\r}{$r_{j+1}$};

\end{tikzpicture}
\]
\caption{An FPUT general dimer (masses normalized to 1 and to $m \ne 1$)}
\label{fig: general dimer}
\end{figure}
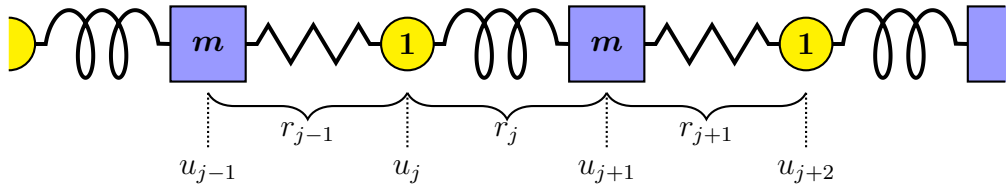

We assume that the spring potentials $\V_1$ and $\V_2$ are real analytic with $\V_j'(0) > 0$ for $j=1$, 2 and $\V_j''(0) \ne 0$ for at least one $j$.
Routine nondimensionalizations (see \cite[Sec.\@ 1]{faver-wright} and \cite[Sec.\@ 1.2]{faver-spring-dimer}) allow us to write
\begin{equation}\label{eqn: dimer masses}
m_j
= \begin{cases}
1, &j \text{ is odd} \\
1/w, &j \text{ is even}
\end{cases}
\end{equation}
for some $w > 0$ and
\begin{equation}\label{eqn: dimer springs}
\V_j'(r)
= \begin{cases}
r + r^2 + \Vscr_1(r), &j \text{ is odd} \\
\kappa{r} + \beta{r}^2 + \Vscr_2(r), &j \text{ is even},
\end{cases}
\qquad
\lim_{r \to 0} \frac{\Vscr_1(r)}{r^2} = \lim_{r \to 0} \frac{\Vscr_2(r)}{r^2} = 0.
\end{equation}

We require $\kappa > 0$ but do not (yet) place any restrictions on $\beta$.
To ensure that the lattice is a dimer, then, we will always take $w \ne 1$ or $\kappa \ne 1$.
Moreover, by relabeling the original nondimensionalized lattice, we may always assume that at least one of $w$ or $\kappa$ is larger than~1.
Finally, when discussing a mass dimer, we will write the potential as 
\[
\V(r)
= r+r^2+\Vscr(r), \qquad \lim_{r \to 0} \frac{\Vscr(r)}{r^2} = 0.
\]

\begin{remark}
Our convention with the dimer's material data in \eqref{eqn: dimer masses} and \eqref{eqn: dimer springs} has been to normalize the ``odd'' parameters to 1, while the ``even'' parameters vary.
This is how \cite{faver-wright} treated the mass dimer, but \cite{faver-spring-dimer} used the reverse convention for the spring dimer.
One should keep this in mind when comparing our set-up, in particular our relative displacement traveling wave problem \eqref{eqn: rel disp tw eqns}, and results to theirs.
\end{remark}

\begin{remark}
The assumption of real analytic spring potentials is not so stringent a restriction as it might seem when we consider that most studies assume a high order of regularity for the potentials. 
For example, in the solitary wave papers \cite{friesecke-pego1, faver-goodman-wright} the potentials are at least $\Cal^5$, while in the nanopteron papers \cite{faver-wright, faver-spring-dimer, faver-hupkes-equal-mass} they are $\Cal^{\infty}$.
We discuss in Remark \ref{rem: why Cinfty potentials} the utility of $\Cal^{\infty}$-potentials for these latter papers and in Remark \ref{rem: why Lombardi real analytic} the technically essential reason why we demand real analytic potentials in the current approach.
\end{remark}

\subsection{Traveling waves in dimers}
We will work in two coordinate systems.
First, we can remain in position coordinates and make the traveling wave ansatz
\begin{equation}\label{eqn: position tw ansatz}
u_j(t)
= \begin{cases}
p_1(j-ct), &j \text{ is odd} \\
p_2(j-ct), &j \text{ is even.}
\end{cases}
\end{equation}
Here the profiles $p_1$ and $p_2$ are functions of a single real variable and $c \in \R$ is the wave speed.
Then we obtain the system of advance-delay differential equations
\begin{equation}\label{eqn: position tw eqns}
\begin{cases}
c^2p_1'' = \V_1'(S^1p_2-p_1)-\V_2'(p_1-S^{-1}p_2) \\
c^2p_2'' = w\V_2'(S^1p_1-p_2)-w\V_1'(p_2-S^{-1}p_1).
\end{cases}
\end{equation}
For $d \in \R$, the operator $S^d$ is the ``shift-by-$d$'' operator $(S^dp)(x) := p(x+d)$.

Although we will work primarily in the position traveling wave coordinates \eqref{eqn: position tw ansatz}, this is not the typical framework in the literature, which instead largely studies traveling waves in relative displacements.
For that, put
\begin{equation}\label{eqn: rel disp tw ansatz}
r_j(t)
= \begin{cases}
\varrho_1(j-ct), &j \text{ is odd} \\
\varrho_2(j-ct), &j \text{ is even}
\end{cases}
\end{equation}
to find that from the relative displacement equations \eqref{eqn: rel disp eqns}, the new profiles $\varrho_1$ and $\varrho_2$ must satisfy 
\begin{equation}\label{eqn: rel disp tw eqns}
\begin{cases}
c^2\varrho_1'' = -(1+w)\V_1'(\varrho_1)+(wS^1+S^{-1})\V_2'(\varrho_2) \\
c^2\varrho_2'' = (S^1+wS^{-1})\V_1'(\varrho_1)-(1+w)\V_2'(\varrho_2).
\end{cases}
\end{equation}

Any solution to the position traveling wave problem \eqref{eqn: position tw eqns} yields a solution to the relative displacement traveling wave problem \eqref{eqn: rel disp tw eqns} via the definition of relative displacement.
Specifically, if the pair $(p_1,p_2)$ solves \eqref{eqn: position tw eqns}, then the pair $(\varrho_1,\varrho_2)$ defined by 
\begin{equation}\label{eqn: tw profile rels}
\varrho_1 := S^1p_2 - p_1 
\quadword{and}
\varrho_2 := S^1p_1-p_2
\end{equation}
solves \eqref{eqn: rel disp tw eqns}.
This follows by factoring $S^{-1}$ out of the second term on the right in each equation in \eqref{eqn: position tw eqns}, e.g., $\V_2'(p_1-S^{-1}p_2) = S^{-1}\V_2'(S^1p_1-p_2)$.
However, a relative displacement solution does not necessarily guarantee a position solution; see Question \ref{ques: pos from rel disp?} below.

\subsection{Prior nanopteron results for dimers}\label{sec: prior nanopteron results}
Now that we possess an adequate vocabulary and notation for traveling wave problems in dimers, we can discuss our historical motivation and then, in the proper context, at last state our results.
Our interest in traveling waves in dimers arises from the fact that over long times, (solitary wave) solutions to certain KdV equations are very good approximations to solutions in relative displacement coordinates of \eqref{eqn: rel disp eqns}.
More precisely, Gaison, Moskow, Wright, and Zhang \cite{gmwz} used techniques from homogenization theory to prove the estimate
\begin{equation}\label{eqn: gmwz est}
\sup_{|t| \le T_0\ep^{-3}} \big|r_j(t) - \ep^2\Ksf_+(\ep(j+\cs{t})) - \ep^2\Ksf_-(\ep(j-\cs{t}))\big|
\le C\ep^{5/2}
\end{equation}
for solutions $r_j$ to \eqref{eqn: rel disp eqns}.
The functions $\Ksf_{\pm}$ are solutions of certain KdV equations whose coefficients are derived from the material data \eqref{eqn: N-periodic data} of the lattice, and $C$ and $T_0$ are (material-dependent) constants.
The $\ep$-dependent scaling $\ep^2\Ksf_{\pm}(\ep\cdot)$ is the classical ``long wave'' scaling; see \cite[Sec.\@ 1.2]{schneider-wayne-water-wave} for a concise historical overview.
The wave speed $\cs = \cs(\kappa,w)$ is the lattice's ``speed of sound,'' and we define it precisely for the dimer in terms of $w$ and $\kappa$ in \eqref{eqn: cs defn}.

The estimate \eqref{eqn: gmwz est} is actually valid for all ``$N$-polyatomic'' lattices satisfying the $N$-periodicity conditions of \eqref{eqn: N-periodic data} and some additional technical hypotheses; in particular, \eqref{eqn: gmwz est} holds for monatomic lattices and dimers without any further hypotheses.
Chirilus-Bruckner, Chong, Prill, and Schneider \cite{chirilus-bruckner-etal} obtained a similar estimate using Bloch wave transforms.
A version of this estimate was first proved by Schneider and Wayne in \cite{schneider-wayne} for monatomic lattices.
See \cite[Thm.\@ 5.2]{gmwz} and \cite[Thm.\@ 1.3]{pankov} for proofs in the framework of differential equations on $\ell^2(\Z)$ that such relative displacement solutions $r_j$ exist in the first place.

A natural question, then, is if these KdV approximate ``solutions'' extend to solitary wave solutions of \eqref{eqn: rel disp eqns} for all time.
Friesecke and Pego \cite{friesecke-pego1} answered this question affirmatively for the monatomic lattice.
Their solutions have the form
\begin{equation}\label{eqn: intro mono sol}
r_j(t)
= \ep^2\varsigma(\ep(j-\cep{t})) + \ep^3\hsf_{\ep}(\ep(j-\cep{t})),
\end{equation}
where $\varsigma$ is a rescaled $\sech^2$-type solution to a certain KdV equation that naturally arises as the monatomic lattice's ``continuum limit.''
The higher-order remainder $\hsf_{\ep} = \hsf_{\ep}(X)$ is exponentially localized in $X$; that is, there exist $C$, $q > 0$ such that $|\hsf_{\ep}(X)| \le Ce^{-q|X|}$ for all $X \in \R$.
The wave speed $\cep$ is ``near-sonic'' in the sense that $\cep^2 = \cs^2 + \O(\ep^2)$.
Friesecke and Wattis \cite{friesecke-wattis} also obtained solitary waves in monatomic lattices using variational techniques.

The results for dimers are rather different.
Via an ansatz established by Beale \cite{beale} for the capillary-gravity water wave problem and techniques refined by Amick and Toland \cite{amick-toland} for a singularly perturbed KdV-type equation, originally derived in \cite{hunter-scheurle}, Faver and Wright proved in \cite{faver-wright} that the mass dimer has relative displacement traveling waves of the form
\begin{equation}\label{eqn: intro dimer nano}
r_j(t)
= \ep^2\varsigma_w(\ep(j-\cep{t}))
+\ep^3\hsf_j^{\ep}(\ep(j-\cep{t})) 
+ \ep^2a_{\ep}\psf_j^{\ep}(\ep(j-\cep{t})).
\end{equation}
As before, $\varsigma_w$ is a scaled $\sech^2$-type solution to a KdV equation that is the dimer's ``continuum limit''; we give the precise formula for $\varsigma_w$ in \eqref{eqn: varsigma-w}.
The wave speed is $\cep^2 = \cs^2+\O(\ep^2)$, where now $\cs$ is the dimer's speed of sound.
Both $\hsf_j^{\ep} = \hsf_j^{\ep}(X)$ and the new term $\psf_j^{\ep} = \psf_j^{\ep}(X)$ are smooth in $X$ and $2$-periodic in $j$.

Although the remainder term $\hsf_j^{\ep}$ is exponentially localized in $X$ as in the Friesecke-Pego solution, the new term $\psf_j^{\ep}$ is periodic in $X$ with frequency $\O(\ep^{-1})$.
Moreover, taking
\begin{equation}\label{eqn: per profs FW}
r_j(t)
= \ep^2a_{\ep}\psf_j^{\ep}(\ep(j-\cep{t}))
\end{equation}
independently solves the relative displacement traveling wave problem \eqref{eqn: rel disp tw eqns}.

The ``amplitude'' coefficient $a_{\ep}$ is small beyond all algebraic orders of $\ep$ in the sense that 
\begin{equation}\label{eqn: small BAAO}
\lim_{\ep \to 0} \ep^{-r}a_{\ep}
= 0
\end{equation}
for all $r > 0$.
Thus the traveling wave \eqref{eqn: intro dimer nano} is a nanopteron, not a solitary wave like \eqref{eqn: intro mono sol}, unless $a_{\ep} = 0$.
Boyd \cite{boyd} coined the term ``nanopteron'' to denote a ``nonlocal'' or ``generalized'' solitary wave whose profile consists of the superposition of an exponentially localized ``core'' and extremely small asymptotic  ``ripples.''
We sketch the nanopteron from \eqref{eqn: intro dimer nano} in Figure \ref{fig: nanopteron}.

A largely similar result holds for relative displacement traveling waves in spring dimers \cite{faver-spring-dimer}; see Section \ref{sec: faver wright comparisons} for a more detailed leading order expansion.
For the sake of contrast with our methods in this paper, we outline in Appendix \ref{app: Beale} the ``Beale's ansatz'' method used in \cite{faver-wright, faver-spring-dimer}.

Beale's ansatz has facilitated the construction of traveling wave solutions to a variety of problems beyond the dimer long wave scenario.
Hoffman and Wright \cite{hoffman-wright} used it to construct nanopterons in the dimer ``small mass'' limit, in which $c$ is now fixed while $m \to 0^+$, thereby making the dimer monatomic.
Faver and Hupkes \cite{faver-hupkes-equal-mass} subsequently produced micropterons (the superposition of a localized core and an algebraically, but not exponentially, small periodic ripple) in the dimer ``equal mass'' limit, in which $c$ is again fixed and $m \to 1$.
See \cite[Sec.\@ 1.2]{faver-hupkes-numerics} for a succinct comparison of the long wave, small mass, and equal mass scenarios.
Faver \cite{faver-mim-nanopteron} studied a related small mass limit for mass-in-mass (MiM) lattices, which are, roughly, monatomic FPUT lattices in which each particle is paired with an internal resonator \cite{cpkd}.
There are numerous technical differences among these problems, on which we do not dwell here, but all three share an obvious difference from the long wave results in \cite{faver-wright, faver-spring-dimer}: the exponentially localized core of the traveling waves in these three ``material'' limits is always $\O(1)$ in the problem's relevant small parameter.
In contrast, as \eqref{eqn: intro dimer nano} specifies, the dimers' long wave core is $\O(\ep^2)$.
However, Johnson and Wright \cite{johnson-wright} used Beale's ansatz to study the gravity-capillary Whitham equation in the long wave limit, and their core was also $\O(\ep^2)$.
This small core naturally arises from the long wave scaling, which is, of course, not present in the material limits.  
It is quite interesting to note that Johnson, Truong, and Wheeler \cite{johnson-truong-wheeler} also studied the gravity-capillary Whitham equation using a nonlocal center manifold reduction and found, effectively, micropterons; they obtained an $\O(\ep^2)$ core but an explicit nonzero $\O(\ep)$ leading order term for their periodic ripples.

Exact traveling waves are far from the only area of interest for FPUT lattices.
We mention just a handful of related results here and discuss others in the context of future problems in Section \ref{sec: future directions}.
Pelinovsky-Schneider \cite{pelinovsky-schneider} considered the dimer small mass limit as an initial value problem in $\ell^2(\Z)$-type sequence spaces and found that if the initial data is close to a solution of the limiting monatomic lattice and the mass ratio is sufficiently small, then the dimer solution remains close to that monatomic solution.
McGinnis and Wright \cite{mcginnis-wright} moved well beyond the polyatomic regime to find that the classical wave equation is a good approximation for linear FPUT lattices with ``random'' material data, i.e., their potentials are $\V_j'(r) = \kappa_jr$, where $\kappa_j$ and the masses $m_j$ are all random variables.
Carmichael \cite{carmichael} proved a KdV approximation result like \eqref{eqn: gmwz est} for monatomic lattices with ``planar'' motion, where the lattice is not constrained to the horizontal like ours but can move in two dimensions.

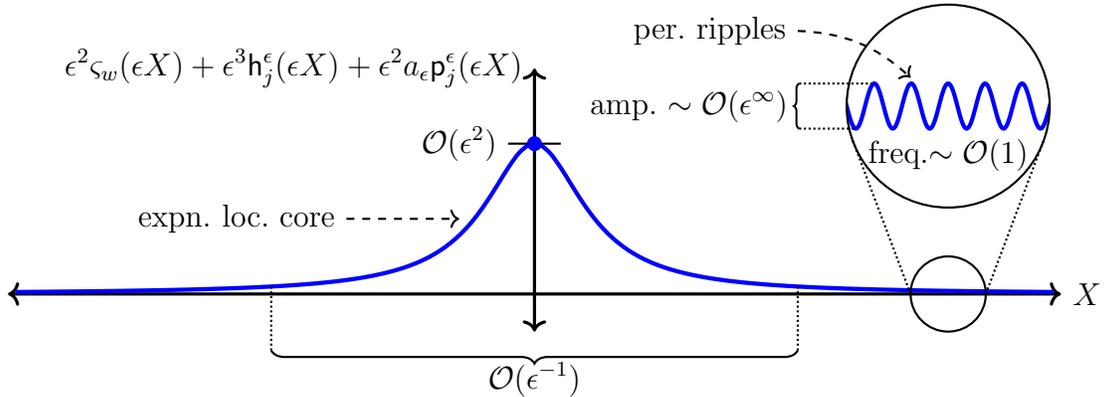
\begin{figure}
\[
\begin{tikzpicture}

\def \nanoCoreAmp{2};
\def \nanoDecayRate{1.25};

\def \perAmp{.3};
\def \perFreq{15};

\draw[very thick,<->] (0,-.5)--(0,3)node[left]{$\ep^2\varsigma_w(\ep{X}) + \ep^3\hsf_j^{\ep}(\ep{X}) + \ep^2a_{\ep}\psf_j^{\ep}(\ep{X})$};
\draw[very thick,<->] (-7,0)--(7,0)node[right]{$X$};

\draw[ultra thick,blue] plot[domain=-6.9:6.9, samples=100, smooth] (\x,{\nanoCoreAmp/(1+(\nanoDecayRate*\x)^2)});

\def \bunderX{3.5};

\draw[thick] (-.35,\nanoCoreAmp)node[left]{$\O(\ep^2)$}--(.35,\nanoCoreAmp);
\fill[blue] (0,\nanoCoreAmp) circle(.1);

\foreach \x in {\bunderX, -\bunderX}
{
\draw[thick,densely dotted] (-\x,{\nanoCoreAmp/(1+(\nanoDecayRate*\x)^2)})--(-\x,-.75);
};

\draw[thick,decoration={brace, amplitude = 5pt,mirror},decorate] (-\bunderX,-.75)--node[midway,below]{$\O(\ep^{-1})$}(\bunderX,-.75);

\def \lowerCenterX{5.5};
\def \lowerCenterY{0};

\def \lowerRad{.5};
\def \upperRad{1.35};

\def \upperCenterX{\lowerCenterX};
\def \upperCenterY{\lowerCenterY+2.5};

\def \lowerRightTanX{\lowerCenterX+\lowerRad};
\def \lowerRightTanY{\lowerCenterY};

\def \lowerLeftTanX{\lowerCenterX-\lowerRad};
\def \lowerLeftTanY{\lowerCenterY};

\def \leftAngle{pi/10};

\def \rightAngle{11*pi/10}

\draw[thick] (\lowerCenterX,\lowerCenterY) circle(\lowerRad);

\draw[densely dotted,thick] (\lowerRightTanX,\lowerRightTanY) -- ({\upperCenterX+\upperRad*cos(-\leftAngle r)},{\upperCenterY+\upperRad*sin(-\leftAngle r)});
\draw[densely dotted,thick] (\lowerLeftTanX,\lowerLeftTanY) -- ({\upperCenterX+\upperRad*cos(\rightAngle r)},{\upperCenterY+\upperRad*sin(\rightAngle  r)});

\def \peaks{3};

\begin{scope}

\clip (\upperCenterX,\upperCenterY) circle(\upperRad);

\node[align=center,outer sep=0pt,inner sep=0pt] at (\upperCenterX, \upperCenterY){\begin{tikzpicture}

\draw[ultra thick,blue] plot[domain=-\upperRad-.1:\upperRad+.1,samples=100,smooth] (\x,{\perAmp*cos((2*\peaks-.5)*pi*\x/\upperRad r)});

\end{tikzpicture}};

\end{scope}

\node[below] at (\upperCenterX,\upperCenterY-\perAmp){freq.\@$\sim\O(1)$};

\draw[thick] (\upperCenterX,\upperCenterY) circle(\upperRad);

\draw[thick,densely dotted] 
({\upperCenterX-(2*\peaks-2)*\upperRad/(2*\peaks-.5)-1},{\upperCenterY+\perAmp})
--
({\upperCenterX-(2*\peaks-2)*\upperRad/(2*\peaks-.5)},{\upperCenterY+\perAmp});

\draw[thick,densely dotted] 
({\upperCenterX-(2*\peaks-2)*\upperRad/(2*\peaks-.5)-1},{\upperCenterY-\perAmp})
--
({\upperCenterX-(2*\peaks-1)*\upperRad/(2*\peaks-.5)},{\upperCenterY-\perAmp});

\draw [line width = .75pt,decorate,decoration={brace,amplitude=2pt,mirror}] 
({\upperCenterX-(2*\peaks-2)*\upperRad/(2*\peaks-.5)-1},{\upperCenterY+\perAmp})
--node[midway,left]{amp.\@ $\sim \O(\ep^{\infty})$}
({\upperCenterX-(2*\peaks-2)*\upperRad/(2*\peaks-.5)-1},{\upperCenterY-\perAmp});

\draw[->,thick,dashed] (-2.5,1)node[left]{expn.\@ loc.\@ core}--(-1,1);
\draw[->,thick,dashed] (3.5,3.5)node[left]{per.\@ ripples} to[bend left=20] (5,2.88);

\end{tikzpicture}
\]
\caption{The nanopteron traveling wave profile from \eqref{eqn: intro dimer nano} for the mass dimer (inspired by \cite[Fig.\@ 2]{faver-wright}).
In the inset, we write $\O(\ep^{\infty})$ to indicate that the amplitude of the periodic ripples is small beyond all orders of $\ep$.}
\label{fig: nanopteron}
\end{figure}

\subsection{Guiding questions about nanopterons in FPUT dimers}\label{sec: guiding questions}
The nanopteron results in \cite{faver-wright, faver-spring-dimer} suggest a number of interesting questions for FPUT dimers.

\begin{enumerate}[label={\bf{Question \arabic*.}}, ref={\arabic*}]

\item
\label{ques: per amp is 0?}
Is it possible that the periodic amplitude coefficient $a_{\ep}$ is 0?
In this case, the nanopteron reduces to a solitary wave.
There is a large body of numerical evidence in mass dimers suggesting that for a discrete sequence of mass ratios that accumulates at 0, the traveling wave problem \eqref{eqn: rel disp tw eqns} has solitary wave solutions \cite{VSWP, SV, lustri-porter, lustri, faver-hupkes-numerics}.
The dimer small mass limit \cite{hoffman-wright} must exclude such a sequence of mass ratios from its nanopteron constructions, although nonexistence of nanopterons is not established; a similar phenomenon occurs in \cite{faver-mim-nanopteron}, although there solitary waves can be constructed at the excluded mass ratios \cite{faver-goodman-wright}.
However, it is less clear if, given a fixed mass ratio, the long wave solution \eqref{eqn: intro dimer nano} can be a solitary wave for some value(s) of $\ep$.
This has been resolved for other long wave problems; Sun \cite{sun-ww-nonexistence} proved nonexistence of solitary waves for the water wave problem, as did Amick and McLeod \cite{amick-mcleod} for the Amick--Toland model KdV equation.

\item
\label{ques: per amp expn upper bound?}
Is it possible to obtain a sharper upper bound on the periodic amplitude in the form of an exponential estimate like $|a_{\ep}| \le A_1e^{-A_2/\ep}$?
Here $A_1$ and $A_2$ should be $\ep$-independent constants.
Boyd \cite{boyd} actually terms nanopterons only those nonlocal solitary waves whose ripples are exponentially small in a relevant small parameter, not small beyond all algebraic orders as in \eqref{eqn: small BAAO}.
Thus our terminology is somewhat of an abuse of Boyd's.

The ripples in Beale's water wave nanopterons, as well as those of Amick and Toland, are ``only'' small beyond all algebraic orders, but Sun and Shen \cite{sun-shen-at-expn-small, sun-shen-ww-expn-small} proved that both families of ripples are indeed exponentially small.
Conversely, Champneys and Lord \cite{champneys-lord} gave numerical evidence for extending the Amick--Toland nanopterons both to larger-amplitude periodic solutions and for larger values of the small parameter than the theory indicates. 

\item
\label{ques: phase shifted per?}
The relative displacement problem is translation invariant in the sense that if the family $\{r_j\}_{j \in \Z}$ solves \eqref{eqn: rel disp eqns}, then so does $\{r_{j+d}\}_{j \in \Z}$ for any $d \in \Z$.
Thus, per \eqref{eqn: per profs FW}, taking 
\[
r_j(t) 
= \ep^2a_{\ep}\psf_j^{\ep}(\ep(j-\cep{t})+\ep{d})
\]
also solves the mass (or spring) dimer version of \eqref{eqn: rel disp eqns}.
Can we then replace $\psf_j^{\ep}$ in \eqref{eqn: intro dimer nano} with this shifted version?
More precisely, are there relative displacement solutions $r_j$ with
\[
\lim_{t \to \pm\infty} \big|r_j(t) -\ep^2\varsigma_w(\ep(j-\cep{t}))- \ep^2a_{\ep}\psf_j^{\ep}(\ep(j-ct \pm d))\big|
= 0
\]
exponentially fast, for some --- or all --- shifts $d$?
For symmetry concerns, we allow a sign change in the phase shift from the asymptotics at $-\infty$ to $+\infty$.

Amick and Toland \cite{amick-toland} achieved this for their model equation; Sun's results in \cite{sun-at-phase-shift} further elaborate on the consequent dependence of the ripple amplitude on the phase shift.
A related construction of phase-shifted nanopterons is proposed, though not executed, for the mass dimer in \cite[Sec.\@ 6.6]{faver-wright}.
Since the proofs in \cite{amick-toland, faver-wright, faver-spring-dimer} hinge in no small part on symmetries (evenness/oddness of various terms in the nanopteron ansatz), and since shifts destroy symmetries (e.g., if $f$ is even and periodic, then $S^df$ is typically not even), this will require some quite delicate attention.

\item
\label{ques: pos from rel disp?}
What, if anything, can be deduced about position traveling waves from these relative displacement results?
We can formally solve for position from relative displacement via identities like
\[
u_j = \sum_{k=-\infty}^{j-1} r_k
\quadword{or}
u_j = \begin{cases}
u_0+\medsum_{k=0}^{j-1} r_k, \ j \ge 1 \\
u_0-\medsum_{k=j}^{-1} r_k, \ j \le -1.
\end{cases}
\]
However, barring some decay conditions on $r_k$, there is no guarantee that the series will converge, while analysis of the finite sums requires at least knowledge of $u_0$ (or $u_{j_*}$ for some fixed $j_* \in \Z$).
For example, the exponential localization of the monatomic solitary waves provides such sufficient decay, see \cite[Prop.\@ 5.5]{friesecke-pego1}, while the Friesecke--Wattis results \cite{friesecke-wattis} were developed in position coordinates.
However, if $r_k$ is a nanopteron, then the periodic tails will almost surely prevent convergence of the series, at least for some time values.

Working at the level of the traveling wave profiles $p_k$ from \eqref{eqn: position tw ansatz} and $\varrho_k$ from \eqref{eqn: rel disp tw ansatz} is no more enlightening.
We might be tempted to try to solve for $p_k$ in terms of $\varrho_k$ via \eqref{eqn: tw profile rels}.
Doing so pointwise is simply fruitless.
Given the reliance of \cite{friesecke-pego1, faver-wright, faver-spring-dimer} on Fourier analysis, we might try to express \eqref{eqn: tw profile rels} using Fourier transforms.
There is no guarantee that all of $p_k$ and $\varrho_k$ have well-defined transforms, and even if they do, we uncover the relation
\[
\begin{pmatrix*}
\hat{\varrho}_1(k) \\
\hat{\varrho}_2(k)
\end{pmatrix*}
= \begin{bmatrix*}[r]
-1 &e^{ik} \\
e^{ik} &-1
\end{bmatrix*}
\begin{pmatrix*}
\hat{p}_1(k) \\
\hat{p}_2(k)
\end{pmatrix*}.
\]
The matrix above is, of course, singular for $k \in 2\pi\Z$, and so there is no clear way to invert the system to solve for $\hat{p}_1$ and $\hat{p}_2$.

All of this is to say that we are not optimistic about working backwards from relative displacement results to position results.
\end{enumerate}

\subsection{The main results}
Motivated by these four questions, we approach the long wave problem in mass and spring dimers from a substantially different perspective from that of \cite{faver-wright, faver-spring-dimer}: spatial dynamics.
We present and interpret our results below and then give an overview of the spatial dynamics method in Section \ref{sec: spatial dynamics method}.
We defer a comparison of these techniques to the prior Beale's ansatz method until Appendix \ref{app: Beale sd comp}, by which point we will have presented sufficient technical detail for both methods to make worthwhile remarks.

Here are our results for the mass dimer, which we prove in Section \ref{sec: proof of main thm mass dimer}.

\begin{theorem}[Mass dimer]\label{thm: main theorem mass dimer}
Suppose $w > 1$ in \eqref{eqn: dimer masses} and $\V_1 = \V_2$ in \eqref{eqn: dimer springs}.
Set 
\begin{equation}\label{eqn: qw}
q_w
:= \left(\frac{6w(1+w)}{w^2-w+1}\right)^{1/2}
\end{equation}
and let $q \in (0,q_w)$.
There are constants $\ep_*$, $\Alpha_0$, $\Alpha_1$, $\Alpha_{\infty}$, $\omega_w > 0$ such that if $0 < \ep < \ep_*$ and
\begin{equation}\label{eqn: cep intro}
\cep 
:= \left(\frac{1+w}{2w}-\ep^2\right)^{-1/2},
\end{equation}
then the following hold.

\begin{enumerate}[label={\bf(\roman*)}, ref={(\roman*)}]

\item\label{part: main theorem mass dimer position per}
\mainthmlabel{Position ``periodic + growing'' solutions}
For $0 < \ep < \ep_*$ and $0 \le \alpha \le \Alpha_1$, there are real analytic maps $\varphi_{j,\ep}^{\alpha}$ and $\Gsf_{\ep}^{\alpha}$ such that 
\begin{equation}\label{eqn: md pos per soln}
u_j(t)
= \alpha\ep^2\varphi_{j,\ep}^{\alpha}(\ep(j-\cep{t}))+ \alpha\ep(\alpha+\ep^2)\Gsf_{\ep}^{\alpha}(\ep(j-\cep{t}))
\end{equation}
solves the equations of motion \eqref{eqn: original equations of motion} in position coordinates.
These maps satisfy the estimate
\begin{equation}\label{eqn: varphi G est md}
\sup_{\substack{0 < \ep < \ep_* \\ 0 \le \alpha \le \Alpha_1}} \left(\sup_{X \in \R} |\varphi_{j,\ep}^{\alpha}(X)| + \sup_{X \in \R\setminus\{0\}} \frac{|\Gsf_{\ep}^{\alpha}(X)|}{|X|}\right)
< \infty.
\end{equation}
The maps $\varphi_{j,\ep}^{\alpha}$ are $2$-periodic in $j$ and $2\pi\ep(\omega_w+\ep\varpi_{\ep}^{\alpha})$-periodic in $X$, where
\[
\sup_{\substack{0 < \ep < \ep_* \\ 0 \le \alpha \le \Alpha_1}} |\varpi_{\ep}^{\alpha}|
< \infty.
\]

\item\label{part: main theorem mass dimer position}
\mainthmlabel{Position ``growing front'' solutions}
For $0 < \ep < \ep_*$ and $\Alpha_0\ep{e}^{-\Alpha_{\infty}/\ep} \le \alpha \le \Alpha_1$, there are real analytic maps $\eta_{j,\ep}^{\alpha}$ and real numbers $\Lsf_{\ep}^{\alpha}$ and $\theta_{\ep}^{\alpha}$ such that if
\begin{equation}\label{eqn: Xi-ep-alpha intro}
\Tsf_{\ep}^{\alpha}(X)
= X+\ep^2\theta_{\ep}^{\alpha}\tanh\left(\frac{q_wX}{2}\right)
\end{equation}
and
\begin{multline}\label{eqn: md pos F}
\Fsf_{j,\ep}^{\alpha}(X)
= \left[\left(\frac{6w(w^2-w+1)}{(1+w)^3}\right)^{1/2}+\ep\Lsf_{\ep}^{\alpha}\right]\tanh\left(\frac{q_wX}{2}\right)
+ \ep\eta_{j,\ep}^{\alpha}(X) 
+ \alpha\ep\varphi_{j,\ep}^{\alpha}(\Tsf_{\ep}^{\alpha}(X)) \\
+ \alpha(\alpha+\ep^2)\Gsf_{\ep}^{\alpha}(\Tsf_{\ep}^{\alpha}(X)),
\end{multline}
with $\varphi_{j,\ep}^{\alpha}$ and $\Gsf_{\ep}^{\alpha}$ defined in part \ref{part: main theorem mass dimer position per}, then
\begin{equation}\label{eqn: md pos tw final}
u_j(t)
= \ep\Fsf_{j,\ep}^{\alpha}(\ep(j-\cep{t}))
\end{equation}
solves \eqref{eqn: original equations of motion}.
The scaled profile $\ep\Fsf_{j,\ep}^{\alpha}(\ep\cdot)$ is sketched in Figure \ref{fig: front}.

There are constants $C_{\pm} > 0$ such that 
\[
C_- < \theta_{\ep}^{\alpha} < C_+
\quadword{and}
|\Lsf_{\ep}^{\alpha}| < C_+
\]
for all $\ep$ and $\alpha$, while $\eta_{j,\ep}^{\alpha}$ is $2$-periodic in $j$ and satisfies
\begin{equation}\label{eqn: md pos est}
\sup_{\substack{0 < \ep < \ep_* \\ \Alpha_0\ep{e}^{-\Alpha_{\infty}/\ep} \le \alpha \le \Alpha_1 \\ X \in \R}} e^{q|X|}|\eta_{j,\ep}^{\alpha}(X)|
< \infty.
\end{equation}

\item\label{part: main theorem mass dimer relative disp per}
\mainthmlabel{Relative displacement exact periodic solutions}
For $0 < \ep < \ep_*$ and $0 \le \alpha \le \Alpha_1$, there are real analytic maps $\tvarphi_{j,\ep}^{\alpha}$ such that
\begin{equation}\label{eqn: md rel disp per soln}
r_j(t)
= \alpha\ep^2\tvarphi_{j,\ep}^{\alpha}(\ep(j-\cep{t}))
\end{equation}
solves the equations of motion \eqref{eqn: rel disp eqns} in relative displacement coordinates.
These maps do not vanish identically and satisfy the estimates
\begin{equation}\label{eqn: md rel disp per nonvanishing}
\inf_{\substack{0 < \ep < \ep_* \\ 0 \le \alpha \le \Alpha_1}} \sup_{X \in \R} |\tvarphi_{j,\ep}^{\alpha}(X)| > 0
\quadword{and}
\sup_{\substack{0 < \ep < \ep_* \\ 0 \le \alpha \le \Alpha_1 \\ X \in \R}} |\tvarphi_{j,\ep}^{\alpha}(X)| < 0.
\end{equation}
They are periodic in $X$, with the same period as $\varphi_{j,\ep}^{\alpha}$, and $2$-periodic in $j$.

\item\label{part: main theorem mass dimer relative disp}
\mainthmlabel{Relative displacement nanopteron solutions}
For $0 < \ep < \ep_*$ and $\Alpha_0\ep{e}^{-\Alpha_{\infty}/\ep} \le \alpha \le \Alpha_1$, there are real analytic maps $\teta_{j,\ep}^{\alpha}$ such that if
\begin{equation}\label{eqn: md rel disp N}
\Nsf_{j,\ep}^{\alpha}(X) 
= \frac{3w}{1+w}\sech^2\left(\frac{q_wX}{2}\right)
+ \alpha\tvarphi_{j,\ep}^{\alpha}(\Tsf_{\ep}^{\alpha}(X))
+ \ep\teta_{j,\ep}^{\alpha}(X),
\end{equation}
with $\Tsf_{\ep}^{\alpha}$ defined in \eqref{eqn: Xi-ep-alpha intro} and $\tvarphi_{j,\ep}^{\alpha}$ in part \ref{part: main theorem mass dimer relative disp per}, then
\begin{equation}\label{eqn: md rel disp tw final}
r_j(t)
= \ep^2\Nsf_{j,\ep}^{\alpha}(\ep(j-\cep{t}))
\end{equation}
solves \eqref{eqn: rel disp eqns}.
The maps $\teta_{j,\ep}^{\alpha}$ are $2$-periodic in $j$ with
\[
\sup_{\substack{0 < \ep < \ep_* \\ \Alpha_0\ep{e}^{-\Alpha_{\infty}/\ep} \le \alpha \le \Alpha_1 \\ X \in \R}}
e^{q|X|}|\teta_{j,\ep}^{\alpha}(X)|
< \infty.
\]
\end{enumerate}
\end{theorem}

Next, we present our position and relative displacement solutions for spring dimers; we prove the following theorem in Section \ref{sec: proof of main thm spring dimer}.

\begin{theorem}[Spring dimer]\label{thm: main theorem spring dimer}
Suppose $w = 1$ in \eqref{eqn: dimer masses} and, in \eqref{eqn: dimer springs}, take $\kappa > 1$ and $\beta \in \R$ with $\kappa +\beta^3 \ne 0$.
Set 
\begin{equation}\label{eqn: qkappa}
q_{\kappa}
:= \left(\frac{6\kappa(1+\kappa)}{\kappa^2-\kappa+1}\right)^{1/2}
\end{equation}
and let $q \in (0,q_{\kappa})$.
There are constants $\ep_*$, $\Alpha_0$, $\Alpha_1$, $\Alpha_{\infty}$, $\omega_{\kappa} > 0$ such that if $0 < \ep < \ep_*$ and
\[
\cep 
:= \left(\frac{1+\kappa}{2\kappa}-\ep^2\right)^{-1/2},
\]
then the following hold.

\begin{enumerate}[label={\bf(\roman*)}, ref={(\roman*)}]

\item
\mainthmlabel{Position ``periodic + growing'' solutions}
For $0 < \ep < \ep_*$ and $0 \le \alpha \le \Alpha_1$, there are real analytic maps $\varphi_{j,\ep}^{\alpha}$ and $\Gsf_{\ep}^{\alpha}$ such that 
\begin{equation}\label{eqn: sd pos per soln}
u_j(t)
= \alpha\ep^2\varphi_{j,\ep}^{\alpha}(\ep(j-\cep{t}))+ \alpha\ep(\alpha+\ep^2)\Gsf_{\ep}^{\alpha}(\ep(j-\cep{t}))
\end{equation}
solves the equations of motion \eqref{eqn: original equations of motion} in position coordinates.

\item
\mainthmlabel{Position ``growing front'' solutions}
For $0 < \ep < \ep_*$ and $\Alpha_0\ep{e}^{-\Alpha_{\infty}/\ep} \le \alpha \le \Alpha_1$, there are real analytic maps $\eta_{j,\ep}^{\alpha}$ and real numbers $\Lsf_{\ep}^{\alpha}$ and $\theta_{\ep}^{\alpha}$ such that if 
\begin{multline}\label{eqn: sd pos F}
\Fsf_{j,\ep}^{\alpha}(X)
= \left(\frac{[6\kappa^3(1+\kappa)(\kappa^2-\kappa+1)]^{1/2}}{2(\beta+\kappa^3)} + \ep\Lsf_{\ep}^{\alpha}\right)\tanh\left(\frac{q_{\kappa}X}{2}\right)
+ \ep\eta_{j,\ep}^{\alpha}(X) 
+ \alpha\ep\varphi_{j,\ep}^{\alpha}(\Tsf_{\ep}^{\alpha}(\ep{x})) \\
+ \alpha(\alpha+\ep^2)\Gsf_{\ep}^{\alpha}(\Tsf_{\ep}^{\alpha}(\ep{x})).
\end{multline}
then
\begin{equation}\label{eqn: sd pos tw final}
u_j(t)
= \ep\Fsf_{j,\ep}^{\alpha}(\ep(j-\cep{t}))
\end{equation}
solves \eqref{eqn: original equations of motion}.

\item
\mainthmlabel{Relative displacement exact periodic solutions}
For $0 < \ep < \ep_*$ and $0 \le \alpha \le \Alpha_1$, there are real analytic maps $\tvarphi_{j,\ep}^{\alpha}$ such that
\begin{equation}\label{eqn: sd rel disp per soln}
r_j(t)
= \alpha\ep^2\tvarphi_{j,\ep}^{\alpha}(\ep(j-\cep{t}))
\end{equation}
solves the equations of motion \eqref{eqn: rel disp eqns} in relative displacement coordinates.

\item
\mainthmlabel{Relative displacement nanopteron-stegoton solutions}
For $0 < \ep < \ep_*$ and $\Alpha_0\ep{e}^{-\Alpha_{\infty}/\ep} \le \alpha \le \Alpha_1$, there are real analytic maps $\teta_{j,\ep}^{\alpha}$ such that if
\begin{equation}\label{eqn: sd rel disp N}
\Nsf_{j,\ep}^{\alpha}(X)
= \kappa^{[(-1)^j+1]/2}\frac{3\kappa^2}{\beta+\kappa^3}\sech^2\left(\frac{q_{\kappa}X}{2}\right)
+ \alpha\tvarphi_{j,\ep}^{\alpha}(\Tsf_{\ep}^{\alpha}(X))
+ \ep\teta_{j,\ep}^{\alpha}(X).
\end{equation}
then
\begin{equation}\label{eqn: sd rel disp tw final}
r_j(t)
= \ep^2\Nsf_{j,\ep}^{\alpha}(\ep(j-\cep{t}))
\end{equation}
solves \eqref{eqn: rel disp eqns}.

\end{enumerate}
The functions $\varphi_{j,\ep}^{\alpha}$, $\tvarphi_{j,\ep}^{\alpha}$, $\Gsf_{\ep}^{\alpha}$, $\eta_{j,\ep}^{\alpha}$, $\teta_{j,\ep}^{\alpha}$, and $\Tsf_{\ep}^{\alpha}$ and the scalars $\theta_{\ep}^{\alpha}$ and $\Lsf_{\ep}^{\alpha}$ have the same properties as their counterparts throughout Theorem \ref{thm: main theorem mass dimer}.
In particular, $\varphi_{j,\ep}^{\alpha}$ and $\tvarphi_{j,\ep}^{\alpha}$ are $2\pi\ep(\omega_{\kappa}+\ep\varpi_{\ep}^{\alpha})$-periodic in $X$.
\end{theorem}

In the following sections we analyze and interpret various consequences of these two theorems.
We address the extent to which our results answer the four motivating questions above, and we compare our conclusions to the prior dimer nanopteron results.

\subsubsection{Growing fronts for position and nanopterons for relative displacement}
We have answered Question \ref{ques: pos from rel disp?} on position traveling waves: the position profiles $\Fsf_{j,\ep}^{\alpha}$ in \eqref{eqn: md pos F} and \eqref{eqn: sd pos F} are ``growing fronts'' with ripples.
We sketch $\Fsf_{j,\ep}^{\alpha}$ in Figure \ref{fig: front}.
To leading order it is dominated by the $\tanh$-term and then at higher order exhibits small ripples via $\varphi_{j,\ep}^{\alpha}$ and (possibly) linear growth in $\Gsf_{\ep}^{\alpha}$, due to the estimate \eqref{eqn: md pos est}.
We discuss below in Remark \ref{rem: is it really periodic?} why we do not necessarily rule out this linear growth.

\begin{figure}
\[
\begin{tikzpicture}

\def\t{.25};
\def\T{1.5*\t};
\def\r{.5};
\def\c{1.1};
\def\ps{\t};

\draw[very thick,<->] (-7,0)--(7,0)node[right]{$X$};
\draw[very thick,<->] (0,-2)--(0,3)node[right]{$\ep\mathsf{F}_{j,\ep}^{\alpha}(\ep{X})$};

\draw[ultra thick,blue,<->] plot[domain=-7:7,smooth,samples=500] (\x,{\c*tanh(\t*\x)+\t^2*\x});

\def\lowerCenterX{4};
\def\lowerCenterY{\c*tanh(\t*\lowerCenterX)+\t^2*\lowerCenterX};

\def \lowerRad{.5};
\def \upperRad{1.35};

\def \upperCenterX{\lowerCenterX};
\def \upperCenterY{\lowerCenterY+2.5};

\def \lowerRightTanX{\lowerCenterX+\lowerRad};
\def \lowerRightTanY{\lowerCenterY};

\def \lowerLeftTanX{\lowerCenterX-\lowerRad};
\def \lowerLeftTanY{\lowerCenterY};

\def \leftAngle{pi/10};

\def \rightAngle{11*pi/10}

\draw[thick] (\lowerCenterX,{\lowerCenterY}) circle(\lowerRad);

\draw[densely dotted,thick] ({\lowerRightTanX},{\lowerRightTanY}) -- ({\upperCenterX+\upperRad*cos(-\leftAngle r)},{\upperCenterY+\upperRad*sin(-\leftAngle r)});
\draw[densely dotted,thick] ({\lowerLeftTanX},{\lowerLeftTanY}) -- ({\upperCenterX+\upperRad*cos(\rightAngle r)},{\upperCenterY+\upperRad*sin(\rightAngle  r)});

\begin{scope}

\clip ({\upperCenterX},{\upperCenterY}) circle(\upperRad);

\node[align=center,outer sep=0pt,inner sep=0pt] at ({\upperCenterX}, {\upperCenterY}){\begin{tikzpicture}

\draw[ultra thick,blue] plot[domain=-3:3,smooth,samples=500] (\x,{.2*\x+.2*sin(20*\x r)});

\end{tikzpicture}};

\end{scope}

\draw[thick] ({\upperCenterX},{\upperCenterY}) circle(\upperRad);

\end{tikzpicture}
\]

\caption{The slowly linearly growing front profile for position traveling waves.}
\label{fig: front}
\end{figure}
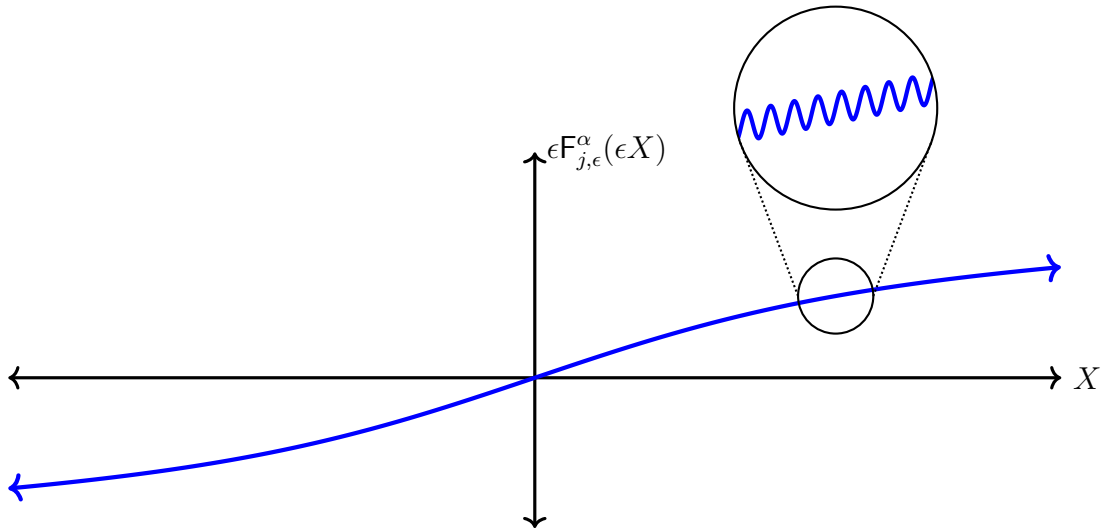

Also, we do not see quite the classical long wave scaling in the position solutions \eqref{eqn: md pos tw final} and \eqref{eqn: sd pos tw final}.
In each case, the leading order term from this scaling is only $\O(\ep)$, not $\O(\ep^2)$.

However, we do recover both the long wave scaling and nanopterons at the level of relative displacement.
The relative displacement profiles \eqref{eqn: md rel disp N} and \eqref{eqn: sd rel disp N} are each the superposition of an exponentially localized function and a periodic function, and the relative displacement solutions \eqref{eqn: md rel disp tw final} and \eqref{eqn: sd rel disp tw final} are $\O(\ep^2)$ at leading order.

\subsubsection{The periodic terms}
The periodic terms $\varphi_{j,\ep}^{\alpha}$ in the position front and $\tvarphi_{j,\ep}^{\alpha}$ in the relative displacement nanopteron exist for all $\alpha$ in the interval $[0,\Alpha_1]$, even though the fronts and nanopterons are only defined for $\alpha$ in the range $[\Alpha_0\ep{e}^{-\Alpha_{\infty}/\ep},\Alpha_1]$.
Inspired by \cite[Fig.\@ 7.2]{Lombardi}, we sketch in Figure \ref{fig-ranges} the different parameter ranges in the $(\ep,\alpha)$-plane for which the periodic terms and the fronts/nanopterons exist.

We remark that the auxiliary position solutions \eqref{eqn: md pos per soln} and \eqref{eqn: sd pos per soln}, which include the periodic terms $\varphi_{j,\ep}^{\alpha}$, need not be periodic themselves, due to the presence of the $\Gsf_{\ep}^{\alpha}$ term.
This is a marked contrast to the exact periodic solutions \eqref{eqn: md rel disp per soln} and \eqref{eqn: sd rel disp per soln} for relative displacement.

The periodic terms in the relative displacement nanopteron profiles are definitely not identically zero, due to the estimates \eqref{eqn: md rel disp per nonvanishing}.
While the amplitude $\alpha$ can be taken to be exponentially small in $\ep$, it can also extend to a small $\O(1)$ threshold independent of $\ep$.
This control over $\alpha$ addresses Questions \ref{ques: per amp is 0?} and \ref{ques: per amp expn upper bound?}.
However, while none of the relative displacement solutions that we construct are solitary waves, we did not prove the {\it{nonexistence}} of supersonic solitary traveling waves in dimers; see Section \ref{sec: future directions} for further discussion of this issue.

Finally, we are not free to pick the phase shift $\theta_{\ep}^{\alpha}$ to be an arbitrary number; instead, as we discuss in Appendix \ref{app: Lombardi}, we must accept the very particular value $\ep^2\theta_{\ep}^{\alpha}$.
Nonetheless, since this is $\O(\ep^2)$, our phase shift is not particularly large.
Because our results do not permit an arbitrary phase shift, we feel that the modifications to Beale's method for periodic phase shifts, as discussed in \cite{faver-wright, amick-toland}, will offer a more satisfactory introduction of shifts into the solutions, and a fuller resolution of Question \ref{ques: phase shifted per?} remains to be accomplished.

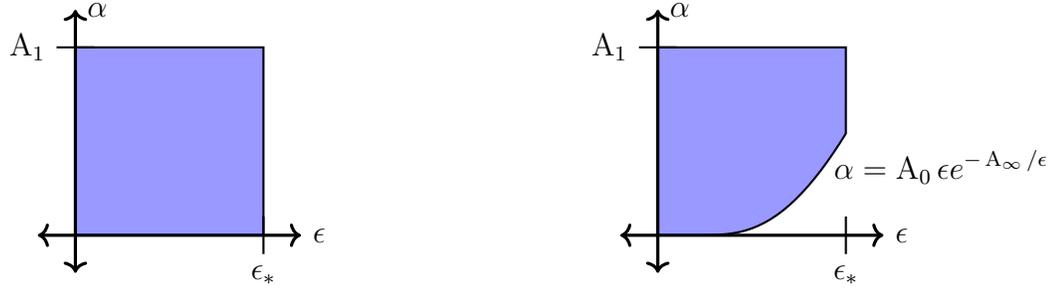
\begin{figure}

\begin{subfigure}[t]{.35\textwidth}

\[
\begin{tikzpicture}

\draw[very thick,<->] (-.5,0)--(3,0)node[right]{$\ep$};
\draw[very thick,<->] (0,-.5)--(0,3)node[right]{$\alpha$};

\draw[fill=blue, fill opacity=.4, thick] (0,2.5)--(2.5,2.5)--(2.5,0)--(0,0)--cycle;

\draw[thick] (2.5,-.25)node[below]{$\ep_*$}--(2.5,.25);
\draw[thick] (-.25,2.5)node[left]{$\Alpha_1$}--(.25,2.5);

\end{tikzpicture}
\]

\caption{Domain of existence for (linearly growing) periodic solutions.}
\end{subfigure}
\qquad
\qquad
\qquad
\begin{subfigure}[t]{.35\textwidth}

\[
\begin{tikzpicture}

\draw[very thick,<->] (-.5,0)--(3,0)node[right]{$\ep$};
\draw[very thick,<->] (0,-.5)--(0,3)node[right]{$\alpha$};

\draw[fill=blue, fill opacity=.4, thick] plot[domain=.5:2.5] (\x,{4*\x*exp(-5/\x)})--(2.5,2.5)--(0,2.5)--(0,0)--(.5,0);

\node[right] at (2.2,{4*2.2*exp(-5/2.2)}){$\alpha = \Alpha_0\ep{e}^{-\Alpha_{\infty}/\ep}$};

\draw[thick] (2.5,-.25)node[below]{$\ep_*$}--(2.5,.25);
\draw[thick] (-.25,2.5)node[left]{$\Alpha_1$}--(.25,2.5);

\end{tikzpicture}
\]

\caption{Domain of existence for (linearly growing) fronts and nanopterons}

\end{subfigure}

\caption{Ranges of the long wave parameter $\ep$ and the amplitude parameter $\alpha$ (based on \cite[Fig.\@ 7.2]{Lombardi})}
\label{fig-ranges}
\end{figure}

\subsubsection{Comparisons with prior dimer nanopterons}\label{sec: faver wright comparisons}
The nanopterons for relative displacement in mass dimers that Faver and Wright constructed via Beale's method \cite[Cor.\@ 6.3]{faver-wright} have the form
\[
r_j(t)
= \nu^2\varsigma_w(\nu(j-C_{\nu}{t})) + \O(\nu^3),
\]
pointwise in $j$ and $t$, where 
\begin{equation}\label{eqn: varsigma-w}
\varsigma_w(X)
:= \frac{3}{2}\left(\frac{1+w}{2w}\right)\sech^2\left(\left(\frac{1+w}{2w}\right)\frac{q_wX}{2}\right)
\end{equation}
and 
\begin{equation}\label{eqn: C-nu}
C_{\nu}
:= \left(\frac{2w}{1+w} + \nu^2\right)^{1/2}.
\end{equation}
The factor $q_w$ was defined in \eqref{eqn: qw}, while $\varsigma_w$ appeared, albeit obliquely, as the leading order term in \eqref{eqn: intro dimer nano}.
We are now writing $\nu$, not $\ep$, for the long wave small parameter here.

If we take 
\begin{equation}\label{eqn: cs rescaling}
\nu
= \frac{2w}{1+w}\ep,
\end{equation}
then 
\[
\nu^2\varsigma_w(\nu{X})
= \frac{3w}{1+w}\ep^2\sech^2\left(\frac{q_w\ep{X}}{2}\right).
\]
A comparison with \eqref{eqn: md rel disp N} and \eqref{eqn: md rel disp tw final} shows that this is exactly the leading order $\sech^2$-term of our nanopterons.
Likewise, with $\cep$ defined in \eqref{eqn: cep intro} and $C_{\nu}$ in \eqref{eqn: C-nu}, we find
\[
\cep^2-C_{\nu}^2
= \left(\frac{1+w}{2w}-\ep^2\right)^{-1}-\left(\frac{2w}{1+w}+\left(\frac{2w}{1+w}\right)^2\ep^2\right)
= \O(\ep^4),
\]
and so the rescaled wave speed from Beale's method agrees with ours to $\O(\ep^4)$.
Finally, we remark that the scaling factor $2w/(1+w)$ in \eqref{eqn: cs rescaling} is exactly the speed of sound for the mass dimer, see \eqref{eqn: cs defn}.

A similar rescaling of the small parameter shows the leading order agreement of our spring dimer nanopterons with those constructed by Faver \cite[Cor.\@ 4.4]{faver-spring-dimer}.
In that paper, nanopterons have the form
\[
r_j(t)
= \nu^2\kappa^{[(-1)^j+1]/2}\varsigma_{\kappa,\beta}(\nu(j-C_{\nu}{t})) + \O(\nu^3),
\]
where now
\[
\varsigma_{\kappa,\beta}(X)
:= \frac{3(1+\kappa)^2}{4(\beta+\kappa^3)}\sech^2\left(\left(\frac{1+\kappa}{2\kappa}\right)\frac{q_{\kappa}X}{2}\right),
\]
\[
C_{\nu}
:= \left(\frac{2\kappa}{1+\kappa}+\nu^2\right)^{1/2},
\]
and $q_{\kappa}$ was defined in \eqref{eqn: qkappa}.
It is, unsurprisingly, necessary to assume $\beta+\kappa^3\ne0$ here, too.
Rescaling 
\[
\nu
= \frac{2\kappa}{1+\kappa}\ep,
\]
we find
\[
\kappa^{[(-1)^j+1]/2}\varsigma_{\kappa,\beta}(X)
= \kappa^{[(-1)^j+1]/2}\frac{3\kappa^2}{\beta+\kappa^3}\ep^2\sech^2\left(\frac{q_{\kappa}\ep{X}}{2}\right),
\]
and this is our leading order $\sech^2$-term in \eqref{eqn: sd rel disp N}.
The scaling factor $2\kappa/(1+\kappa)$ is now the speed of sound for the spring dimer.

\subsubsection{Stegotons in the spring dimer}
At first glance Theorems \ref{thm: main theorem mass dimer} and \ref{thm: main theorem spring dimer} may look identical, and they essentially are, up to one key difference.
The spring dimer relative displacement profile, given in \eqref{eqn: sd rel disp N}, contains a factor on its leading-order localized term that is either $\kappa$ or 1, alternating with the parity of lattice site $j$.
Since we assume $\kappa > 1$ for the spring dimer, this means the leading order term is either slightly larger or smaller depending on $j$.
This phenomenon is observed in \cite{faver-spring-dimer}, which uses the parlance ``stegoton'' of LeVeque and Yong \cite{leveque1, leveque2} to describe this alternating behavior in relative displacement; our theorem below, like the result in \cite{faver-spring-dimer}, analytically confirms the stegoton's existence.

\subsubsection{Some physical interpretations}
Here is how we physically interpret Theorem \ref{thm: main theorem mass dimer}; the same will hold for Theorem \ref{thm: main theorem spring dimer}.
The ``growing front'' results for position coordinates mean that, over time, a fixed particle drifts further and further away from its equilibrium position.
If the term $\Gsf_{\ep}^{\alpha}$ is indeed bounded and not truly linearly growing, then that drift will asymptote to some constant length with small ``wiggles'' due to the ripples $\varphi_{j,\ep}^{\alpha}$.

At first glance the nanopteron results for relative displacement many seem strange; after all, while the $\sech^2$-term dominates the nanopteron, as this term decays the periodic ripples could induce negative values in relative displacement (especially if we do not select the amplitude $\alpha$ to be exponentially small).
That, however, would allow pairs of particles to ``cross'' or ``move through'' each other, which is physically bizarre.
(To be fair, with the linearly growing front we are allowing the possibility that the springs become arbitrarily long, which is also physically bizarre.)
But recall our first nondimensionalization of the problem in \eqref{eqn: equilibrium length cov} to take the spring equilibrium lengths to be 0.
And so, with $y_j$ as the original position coordinates from \eqref{eqn: original equations of motion with ell-j}, we find from \eqref{eqn: rel disp y and u} that
\[
y_{j+1}-y_j
= r_j + \ell_j,
\]
where $r_j$ is the nanopteron from \eqref{eqn: md rel disp tw final}.
By restricting $\ep$ and $\alpha$, we can make this $r_j$ as small as we like uniformly in time, and so over very long times, the distance between pairs of particles just settles down to oscillate around the equilibrium length of the spring connecting the particles.

\subsection{The spatial dynamics method}\label{sec: spatial dynamics method}
The progenitors of this ``spatial dynamics'' method for lattice problems were Iooss and Kirchg\"{a}ssner \cite{iooss-kirchgassner}, who applied it to the position traveling wave problem for Klein--Gordon lattices.
Such lattices consist, roughly, of a monatomic FPUT lattice with linear nearest-neighbor spring forces and an additional ``on-site'' potential applied to each particle.
The Iooss--Kirchg\"{a}ssner methods have been adapted to a host of subsequent lattice problems, including \cite{calleja-sire, iooss-fpu, james-sire, sire, sire-james, iooss-pelinovsky, iooss-james, hilder-derijk-schneider}.
Accessible surveys of their techniques appear in \cite{james-sire-review} and \cite[Sec.\@ 5.2.3]{haragus-iooss}.
In broad strokes, they make a special change of variables to convert their second-order, nonlocal traveling wave equation into a first-order ordinary differential equation on a particular infinite-dimensional Banach space, with the ``spatial'' variable of the traveling wave profile now taking the role of the ``time'' variable in this differential equation, hence the term ``spatial dynamics'' for this overall approach.
A key idea is that instead of shifting with $(S^{\pm1}p)(x) = p(x\pm1)$ as in \eqref{eqn: position tw eqns}, one replaces the shifts with ``continuous variables'' by putting $P(x,v) := p(x+v)$, so that the shift operator $S^{\pm1}p$ becomes the ``evaluation'' operator $P(x,\cdot) \mapsto P(x,\pm1)$.
We perform this change of variables for our problem \eqref{eqn: position tw eqns} in Section \ref{sec: IK-cov}.
Prior to studying the Klein--Gordon lattice, Iooss and Kirchg\"{a}ssner had earlier applied the spatial dynamics method to water wave problems \cite{iooss-kirchgassner-ww}; see the introduction to that paper for a detailed historical overview of the method in the water wave context, and also the articles \cite{dias-iooss, iooss-peyresq, Lombardi-iooss-gravity}.
The myriad underlying techniques date back to Kirchg\"{a}ssner \cite{kirchgassner-wave-solutions} and Mielke \cite{mielke-reduction}.

For future reference, let us write the first-order equation that Iooss and Kirchg\"{a}ssner, and successors, would study as 
\begin{equation}\label{eqn: sd intro}
U'(x)
= \L_0U(x) + \mu\L_1(\mu)U(x) + \nl(U(x),\mu),
\end{equation}
where $\mu$ is the appropriate expression of the long wave parameter in Iooss--Kirchg\"{a}ssner coordinates.
The operators $\L_0$ and $\L_1(\mu)$ are linear and $\nl$ is nonlinear.
A deft combination of spectral theory, normal form analysis, and center manifold theory can then capture, in different regimes and to different degrees of detail, small solutions to \eqref{eqn: sd intro} --- in particular, nanopterons. 
The existence of nanopterons is typically established by invoking the comprehensive results of Lombardi's magisterial monograph \cite{Lombardi}, which, together with the articles \cite{Lombardi-peyresq,Lombardi-orbits97, Lombardi-orbits96, Lombardi-nonpersistence}, provides a far-reaching set of hypotheses under which a problem like \eqref{eqn: sd intro} has nanopteron solutions.
These hypotheses are valid in general Banach spaces and not specific to lattice-type problems.
We give a careful overview of Lombardi's method in Appendix \ref{app: Lombardi} to put our language and methods in better dialogue with his.
The interested reader may wish at this point to consult our phrasing of Lombardi's nanopteron existence results in Theorem \ref{thm: Lombardi}.

We were specifically inspired to deploy the spatial dynamics method on FPUT lattices by Venney and Zimmer \cite{venney-zimmer}, who studied monatomic lattices with nearest {\it{and}} next-to-nearest neighbor couplings --- a spring dimer, after a fashion, though of a rather different genus from ours.
Venney and Zimmer, in turn, adapted techniques that Iooss and James \cite{iooss-james} used to construct traveling breathers in monatomic lattices.
Neither Venney and Zimmer, nor Iooss and James, obtained quite the same problem that Iooss and Kirchg\"{a}ssner encountered, and neither the Venney--Zimmer nor the Iooss--James problem was directly amenable to Lombardi's techniques.
The chief obstacle in \cite{venney-zimmer, iooss-james} to a direct application of Lombardi's results was that the center spectrum of their versions of the operator $\L_0$ from \eqref{eqn: sd intro} consisted of 0 (with algebraic multiplicity 4 and geometric multiplicity 1) and $\pm{i}\omega$ (both simple eigenvalues) for a certain ``critical frequency'' $\omega > 0$.
This is the same situation as in our FPUT dimers.
However, Lombardi calls for a {\it{double}} eigenvalue at 0.
By exploiting translation invariance and a conserved first integral inherent to their respective lattice problems, both pairs of authors were able to change variables and work on a ``reduced'' version of \eqref{eqn: sd intro} for which the linearization's eigenvalue at 0 does have multiplicity 2.

Building on the work of our spatial dynamics predecessors, here are the chief contributions of this paper to the spatial dynamics method for lattices, in addition to the FPUT-specific results in Theorems \ref{thm: main theorem mass dimer} and \ref{thm: main theorem spring dimer}.

\begin{enumerate}[label={\bf\arabic*.}]

\item
We state seven exact hypotheses under which we can solve an equation like \eqref{eqn: sd intro}, posed in an arbitrary Banach space, when $\L_0$ has an eigenvalue of algebraic multiplicity 4 at 0.
We enumerate these hypotheses in Section \ref{sec: hypos} and state their consequence for \eqref{eqn: sd intro} as Theorem \ref{thm: main abstract}.
The language of our hypotheses is inspired by Lombardi's assumptions in \cite[Ch.\@ 8]{Lombardi} and intended to facilitate the invocation of his results on our reduced problem, for which the linearization at $\mu=0$ has a double eigenvalue at 0.
Our arrangement of the hypotheses also arose by observing the common threads among the problems of \cite{iooss-james, venney-zimmer} and our FPUT dimers and distilling them into more general terminology than previously employed (particularly in the cases of our Hypotheses \ref{hypo: nondegen} and \ref{hypo: opt reg}).
Working at the level of these abstract hypotheses allows us to see precisely what the mass and spring dimers have in common, where they differ, and what is true for a general dimer, all without being obscured by the complex dependencies on the four parameters $c$, $w$, $\kappa$, and $\beta$.
We check the hypotheses for the FPUT dimers in Section \ref{sec: tw prob redux} and then analyze the consequences of the abstract Theorem \ref{thm: main abstract} for dimer position traveling waves in Section \ref{sec: proofs of main lattice theorems}.

\item
We provide a thorough treatment of the reduced problem that converts \eqref{eqn: sd intro} into the algebraic multiplicity 2 situation.
For reference, write this reduced problem as
\begin{equation}\label{eqn: sd intro reduced}
W' 
= \tL_0W + \mu\tL_1(\mu)W + \tnl(W,\mu)
\end{equation}
for some linear operators $\tL_0$ and $\tL_1(\mu)$ and a nonlinear operator $\tnl$.
In Section \ref{sec: 02+iomega reduction}, we show how any (small) solution of \eqref{eqn: sd intro} also solves \eqref{eqn: sd intro reduced}, and, importantly, we explain how any solution to the reduced problem really does yield a solution to the full original problem \eqref{eqn: sd intro}.
This reduced problem is not the finite-dimensional system that would arise from the center manifold theorem, which we do not use, so as to preserve the analyticity of various maps, which in turn permits the application of Lombardi's analytic-dependent hypotheses.
The operators $\tL_0$, $\tL_1(\mu)$, and $\tnl$ in \eqref{eqn: sd intro reduced} are really $\O(1)$ perturbations of the original operators $\L_0$, $\L_1(\mu)$, and $\nl$ from \eqref{eqn: sd intro}, and as such we cannot capture the properties of these new operators using ordinary perturbation theory.
With an eye toward our eventual nanopteron construction in relative displacement coordinates, we also prove a host of exact estimates on various terms and operators in the reduced problem.
These are all details that are not treated with particular precision in \cite{iooss-james, venney-zimmer}.
They are necessary not only for logical completeness but also to allow us to compare the leading order behavior of our spatial dynamics nanopterons against those previously found by Beale's method, as we discussed in Section \ref{sec: faver wright comparisons}.

\item
In Appendix \ref{app: Lombardi} we generalize Lombardi's nanopteron program from \cite{Lombardi}.
Since Lombardi originally aimed his treatment at the water wave problem, we have adapted some of his hypotheses to encompass broader problems.
We feel that our gloss of Lombardi's intricate methods offers an efficient, accessible outline that will be useful not just to work in concert with our hypotheses but for future nanopteron constructions in settings not limited to water waves or lattices.
\end{enumerate}

\subsection{Notation}
For clarity, we remark on some aspects of our (mostly standard) notation.

\begin{enumerate}[label=$\bullet$]

\item
We abbreviate $\R_+ := (0,\infty)$.

\item
If $\X$ is a vector space, then $\ind_{\X}$ is the identity operator on $\X$.

\item
If $\X$ and $\Y$ are normed spaces, then $\b(\X,\Y)$ is the space of bounded linear operators from $\X$ to $\Y$.
We write $\b(\X) := \b(\X,\X)$.
We denote the dual space of $\X$ by $\X^* = \b(\X,\R)$.

\item
If $\X$ and $\Y$ are vector spaces and $\Zcal\subseteq \X$ is a subspace of $\X$, and if $\A \colon \X \to \Y$ is a linear operator, then we write $\restr{\A}{\Zcal}$ for the restriction of $\A$ to $\Zcal$.

\item
If $\X$ and $\Y$ are normed spaces, $f \colon \X \to \Y$ is differentiable, and $U$, $\grave{U} \in \X$, then we denote the derivative of $f$ at $U$ evaluated at $\grave{U}$ by $Df(U)\grave{U}$.
Likewise, the second derivative of $f$ at $U$ evaluated at $(\grave{U},\breve{U})$ is $D^2f(U)[\grave{U},\breve{U}]$.
Occasionally we will write $D_Uf$ instead of $Df$ if $f$ is a function of more than just $U$.
For a function $f = f(x)$ of a real variable $x$ (complex variable $z$), we sometimes write $\partial_x[f] = f'$ ($\partial_z[f] = f'$).

\item
If $\X$ is a normed space, then for a function $f \colon \R \to \X$ we put
\[
\Lip(f)
:= \sup_{\substack{x,\grave{x} \in \R \\ x \ne \grave{x}}} \frac{\norm{f(x)-f(\grave{x})}_{\X}}{|x-\grave{x}|}.
\]

\item
If a quantity $f$ depends on the variable $x$ and some additional ``parameters,'' say, $\mu$ and $\nu$, we will often write $f(x;\mu,\nu)$; in general, data ``after the semi-colon'' is not part of the independent variable.
If we are considering the derivative of $f$ just with respect to $x$, we will write $Df(x;\mu,\nu)$.

\item
We review a number of conventions and notations for spectral theory in Appendix \ref{app: spectral theory}.
\end{enumerate}

\subsection{Acknowledgments}
We acknowledge support from the Netherlands Organization for Scientific Research (NWO) (grant 639.032.612).
\section{The Abstract Problem}\label{sec: abstract}

\subsection{The hypotheses}\label{sec: hypos}
Let $\D$ and $\X$ be Banach spaces with $\D$ continuously embedded in $\X$.
We will denote the norms of these spaces by $\norm{\cdot}_{\D}$ and $\norm{\cdot}_{\X}$, respectively.
For some $\mu_0 > 0$, we consider a family of maps $\F \colon \D \times [0,\mu_0] \to \X$ of the form \eqref{eqn: F-mu} given below, and we will construct solutions to the differential equation
\begin{equation}\label{eqn: abstract ode}
U'(x)
= \F(U(x),\mu)
\end{equation}
that are the superposition of a front, an exponentially localized term, a small-amplitude periodic term, and a (possibly) linearly growing term. 
To be clear, a solution to \eqref{eqn: abstract ode} is a map $U \in \Cal(\R,\D) \cap \Cal^1(\R,\X)$ that satisfies \eqref{eqn: abstract ode} for each $x \in \R$.
We are not making any assumptions on the well-posedness of \eqref{eqn: abstract ode}, and in particular we do not treat it as an initial value problem.

We assume a number of hypotheses on the map $\F$ and its constituent terms.
These hypotheses will permit us to convert \eqref{eqn: F-mu} to the ``reduced'' problem \eqref{eqn: fully reduced}, which we will then solve using Lombardi's nanopteron theory.
We will then undo our reduction procedure and recover solutions to the original problem \eqref{eqn: F-mu}.
We state our main result on solutions to \eqref{eqn: F-mu} in Theorem \ref{thm: main abstract}.

\begin{hypothesis}[structural properties of $\F$]\label{hypo: F structure}
The map $\F$ has the form
\begin{equation}\label{eqn: F-mu}
\F(U,\mu)
:= \L_0U + \mu\L_1(\mu)U + \nl(U,\mu),
\end{equation}
where $\L_0 \in \b(\D,\X)$ and $\L_1 \colon [0,\mu_0] \to \b(\D,\X)$ and $\nl \colon \X \times [0,\mu_0] \to \X$ are analytic.
The operator $\nl$ is quadratic in the sense that 
\begin{equation}\label{eqn: nl hypo}
\nl(0,\mu) = 0
\quadword{and}
D_U\nl(0,\mu) = 0
\end{equation}
for all $\mu$.
\end{hypothesis}

This hypothesis, along with the analyticity of $\nl$, implies that we have the expansion
\begin{equation}\label{eqn: nl expn}
\nl(U,\mu)
= \nl_0(U,U) + \nl_1(U,\mu)
\end{equation}
for some bounded bilinear operator $\nl_0 \colon \X \times \X \to \X$ and analytic map $\nl_1 \colon \X \times [0,\mu_0] \to \X$ satisfying
\begin{equation}\label{eqn: nl ests}
 \norm{\nl_1(U,\mu)}_{\X} 
\le C\mu\norm{U}_{\X}^2 + C\norm{U}_{\X}^3.
\end{equation}
for some $C > 0$ when $\norm{U}_{\X} \le 1$.
Equivalently, if $\nl$ has the form \eqref{eqn: nl expn}, then of course $\nl$ satisfies \eqref{eqn: nl hypo}.

\begin{hypothesis}[reversible symmetry]\label{hypo: symmetry}
There exists an operator $\S \in \b(\X)$ such that $\S^2 = \ind_{\X}$, $\norm{\S}_{\b(\X)} = 1$, and
\[
\S\L_0 = -\L_0\S,
\qquad
\S\L_1(\mu) = -\L_1(\mu)\S,
\quadword{and}
\nl(\S{U},\mu) = -\S\nl(U,\mu)
\]
for all $U \in \D$ and $0 \le \mu \le \mu_0$.
\end{hypothesis}

\begin{definition}\label{defn: reversible}
A map $f \colon \R \to \X$ is \defn{$\S$-reversible} if $\S{f}(x) = f(-x)$ for all $x \in \R$ and \defn{$\S$-antireversible} if $\S{f}(x) = -f(-x)$ for all $x$.
\end{definition}

The symmetry $\S$ does not play a particularly large role in our development in this section, beyond serving as a condition to verify for the invocation of Lombardi's nanopteron theorem.
In Lombardi's nanopteron method, however, the symmetry is a key feature to remove certain redundancies and prevent overdetermined systems, and we point out some of its specific uses at various times in Appendix \ref{app: Lombardi}.

Our next two hypotheses control the center spectrum of $\L_0$, considered now as an operator in $\X$ with domain $\D$, and its spectral projection.
See Appendix \ref{app: spectral theory} for our spectral theoretic conventions.

\begin{hypothesis}[center spectrum of $\L_0$]\label{hypo: spectrum}
There exists $\omega >0$ such that $\sigma(\L_0) \cap i\R = \{0,\pm{i}\omega\}$.
All three points are eigenvalues; 0 has algebraic multiplicity 4 and $\pm{i}\omega$ each has algebraic multiplicity 1.
Each eigenvalue is geometrically simple, and there exists $\lambda_0 > 0$ such that if $\lambda \in \sigma(\L_0)\setminus{i}\R$, then $|\re(\lambda)| \ge \lambda_0$.
\end{hypothesis}

\begin{hypothesis}[spectral projection for $\L_0$]\label{hypo: eigenproj}
Let $\Pi_0$ be the spectral projection for $\L_0$ corresponding to 0.
Write
\begin{equation}\label{eqn: abstract 0 proj}
\Pi_0U
= \sum_{k=0}^3 \chi_k^*[U]\chi_k,
\end{equation}
where $(\chi_0,\chi_1,\chi_2,\chi_3) \in \X^4$ is a Jordan chain for $\L_0$ associated with 0, and $\chi_k^* \in \X^*$ are functionals.
The vectors $\chi_k$ and the functionals $\chi_k^*$ have the following additional properties.

\begin{enumerate}[label={\bf(\roman*)}, ref={(\roman*)}]

\item\label{part: L1 chi0}
$\L_1(\mu)\chi_0 = 0$ for all $0 \le \mu \le \mu_0$.

\item\label{part: Q chi0}
(translation invariance) $\nl(U+\gamma\chi_0,\mu) = \nl(U,\mu)$ for all $U \in \D$, $\gamma \in \R$, and $0 \le \mu \le \mu_0$.

\item\label{part: S and chik} 
$\S\chi_k = (-1)^{k+1}\chi_k$.
\end{enumerate}

\end{hypothesis}

\begin{remark}\label{rem: 04+iomega}
Following \cite[Rem.\@ 3.1.12, 3.1.16]{Lombardi}, we say that the nonlinear operator $\F$ has a ``$0^{4-}i\omega$ bifurcation at $(0,0)$,'' since

\begin{enumerate}[label={\bf(\roman*)}]

\item
$\F(0,\mu) = 0$ for all $\mu$;

\item
The center spectrum of $\L_0 = D_U\F(0,0)$ is $\{0,\pm{i}\omega\}$, where 0 is an eigenvalue of algebraic multiplicity 4 and $\pm{i}\omega$ are eigenvalues of algebraic multiplicity 1;

\item
$\S\chi_0 = -\chi_0$, where $\chi_0$ is an eigenvector of $\L_0$ corresponding to 0.
\end{enumerate}
Our reduction procedure in Section \ref{sec: 02+iomega reduction}, modeled on that of \cite{venney-zimmer, iooss-james}, will allow us to consider a problem whose linearization has the more well-known ``$0^{2+}i\omega$'' bifurcation, in which 0 is an eigenvalue of algebraic multiplicity only 2.  
See \cite[Sec.\@ 4.3]{haragus-iooss} and \cite[Ex.\@ 3.2.9]{Lombardi} for more details on this bifurcation in both finite and infinite dimensions.
\end{remark}

Our antepenultimate hypothesis introduces a map $\J_{\mu}$ that serves as first integral for the problem \eqref{eqn: abstract ode}; in particular, per part \ref{part: Djmu on F} below, if $U$ solves \eqref{eqn: abstract ode}, then $x \mapsto J_{\mu}(U(x))$ is constant.
The other conditions of this hypothesis require $\J_{\mu}$ to interact with the symmetry $\S$, the eigenvectors $\chi_k$, and the eigenfunctional $\chi_3^*$ in suitable ways.

\begin{hypothesis}[first integral]\label{hypo: first int}
For each $0 \le \mu \le \mu_0$, there is a map $\J_{\mu} \colon \X \to \R$ with the following properties.

\begin{enumerate}[label={\bf(\roman*)},ref={(\roman*)}]

\item\label{part: first int diff}
The map $\X \times [0,\mu_0] \to \R \colon (U,\mu) \mapsto \J_{\mu}(U)$ is analytic.

\item\label{part: Djmu on F}
$D\J_{\mu}(U)\F(U,\mu) = 0$ for all $U \in \X$, where $\F$ is defined in \eqref{eqn: F-mu}.

\item\label{part: Jmu and S}
$\J_{\mu}(\S{U}) = \J_{\mu}(U)$ for all $U \in \X$.

\item\label{part: Jstar}
$\J_0(0) = 0$, and there is a functional $\Jscr_*(\mu) \in \X^*$ such that 
\begin{equation}\label{eqn: Jmu J0 quad}
\J_{\mu}(U)-\J_0(U)
= \mu\Jscr_*(\mu)U
\end{equation}
for all $U \in \D$.
The mapping $[0,\mu_0] \to \X^* \colon \mu \mapsto \Jscr_*(\mu)$ is analytic.

\item\label{part: Jmu on chi0}
$\J_{\mu}(U + \gamma\chi_0) = \J_{\mu}(U)$ for all $U \in \D$, $\gamma \in \R$.

\item\label{part: chi3 J0}
$\chi_3^*[U] = D\J_0(0)U$.
\end{enumerate}
\end{hypothesis}

The following two quantities are essential values for the precise expressions of our solutions to \eqref{eqn: abstract ode}, and they control the leading order linear and quadratic behavior of the eventual reduction of \eqref{eqn: abstract ode}.

\begin{hypothesis}[nondegeneracies]\label{hypo: nondegen}
Define
\begin{equation}\label{eqn: frak defn}
\Lfrak_0 
:= \chi_2^*[\L_1(0)\chi_1] - \Jscr_*(0)\chi_1
\quadword{and}
\Qfrak_0 
:= \chi_2^*\big[\nl_0(\chi_1,\chi_1)\big] - \frac{D^2\J_0(0)[\chi_1,\chi_1]}{2}.
\end{equation}
Then
\begin{equation}\label{eqn: nondegen}
\Lfrak_0 > 0
\quadword{and}
\Qfrak_0 \ne 0.
\end{equation}
\end{hypothesis}

Our final hypothesis concerns the solutions of certain affine first-order differential equations involving $\L_0$, and our phrasing of this hypothesis is arguably our greatest departure from the abstract structure implicit in \cite{venney-zimmer, iooss-james}.
To state this hypothesis, we need the machinery of Appendix \ref{app: opt reg}, specifically Definition \ref{defn: opt reg} for localized optimal regularity, Definition \ref{defn: per opt reg} for periodic optimal regularity, and Definition \ref{defn: subopt reg} for suboptimal regularity.
We do not elaborate on these concepts here, and they do not play a major role in the reduction process of Section \ref{sec: 02+iomega reduction}; however, they are critical to invoking Lombardi's results for our final problem.

Let $\Pi$ be the spectral projection for $\L_0$ corresponding to $\{0,\pm{i}\omega\}$; by the remarks at the conclusion of the proof of Theorem 6.17 in \cite{kato}, we may write
\begin{equation}\label{eqn: Pi spectral proj}
\Pi{U}
= \Pi_0U + \chi_+^*[U]\chi_+ + \chi_-^*[U]\chi_-
\end{equation}
for some functionals $\chi_{\pm}^* \in \X^*$, where $\chi_{\pm}$ are eigenvectors of $\L_0$ corresponding to $\pm{i}\omega$.
Put
\begin{equation}\label{eqn: hyperbolic space X D}
\X_{\hsf} := (\ind_{\X}-\Pi)\X
\quadword{and}
\D_{\hsf} := \D \cap \X_{\hsf}.
\end{equation}
Let $\X_{\hsf}$ and $\D_{\hsf}$ have the norms $\norm{\cdot}_{\X}$ and $\norm{\cdot}_{\D}$, respectively.

\begin{hypothesis}[optimal regularity]\label{hypo: opt reg}
There is a subspace $\Y_{\hsf}$ of $\X_{\hsf}$ such that 
\[
(\ind_{\X}-\Pi)\L_1(\mu)W \in \Y_{\hsf} 
\quadword{and}
(\ind_{\X}-\Pi)\nl(W,\mu) \in \Y_{\hsf}
\]
for all $W \in \D_{\hsf}$.
There exist $b \in (0,\pi)$, $q \in (0,\Lfrak_0^{1/2})$, and $\grave{q} < 0$ such that on the triple $(\D_{\hsf},\Y_{\hsf},\X_{\hsf})$, the operator $\restr{\L_0}{\D_{\hsf}}$ has the localized optimal regularity property with decay rate $q\Lfrak_0^{-1/2}$ and strip width $b$; the periodic optimal regularity property with base frequency $\omega$; and the suboptimal regularity property with growth rate $\grave{q}$.
\end{hypothesis}

Before we state the main result in Theorem \ref{thm: main abstract}, we collect a variety of immediate and useful consequences of the preceding hypotheses in the following lemma.

\begin{lemma}\label{lem: odds and ends}
The following hold for all $U \in \X$ and $0 \le \mu \le \mu_0$.

\begin{enumerate}[label={\bf(\roman*)}, ref={(\roman*)}]

\item\label{part: chi-k+1-star on L0}
$\chi_{k+1}^*[\L_0U] = \chi_k^*[U]$ for $k=0,1,2$, while $\chi_3^*[\L_0U]= 0$.

\item\label{part: chik-star and S}
$\chi_k^*[\S{U}] = (-1)^{k+1}\chi_k^*[U]$.

\item\label{part: chi3-star on L1 on chi3}
$\chi_3^*[\L_1(\mu)\chi_3] = 0$.

\item\label{part: Djmu on chi0}
$D\J_{\mu}(U)\chi_0 = 0$.

\item\label{part: Djmu0 on chi1 and chi3}
$D\J_0(0)\chi_1 = 0$ and $D\J_0(0)\chi_3 = 1$.

\item\label{part: Jscr and S}
$\Jscr_*(\mu)\S{U} = \Jscr_*(\mu)U$.

\item\label{part: Jscr on chi0}
$\Jscr_*(\mu)\chi_0 = 0$.

\end{enumerate}
\end{lemma}

\begin{proof}

\begin{enumerate}[label={\bf(\roman*)}]

\item
This follows from the commutativity of $\L_0$ and the spectral projection $\Pi_0$ defined in \eqref{eqn: abstract 0 proj}, the identities $\L_0\chi_0 = 0$ and $\L_0\chi_{k+1} = \chi_k$ for $k=1$, $2$, $3$, and the linear independence of the $\chi_k$.

\item
Since $\S$ anticommutes with $\L_0$ by Hypothesis \ref{hypo: symmetry}, the spectral projection $\Pi_0$ and $\S$ commute; see the remarks preceding Theorem 2.5 in Section 2.3.3 of \cite{haragus-iooss}.
Then the result follows from Hypothesis \ref{hypo: eigenproj}--\ref{part: S and chik} and the linear independence of the $\chi_k$.

\item
Since $\S^2 = \ind_{\X}$, we may use Hypothesis \ref{hypo: eigenproj}--\ref{part: S and chik} to calculate
\[
\L_1(\mu)\chi_3
= \L_1(\mu)\S^2\chi_3
= \big(\L_1(\mu)\S\big)\big(\S\chi_3\big)
= -\S\L_1(\mu)\chi_3.
\]
We combine this with another invocation of Hypothesis \ref{hypo: eigenproj}--\ref{part: S and chik} to calculate
\[
\chi_3^*[\L_1(\mu)\chi_3]
= -\chi_3^*[\S\L_1(\mu)\chi_3]
= -\chi_3^*[\L_1(\mu)\chi_3],
\]
and so $\chi_3^*[\L_1(\mu)\chi_3] = 0$.

\item
Use Hypothesis \ref{hypo: first int}--\ref{part: Jmu on chi0} to calculate the (directional) derivative
\[
D\J_{\mu}(U)\chi_0
= \lim_{\gamma \to 0} \frac{\J_{\mu}(U+\gamma\chi_0)-\J_{\mu}(U)}{\gamma}
= \lim_{\gamma \to 0} \frac{\J_{\mu}(U)-\J_{\mu}(U)}{\gamma}
= 0.
\]

\item
This follows at once from Hypothesis \ref{hypo: first int}--\ref{part: chi3 J0}.

\item
Use the definition of $\Jscr_*$ and Hypothesis \ref{hypo: first int}--\ref{part: Jmu and S}.

\item
Use part \ref{part: Jscr and S} and Hypothesis \ref{hypo: eigenproj}--\ref{part: S and chik} to calculate
\[
\Jscr_*(\mu)\chi_0
= \Jscr_*(\mu)\S\chi_0
= -\Jscr_*(\mu)\chi_0.
\qedhere
\]
\end{enumerate}
\end{proof}

\subsection{The main abstract result}\label{sec: main abstract result}
Now we can state our result on the existence of certain solutions to the abstract problem \eqref{eqn: abstract ode}.
We prove this theorem in Section \ref{sec: proof of thm: main abstract} after the reduction procedures of Section \ref{sec: 02+iomega reduction}.
We think of this theorem as the abstract existence result for near-sonic position traveling waves in dimer FPUT lattices, for we will show in Section \ref{sec: tw prob redux} that the position traveling wave problem \eqref{eqn: position tw eqns} satisfies, after a fashion, all the hypotheses above.

\begin{theorem}\label{thm: main abstract}
Assume Hypotheses \ref{hypo: F structure}, \ref{hypo: symmetry}, \ref{hypo: spectrum}, \ref{hypo: eigenproj}, \ref{hypo: first int}, \ref{hypo: nondegen}, and \ref{hypo: opt reg}.
Then there are constants $\Alpha_0$, $\Alpha_1$, $\mu_* > 0$ such that if
\begin{equation}\label{eqn: Ascr-mu}
\mu \in (0,\mu_*)
\quadword{and}
\alpha \in \left[\Alpha_0\mu\exp\left(-\frac{b\omega}{\Lfrak_0^{1/2}\mu^{1/2}}\right), \Alpha_1\right] =: \Ascr_{\mu},
\end{equation}
there is a solution to \eqref{eqn: abstract ode} of the form
\begin{multline}\label{eqn: abstract drifting nanopteron}
\Usf_{\mu}^{\alpha}(x)
= \mu^{1/2}\left[
\left(-\frac{3\Lfrak_0^{1/2}}{\Qfrak_0}+\mu^{1/2}\Lup_{\mu}^{\alpha}\right)
\tanh\left(\frac{\Lfrak_0^{1/2}\mu^{1/2}x}{2}\right)
+ \mu^{1/2}\Upsilon_{\mu}^{\alpha,0}(\mu^{1/2}x)
\right]
\chi_0 \\
-\frac{3\Lfrak_0}{2\Qfrak_0}\mu\sech^2\left(\frac{\Lfrak_0^{1/2}\mu^{1/2}x}{2}\right)\chi_1
+ \mu^{3/2}\Upsilon_{\mu}^{\alpha,*}(\mu^{1/2}x) \\
+\alpha\mu\Phi_{\mu}^{\alpha}(\Tup_{\mu}^{\alpha}(\mu^{1/2}x))
+ \alpha\mu^{1/2}(\alpha+\mu)\left(\int_0^{\Tup_{\mu}^{\alpha}(\mu^{1/2}x)} \Phi_{\mu}^{\alpha,\int}(s)\ds\right)\chi_0,
\end{multline}
where 
\begin{equation}\label{eqn: Tsf-mu-alpha}
\Tup_{\mu}^{\alpha}(X)
:= X + \mu\vartheta_{\mu}^{\alpha}\tanh\left(\frac{\Lfrak_0^{1/2}X}{2}\right)
\end{equation}
and the terms and coefficients above have the following properties.

\begin{enumerate}[label={\bf(\roman*)}]

\item
The maps $\Upsilon_{\mu}^{\alpha,0} \colon \R \to \R$ and $\Upsilon_{\mu}^{\alpha,*} \colon \R \to \D$ are real analytic and exponentially localized with
\begin{equation}\label{eqn: main expn loc}
\sup_{\substack{0 < \mu < \mu_* \\ \alpha \in \Ascr_{\mu} \\ X \in \R}} 
e^{q|X|}
\big(
|\Upsilon_{\mu}^{\alpha,0}(X)|
+|\partial_X[\Upsilon_{\mu}^{\alpha,0}](X)|
+\norm{\Upsilon_{\mu}^{\alpha,*}(X)}_{\D}
\big)
< \infty
\end{equation}
and $\chi_0^*[\Upsilon_{\mu}^{\alpha,*}(X)] = 0$ for all $X \in \R$.

\item
The real coefficient $\Lup_{\mu}^{\alpha}$ is uniformly bounded with
\[
\sup_{\substack{0 < \mu < \mu_* \\ \alpha \in \Ascr_{\mu} \\ X \in \R}} 
|\Lup_{\mu}^{\alpha}|
< \infty.
\]

\item
There are constants $0 < C_{\vartheta}^- < C_{\vartheta}^+$ such that the phase shift $\vartheta_{\mu}^{\alpha} \in \R$ satisfies
\begin{equation}\label{eqn: main phase shift}
0
< C_{\vartheta}^-
\le \vartheta_{\mu}^{\alpha} 
\le C_{\vartheta}^+
\end{equation}
for all $0 < \mu < \mu_*$ and $\alpha \in \Ascr_{\mu}$.

\item
The maps $\Phi_{\mu}^{\alpha} \colon \R \to \D$ and $\Phi_{\mu}^{\alpha,\int} \colon \R \to \R$ are defined for all $0 < \mu < \mu_*$ and $0 \le \alpha \le \Alpha_1$.
They are real analytic, periodic, and uniformly bounded with
\begin{equation}\label{eqn: main periodic est}
\sup_{\substack{0 < \mu < \mu_* \\
0 \le \alpha \le \Alpha_1 \\
X \in \R}} 
\left(
\norm{\Phi_{\mu}^{\alpha}(X)}_{\D} 
+ |\Phi_{\mu}^{\alpha,\int}(X)|
\right)
< \infty
\end{equation}
There is a constant $C_{\Phi} > 0$ such that 
\begin{equation}\label{eqn: main periodic lip}
\sup_{0 \le \alpha \le \Alpha_1}
\left(
\Lip(\Phi_{\mu}^{\alpha}) 
+ \Lip(\Phi_{\mu}^{\alpha,\int})
\right)
< C_{\Phi}\mu^{-1/2}
\end{equation}
for all $0 < \mu < \mu_*$.
The map $\Phi_{\mu}^{\alpha}$ satisfies the additional properties \eqref{eqn: Lombardi periodic1}, \eqref{eqn: Lombardi periodic frequency}, and \eqref{eqn: Lombardi periodic2} of part \ref{part: Lombardi periodic} of Theorem \ref{thm: Lombardi}, as well as $\chi_0^*[\Phi_{\mu}^{\alpha}(X)] = 0$ for all $X \in \R$.

\item
For each $0 < \mu < \mu_*$ and $0 \le \alpha \le \Alpha_1$, the map
\begin{equation}\label{eqn: Psf-mu-alpha}
\Psf_{\mu}^{\alpha}(x)
:= \alpha\mu\Phi_{\mu}^{\alpha}(\mu^{1/2}x) 
+ \alpha\mu^{1/2}(\alpha+\mu)\left(\int_0^{\mu^{1/2}x} \Phi_{\mu}^{\alpha,\int}(s) \ds\right)\chi_0
\end{equation}
also solves \eqref{eqn: abstract ode}.

\item
The solutions $\Usf_{\mu}^{\alpha}$ and $\Psf_{\mu}^{\alpha}$ are $\S$-reversible.
\end{enumerate}
\end{theorem}

The map $\Usf_{\mu}^{\alpha}$ defined in \eqref{eqn: abstract drifting nanopteron} is a kind of ``growing front,'' a version of which we previously drew in Figure \ref{fig: front}.
Its leading-order term is the $\O(\mu^{1/2})$ product of a $\tanh$-type coefficient and the $\chi_0$ eigenvector.
This induces the ``front'' behavior.
The next higher-order terms are $\O(\mu)$ exponentially localized terms; one of these has the abstract coefficient $\Upsilon_{\mu}^{\alpha,0}$ on $\chi_0$, while the other has the very explicit $\sech^2$-type coefficient on the generalized eigenvector $\chi_1$.
For the purposes of our lattice calculations in Section \ref{sec: proofs of main lattice theorems}, it is important to separate the abstract localized coefficient from the explicit.

The final higher-order terms are the $\O(\mu^{3/2})$ localized term $\Upsilon_{\mu}^{\alpha,*}$, the genuinely periodic term $\Phi_{\mu}^{\alpha}$, and and an integral term.
The integral term may cause the ``growing'' behavior of the front if it is not also periodic; we discuss in Remark \ref{rem: is it really periodic?} why we are not more specific about its potential periodicity.

The parameter $\alpha$ controls the amplitude of the genuinely periodic term.
The interval $\Ascr_{\mu}$, defined in \eqref{eqn: Ascr-mu}, to which $\alpha$ belongs, has a nonzero, exponentially-small-in-$\mu$ lower bound and an $\O(1)$ upper bound.
Thus the periodic amplitude can be quite small, but nonvanishing, and also somewhat large.

Although the integral term could grow, at worst, linearly, it bears the factor $\alpha^2\mu^{1/2} + \alpha\mu^{3/2}$.
Consequently, for $\alpha$ small, this integral term is smaller ``in $\mu$'' than the genuinely periodic term.
That the integral appears as a factor on the eigenvector $\chi_0$ will be important for the lattice calculations in Section \ref{sec: proofs of main lattice theorems}.

We do not have the freedom to choose the phase shift $\vartheta_{\mu}^{\alpha}$ that appears in \eqref{eqn: Tsf-mu-alpha}; it is given to us by Lombardi's nanopteron method.
However, the phase shift is $\O(\mu)$, so, at least, it is small.

\subsection{Reduction to the $0^{2+}i\omega$ center spectrum}\label{sec: 02+iomega reduction}
We show that to solve \eqref{eqn: abstract ode}, it suffices to solve a simpler ``reduced'' differential equation, stated in \eqref{eqn: fully reduced}, and this simpler equation turns out to possess the $0^{2+}i\omega$ bifurcation from Remark \ref{rem: 04+iomega}.  
To convert \eqref{eqn: abstract ode} into \eqref{eqn: fully reduced}, we make several changes of variables, and we summarize the related notation in Figure \ref{fig: reduction}.
We will apply Lombardi's nanopteron theorem (Theorem \ref{thm: Lombardi}) to this simpler problem and then undo the reductions to recover solutions to the original problem.

\begin{figure}
\[
\begin{tikzpicture}

\node[draw, thick, draw = blue, rectangle, rounded corners] 
at (0,0)
{\eqref{eqn: abstract ode} $U' = \L_0U + \mu\L_1(\mu)U + \nl(U,\mu) \qquad \sigma(\L_0) \cap i\R = 0^{4+}i\omega$ 
};

\draw[thick,blue,->] 
(0,-.5)
--node[pos=.4,right,black]
{\begin{minipage}{1in}
\begin{align*} U &= \mu\chi_0+V 
\\ \mu &\in \R, \ \chi_0^*[V] = 0
\end{align*} 
\end{minipage}
}(0,-2);

\node[draw, thick, draw = blue, rectangle, rounded corners] 
at (0,-2.5)
{\eqref{eqn: V} $V' = \hL_0V + \mu\hL_1(\mu)V + \hnl(V,\mu) \qquad \sigma(\hL_0) \cap i\R = 0^{3+}i\omega$
};

\draw[thick,blue,->] 
(0,-3)
--node[pos=.4,right,black]
{\begin{minipage}{1in} 
\begin{align*} V &= \tau\chi_3+W 
\\ \tau &\in \R, \ \chi_0^*[W] = \chi_3^*[W] = 0 
\end{align*} 
\end{minipage}
}(0,-4.5);

\node[draw, thick, draw = blue, rectangle, rounded corners] 
at (0,-5){\eqref{eqn: fully reduced} $W' = \tL_0W + \mu\tL_1(\mu)W + \tnl(W,\mu) \qquad \sigma(\tL_0) \cap i\R = 0^{2+}i\omega$
};

\end{tikzpicture}
\]
\caption{The two reductions of the original system \eqref{eqn: abstract ode} to the simpler problem \eqref{eqn: fully reduced}.
The center spectrum notation follows Remark \ref{rem: 04+iomega}.}
\label{fig: reduction}
\end{figure}

\subsubsection{Removal of translation invariance}
By Hypothesis \ref{hypo: eigenproj}, we can write any $U \in \X$ uniquely as $U = \gamma\chi_0 + V$, where $\gamma \in \R$ and $\chi_0^*[V] = 0$.
This motivates the decompositions
\[
\hX := \set{V \in \X}{\chi_0^*[V]= 0}
\quadword{and}
\hD := \D \cap \hX.
\]
The spaces $\hX$ and $\hD$ retain the norms of $\X$ and $\D$, respectively.
Put
\begin{equation}\label{eqn: hL0}
\hL_0V
:= \L_0V - \chi_1^*[V]\chi_0
\end{equation}
and set
\begin{equation}\label{eqn: hL1 and hnl}
\hL_1(\mu)V := \L_1(\mu)V-\chi_0^*[\L_1(\mu)V]\chi_0
\quadword{and}
\hnl(V,\mu) := \nl(V,\mu)-\chi_0^*[\nl(V,\mu)]\chi_0.
\end{equation}
Last, define
\begin{equation}\label{eqn: Gamma-mu}
\Gamma_{\mu}(V)
:= \chi_1^*[V]+ \Gamma_{\mu}^*(V),
\qquad
\Gamma_{\mu}^*(V) := \mu\chi_0^*[\L_1(\mu)V] + \chi_0^*[\nl(V,\mu)],
\end{equation}
and
\begin{equation}\label{eqn: hFmu}
\hF(V,\mu) 
:= \hL_0V + \mu\hL_1(\mu)V + \hnl(V,\mu).
\end{equation}

Hypotheses \ref{hypo: eigenproj}--\ref{part: L1 chi0} and \ref{hypo: eigenproj}--\ref{part: Q chi0} then imply that \eqref{eqn: F-mu} is equivalent to the decoupled system
\begin{numcases}{}
\gamma'(x) = \Gamma(V(x),\mu) \label{eqn: mu} \\
V'(x) = \hF(V(x),\mu) \label{eqn: V}.
\end{numcases}
Given a solution $V$ to \eqref{eqn: V}, we can of course solve \eqref{eqn: mu} for $\gamma$, up to a constant of integration $\gamma_0 \in \R$:
\begin{equation}\label{eqn: mu soln}
\gamma(x)
= \gamma_0 + \int_0^x \Gamma_{\mu}(V(s)) \ds.
\end{equation}
Then the full solution to \eqref{eqn: F-mu} will have the form
\[
U(x)
= \gamma_0\chi_0 + \bunderbrace{\left(\int_0^x \Gamma_{\mu}(V(s)) \ds\right)\chi_0 + V(x)}{U_0(x)}.
\]
By Hypotheses \ref{hypo: eigenproj}--\ref{part: L1 chi0} and \ref{hypo: eigenproj}--\ref{part: Q chi0}, if a function $U$ of this form solves \eqref{eqn: F-mu}, then so does $U_0$, which depends only on $V$.
Thus it suffices to solve only \eqref{eqn: V}.
Manipulating the integral term that multiplies $\chi_0$ in $U_0$ will be one of our most delicate tasks later, in Sections \ref{sec: periodic + growing construction} and \ref{sec: nano + growing construction}.

Our work here has followed \cite[Sec.\@ 3.2]{venney-zimmer}; one difference is that we assume neither $\chi_0^*[\L_1(\mu)V] = 0$ nor $\chi_0^*[\nl(V,\mu)] = 0$, and so we must keep these terms in $\Gamma_1$, and thus in $\Gamma$, unlike \cite[Eqn.\@ (38)]{venney-zimmer}.
See Remark \ref{rem: mass dimer chi0-star} for the precise reason, in the context of our dimer FPUT problems, why we do not assume that these functionals vanish on these kinds of inputs.
However, the terms in $\Gamma_1$ are small, and so their effect is negligible.
See also \cite[Sec. 3]{iooss-fpu}, \cite[Sec. II.C.1]{iooss-james}, and \cite[Sec.\@ 2.2]{hilder-derijk-schneider} for the same sort of reduction of translation invariance.

Before proceeding, we note that $\J_{\mu}$ remains a first integral for \eqref{eqn: V}, and we also prove some useful properties of $\Gamma$.

\begin{lemma}\label{lem: first int reduced}
With $\hF$ defined in \eqref{eqn: hFmu}, we have $D\J_{\mu}(V)\hF(V,\mu) = 0$ for all $V \in \hD$.
\end{lemma}

\begin{proof}
Rewrite
\[
\hF(V,\mu)
= \F(V,\mu) -  \big(\chi_1^*[V]+ \mu\chi_0^*[\L_1(\mu)V] + \chi_0^*[\nl(V,\mu)]\big)\chi_0
\]
and use Hypothesis \ref{hypo: first int}--\ref{part: Djmu on F} and part \ref{part: Djmu on chi0} of Lemma \ref{lem: odds and ends}.
\end{proof}

\begin{lemma}\label{lem: Gamma-mu props}
The map $\Gamma_{\mu}$, defined in \eqref{eqn: Gamma-mu}, has the following properties.

\begin{enumerate}[label={\bf(\roman*)}, ref={(\roman*)}]

\item\label{part: Gamma-mu and S}
$\Gamma_{\mu}(\S{V}) = \Gamma_{\mu}(V)$ and $\Gamma_{\mu}^*(\S{V}) = \Gamma_{\mu}^*(\S{V})$ for all $V \in \X$.

\item
Suppose that $f \colon \R \to \X$ is $\S$-reversible (in the sense of Definition \ref{defn: reversible}).
Then $\Gamma_{\mu}\circ{f}$ and $\Gamma_{\mu}^*\circ{f}$ are even.

\item\label{part: Gamma-mu-star}
The map $\Gamma_{\mu}^*$ satisfies
\[
\sup_{\substack{0 < \mu < \mu_0 \\
0 < \norm{V}_{\X} \le 1 \\
0 < \alpha, \tau_1,\tau_2 < 1}} \big(\alpha\mu^2\big)^{-1} |\Gamma_{\mu}^*(\alpha\mu\tau_1\chi_3 + \alpha\mu\tau_2V)|
< \infty
\]
and
\[
\sup_{\substack{0 < \mu < \mu_0 \\
0 < \norm{V}_{\X} + \norm{\grave{V}}_{\X} < 1}}
(\mu^2\norm{\grave{V}}_{\X})^{-1}|\Gamma_{\mu}^*(\mu{V}+\mu{\grave{V}})-\Gamma_{\mu}^*(\mu{V})|
< \infty.
\]
\end{enumerate}
\end{lemma}

\begin{proof}

\begin{enumerate}[label={\bf(\roman*)}]

\item
This is a direct calculation using the definition of $\Gamma_{\mu}$, Hypotheses \ref{hypo: symmetry} and \ref{hypo: eigenproj}--\ref{part: S and chik}, and part \ref{part: chik-star and S} of Lemma \ref{lem: odds and ends}.

\item
We have $\Gamma_{\mu}(f(-x)) = \Gamma_{\mu}(\S{f}(x))$ since $f$ is $\S$-reversible, and $\Gamma_{\mu}(\S{f}(x)) = \Gamma_{\mu}(f(x))$ by part \ref{part: Gamma-mu and S}.
The same holds for $\Gamma_{\mu}^*$.

\item
This estimate follows from the definition of $\Gamma_{\mu}^*$ and the quadraticity of $\nl$ in its first variable.
\qedhere
\end{enumerate}
\end{proof}

\subsubsection{Fixing a value for the first integral}\label{sec: chi3 cov}
We can write any $V \in \hD$ uniquely as $V = \tau\chi_3 + W$, where $\tau \in \R$ and $\chi_3^*[W] = 0$, and so \eqref{eqn: V} is equivalent to
\begin{numcases}{}
\tau'(x) = \chi_3^*\big[\hF(\tau(x)\chi_3+W(x),\mu)\big] \label{eqn: tau chi3} \\
W'(x) = \hF(\tau(x)\chi_3+W(x),\mu)-\chi_3^*\big[\hF(\tau(x)\chi_3+W(x),\mu)\big]\chi_3 \label{eqn: W chi3}.
\end{numcases}
Now let
\begin{equation}\label{eqn: tilde spaces}
\tX := \set{W \in \hX}{\chi_3^*[W] = 0}
\quadword{and}
\tD := \hD \cap \tX.
\end{equation}

We will show that the values of the first integral $\J_{\mu}(V)$ and the ``orthogonal'' component $W$ completely determine $\chi_3^*(V)$, at least when everything is small.
In other words, we can solve $j = \J_{\mu}(\tau\chi_3+W)$ for $\tau$ given small $\mu$, $j \in \R$ and $W \in \tX$.
Since $\J_{\mu}$ is constant on solutions of the coupled system \eqref{eqn: tau chi3}--\eqref{eqn: W chi3} by Lemma \ref{lem: first int reduced}, this system therefore reduces to a problem in $W$ alone, at least for ``small-amplitude'' solutions.
This is the same procedure that Iooss and James performed to obtain \cite[Eqn.\@ (33)]{iooss-james} and likewise that Venney and Zimmer began at the start of \cite[Sec.\@ 3.5]{venney-zimmer}.
However, these papers did not address a certain ``consistency'' problem with this approach as we do; see Lemma \ref{lem: chi3 reduction consistency} and the comments preceding it.
Moreover, we need quite precise estimates on how $\tau$ is so ``completely determined'' by the values of $W$ and $\J_{\mu}(V)$.
We mention that quite a related result appears in \cite[Lem.\@ 4.1]{hilder-derijk-schneider}, albeit with the intention of different estimates and applications; our goals are sufficiently distinct as to warrant the full proof below.

Part \ref{part: tau = T Jmu} of the following lemma makes precise how $W$ and $J_{\mu}(V)$ completely determine $\chi_3^*(V)$.
We also include several detailed expansions and estimates whose usefulness will only be apparent later.
We will work on the balls
\[
\Bfrak(r)
:= \set{(W,j,\mu) \in \tX \times \R \times [0,\mu_0]}{ \max\{\norm{W}_{\X}, |j|, \mu\} < r}.
\]

\begin{lemma}\label{lem: big bad chi3 lemma}
There exists $\mu_{\T} \in (0,\min\{\mu_0,1\})$ and an analytic map $\T \colon \Bfrak(\mu_{\T}) \to \R$ with the following properties.

\begin{enumerate}[label={\bf(\roman*)},ref={(\roman*)}]

\item\label{part: T0 = 0}
$\T(0,0,\mu) = 0$ for $0 \le \mu \le \mu_{\T}$.

\item\label{part: j = Jmu T}
$j = \J_{\mu}(\T(W,j,\mu)\chi_3+W)$ for $(W,j,\mu) \in \Bfrak(\mu_{\T})$.

\item\label{part: tau = T Jmu}
$\tau = \T(W,\J_{\mu}(\tau\chi_3+W),\mu)$ for $\max\{\norm{W}_{\X},|\tau|, \mu\} < \mu_{\T}$.

\item\label{part: T and S}
$\T(\S{W},j,\mu) = \T(W,j,\mu)$ for $(W,j,\mu) \in \Bfrak(\mu_{\T})$.

\item\label{part: deriv of T}
$D_W\T(W,j,\mu)\grave{W} = -\frac{D\J_{\mu}(\T(W,j,\mu)\chi_3+W)\grave{W}}{D\J_{\mu}(\T(W,j,\mu)\chi_3+W)\chi_3}$
and $D_W\T(0,0,\mu)\grave{W} = -\frac{\mu\Jscr_*(\mu)\grave{W}}{1+\mu\Jscr_*(\mu)\chi_3}$ for all $\grave{W}$, $\breve{W} \in \tX$ and $(W,j,\mu) \in \Bfrak(\mu_{\T})$.

\item\label{part: 2deriv of T}
$D_{WW}^2\T(0,0,0)[\grave{W},\breve{W}]
= -D^2\J_0(0)[\grave{W},\breve{W}] + D^2\J_0(0)[\grave{W},\chi_3]D\J_0(0)\breve{W}$ for all $\grave{W}$, $\breve{W} \in \tX$.

\item\label{part: T quad est}
$\sup_{\substack{0 < |\alpha| \le 1 \\ \norm{W}_{\X} \le \mu_{\T} \\ 0 < \mu < \mu_{\T}}} \alpha^{-1}\mu^{-2}|\T(\alpha\mu{W},0,\mu)| < \infty$.

\item\label{part: T Lip est}
$\sup_{\substack{0 < |\alpha| \le 1 \\ \norm{W}_{\X}+\norm{\grave{W}}_{\X} \le \mu_{\T} \\ 0 < \mu < \mu_{\T}}} \mu^{-2}|\T(\alpha\mu{W}+\mu\grave{W},0,\mu)-\T(\alpha\mu{W},0,\mu)| < \infty$.
\end{enumerate}
\end{lemma}

\begin{proof}
We first construct the map $\T$.
Put
\begin{equation}\label{eqn: Jsf}
\Jsf(\tau,W,j,\mu)
:= j-\J_{\mu}(\tau\chi_3+W),
\end{equation}
so $\Jsf$ is real analytic on $\R \times \tX \times \R \times [0,\mu_0]$ by Hypothesis \ref{hypo: first int}--\ref{part: first int diff}.
Hypothesis \ref{hypo: first int}--\ref{part: Jstar} gives
\begin{equation}\label{eqn: Jsf-0-0-0-mu}
\Jsf(0,0,0,\mu)
= -\J_0(0)
= 0
\end{equation}
for all $\mu$, and part \ref{part: Djmu0 on chi1 and chi3} of Lemma \ref{lem: odds and ends} implies
\[
D_{\tau}\Jsf(0,0,0,0)
= -D\J_0(0)\chi_3
= -1.
\]
The analytic implicit function theorem then provides $\mu_1 \in (0,\mu_0]$ and $\mu_2 > 0$ and an analytic map $\T \colon \Bfrak(\mu_1) \to (-\mu_2,\mu_2)$ such that if $(W,j,\mu) \in \Bfrak(\mu_1)$ and $|\tau| < \mu_2$, then $\Jsf(\tau,W,j,\mu) = 0$ if and only if $\tau = \T(W,j,\mu)$.

Now we prove the more specialized properties of $\T$.

\begin{enumerate}[label={\bf(\roman*)}]

\item 
The identity \eqref{eqn: Jsf-0-0-0-mu} implies that $\T(0,0,\mu) = 0$ for all $\mu$.

\item 
The identity $\Jsf(\T(W,j,\mu),W,j,\mu) = 0$ and the definition of $\Jsf$ in \eqref{eqn: Jsf} give 
\[
j 
= \J_{\mu}(\T(W,j,\mu)\chi_3+W).
\]

\item 
Since $\J_0(0) = 0$ by Hypothesis \ref{hypo: first int}--\ref{part: Jstar}, the continuity result from Hypothesis \ref{hypo: first int}--\ref{part: first int diff} gives $\mu_3 \in (0,\min\{\mu_1,\mu_2\})$ such that if 
\begin{equation}\label{eqn: W tau mu cond}
\min\{\norm{W}_{\X}, |\tau|, \mu\} 
< \mu_3
\end{equation}
then $|\J_{\mu}(\tau\chi_3+W)| \le \mu_1$.
Assume that  \eqref{eqn: W tau mu cond} holds.
Then $(W,\J_{\mu}(\tau\chi_3+W),\mu) \in \Bfrak(\mu_1)$ and $|\tau| < \mu_2$, while a direct calculation yields $\Jsf(\tau,W,\J_{\mu}(\tau\chi_3+W),\mu) = 0$.
Consequently, \eqref{eqn: W tau mu cond} forces
\[
\tau 
= \T(W,\J_{\mu}(\tau\chi_3+W),\mu).
\]

\item 
First, since $\norm{\S{W}}_{\X} \le \norm{W}_{\X}$ by Hypothesis \ref{hypo: symmetry}, we have $(\S{W},j,\mu) \in \Bfrak(\mu_1)$ whenever $(W,j,\mu) \in \Bfrak(\mu_1)$.
Thus $\T(\S{W},j,\mu)$ is in fact defined for $(W,j,\mu) \in \Bfrak(\mu_1)$.
Next, we already know that $\Jsf(W,\T(W,j,\mu),j,\mu) = 0$. 
If we can also show that $\Jsf(W,\T(\S{W},j,\mu),j,\mu) = 0$, then we must have $\T(\S{W},j,\mu) = \T(W,j,\mu)$.

The definition of $\Jsf$ in \eqref{eqn: Jsf},  Hypotheses \ref{hypo: eigenproj}--\ref{part: S and chik} and \ref{hypo: first int}--\ref{part: Jmu and S}, and part \ref{part: chik-star and S} of Lemma \ref{lem: odds and ends} enable us to calculate $\Jsf(\S{W},\tau,j,\mu) = \Jsf(W,\tau,j,\mu)$ for any $\tau$.
Then
\[
0
= \Jsf(\S{W},\T(\S{W},j,\mu),j,\mu)
= \Jsf(W,\T(\S{W},j,\mu),j,\mu),
\]
as desired.

\item 
Part \ref{part: Djmu0 on chi1 and chi3} of Lemma \ref{lem: odds and ends} tells us that $D\J_0(0)\chi_3 = 1$.
The analyticity of $\J$ and $\T$ then give $\mu_{\T} \in (0,\min\{\mu_3,1\})$ such that if $(W,j,\mu) \in \Bfrak(\mu_{\T})$, then $D\J_{\mu}(\T(W,j,\mu)\chi_3+W)\chi_3 \ne 0$.

Fix $j$, $\mu \in \R$ with $\max\{|j|,\mu\} < \mu_{\T}$.
By part \ref{part: j = Jmu T} above, the map $\Jfrak(W) := \J_{\mu}(\T(W,j,\mu)\chi_3+W)$ is constant for $\norm{W}_{\X} < \mu_{\T}$, and so for any $\grave{W} \in \tX$, we have
\begin{equation}\label{eqn: derivT}
0
= D\Jfrak(W)\grave{W}
= D\J_{\mu}(\T(W,j,\mu)\chi_3+W)\big[(D_W\T(W,j,\mu)\grave{W})\chi_3+\grave{W}\big].
\end{equation}
This implies
\[
D_W\T(W,j,\mu)\grave{W} 
= -\frac{D\J_{\mu}(\T(W,j,\mu)\chi_3+W)\grave{W}}{D\J_{\mu}(\T(W,j,\mu)\chi_3+W)\chi_3}.
\]

In the special case $W = j = 0$, we use the identity $\T(0,0,\mu)$ from part \ref{part: T0 = 0} above and the formula \eqref{eqn: derivT} to calculate
\[
D_W\T(0,0,\mu)\grave{W}
= -\frac{D\J_{\mu}(0)\grave{W}}{D\J_{\mu}(0)\chi_3}.
\]
By Hypotheses \ref{hypo: first int}--\ref{part: Jstar} and \ref{hypo: first int}--\ref{part: chi3 J0}, we have
\[
D\J_{\mu}(0)U
= D\J_0(0)U + \big(D\J_{\mu}(0)-D\J_0(0)\big)U
= \chi_3^*[U] + \mu\Jscr_*(\mu)U
\]
for any $U \in \X$.
Since $\chi_3^*[\grave{W}] = 0$ for $\grave{W} \in \tX$ and $\chi_3^*[\chi_3] = 1$, we conclude
\[
D_W\T(0,0,\mu)\grave{W}
= -\frac{\mu\Jscr_*(\mu)U}{1+\mu\Jscr_*(\mu)\chi_3}.
\]

\item 
This follows from part \ref{part: deriv of T} and a lengthy calculation involving the product and chain rule in Banach spaces.

\item 
Since $\T(0,0,\mu) = 0$, we may use the fundamental theorem of calculus twice and part \ref{part: deriv of T} to write
\begin{multline*}
\T(\alpha\mu{W},0,\mu)
= \int_0^1\int_0^1 D_{WW}^2\T(st\alpha\mu{W},0,\mu)[t\alpha\mu{W},\alpha\mu{W}] \ds\dt + D_W\T(0,0,\mu)(\alpha\mu{W}) \\
= \alpha\mu^2\left(\alpha\int_0^1\int_0^1 D_{WW}^2\T(st\alpha\mu{W},0,\mu)[tW,W] \ds\dt - \frac{\Jscr_*(\mu)W}{1+\mu\Jscr_*(\mu)\chi_3}\right).
\end{multline*}
The estimate then follows from the analyticity of $\T$.

\item 
Similar to part \ref{part: T quad est}, we use the fundamental theorem of calculus twice to rewrite the difference $\T(\alpha\mu{W}+\mu\grave{W},\mu,0)-\T(\alpha\mu{W},\mu,0)$ in the form
\[
\mu^2\left(\int_0^1\int_0^1 D_{WW}^2\T(s(\alpha\mu{W}+t\mu\grave{W}),0,\mu)[\alpha{W}+t\grave{W},\grave{W}] \ds\dt- \frac{\Jscr_*(\mu)\grave{W}}{1+\mu\Jscr_*(\mu)\chi_3}\right).
\qedhere
\]
\qedhere
\end{enumerate}
\end{proof}

Now let $\mu_1$ and $\mu_2$ be as in Lemma \ref{lem: big bad chi3 lemma}.
Suppose that $\tau_{\mu} \colon \R \to \R$ and $W_{\mu} \colon \R \to \tD$ solve the system \eqref{eqn: tau chi3}--\eqref{eqn: W chi3} with $0 \le \mu  \le \mu_{\T}$.
Put $j_{\mu}(x) := J_{\mu}(\tau_{\mu}(x)\chi_3+W_{\mu}(x))$, so $j_{\mu}$ is constant by Hypothesis \ref{hypo: first int}--\ref{part: Djmu on F}. 
Suppose also that $\tau_{\mu}$ and $W_{\mu}$ are small enough that $\max\{\norm{W_{\mu}(x)}_{\X},|\tau_{\mu}(x)|,\mu\} < \mu_{\T}$ for all $x \in \R$. 
Then part \ref{part: tau = T Jmu} of Lemma \ref{lem: big bad chi3 lemma} implies $\tau_{\mu}(x) = \T(W_{\mu}(x),j_{\mu}(0),\mu)$ for all $x$.
That is, the coefficient on $\chi_3$ is completely determined by the part of the solution orthogonal to $\chi_3$ and the first integral --- at least if everything is sufficiently small.

Conversely, this suggests a strategy for solving \eqref{eqn: V}, equivalently the coupled system \eqref{eqn: tau chi3}--\eqref{eqn: W chi3}: assume $\tau(x) = \T(W(x),j,\mu)$ for $j \in \R$ fixed and attempt to solve
\begin{equation}\label{eqn: W chi3 no tau}
W'(x)
= \hF(\T(W(x),j,\mu)\chi_3+W(x),\mu)-\chi_3^*\big[\hF(\T(W(x),j,\mu)\chi_3+W(x),\mu)\big]\chi_3.
\end{equation}
This equation is now an ordinary differential equation in $W$ with $\mu$ and $j$ as parameters.
Of course, here we assume that $W$, $j$, and $\mu$ small enough for $\T(W(x),j,\mu)$ to be defined.

This is precisely what we shall do, except there is a nagging question of consistency.
If $W$ solves \eqref{eqn: W chi3 no tau} and is sufficiently small, does putting $\tau(x) := \T(W(x),j,\mu)$ solve \eqref{eqn: tau chi3}?
Happily, this is the case, at least if we pick $j = 0$.
In other words, we are freezing the value of the first integral to be 0.

\begin{lemma}\label{lem: chi3 reduction consistency}
Suppose that $W \colon \R \to \tD$ solves \eqref{eqn: W chi3 no tau} with $j = 0$ and $\max\{\norm{W(x)}_{\X},\mu\} < \mu_{\T}$ for all $x$.
Then the map $\tau(x) := \T(W(x),0,\mu)$ solves \eqref{eqn: tau chi3}.
\end{lemma}

\begin{proof}
Since $(W(x),0,\mu) \in \Bfrak(\mu_{\T})$ for all $x \in \R$, we may differentiate $\tau$ directly using part \ref{part: deriv of T} from Lemma \ref{lem: big bad chi3 lemma} and find, suppressing all $x$-dependencies,
\[
\tau_{\mu}'
= D_W\T(W,0,\mu)W'
= -\frac{D\J_{\mu}(\T(W,0,\mu)\chi_3+W)W'}{D\J_{\mu}(\T(W,0,\mu)\chi_3+W)\chi_3}.
\]
We use \eqref{eqn: W chi3 no tau} to calculate
\begin{multline*}
D\J_{\mu}(\T(W,0,\mu)\chi_3+W)W' \\
= D\J_{\mu}(\T(W,0,\mu)\chi_3+W)\hF(\T(W,0,\mu)\chi_3+W,\mu)) \\
-\chi_3^*\big[\hF(\T(W,0,\mu)\chi_3+W,\mu))\big]D\J_{\mu}(\T(W,0,\mu)\chi_3+W)\chi_3.
\end{multline*}
Lemma \ref{lem: first int reduced} implies
\[
D\J_{\mu}(\T(W,0,\mu)\chi_3+W)\hF(\T(W,0,\mu)\chi_3+W,\mu))
= 0,
\]
and so
\begin{multline*}
\tau_{\mu}'
= \frac{\chi_3^*\big[\hF(\T(W,0,\mu)\chi_3+W,\mu))\big]D\J_{\mu}(\T(W,0,\mu)\chi_3+W)\chi_3}{D\J_{\mu}(\T(W,0,\mu)\chi_3+W)\chi_3} \\
= \chi_3^*\big[\hF(\T(W,0,\mu)\chi_3+W,\mu))\big] 
= \chi_3^*\big[\hF(\tau_{\mu}\chi_3+W,\mu)\big].
\qedhere
\end{multline*}
\end{proof}

\subsubsection{The fully reduced system}
From now on, we assume $0 \le \mu \le \mu_{\T}$, where $\mu_{\T}$ was defined in Lemma \ref{lem: big bad chi3 lemma}, and, for $W \in \tX$ with $\norm{\W}_{\X} \le \mu_{\T}$, we put
\[
\T_{\mu}(W)
:= \T(W,0,\mu),
\]
with $\T$ also defined in Lemma \ref{lem: big bad chi3 lemma}.
We will study the ``fully reduced'' problem
\begin{equation}\label{eqn: fully reduced}
W'(x)
= \tF(W(x),\mu),
\end{equation}
where
\begin{equation}\label{eqn: tFmu}
\tF(W,\mu)
:= \hF(\T_{\mu}(W)\chi_3+W,\mu)-\chi_3^*\big[\hF(\T_{\mu}(W)\chi_3+W,\mu)\big]\chi_3.
\end{equation}
The work in the previous two sections implies that if $W$ solves \eqref{eqn: fully reduced} and is sufficiently small, then 
\begin{equation}\label{eqn: W-T-int soln to orig}
U(x)
:= W(x) 
+ \T_{\mu}(W(x))\chi_3
+ \left(\int_0^x \Gamma_{\mu}\big(\T_{\mu}(W(s))\chi_3 + W(s)\big) \ds\right)\chi_0
\end{equation}
solves the original problem \eqref{eqn: abstract ode}.

We rewrite $\tF$ in the form
\begin{equation}\label{eqn: tFmu2}
\tF(W,\mu)
= \tL_0W
+ \mu\tL_1(\mu)W
+ \tnl(W,\mu),
\end{equation}
where we use the definition of $\hF$ in \eqref{eqn: hFmu} and its constituent terms from \eqref{eqn: hL0} and \eqref{eqn: hL1 and hnl} to expand
\begin{equation}\label{eqn: tL0}
\tL_0W 
:= \hL_0W
= \L_0W - \chi_1^*[W]\chi_0,
\end{equation}
\begin{equation}\label{eqn: tL1}
\begin{aligned}
\tL_1(\mu)W
&:= \L_1(\mu)W-\chi_0^*[\L_1(\mu)W]\chi_0-\chi_3^*[\L_1(\mu)W]\chi_3
- \left(\frac{\Jscr_*(\mu)W}{1+\mu\Jscr_*(\mu)\chi_3}\right)\chi_2 \\
&-\mu\left(\frac{\Jscr_*(\mu)W}{1+\mu\Jscr_*(\mu)\chi_3}\right)\big(\L_1(\mu)\chi_3-\chi_0^*[\L_1(\mu)\chi_3]\chi_0\big),
\end{aligned}
\end{equation}
and
\begin{equation}\label{eqn: tnl}
\begin{aligned}
\tnl(W,\mu)
&:=\nl(\T_{\mu}(W)\chi_3+W,\mu)
-\chi_0^*\big[\nl(\T_{\mu}(W)\chi_3+W,\mu)\big]\chi_0 \\
&-\chi_3^*\big[\nl(\T_{\mu}(W)\chi_3+W,\mu)\big]\chi_3 
+ \nl_{\T}(W,\mu)\chi_2 \\
&+ \mu\nl_{\T}(W,\mu)\big(\L_1(\mu)\chi_3-\chi_0^*[\L_1(\mu)\chi_3]\chi_0\big),
\end{aligned}
\end{equation}
with
\[
\nl_{\T}(W,\mu)
:= \T_{\mu}(W)-D\T_{\mu}(0)W
= \T_{\mu}(W)+\mu\left(\frac{\Jscr_*(\mu)W}{1+\mu\Jscr_*(\mu)\chi_3}\right).
\]
We have used the various hypotheses, part \ref{part: chi3-star on L1 on chi3} of Lemma \ref{lem: odds and ends}, and part \ref{part: deriv of T} of Lemma \ref{lem: big bad chi3 lemma} to simplify and cancel some terms to obtain the expansion \eqref{eqn: tFmu2}.
Observe that 
\[
\chi_k^*[\tL_0W]
= \chi_k^*[\tL_1(\mu)W]
= \chi_k^*[\tnl(W,\mu)]
= 0, \ k = 0, 3
\]
for all $W \in \tD$, and so $\tF$ and its constituent terms do indeed map $\tD$ into $\tX$.

The next lemma shows that $\tF$ satisfies all the properties of Theorem \ref{thm: Lombardi}, which guarantees nanopteron solutions to the reduced problem \eqref{eqn: fully reduced}.

\begin{lemma}\label{lem: ultimate tFmu props}
For $0 \le \mu \le \mu_1$, the map $\tF$ defined in \eqref{eqn: tFmu2} has the following properties.

\begin{enumerate}[label={\bf(\roman*)}, ref={(\roman*)}]

\item
The mappings $\mu \mapsto \tL_1(\mu)$ and $(W,\mu) \mapsto \tnl(W,\mu)$ are analytic with
\[
\sup_{0 \le \mu \le \mu_1} \norm{\tL_1(\mu)}_{\b(\tD,\tX)} \le C
\quadword{and}
\sup_{0 \le \mu \le \mu_1} \norm{\tnl(W,\mu)}_{\X} \le C\norm{W}_{\X}^2.
\]

\item
The symmetry $\S$ anticommutes with each term in $\tF$:
\[
\S\tL_0 = -\tL_0\S,
\qquad
\S\tL_1(\mu) = -\tL_1(\mu)\S,
\quadword{and}
\S\tnl(W,\mu) = - \tnl(\S{W},\mu).
\]

\item\label{part: Fmu center spectrum}
The center spectrum of $\tL_0$ as an operator in $\tX$ with domain $\tD$ is 
\[
\sigma(\tL_0) \cap i\R
= \{0,\pm{i}\omega\},
\]
where $\pm{i}\omega$ are the pure imaginary eigenvalues of $\L_0$ from Hypothesis \ref{hypo: spectrum}.
The point 0 is an eigenvalue of multiplicity 2 with generalized eigenvectors $\chi_1$ and $\chi_2$ satisfying $\tL_0\chi_1 = 0$, $\tL_0\chi_2 = \chi_1$, and $\S\chi_1 = \chi_1$.
The points $\pm{i}\omega$ are eigenvalues of algebraic multiplicity 1. 

\item
The spectral projection for $\tL_0$ corresponding to 0 has the form 
\begin{equation}\label{eqn: tPi0}
\tPi_0W 
= \chi_1^*[W]\chi_1 + \chi_2^*[W]\chi_2,
\end{equation}
where
\begin{equation}\label{eqn: tFmu nondegen}
\chi_2^*\big[\tL_1(0)\chi_1\big]
> 0
\quadword{and}
\chi_2^*\big[D_{WW}^2\tF(0,0)[\chi_1,\chi_1]\big]
\ne 0.
\end{equation}

\item
Let $\tPi$ be the spectral projection for $\tL_0$ corresponding to $\{0,\pm{i}\omega\}$ and put
\begin{equation}\label{eqn: hyperbolic space tX tD}
\tX_{\hsf} := (\ind_{\tX}-\tPi)(\tX)
\quadword{and}
\tD_{\hsf} := \tX_{\hsf} \cap \tD.
\end{equation}
Put $\tY_{\hsf} := \Y_{\hsf}$, where $\Y_{\hsf}$ was defined in Hypothesis \ref{hypo: opt reg}, and take $b \in (0,\pi)$ and $q \in (0,\Lfrak_0^{1/2})$ from that hypothesis as well.
Then on the triple $(\tD_{\hsf},\tY_{\hsf},\tX_{\hsf})$, the operator $\restr{\tL_0}{\D_{\hsf}}$ has the localized optimal regularity property with decay rate $q\Lfrak_0^{-1/2}$ and strip width $b$; the periodic optimal regularity property with base frequency $\omega$; and the suboptimal regularity property with growth rate $\grave{q}$.
\end{enumerate}
\end{lemma}

\begin{proof}
\begin{enumerate}[label={\bf(\roman*)}]

\item
This is obvious from the definitions in \eqref{eqn: tL0}, \eqref{eqn: tL1}, and \eqref{eqn: tnl}.

\item
This follows from Hypothesis \ref{hypo: symmetry}, Hypothesis \ref{hypo: eigenproj}--\ref{part: S and chik}, parts \ref{part: chik-star and S} and  \ref{part: Jscr and S} of Lemma \ref{lem: odds and ends}, and part \ref{part: T and S} of Lemma \ref{lem: big bad chi3 lemma}.

\item
This follows from Hypothesis \ref{hypo: spectrum} and Lemma \ref{lem: main spectral theory lemma}.

\item
That the spectral projection has this form follows from Lemma \ref{lem: main spectral theory lemma}.
For the inequality in \eqref{eqn: tFmu nondegen}, we use the definition of $\tL_1$ in \eqref{eqn: tL1} and the identities $\chi_2^*[\chi_0] = \chi_2^*[\chi_3] = 0$ to calculate
\[
\chi_2^*[\tL_1(0)\chi_1]
= \chi_2^*[\L_1(0)\chi_1]- \Jscr_*(0)\chi_1
= \Lfrak_0
> 0,
\]
where the positivity is Hypothesis \ref{hypo: nondegen}.

The proof of the nonequality in \eqref{eqn: tFmu nondegen} is slightly more involved.
From the definition of $\tF$ in \eqref{eqn: tFmu2}, we calculate
\[
\chi_2^*[\tF(W,0)]
= \chi_2^*[\tL_0W] + \chi_2^*\big[\nl(\T_{\mu}(W)\chi_3 + W,\T_{\mu}(W)\chi_3+W)\big] + \T_{\mu}(W).
\]
We immediately have
\[
\chi_2^*[\tL_0W]
= \chi_2^*[\L_0W]-\chi_1^*[W]\chi_2^*[\chi_0]
= \chi_3^*[W]
= 0.
\]
Next, we rewrite
\[
\nl(\T_{\mu}(W)\chi_3 + W,\T_{\mu}(W)\chi_3+W,0)
= \nl_0(\T_{\mu}(W)\chi_3+W,\T_{\mu}(W)\chi_3+W) + \nl_1(\T_{\mu}(W)\chi_3+W,0)
\]
from the expansion of $\nl$ in \eqref{eqn: nl expn}.

By parts \ref{part: T0 = 0} and \ref{part: deriv of T} of Lemma \ref{lem: big bad chi3 lemma}, we have 
\begin{equation}\label{eqn: T-quad est}
|\T_0(W)| 
\le C\norm{W}_{\X}^2
\end{equation}
for $\norm{W}_{\X} < \mu_1$.
The bilinearity of $\nl_0$ implies
\begin{equation}\label{eqn: nl0 with T}
\nl_0(\T_{\mu}(W)\chi_3+W,\T_{\mu}(W)\chi_3+W)
= \T_{\mu}(W)^2\nl_0(\chi_3,\chi_3) + 2\T_{\mu}(W)\nl_0(\chi_3,W) + \nl_0(W,W).
\end{equation}
The estimate \eqref{eqn: T-quad est} and the bilinearity of $\nl_0$ imply that the first two terms in \eqref{eqn: nl0 with T} are quartic and cubic in $W$, respectively.
Next, the term $\nl_1(\T_{\mu}(W)\chi_3+W,0)$ is then cubic in $W$ by \eqref{eqn: nl ests} and the estimate \eqref{eqn: T-quad est}.

We conclude
\begin{align*}
\chi_2^*\big[D_{WW}^2\tF(0,0)[\chi_1,\chi_1]\big]
&= 2\chi_2^*\big[\nl_0(\chi_1,\chi_1)\big]+ D_{WW}^2\T(0)[\chi_1,\chi_1] \\
&= 2\chi_2^*\big[\nl_0(\chi_1,\chi_1)\big] - D^2\J_0(0)[\chi_1,\chi_1]
\end{align*}
by part \ref{part: 2deriv of T} of Lemma \ref{lem: big bad chi3 lemma}.
This final quantity is $2\Qfrak_0 \ne 0$, per Hypothesis \ref{hypo: nondegen}.

\item
First, the spectral projection for $\tL_0$ onto $\{0,\pm{i}\omega\}$ has the form
\[
\tPi{W}
= \tPi_0W + \chi_+^*[W]\chi_+ + \chi_-^*[W]\chi_-, \qquad W \in \tD_{\hsf}.
\]
To verify this, use Lemma \ref{lem: main spectral theory lemma} and the uniqueness of the spectral projection discussed in Appendix \ref{app: elementary spectral theory}.
Then $\tPi{W} = \Pi{W}$ for all $W \in \tD_{\hsf}$, where $\Pi$ was defined in \eqref{eqn: Pi spectral proj}.
It then follows from the definition of $\tX$ and $\tD$ in \eqref{eqn: tilde spaces} that
\[
\tX_{\hsf} = \X_{\hsf},
\qquad
\tD_{\hsf} = \D_{\hsf},
\quadword{and}
\tY_{\hsf} = \Y_{\hsf} \subseteq \tX_{\hsf}.
\]
The spaces $\X_{\hsf}$ and $\D_{\hsf}$ were defined in \eqref{eqn: hyperbolic space X D} and $\tX_{\hsf}$ and $\tD_{\hsf}$ in \eqref{eqn: hyperbolic space tX tD}.

We use \eqref{eqn: tL1} to calculate
\begin{align*}
(\ind_{\tX}-\tPi)\tL_1(\mu)W
&= (\ind_{\X}-\Pi)\tL_1(\mu)W \\
&= (\ind_{\X}-\Pi)\left(\L_1(\mu)W-\mu\left(\frac{\Jscr_*(\mu)W}{1+\mu\Jscr_*(\mu)\chi_3}\right)\L_1(\mu)\chi_3\right)
\in \Y_{\hsf}
\end{align*}
and \eqref{eqn: tnl} to calculate 
\begin{align*}
(\ind_{\tX}-\tPi)\tnl(W,\mu)
&= (\ind_{\X}-\Pi)\tnl(W,\mu) \\
&= (\ind_{\X}-\Pi)\big(\nl(\T_{\mu}(W)\chi_3+W,\mu) + \mu\nl_{\T}(W,\mu)\L_1(\mu)\chi_3\big)
\in \Y_{\hsf},
\end{align*}
with both inclusions from Hypothesis~\ref{hypo: opt reg}.

Finally, $\restr{\tL_0}{\tD_{\hsf}} = \restr{\L_0}{\D_{\hsf}}$, so we can just directly import the assumptions of Hypothesis \ref{hypo: opt reg} on the triple $(\tD_{\hsf}, \tY_{\hsf},\tX_{\hsf}) = (\D_{\hsf},\X_{\hsf},\Y_{\hsf})$.
\qedhere
\end{enumerate}
\end{proof}

\subsection{The proof of Theorem \ref{thm: main abstract}}\label{sec: proof of thm: main abstract}
We will use the following term for brevity.

\begin{definition}\label{defn: unif bd}
Let $\Y$ be a Banach space.
A family of functions $f_{\mu}^{\alpha} \colon \R \to \Y$ defined for $\mu \in \mathscr{I}_1 \subseteq (0,1)$ and $\alpha \in \mathscr{I}_2 \subseteq \R$ is \defn{uniformly bounded} if
\[
\sup_{\substack{\mu \in \mathscr{I}_1 \\ \alpha \in \mathscr{I}_2 \\ X \in \R}}
\norm{f_{\mu}^{\alpha}(X)}_{\Y}
< \infty.
\]
We say that the family is \defn{uniformly bounded and $q$-localized} if the estimate above is true when $\norm{f_{\mu}^{\alpha}(X)}_{\Y}$ is replaced by $e^{q|X|}\norm{f_{\mu}^{\alpha}(X)}_{\Y}$, with $q \in \R$.
\end{definition}

Now we begin the proof of Theorem \ref{thm: main abstract}.
Lemma \ref{lem: ultimate tFmu props} allow us to invoke Theorem \ref{thm: Lombardi} to obtain nanopteron solutions to the fully reduced problem \eqref{eqn: fully reduced}.
Specifically, these solutions have the form $W = \Wsf_{\mu}^{\alpha}$, where
\begin{equation}\label{eqn: W nano}
\Wsf_{\mu}^{\alpha}(x)
= -\frac{3\Lfrak_0}{2\Qfrak_0}\mu\sech^2\left(\frac{\Lfrak_0^{1/2}\mu^{1/2}x}{2}\right)\chi_1
+ \mu^{3/2}\tUpsilon_{\mu}^{\alpha}(\mu^{1/2}x)
+ \alpha\mu\tPhi_{\mu}^{\alpha}(\Tup_{\mu}^{\alpha}(\mu^{1/2}x))
\end{equation}
and $\Tup_{\mu}^{\alpha}$ is defined in \eqref{eqn: Tsf-mu-alpha}.
Taking $W = \tPsf_{\mu}^{\alpha}$, where
\begin{equation}\label{eqn: tPsf-mu-alpha}
\tPsf_{\mu}^{\alpha}(x) 
:= \alpha\mu\tPhi_{\mu}^{\alpha}(\mu^{1/2}x),
\end{equation}
also solves \eqref{eqn: fully reduced}.

The maps $\tUpsilon_{\mu}^{\alpha}$ and $\tPhi_{\mu}^{\alpha}$ have the same properties as (their tilde-denuded counterparts) in Theorem \ref{thm: Lombardi}.
In particular, $\tUpsilon_{\mu}^{\alpha} \colon \R \to \tD_{\hsf}$ is uniformly bounded and $q$-localized, while $\tPhi_{\mu}^{\alpha} \colon \R \to \tD_{\hsf}$ is uniformly bounded.
Both maps are $\S$-reversible. 
Also,
\begin{equation}\label{eqn: chi0-star on tPsf-mu-alpha}
\chi_0^*[\tPsf_{\mu}^{\alpha}(X)]
= \chi_0^*[\tUpsilon_{\mu}^{\alpha}(X)]
= 0.
\end{equation}

All of the functions defined above are valid for $\mu$ and $\alpha$ in the ranges
\begin{equation}\label{eqn: Lombardi intervals}
0 < \mu < \tmu_*
\quadword{and}
\tAlpha_0\exp\left(-\frac{b\omega}{\Lfrak_0^{1/2}\mu^{1/2}}\right) \le \alpha \le \tAlpha_1,
\end{equation}
where $\tAlpha_1 > 0$ is independent of $b$ and $q$, while $\tmu_*$ and $\tAlpha_0$ may depend on $b$ and $q$.
We may assume that $\tAlpha_0$, $\tAlpha_1$, and $\tmu_*$ are small enough that $\T_{\mu}(\Wsf_{\mu}^{\alpha}(x))$ and $\T_{\mu}(\tPsf_{\mu}^{\alpha}(x))$ are defined for all $x \in \R$ and $\mu$ and $\alpha$  in the intervals \eqref{eqn: Lombardi intervals}.

It will also be convenient to abbreviate
\begin{equation}\label{eqn: tSigma-mu-alpha}
\tSigma_{\mu}^{\alpha}(X)
:= -\frac{3\Lfrak_0}{2\Qfrak_0}\sech^2\left(\frac{\Lfrak_0^{1/2}X}{2}\right)\chi_1
+ \mu^{1/2}\tUpsilon_{\mu}^{\alpha}(X),
\end{equation}
so that $\tSigma_{\mu}^{\alpha}$ is uniformly bounded and $q$-localized, while the solution $\Wsf_{\mu}^{\alpha}$ has the ``localized + (asymptotically) periodic structure''
\begin{equation}\label{eqn: W nano loc + per}
\Wsf_{\mu}^{\alpha}(x)
= \mu\tSigma_{\mu}^{\alpha}(\mu^{1/2}x) + \alpha\mu\tPhi_{\mu}^{\alpha}(\Tup_{\mu}^{\alpha}(\mu^{1/2}x)).
\end{equation}

In the following sections we will convert these solutions to the reduced problem \eqref{eqn: fully reduced} into solutions to the original problem \eqref{eqn: abstract ode} by undoing the changes of variables above.
The subsequent work is not particularly difficult, but we need to keep careful track of various powers of $\mu$ and $\alpha$.

\subsubsection{Construction of the ``periodic + growing'' solutions $\Psf_{\mu}^{\alpha}$}\label{sec: periodic + growing construction}
Per \eqref{eqn: W-T-int soln to orig}, since the map $\tPsf_{\mu}^{\alpha}(s) = \alpha\mu\tPhi_{\mu}^{\alpha}(\mu^{1/2}s)$ solves the reduced problem \eqref{eqn: fully reduced}, the map
\begin{multline}\label{eqn: Psf-mu-alpha origin}
\Psf_{\mu}^{\alpha}(x)
:= \left(\int_0^x \Gamma_{\mu}\big(\T_{\mu}(\alpha\mu\tPhi_{\mu}^{\alpha}(\mu^{1/2}s))\chi_3
+\alpha\mu\tPhi_{\mu}^{\alpha}(\mu^{1/2}s)\big)\ds\right)\chi_0
+ \T_{\mu}(\alpha\mu\tPhi_{\mu}^{\alpha}(\mu^{1/2}x))\chi_3 \\
+ \alpha\mu\tPhi_{\mu}^{\alpha}(\mu^{1/2}x)
\end{multline}
solves the original problem \eqref{eqn: abstract ode}.
We will rewrite this expression in the form \eqref{eqn: Psf-mu-alpha}.

We first put
\begin{equation}\label{eqn: tPhi-T}
\tPhi_{\mu}^{\alpha,\T}(X)
:= \alpha^{-1}\mu^{-2}\T_{\mu}(\alpha\mu\tPhi_{\mu}^{\alpha}(X)).
\end{equation}
Then $\tPhi_{\mu}^{\alpha,\T}$ is periodic, and it is uniformly bounded in the sense of Definition \ref{defn: unif bd} by part \ref{part: T quad est} of Lemma \ref{lem: big bad chi3 lemma}.
It is also even, since the $\S$-reversibility of $\tPhi_{\mu}^{\alpha}$ and part \ref{part: T and S} of Lemma \ref{lem: big bad chi3 lemma} imply
\begin{equation}\label{eqn: tPhi-mu-alpha-T even}
\alpha\mu^2\tPhi_{\mu}^{\alpha,\T}(-X)
= \T_{\mu}(\alpha\mu\tPhi_{\mu}^{\alpha}(-X))
= \T_{\mu}(\alpha\mu\S\tPhi_{\mu}^{\alpha}(X))
= \T_{\mu}(\alpha\mu\tPhi_{\mu}^{\alpha}(X)).
\end{equation}

Next, we define
\begin{equation}\label{eqn: Phi-mu-alpha nano defn}
\Phi_{\mu}^{\alpha}(X)
:= \mu\tPhi_{\mu}^{\alpha,\T}(X)\chi_3 + \tPhi_{\mu}^{\alpha}(X)
\end{equation}
to see that $\Phi_{\mu}^{\alpha}$ is uniformly bounded and
\begin{equation}\label{eqn: TPchi3 + P id}
\T_{\mu}(\alpha\mu\tPhi_{\mu}^{\alpha}(X))\chi_3 
+ \alpha\mu\tPhi_{\mu}^{\alpha}(X)
= \alpha\mu\Phi_{\mu}^{\alpha}(X).
\end{equation}

It remains for us to treat the integral in \eqref{eqn: Psf-mu-alpha origin}.
With $\Gamma_{\mu}$ and $\Gamma_{\mu}^*$ defined in \eqref{eqn: Gamma-mu}, we have
\[
\Gamma_{\mu}(\T_{\mu}\big(\alpha\mu\tPhi_{\mu}^{\alpha}(X))\chi_3
+\alpha\mu\tPhi_{\mu}^{\alpha}(X)\big)
= \alpha\mu\chi_1^*[\tPhi_{\mu}^{\alpha}(X)]
+ \Gamma_{\mu}^*\big(\T(\alpha\mu\tPhi_{\mu}^{\alpha}(X),\mu,0)\chi_3
+\alpha\mu\tPhi_{\mu}^{\alpha}(X)\big).
\]
Set
\begin{equation}\label{eqn: tPhi-mu-alpha-chi}
\tPhi_{\mu}^{\alpha,\chi}(X)
:= \alpha^{-1}\chi_1^*[\tPhi_{\mu}^{\alpha}(X)].
\end{equation}
This map is periodic and uniformly bounded by the expansions \eqref{eqn: Lombardi periodic1} and \eqref{eqn: Lombardi periodic2}.
It is also even, by a calculation like \eqref{eqn: tPhi-mu-alpha-T even}.

Next, let
\[
\tPhi_{\mu}^{\alpha,\Gamma}(X)
:= (\alpha\mu^2)^{-1}\Gamma_{\mu}^*\big(\T_{\mu}(\alpha\mu\tPhi_{\mu}^{\alpha}(X))\chi_3
+\alpha\mu\tPhi_{\mu}^{\alpha}(X)\big).
\]
Part \ref{part: Gamma-mu-star} of Lemma \ref{lem: Gamma-mu props} guarantees that $\tPhi_{\mu}^{\alpha,\Gamma}$ is  uniformly bounded, and it is also periodic.
Finally, it is even, since
\[
\Gamma_{\mu}\big(\T_{\mu}\big(\alpha\mu\tPhi_{\mu}^{\alpha}(-X))\chi_3 + \alpha\mu\tPhi_{\mu}^{\alpha}(-X)\big)
= \Gamma_{\mu}\big(\S(\T_{\mu}(\alpha\mu\tPhi_{\mu}^{\alpha}(X))\chi_3+\alpha\tPhi_{\mu}^{\alpha}(X)\big)
\]
by the $\S$-reversibility of $\tPhi_{\mu}^{\alpha}$, part \ref{part: T and S} of Lemma \ref{lem: big bad chi3 lemma}, part \ref{part: chik-star and S} of Lemma \ref{lem: odds and ends}, and part \ref{part: Gamma-mu and S} of Lemma \ref{lem: Gamma-mu props}.

We now abbreviate
\begin{equation}\label{eqn: Phi-mu-alpha-int}
\Phi_{\mu}^{\alpha,\int}(X)
:= \big(\alpha\mu(\alpha+\mu)\big)^{-1}\big(\alpha^2\mu\tPhi_{\mu}^{\alpha,\chi}(X) + \alpha\mu^2\tPhi_{\mu}^{\alpha,\Gamma}(X)\big),
\end{equation}
so that $\Phi_{\mu}^{\alpha,\int}$ is, once again, periodic, uniformly bounded, and even.
We integrate and find
\begin{equation}\label{eqn: Mu-T-Phi id}
\int_0^x \Gamma_{\mu}\big(\T_{\mu}(\alpha\mu\tPhi_{\mu}^{\alpha}(\mu^{1/2}s))\chi_3
+\alpha\mu\tPhi_{\mu}^{\alpha}(\mu^{1/2}s)\big)\ds
= \alpha\mu^{1/2}(\alpha+\mu)\int_0^{\mu^{1/2}x} \Phi_{\mu}^{\alpha,\int}(s) \ds.
\end{equation}

The estimate \eqref{eqn: main periodic est} follows from all the uniform bounds mentioned above.
The Lipschitz estimate \eqref{eqn: main periodic lip} follows from \eqref{eqn: Lombardi periodic1} and \eqref{eqn: Lombardi periodic frequency} and the various uniform bounds above.
Finally, the formula \eqref{eqn: Psf-mu-alpha} follows from \eqref{eqn: TPchi3 + P id} and \eqref{eqn: Mu-T-Phi id}.

\begin{remark}\label{rem: is it really periodic?}
For the map $\Psf_{\mu}^{\alpha}$ to be periodic, it is necessary and sufficient that the integral term in \eqref{eqn: Psf-mu-alpha origin} be periodic.
Since the integrand is periodic and real analytic, the periodicity of the integral is equivalent to the vanishing of the zeroth Fourier coefficient of the integrand.
However, the integrand here is quite complicated due to the three terms in $\Gamma_{\mu}$ from \eqref{eqn: Gamma-mu}, where now each of these terms is a composition involving both $\T_{\mu}$ and $\tPhi_{\mu}^{\alpha}$.
We know a great deal about these maps from Lemma \ref{lem: big bad chi3 lemma} and part \ref{part: Lombardi periodic} of Theorem \ref{thm: Lombardi}, but our present store of knowledge is not sufficient to help us divine whether or not the zeroth Fourier mode vanishes.
\end{remark}

\subsubsection{Construction of the full ``nanopteron + growing'' solutions $\Usf_{\mu}^{\alpha}$}\label{sec: nano + growing construction}
We may use \eqref{eqn: W-T-int soln to orig} to see that, with $\Wsf_{\mu}^{\alpha}$ defined in \eqref{eqn: W nano}, the map
\begin{equation}\label{eqn: UsF-mu-alpha origin}
\Usf_{\mu}^{\alpha}(x)
:= \left(\int_0^x \Gamma_{\mu}\big(\T_{\mu}(\Wsf_{\mu}^{\alpha}(s))\chi_3 + \Wsf_{\mu}^{\alpha}(s)\big) \ds\right)\chi_0
+ \T_{\mu}(\Wsf_{\mu}^{\alpha}(x))\chi_3
+ \Wsf_{\mu}^{\alpha}(x)
\end{equation}
solves the original problem \eqref{eqn: abstract ode}.
As in Section \ref{sec: periodic + growing construction}, we rewrite a number of the terms above to highlight certain powers of $\mu$ and $\alpha$, and now also with an eye toward separating exponentially localized terms from (asymptotically) periodic ones.

We begin this time with the integral in \eqref{eqn: UsF-mu-alpha origin}.
First, from the definitions of $\Gamma$ and $\Gamma_1$ in \eqref{eqn: Gamma-mu} and the definition of $\tPhi_{\mu}^{\alpha,\chi}$ in \eqref{eqn: tPhi-mu-alpha-chi}, we calculate
\begin{multline}\label{eqn: Mu(T(W))1}
\Gamma_{\mu}\big(\T_{\mu}(\Wsf_{\mu}^{\alpha}(s))\chi_3 + \Wsf_{\mu}^{\alpha}(s)\big)
= -\frac{3\Lfrak_0}{2\Qfrak_0}\mu\sech^2\left(\frac{\Lfrak_0^{1/2}\mu^{1/2}s}{2}\right)
+ \mu^{3/2}\chi_1^*[\tUpsilon_{\mu}^{\alpha}(\mu^{1/2}s)] \\
+ \alpha^2\mu\tPhi_{\mu}^{\alpha,\chi}(\Tup_{\mu}^{\alpha}(\mu^{1/2}s))
+ \Gamma_{\mu}^*\big(\T_{\mu}(\Wsf_{\mu}^{\alpha}(s))\chi_3 + \Wsf_{\mu}^{\alpha}(s)\big).
\end{multline}

Put
\[
\tSigma_{\mu}^{\alpha,\T}(X)
:= \mu^{-2}\big(\T_{\mu}(\mu\tSigma_{\mu}^{\alpha}(X) + \alpha\mu\tPhi_{\mu}^{\alpha}(X))-\T_{\mu}(\alpha\mu\tPhi_{\mu}^{\alpha}(X))\big),
\]
where $\tSigma_{\mu}^{\alpha}$ was defined in \eqref{eqn: tSigma-mu-alpha}.
Part \ref{part: T Lip est} of Lemma \ref{lem: big bad chi3 lemma} guarantees that $\tSigma_{\mu}^{\alpha,\T}$ is uniformly bounded.
The methods of the preceding section also show that $\tSigma_{\mu}^{\alpha,\T}$ is even.
Finally, the expansion of $\Wsf_{\mu}^{\alpha}$ in \eqref{eqn: W nano loc + per} gives
\[
\T_{\mu}(\Wsf_{\mu}^{\alpha}(x))
= \mu^2\tSigma_{\mu}^{\alpha,\T}(\mu^{1/2}x) + \alpha\mu^2\tPhi_{\mu}^{\alpha,\T}(\Tup_{\mu}^{\alpha}(\mu^{1/2}x)),
\]
where $\tPhi_{\mu}^{\alpha,\T}$ was defined in \eqref{eqn: tPhi-T}.

Now set
\begin{multline*}
\tSigma_{\mu}^{\alpha,\Gamma}(X)
:= \mu^{-2}\Gamma_{\mu}^*\big([\alpha\mu^2\tPhi_{\mu}^{\alpha,\T}(\Tup_{\mu}^{\alpha}(X))\chi_3+\alpha\mu\tPhi_{\mu}^{\alpha}(\Tup_{\mu}^{\alpha}(X))] + [\mu^2\tSigma_{\mu}^{\alpha,\T}(X)\chi_3 + \mu\tSigma_{\mu}^{\alpha}(X)]\big) \\
-  \mu^{-2}\Gamma_{\mu}^*\big(\alpha\mu^2\tPhi_{\mu}^{\alpha,\T}(\Tup_{\mu}^{\alpha}(X))\chi_3+\alpha\mu\tPhi_{\mu}^{\alpha}(\Tup_{\mu}^{\alpha}(X))\big).
\end{multline*}
Part \ref{part: Gamma-mu-star} of Lemma \ref{lem: Gamma-mu props} implies that $\tSigma_{\mu}^{\alpha,\Gamma}$ is uniformly bounded and $q$-localized in the sense of Definition \ref{defn: unif bd}, and the methods of Section \ref{sec: periodic + growing construction} show that it is even.
Last, abbreviate
\[
\tUpsilon_{\mu}^{\alpha,\int}(X)
:= \chi_1^*[\tUpsilon_{\mu}^{\alpha}(X)] + \mu^{1/2}\tSigma_{\mu}^{\alpha,\Gamma}(X),
\]
so $\tUpsilon_{\mu}^{\alpha,\int}$ is uniformly bounded and $q$-localized and even.

We may now use the work above and the definition of $\Phi_{\mu}^{\alpha,\int}$ in \eqref{eqn: Phi-mu-alpha-int} to rearrange \eqref{eqn: Mu(T(W))1} into
\begin{multline*}
\Gamma_{\mu}\big(\T_{\mu}(\Wsf_{\mu}^{\alpha}(s))\chi_3 + \Wsf_{\mu}^{\alpha}(s)\big)
= -\frac{3\Lfrak_0}{2\Qfrak_0}\mu\sech^2\left(\frac{\Lfrak_0^{1/2}\mu^{1/2}s}{2}\right)
+ \mu^{3/2}\Upsilon_{\mu}^{\alpha,\int}(\mu^{1/2}s) \\
+ \alpha\mu(\alpha+\mu)\Phi_{\mu}^{\alpha,\int}(\Tup_{\mu}^{\alpha}(\mu^{1/2}s)),
\end{multline*}
and so
\begin{multline}\label{eqn: Mu int W}
\int_0^x \Gamma_{\mu}\big(\T_{\mu}(\Wsf_{\mu}^{\alpha}(s))\chi_3 + \Wsf_{\mu}^{\alpha}(s)\big) \ds
= -\frac{3\Lfrak_0^{1/2}}{\Qfrak_0}\mu^{1/2}\tanh\left(\frac{\Lfrak_0^{1/2}\mu^{1/2}x}{2}\right)
+ \mu\int_0^{\mu^{1/2}x}\tUpsilon_{\mu}^{\alpha,\int}(s) \ds \\
+ \alpha\mu^{1/2}(\alpha+\mu)\int_0^{\mu^{1/2}x}\Phi_{\mu}^{\alpha,\int}(\Tup_{\mu}^{\alpha}(s)) \ds.
\end{multline}

We rewrite the remaining two integrals in \eqref{eqn: Mu int W} to isolate their leading-order behavior.
Part \ref{part: tanh3} of Lemma \ref{lem: the big tanh lemma} (with $\nu = \mu^{1/2}$ and $q_* = \Lfrak_0^{1/2}$) allows us to write
\begin{equation}\label{eqn: Upsilon1}
\int_0^X \tUpsilon_{\mu}^{\alpha,\int}(s) \ds
= \Lup_{\mu,1}^{\alpha}\tanh\left(\frac{\Lfrak_0^{1/2}X}{2}\right)
+ \Upsilon_{\mu}^{\alpha,1}(X),
\end{equation}
where $\Lup_{\mu}^{\alpha,\int}$ is uniformly bounded and $\Upsilon_{\mu}^{\alpha,1}$ is uniformly bounded and $q$-localized.
Part \ref{part: tanh2} of Lemma \ref{lem: the other big tanh lemma} (now with $q_* = \Lfrak_0^{1/2}/2$) allows us to write
\begin{equation}\label{eqn: Upsilon2}
\int_0^X \Phi_{\mu}^{\alpha,\int}(\Tup_{\mu}^{\alpha}(s)) \ds
= \int_0^{\Tup_{\mu}^{\alpha}(X)} \Phi_{\mu}^{\alpha,\int}(s) \ds
+ \mu^{1/2}\Lup_{\mu,2}^{\alpha,2}\tanh\left(\frac{\Lfrak_0^{1/2}X}{2}\right)
+ \mu^{1/2}\Upsilon_{\mu}^{\alpha,2}(X),
\end{equation}
where $\Lup_{\mu}^{\alpha,\infty}$ is uniformly bounded and $\Upsilon_{\mu}^{\alpha,2}$ is uniformly bounded and $\Lfrak_0^{1/2}$-localized, and therefore $q$-localized, since $0 < q < \Lfrak_0^{1/2}$.

Put
\[
\Lup_{\mu}^{\alpha}
:= \Lup_{\mu}^{\alpha,1} + \alpha(\alpha+\mu)\Lup_{\mu}^{\alpha,2}
\]
and
\begin{equation}\label{eqn: Upsilon-mu-alpha-0}
\Upsilon_{\mu}^{\alpha,0}(X)
:= \Upsilon_{\mu}^{\alpha,1}(X) + \alpha(\alpha+\mu)\Upsilon_{\mu}^{\alpha,2}(X)
\end{equation}
Then $\Lup_{\mu}^{\alpha}$ is uniformly bounded and $\Upsilon_{\mu}^{\alpha,0}$ is uniformly bounded and $q$-localized.
By inspection of \eqref{eqn: Upsilon1} and \eqref{eqn: Upsilon2}, we see that $\partial_X[\Upsilon_{\mu}^{\alpha,0}]$ is also uniformly bounded and $q$-localized.

Last, define
\begin{equation}\label{eqn: Upsilon-mu-alpha-star}
\Upsilon_{\mu}^{\alpha,*}(X)
:= \mu^{1/2}\tSigma_{\mu}^{\alpha,\T}(X)\chi_3 + \tUpsilon_{\mu}^{\alpha}(X).
\end{equation}
By \eqref{eqn: chi0-star on tPsf-mu-alpha}, we have $\chi_0^*[\Upsilon_{\mu}^{\alpha,*}(X)] = 0$.
Then the definition of $\Wsf_{\mu}^{\alpha}$ in \eqref{eqn: W nano}, the definition of $\Phi_{\mu}^{\alpha}$ in \eqref{eqn: Phi-mu-alpha nano defn}, and the integral identity \eqref{eqn: Mu int W} allow us to rewrite the solution $\Usf_{\mu}^{\alpha}$ from \eqref{eqn: UsF-mu-alpha origin} into the version \eqref{eqn: abstract drifting nanopteron} foretold in the statement of Theorem \ref{thm: main abstract}.

\subsubsection{Reversibility proofs}
We show that each of the terms in the definition of $\Usf_{\mu}^{\alpha}$ in \eqref{eqn: abstract drifting nanopteron} is $\S$-reversible.
We will need the following lemma, whose proof is a direct calculation using Definition \ref{defn: reversible} and Hypothesis \ref{hypo: eigenproj}--\ref{part: S and chik}.

\begin{lemma}\label{lem: odd/even reversible lemma}
Let $f \colon \R \to \R$ be a function and define $g_k(X) := f_k(X)\chi_k$, for $k=1,\ldots,4$, where $\chi_k$ is one of the generalized eigenvectors from Hypothesis \ref{hypo: eigenproj}.
Then $g_k$ is reversible if $f$ is an odd function and $k$ is an even integer, or if $f$ is an even function and $k$ is an odd integer.
\end{lemma}

It follows from \eqref{eqn: Upsilon1} and \eqref{eqn: Upsilon2} and the evenness of $\tUpsilon_{\mu}^{\alpha,\int}$ and $\Phi_{\mu}^{\alpha,\int}$ that the maps $\Upsilon_{\mu,1}^{\alpha,0}$ and $\Upsilon_{\mu}^{\alpha,1}$ are odd.
Then from its definition in \eqref{eqn: Upsilon-mu-alpha-0}, the map $\Upsilon_{\mu}^{\alpha,0}$ is also odd.
Consequently, the prefactor functions on $\chi_0$ in \eqref{eqn: abstract drifting nanopteron} are odd, and so Lemma \ref{lem: odd/even reversible lemma} shows that the $\chi_0$-term in \eqref{eqn: abstract drifting nanopteron} is $\S$-reversible.

Next, the $\sech^2$-prefactor on $\chi_1$ is even, so that term in \eqref{eqn: abstract drifting nanopteron} is $\S$-reversible.
Also, the map $\tSigma_{\mu}^{\alpha,\T}$ is even, so the $\chi_3$-term in the definition of $\Upsilon_{\mu}^{\alpha,*}$ in \eqref{eqn: Upsilon-mu-alpha-star} is $\S$-reversible.
Last, the map $\tUpsilon_{\mu}^{\alpha}$ is $\S$-reversible by Theorem \ref{thm: Lombardi}.
Thus $\Upsilon_{\mu}^{\alpha,*}$ is $\S$-reversible.

Since $\Tup_{\mu}^{\alpha}$ is odd, the last two terms in \eqref{eqn: abstract drifting nanopteron} will be $\S$-reversible if we know that $\Psf_{\mu}^{\alpha}$, defined in \eqref{eqn: Psf-mu-alpha}, is $\S$-reversible.
From its definition in \eqref{eqn: Phi-mu-alpha nano defn}, the map $\Phi_{\mu}^{\alpha}$ is $\S$-reversible by the same arguments as for $\Upsilon_{\mu}^{\alpha,*}$.
The $\chi_0$-term in \eqref{eqn: Psf-mu-alpha} is $\S$-reversible by the evenness of the integrand $\Phi_{\mu}^{\alpha,\int}$ and Lemma \ref{lem: odd/even reversible lemma}.
Thus $\Psf_{\mu}^{\alpha}$ is $\S$-reversible. 
\section{The Position Traveling Wave Problem}\label{sec: tw prob redux}

We will reformulate the position traveling wave problem \eqref{eqn: position tw eqns} for the general dimer into a first-order differential equation posed on an infinite-dimensional Banach space.
After choosing a small parameter that controls the distance of the wave speed from the speed of sound, we will show that this equation satisfies the hypotheses of Section \ref{sec: hypos}.
We are able to verify all the hypothesis except the symmetry for the general dimer; to obtain a symmetry, we must specialize to either a mass or a spring dimer.

We develop this first-order equation in Section \ref{sec: IK-cov} and introduce the near-sonic small parameter in Section \ref{sec: near-sonic}.
In between and afterward, we uncover a host of properties of this equation that are needed to verify the hypotheses.
Table \ref{table: hypo proof location} outlines precisely where the different hypotheses are considered for the concrete lattice problems.
Table \ref{table: notation} connects the notation of various operators and quantities in the abstract problem of Section \ref{sec: abstract} to the concrete lattice situation of this section.

\begin{table}
\[
\begin{tabular}{l|l}
\hline
{\bf{Hypothesis}} &{\bf{Discussion and Verification}} \\
\hline
\ref{hypo: F structure}: structural properties &Section \ref{sec: near-sonic} \\
\ref{hypo: symmetry}: symmetry &Section \ref{sec: symmetries} \\
\ref{hypo: spectrum}: center spectrum &Propositions \ref{prop: eigenprops} and \ref{prop: Lambda props}  \\
\ref{hypo: eigenproj}: zero eigenprojection &Sections \ref{sec: gen ev}, \ref{sec: eigenproj}, and \ref{sec: omega ev}  \\
\ref{hypo: first int}: first integral &Section \ref{sec: first int} \\
\ref{hypo: nondegen}: nondegeneracies &Section \ref{sec: nondegen} \\
\ref{hypo: opt reg}: optimal regularity &Section \ref{sec: opt reg} \\
\hline
\end{tabular}
\]

\

\caption{Location of the proofs of the various hypotheses}
\label{table: hypo proof location}
\end{table}

\begin{table}
\[
\begin{tabular}{Sc|Sc|Sl}
\hline
{\bf{Section \ref{sec: abstract}}} &{\bf{Section \ref{sec: tw prob redux}}} &{\bf{Definition in Section \ref{sec: tw prob redux}}} \\
\hline
$\X$ &$\X$ & \eqref{eqn: X defn} \\
\hline
$\D$ &$\D$ & \eqref{eqn: D defn} \\
\hline
$U$ &$\Ub$ &\eqref{eqn: Ub} \\
\hline
$\F(U,\mu)$ &$\F(\Ub,\mu;\kappa,\beta,\Vscr_1,\Vscr_2,w)$ &\eqref{eqn: F-ep-kappa-beta-w} \\
\hline
$\L_0$ &$\L_0(\kappa,w)$ &\eqref{eqn: star simplicity} \\
\hline
$\L_1(\mu)$ &$\L_1(\kappa,w)$ &\eqref{eqn: Lsuper1} \\
\hline
$\nl(U,\mu)$ &$\nl(\Ub,\mu;\kappa,\beta,\Vscr_1,\Vscr_2,w)$ &\eqref{eqn: concrete nl} \\
\hline
$\nl_0(U,\grave{U})$ &$\nl_0(\Ub,\grave{\Ub};\kappa,\beta,w)$ &\eqref{eqn: concrete nl0} \\
\hline
$\S$ &$\S_{\Mb}$ or $\S_{\Kb}$ &\eqref{eqn: md symm}, \eqref{eqn: sd symm} \\
\hline
$\omega$ &$\omega_*(\kappa,w)$ &\eqref{eqn: star simplicity}; Proposition \ref{prop: Lambda props}, parts \ref{part: Lambda-1}, \ref{part: Lambda-3}  \\
\hline
$\chi_k$ &$\chib_k(\kappa,w)$ &\eqref{eqn: chib0}, \eqref{eqn: chib1}, \eqref{eqn: chib2}, \eqref{eqn: chib3} \\
\hline
$\chi_k^*[U]$ &$\chib_k^*[\Ub;\kappa,w]$ &\eqref{eqn: chib0-star}, \eqref{eqn: chib1-star}, \eqref{eqn: chib2-star}, \eqref{eqn: chib3-star} \\
\hline
$\chi_{\pm}$ &$\chib_{\pm}(\kappa,w)$ &\eqref{eqn: chib-pm-omega} \\
\hline 
$\chi_{\pm}^*[U]$ &$\chib_{\pm}^*[\Ub;\kappa,w]$ &Section \ref{sec: omega ev} \\
\hline
$\J_{\mu}(U)$ &$\J_{\mu}(\Ub;\V_1,\V_2,\kappa,w)$ &\eqref{eqn: concrete first int} \\
\hline
$\Jscr_*(\mu)$ &$\Jscr_*(\mu;\kappa,w)$ &\eqref{eqn: Jscr concrete} \\
\hline
$\Lfrak_0$ &$\Lfrak_0(\kappa,w)$ &\eqref{eqn: lin nondegen concrete} \\
\hline
$\Qfrak_0$ &$\Qfrak_0(\kappa,\beta,w)$ &\eqref{eqn: quad nondegen concrete} \\
\hline
\end{tabular}
\]

\

\caption{Correspondence of notation between Sections \ref{sec: abstract} and \ref{sec: tw prob redux}}
\label{table: notation}
\end{table}

\subsection{The Iooss--Kirchg\"{a}ssner change of variables}\label{sec: IK-cov}
Recall that $p_1$ and $p_2$ are the traveling wave profiles in position coordinates, and $p_1$ and $p_2$ must solve \eqref{eqn: position tw eqns}.
Following Iooss and Kirchg\"{a}ssner \cite[Sec.\@ 2]{iooss-kirchgassner}, we set $\xi_j = p_j'$ and then define functions $P_j$ on $\R \times [-1,1]$ by
\begin{equation}\label{eqn: Pj}
P_j(x,v)
:= p_j(x+v).
\end{equation}
Observe that 
\[
P_j(x,0) = p_j(x+0) = p_j(x)
\quadword{and}
\partial_x[P_j](x,v) = p_j(x+v) = \partial_v[P_j](x,v).
\]
We collect these new coordinates into one vector as
\begin{equation}\label{eqn: Ub}
\Ub(x)
:= \big(p_1(x),p_2(x),\xi_1(x),\xi_2(x),P_1(x,\cdot),P_2(x,\cdot)\big),
\end{equation}
so that $\Ub(x) \in \X$, where 
\begin{equation}\label{eqn: X defn}
\X := \set{(p_1,p_2,\xi_1,\xi_2,P_1,P_2) \in \R^4 \times \Cal([-1,1]) \times \Cal([-1,1])}{P_1(0) = p_1, \ P_2(0) = p_2}
\end{equation}
We will be rather cavalier about whether we write elements of $\X$ as ``row'' or ``column'' vectors.
We will also work on the space
\begin{equation}\label{eqn: D defn}
\D
:= \X \cap \big(\R^4 \times \Cal^1([-1,1]) \times \Cal^1([-1,1])\big).
\end{equation}

The spaces $\D$ and $\X$ are Banach spaces with the usual maximum norms.
For $\Ub \in \X$ we will denote sometimes its components by $(\Ub)_j$ for $j=1,\ldots,6$, so that if $\Ub = (p_1,p_2,\xi_1,\xi_2,P_1,P_2) \in \X$, then
\begin{equation}\label{eqn: component convention}
(\Ub)_1 := p_1, 
\quad
(\Ub)_2 := p_2, 
\quad
(\Ub)_3 := \xi_1,
\quad
(\Ub)_4 := \xi_2,
\quad
(\Ub)_5 := P_1,
\quad\text{and}\quad
(\Ub)_6 := P_2.
\end{equation}

We are almost ready to express the traveling wave problem \eqref{eqn: position tw eqns} as a first-order equation in $\X$.
Define the evaluation operators $\delta^{\pm1}$ by
\begin{equation}\label{eqn: delta defn}
(\delta^{\pm1}P_j)(x,v)
:= P_j(x,\pm1)
= p_j(x\pm1).
\end{equation}
If $f = f(v)$ is a function of a single variable $v$, we will still write $\delta^{\pm1}f := f(\pm1)$.

The traveling wave problem \eqref{eqn: position tw eqns} is then equivalent to
\begin{equation}\label{eqn: IK system}
\Ub'(x)
= \F(\Ub(x);\V_1,\V_2,w,c),
\end{equation}
where, for $\Ub = (p_1,p_2,\xi_1,\xi_2,P_1,P_2) \in \X$, we define
\begin{equation}\label{eqn: Fc}
\F(\Ub;\V_1,\V_2,w,c)
:= \begin{pmatrix*}
\xi_1 \\
\xi_2 \\
c^{-2}\V_1'(\delta^1P_2-p_1)-c^{-2}\V_2'(p_1-\delta^{-1}P_2) \\
c^{-2}w\V_2'(\delta^1P_1-p_2)-c^{-2}w\V_1'(p_2-\delta^{-1}P_1) \\
\partial_v[P_1] \\
\partial_v[P_2]
\end{pmatrix*}.
\end{equation}
The parameter $w > 0$ is the (reciprocal of) the mass ratio, per \eqref{eqn: dimer masses}, and the functions $\V_1$ and $\V_2$ are the spring potentials from \eqref{eqn: dimer springs}.

We can expand $\F$ as the superposition of linear, quadratic, and superquadratic terms:
\begin{equation}\label{eqn: Fc L+Q0+Q1}
\F(\Ub;c)
= \L(\kappa,w,c)\Ub + c^{-2}\nl_0(\Ub,\Ub;\beta,w) + c^{-2}\nl_1(\Ub;\Vscr_1,\Vscr_2,w).
\end{equation}
We define these terms as follows.
First, the linear operator $\L(\kappa,w,c)$ is the block diagonal matrix
\begin{equation}\label{eqn: Lc defn}
\L(\kappa,w,c)
:= \begin{bmatrix*}
0 &\ind &0 \\[10pt]
-c^{-2}(1+\kappa)\diag(1,w) &0 &c^{-2}\Delta(\kappa,w) \\
0 &0 &\partial_v\ind
\end{bmatrix*},
\qquad
\ind := \begin{bmatrix*}
1 &0 \\
0 &1
\end{bmatrix*},
\end{equation}
and the component operator $\Delta(\kappa,w)$ is
\begin{equation}\label{eqn: Delta defn}
\Delta(\kappa,w)
:= \begin{bmatrix*}
0 &\big(\delta^1+\kappa\delta^{-1}\big) \\
w\big(\kappa\delta^1+\delta^{-1}\big) &0
\end{bmatrix*}.
\end{equation}
Next, the quadratic term $\nl_0$ is
\begin{equation}\label{eqn: nl0 defn}
\nl_0(\Ub,\grave{\Ub};\kappa,\beta,w)
:= \begin{pmatrix*}
0 \\
0 \\
(\delta^1P_2-p_1)(\delta^1\grave{P}_2-\grave{p}_1) - \beta(p_1-\delta^{-1}P_2)(\grave{p}_1-\delta^{-1}\grave{P}_2) \\
w\beta(\delta^1P_1-p_2)(\delta^1\grave{P}_1-\grave{p}_2) - w(p_2-\delta^{-1}P_1)(\grave{p}_2-\delta^{-1}\grave{P}_1) \\
0 \\
0
\end{pmatrix*}.
\end{equation}
Here we are writing $\grave{\Ub} = (\grave{p}_1,\grave{p}_2,\grave{\xi}_1,\grave{\xi}_2,\grave{P}_1,\grave{P}_2)$.
Finally, the superquadratic term is
\begin{equation}\label{eqn: nl1 defn}
\nl_1(\Ub;\Vscr_1,\Vscr_2,w)
= \begin{pmatrix*}
0 \\
0 \\
\Vscr_1(\delta^1P_2-p_1)-\Vscr_2(p_1-\delta^{-1}P_2) \\
w\Vscr_2(\delta^1P_1-p_2)-w\Vscr_1(p_2-\delta^{-1}P_1) \\
\partial_v[P_1] \\
\partial_v[P_2]
\end{pmatrix*}.
\end{equation}
The functions $\Vscr_1$ and $\Vscr_2$ represent the superquadratic terms in the spring forces, per \eqref{eqn: dimer springs}.

The problem \eqref{eqn: IK system} is not the one to which we will apply our abstract theory from Section \ref{sec: abstract}.
In particular, we have not yet introduced a small parameter into \eqref{eqn: IK system}; we will do so in Section \ref{sec: near-sonic} by taking the wave speed $c$ to be sufficiently close to the speed of sound.
Before that, however, we develop some general properties of \eqref{eqn: IK system} that are valid for arbitrary $c$.

\subsection{Lattice symmetries}\label{sec: symmetries}
We first observe a symmetry at the level of the original position coordinates. 
For a mass dimer ($m_{j+2} = m_j$ and $\V_j = \V_0$ for all $j$), one can check that if $\{u_j\}_{j \in \Z}$ is a solution set for \eqref{eqn: original equations of motion}, then so is $\{-u_{-j}\}_{j \in \Z}$.
For a spring dimer ($\V_{j+2} = \V_j$ and $m_j = m_0$ for all $j$), the set $\{-u_{-j+1}\}_{j \in \Z}$ is a solution whenever $\{u_j\}_{j \in \Z}$ is.
This mismatch in symmetries between the mass and spring dimers appears to be a perennial leitmotif of the respective traveling wave problems.
For example, at the level of relative displacement traveling waves, with the traveling wave profiles denoted by $\varrho_1$ and $\varrho_2$ as in \eqref{eqn: rel disp tw ansatz}, the articles \cite{faver-wright, hoffman-wright, faver-hupkes-equal-mass} on the mass dimer assume that $\varrho_1$ is even and $\varrho_2$ is odd, while for the spring dimer \cite{faver-spring-dimer} takes both $\varrho_1$ and $\varrho_2$ to be even.

We now define operators $\S_{\Mb}$, $\S_{\Kb} \in \b(\X)$ such that $\S_{\Mb}^2 = \S_{\Kb}^2 = \ind_{\X}$,  $\S_{\Mb}$ anticommutes with the mass dimer's system, and $\S_{\Kb}$ anticommutes with the spring dimer's system.
That is, for the mass dimer, we need
\begin{equation}\label{eqn: mass dimer anticommute}
\begin{cases}
\L(1,w,c)\S_{\Mb}\Ub = -\S_{\Mb}\L(1,w,c)\Ub \\
\nl_0(\S_{\Mb}\Ub,\S_{\Mb}\Ub;1,w) = -\S_{\Mb}\nl_0(\Ub,\Ub;1,w) \\
\nl_1(\S_{\Mb}\Ub;\Vscr,\Vscr,w) = -S_{\Mb}\nl_1(\Ub,\Vscr,\Vscr,w).
\end{cases}
\end{equation}
For the spring dimer, we need
\begin{equation}\label{eqn: spring dimer anticommute}
\begin{cases}
\L(\kappa,1,c)\S_{\Kb}\Ub = -\S_{\Kb}\L(\kappa,1,c)\Ub \\
\nl_0(\S_{\Kb}\Ub,\S_{\Kb}\Ub;\beta,1) = -\S_{\Kb}\nl_0(\Ub,\Ub;\beta,1) \\
\nl_1(\S_{\Kb}\Ub;\Vscr_1,\Vscr_2,1) = -\S_{\Kb}\nl_1(\Ub;\Vscr_1,\Vscr_2,1).
\end{cases}
\end{equation}

To construct these symmetries, first define a ``reflection'' operator by
\begin{equation}\label{eqn: R refl}
(Rf)(v)
:= f(-v),
\end{equation}
where $f$ is any function defined on $[-1,1]$.
The chain rule tells us that $R$ and $\partial_v$ anticommute.
With $(\delta^{\pm1}f)(v) := f(\pm1)$ as in \eqref{eqn: delta defn}, we have $\delta^{\pm}R = \delta^{\mp}$.
Then the operators 
\begin{equation}\label{eqn: md symm}
\S_{\Mb}
:= \begin{bmatrix*}
-\ind &0 &0 \\
0 &\ind &0 \\
0 &0 &-R\ind
\end{bmatrix*}
\end{equation}
and
\begin{equation}\label{eqn: sd symm}
\S_{\Kb}
:= \begin{bmatrix*}
-\Jbb &0 &0 \\
0 &\Jbb &0 \\
0 &0 &-R\Jbb
\end{bmatrix*},
\qquad \Jbb
:= \begin{bmatrix*}
0 &1 \\
1 &0
\end{bmatrix*}
\end{equation}
satisfy \eqref{eqn: mass dimer anticommute} and \eqref{eqn: spring dimer anticommute}, respectively.
It is straightforward to check that the operator norms of $\S_{\Mb}$ and $\S_{\Kb}$ are both 1.

The structures of these symmetries are related to our earlier observation that if $\{u_j\}_{j \in \Z}$ solves \eqref{eqn: original equations of motion}, then $\{-u_{-j}\}_{j \in \Z}$ is also a solution set for the mass dimer, while $\{-u_{-j-1}\}_{j \in \Z}$ is a solution set for the spring dimer.
We view the reflection operator $R$ in both cases as arising from the switching of $j$ to $-j$, while the various factors of $-1$ on the diagonal in both $\S_{\Mb}$ and $S_{\Kb}$ come from the prefactor of $-1$ on $u_{-j}$ in the mass dimer and $u_{-j-1}$ in the spring dimer.
Last, the ``flip'' operator $\Jbb$ for the spring dimer arises from the offset index $-j-1$.

We note with considerable interest, and frustration, that there does not appear to be a general symmetry that anticommutes with the terms of \eqref{eqn: Fc} when we are not specifically in the situation of a mass dimer or a spring dimer.
To be clear, we have {\it{not}} proved nonexistence of such a symmetry.
However, in light of the different even-odd and even-even symmetries for the prior relative displacement problems, and, indeed, the asymmetric appearance of the general dimer in Figure \ref{fig: general dimer}, we think it is unlikely that such a symmetry exists.
This absence of symmetry is the chief reason why we report results separately in Theorems \ref{thm: main theorem mass dimer} and \ref{thm: main theorem spring dimer} for mass and spring dimers.
Nonetheless, we will be able to verify all of the non-symmetry hypotheses for the general dimer, which suggests that if a substitute could be found for the role of the symmetry in Lombardi's methods, our methods would readily extend to the general dimer.

\subsection{Spectral analysis of the operator $\L(\kappa,w,c)$}\label{sec: spectral analysis}
We exhaustively analyze the operator $\L(\kappa,w,c)$, defined in \eqref{eqn: Lc defn} as an operator in the space $\X$ from \eqref{eqn: X defn} with domain $\D$ from \eqref{eqn: D defn}.
In particular, we calculate its spectrum and a family of generalized eigenvectors associated with the eigenvalue 0.
We summarize all of our spectral-theoretic conventions in Appendix \ref{app: spectral theory}.

\subsubsection{A characterization of the spectrum of $\L(\kappa,w,c)$}\label{sec: spectrum}
Let $z \in \C$ and suppose that $\Ub = (p_1,p_2,\xi_1,\xi_2,P_1,P_2) \in \D$ and $\grave{\Ub} = (\grave{p}_1,\grave{p}_2,\grave{\xi}_1,\grave{\xi}_2,\grave{P}_1,\grave{P}_2) \in \X$
satisfy
\[
z\Ub-\L(\kappa,w,c)\Ub = \grave{\Ub}.
\]
We will find a formula for $\Ub$ in terms of $z$ and $\grave{\Ub}$; that is, we calculate the resolvent of $\L(\kappa,w,c)$.

Using the definition of $\L(\kappa,w,c)$ in \eqref{eqn: Lc defn}, we first see that $P_j$, $\grave{P}_j$, and $p_j$ must satisfy the initial value problems
\[
\begin{cases}
\partial_v[P_j] -z{P}_j = -\grave{P}_j \\
P_j(0) = p_j.
\end{cases}
\]
We solve them with Duhamel's formula:
\begin{equation}\label{eqn: P-resolvent}
P_j(v) 
= e^{z{v}}p_j + \I_{z}[\grave{P}_j](v),
\end{equation}
where, for $P \in \Cal([-1,1])$, we put
\begin{equation}\label{eqn: Ilambda}
\I_{z}[P](v)
:= -\int_0^v e^{z(v-s)}P(s) \ds, \qquad P \in \Cal([-1,1]).
\end{equation}
Thus to find $P_j$ it suffices to discover $p_j$.

Again using the definition of $\L(\kappa,w,c)$ in \eqref{eqn: Lc defn}, we see that $p_j$, $\xi_j$, and $\grave{p}_j$ must satisfy
\[
z{p}_j-\xi_j 
= \grave{p}_j,
\]
and so 
\begin{equation}\label{eqn: xi-resolvent}
\xi_j
= z{p}_j - \grave{p}_j.
\end{equation}
Then $p_j$, $P_j$, and $\grave{\xi}_j$ must satisfy
\begin{equation}\label{eqn: xi P p}
\begin{cases}
z\xi_1 + \left(\frac{1+\kappa}{c^2}\right)p_1 -\left(\frac{\delta^1+\kappa\delta^{-1}}{c^2}\right)P_2 = \grave{\xi}_1 \\
\\
z\xi_2 + \left(\frac{w(1+\kappa)}{c^2}\right)p_2 - \left(\frac{w(\kappa\delta^1+\delta^{-1})}{c^2}\right)P_1 = \grave{\xi}_2.
\end{cases}
\end{equation}

We calculate
\[
\delta^{\pm1}P_j
= P_j(\pm1)
= e^{\pm{z}}p_j(x) + \I_{z}[\grave{P}_j](\pm1)
\]
and replace $P_j$ and $\xi_j$ in \eqref{eqn: xi P p} with their $p_j$-dependent formulas to find
\begin{equation}\label{eqn: spectrum pre prod}
\begin{cases}
z(z{p}_1 - \grave{p}_1)
+ \left(\frac{1+\kappa}{c^2}\right)p_1
-\left(\frac{e^{z}+\kappa{e}^{-z}}{c^2}\right)p_2
- \frac{\delta^1\I_{z}[\grave{P}_2]+\kappa\delta^{-1}\I_{z}[\grave{P}_2]}{c^2} = \grave{\xi}_1 \\
\\
z(z{p}_2 - \grave{p}_2)
+ \left(\frac{w(1+\kappa)}{c^2}\right)p_2
- \left(\frac{\kappa{w}{e}^{z}+we^{-z}}{c^2}\right)p_1
- \frac{\kappa{w}\delta^{1}\I_{z}[\grave{P}_1]+w\delta^{-1}\I_{z}[\grave{P}_1]}{c^2} = \grave{\xi}_2.
\end{cases}
\end{equation}

After some rearrangements, this will be equivalent to a matrix-vector equation in $\C^2$.
We define a matrix $\M(z;\kappa,w,c) \in \C^{2 \times 2}$ by 
\begin{equation}\label{eqn: Mc defn}
\M(z;\kappa,w,c)
:= \begin{bmatrix*}
\big(c^2z^2+1+\kappa\big) &-\big(e^{z}+\kappa{e}^{-z}\big) \\[5pt]
-w\big(\kappa{e}^{z}+e^{-z}\big) &\big(c^2z^2+w(1+\kappa)\big)
\end{bmatrix*}
\end{equation}
and an operator $\B(z;\kappa,w,c) \colon \X \to \C^2$ by
\begin{equation}\label{eqn: Bc defn}
\B(z;\kappa,w,c)
\grave{\Ub}
:= \begin{pmatrix*}
c^2z\grave{p}_1 + c^2\grave{\xi}_1 + \delta^1\I_{z}[\grave{P}_2]+\kappa\delta^{-1}\I_{z}[\grave{P}_2] \\[5pt]
c^2z\grave{p}_2 + c^2\grave{\xi}_2 + \kappa{w}\delta^{1}\I_{z}[\grave{P}_1]+w\delta^{-1}\I_{z}[\grave{P}_1]
\end{pmatrix*}.
\end{equation}
Then with $\pb = (p_1,p_2) \in \C^2$, the system \eqref{eqn: spectrum pre prod} becomes
\begin{equation}\label{eqn: matrix-vector product}
\M(z;\kappa,w,c)\pb
= \B(z;\kappa,w,c)\grave{\Ub}.
\end{equation}
We can solve \eqref{eqn: matrix-vector product} uniquely for $\pb$ given $\grave{\Ub}$ if and only if 
\begin{equation}\label{eqn: Mc det}
\det\big(\M(z;\kappa,w,c)\big)
= c^4z^4+c^2(1+\kappa)(1+w)z^2+2\kappa{w}(1-\cosh(2z))
\ne 0.
\end{equation}

Since we can always solve for $P_j$ and $\xi_j$ via \eqref{eqn: P-resolvent} and \eqref{eqn: xi-resolvent}, we can solve $z\Ub - \L(\kappa,w,c)\Ub = \grave{\Ub}$ uniquely for $\Ub \in \D$ given $\grave{\Ub} \in \X$ if and only if we can solve \eqref{eqn: matrix-vector product} for $\pb \in \C^2$ given $\grave{\Ub} \in \X$.
And so we have the following characterization of the spectrum.

\begin{proposition}\label{prop: eigenprops}
The spectrum of $\L(\kappa,w,c) \colon \D \to \X$ is the set
\begin{equation}\label{eqn: spectrum}
\sigma(\L(\kappa,w,c)) 
= \set{z \in \C}{\det\big(\M(z;\kappa,w,c)\big) = 0}.
\end{equation}
Suppose $|c| > 1$.
Then the following hold.

\begin{enumerate}[label={\bf(\roman*)}, ref={(\roman*)}]

\item
The spectrum consists entirely of eigenvalues.

\item
Each eigenvalue is an isolated point of the spectrum, and the algebraic multiplicity of an eigenvalue equals its multiplicity as a root of the map $\det\big(\M(\cdot;\kappa,w,c)\big)$.

\item\label{part: ew geo simple}
Each eigenvalue is geometrically simple.

\item
The spectrum is symmetric with respect to the real and complex axes in $\C$: if $z \in \sigma(\L(\kappa,w,c))$, then $\overline{z}, -z \in \sigma(\L(\kappa,w,c))$, too.

\item\label{part: resolvent}
Let $\Ub = (p_1,p_2,\xi_1,\xi_2,P_1,P_2) \in \X$ and $z \in \rho(\L(\kappa,w,c))$.
The six components of the resolvent operator $\rhs(z;\kappa,w,c)\Ub := \big(z\ind_{\D}-\L(\kappa,w,c)\big)^{-1}\Ub$ are
\begin{align}
\begin{split}
\big(\rhs(z;\kappa,w,c)\Ub\big)_1
&= \frac{c^2z^2+w(1+\kappa)}{\det\big(\M(z;\kappa,w,c)\big)}\big(c^2\xi_1+c^2z{p}_1 + \delta^1\I_{z}[P_2] + \kappa\delta^{-1}\I_{z}[P_2]\big) \\
&+\frac{e^{z}+\kappa{e}^{-z}}{\det\big(\M(z;\kappa,w,c)\big)}\big(c^2\xi_2 + c^2z{p}_2 + \kappa{w}\delta^1\I_{z}[P_1] + w\delta^{-1}\I_{z}[P_1]\big) 
\end{split} \label{eqn: resolv comp1} \\
\nonumber \\
\begin{split}
\big(\rhs(z;\kappa,w,c)\Ub\big)_2
&= \frac{w(\kappa{e}^{z}+e^{-z})}{\det\big(\M(z;\kappa,w,c)\big)}\big(c^2\xi_1+c^2z{p}_1 + \delta^1\I_{z}[P_2] + \kappa\delta^{-1}\I_{z}[P_2]\big) \\
&+\frac{c^2z^2+1+\kappa}{\det\big(\M(z;\kappa,w,c)\big)}\big(c^2\xi_2 + c^2z{p}_2 + \kappa{w}\delta^1\I_{z}[P_1] + w\delta^{-1}\I_{z}[P_1]\big)
\end{split} \label{eqn: resolv comp2} \\
\nonumber \\
\big(\rhs(z;\kappa,w,c)\Ub\big)_3
&= z\big(\rhs(z;\kappa,w,c)\Ub\big)_1-p_1 \\
\nonumber \\
\big(\rhs(z;\kappa,w,c)\Ub\big)_4
&= z\big(\rhs(z;\kappa,w,c)\Ub\big)_2-p_2 \\
\nonumber \\
\big(\rhs(z;\kappa,w,c)\Ub\big)_5(v)
&= e^{zv}\big(\rhs(z;\kappa,w,c)\Ub\big)_1+\I_z[P_1], \ -1 \le v \le 1 \\
\nonumber \\
\big(\rhs(z;\kappa,w,c)\Ub\big)_6(v)
&= e^{zv}\big(\rhs(z;\kappa,w,c)\Ub\big)_2+\I_z[P_2], \ -1 \le v \le 1.
\end{align}
\end{enumerate}
\end{proposition}

\begin{proof}
For simplicity of notation in this proof, we suppress the dependence of (almost) all operators and functions on $c$, $\kappa$, and $w$; in particular, we just write $\M(z;c)$ for the matrix in \eqref{eqn: Mc defn} and $\L(c)$ instead of $\L(\kappa,w,c)$.
The set equality \eqref{eqn: spectrum} was established by the remarks preceding the statement of this proposition.

\begin{enumerate}[label={\bf(\roman*)}]

\item
Let $z \in \sigma(\L)$.
We first claim that if $|c| > 1$ and $z \in \sigma(\L)$, then 
\begin{equation}\label{eqn: spectrum claim}
c^2z^2 + 1+ \kappa
\ne 0.
\end{equation}
Otherwise, if $c^2z^2+1+\kappa = 0$, then we can solve for $z$ as
\[
z
= \pm{i}\frac{\sqrt{1+\kappa}}{|c|}.
\]
Since $z \in \sigma(\L)$, we have
\[
\det\left(\M\left(\frac{\sqrt{1+\kappa}}{|c|}\right);c\right)
= 0.
\]
Some straightforward algebraic rearrangements then reveal that
\begin{equation}\label{eqn: spectrum claim aux}
\cos\left(\frac{2\sqrt{1+\kappa}}{|c|}\right)
= -\left(\frac{\kappa^2+1}{2}\right)
< -1,
\end{equation}
which is impossible.

So, we may put
\begin{equation}\label{eqn: Eb-c}
p_1
= E(z;\kappa,c)
:= \frac{e^{z}+\kappa{e}^{-z}}{c^2z^2+1+\kappa}
\quadword{and}
\pb := (E(z;\kappa,c),1).
\end{equation}
Then $\M(z;c)\pb = 0$, and so with 
\begin{equation}\label{eqn: Ub-kappa-lambda}
\Eb
= \Eb(z;\kappa,c)
:= (E(z;\kappa,c), 1, zE(z;\kappa,c),z, E(z;\kappa,c)e^{z\cdot}, e^{z\cdot}), 
\end{equation}
it is a direct calculation that $\L\Eb=z\Eb$.
That is, $z$ is an eigenvalue of $\L$.

\item
Suppose that $z \in \C\setminus\sigma(\L) = \rho(\L)$.
Using \eqref{eqn: P-resolvent}, \eqref{eqn: xi-resolvent}, and \eqref{eqn: matrix-vector product} we can write the resolvent operator of $\L$ at $z$ in the form
\begin{equation}\label{eqn: resolvent0}
\rhs(z)
= \frac{1}{\det(\M(z;c))}\rhs_1(z) + \rhs_2(z),
\end{equation}
where $\rhs_1$, $\rhs_2 \colon \rho(\L) \to \b(\X)$ are analytic.
From its definition in \eqref{eqn: Mc defn}, we see that the mapping $\M(\cdot,z) \colon \C \to \C^{2 \times 2}$ is entire, and so $\det(\M(\cdot;c))$ is also entire.
Lemma \ref{lem: alg mult} then guarantees that the eigenvalues of $\L$ are isolated and that the algebraic multiplicity of an eigenvalue equals its multiplicity as a root of $\det(\M(\cdot;c))$.

\item
Fix $z \in \sigma(\L)$ and suppose $\L\Ub = z\Ub$ for some $\Ub = (p_1,p_2,\xi_1,\xi_2,P_1,P_2) \in \X$.
We use \eqref{eqn: matrix-vector product} with $\grave{\Ub} = 0$ to obtain $p_1 = E(z;\kappa,c)p_2$, with $E$ defined in \eqref{eqn: Eb-c}.
Then we use \eqref{eqn: P-resolvent} and \eqref{eqn: xi-resolvent}, again with $\grave{\Ub} = 0$, to solve for the other components of $\Ub$ as
\[
\xi_1 = zE(z;\kappa,c)p_2,
\quad
\xi_2 = z{p}_2,
\quad
P_1(v) = zE(z;\kappa,c)p_2e^{z{v}},
\quad\text{and}\quad
P_2(v) = p_2e^{z{v}}.
\]
That is,
\begin{equation}\label{eqn: eigenvector}
\Ub
= p_2\big(E(z;\kappa,c), 1, zE(z;\kappa,c), z, E(z;\kappa,c)e^{z\cdot}, e^{z\cdot}\big)
= p_2\Eb(z;\kappa,c)
\end{equation}
as in \eqref{eqn: Ub-kappa-lambda}, and so the eigenspace corresponding to $z$ is 1-dimensional.

\item
The determinant \eqref{eqn: Mc det} is even in $z$, so if $z \in \sigma(\L)$, then $-z \in \sigma(\L)$.
The coefficients in \eqref{eqn: Mc det} are real and $\overline{\cosh(2z)} = \cosh(2\overline{z})$, so if $z \in \sigma(\L)$, then $\overline{z} \in \sigma(\L)$.

\item
This follows by unraveling the calculations in \eqref{eqn: P-resolvent}, \eqref{eqn: xi-resolvent}, and \eqref{eqn: matrix-vector product}.
\qedhere
\end{enumerate}
\end{proof}

\subsubsection{The center spectrum of $\L(\kappa,w,c)$}
We study $\sigma(\L(\kappa,w,c)) \cap i\R$.
Given $k \in \R$, we have $ik \in \sigma(\L(\kappa,w,c))$ if and only if $k$ satisfies the ``dispersion relation'' 
\begin{equation}\label{eqn: Lambdac defn}
0
= \det(\M(ik;\kappa,w,c)) 
= \bunderbrace{c^4k^4-c^2(1+w)(1+\kappa)k^2+2\kappa{w}(1-\cos(2k))}{\Lambda(k;\kappa,w,c)}.
\end{equation}
Clearly $\Lambda(0;\kappa,w,c) = 0$; we are interested in the multiplicity of the root 0, and in the behavior of other roots of $\Lambda(\cdot;\kappa,w,c)$.

\begin{proposition}\label{prop: Lambda props}
Let $\kappa$, $w > 0$ with $\max\{\kappa,w\} > 1$.

\begin{enumerate}[label={\bf(\roman*)}, ref={(\roman*)}]

\item\label{part: Lambda-1}
There exists $c_-(\kappa,w) \in (0,1)$ such that if $|c| \ge c_-(\kappa,w)$, then there exists a unique $\omega_c = \omega_c(\kappa,w) > 0$ such that $\Lambda(k;\kappa,w,c) = 0$ if and only if $k =0$ or $k = \pm\omega_c(\kappa,w)$.
 
\item\label{part: Lambda-2}
The numbers $\pm\omega_c(\kappa,w)$ from part \ref{part: Lambda-1} are both simple roots of $\Lambda(\cdot;\kappa,w,c)$.

\item\label{part: Lambda-3}
Set
\begin{equation}\label{eqn: cs defn}
\cs(\kappa,w)
:= \sqrt{\frac{4\kappa{w}}{(1+\kappa)(1+w)}}.
\end{equation}
Then 0 has multiplicity 4 as a root of $\Lambda(\cdot;\kappa,w,\cs(\kappa,w))$ and multiplicity 2 as a root of $\Lambda(\cdot;\kappa,w,c)$ when $c \ne \cs(\kappa,w)$.

\item\label{part: Lambda-4}
There exists $\lambda_0(\kappa,w) > 0$ such that if $\M(z;\kappa,w,\cs(\kappa,w)) = 0$ with $z \not\in i\R$, then $|\re(z)| \ge \lambda_0(\kappa,w)$.
\end{enumerate}
\end{proposition}

\begin{proof}
For simplicity throughout the proof, we will again suppress most notational dependencies of quantities on $c$, $\kappa$, and $w$.
In particular, we write $\Lambda(k;c)$ instead of $\Lambda(k;\kappa,w,c)$ and $\cs$ for $\cs(\kappa,w)$.

\begin{enumerate}[label={\bf(\roman*)}]

\item
The equation $\Lambda(k;c) = 0$ is really a quadratic equation in the unknown $c^2k^2$, and so by the quadratic formula we have $\Lambda(k;c) = 0$ if and only if 
\begin{equation}\label{post qf}
c^2k^2-\tlambda_{\pm}(k) = 0,
\end{equation}
where
\begin{equation}\label{eqn: tlambda}
\tlambda_{\pm}(k)
:= \frac{(1+\kappa)(1+w)}{2} \pm \frac{\sqrt{(1+w)^2(1-\kappa)^2+4\kappa((1-w)^2+4w\cos^2(k))}}{2}.
\end{equation}
It was shown in \cite[Prop.\@ 2.2.1, part (vii)]{faver-dissertation} that if at least one of $w$, $\kappa$ is greater than 1, then there exists $c_- = c_-(\kappa,w) \in (0,1)$ such that if $|c| > c_-$, then there is a unique positive solution $k = \omega_c =  \omega_c(\kappa,w)$ to \eqref{post qf}. 

\item
It was also shown in \cite[Prop.\@ 2.2.1, part (vii)]{faver-dissertation} that 
\[
\inf_{|c| > c_-} |\partial_k[\Lambda](\omega_c;c)|
= \inf_{|c| > c_-}  |2c^2\omega_c-\partial_k[\tlambda_{\pm}](\omega_c)| 
> 0.
\]
Thus the roots $\pm\omega_c$ are always simple.

\item
Since $\Lambda(\cdot;c)$ is even, we have $\partial_k[\Lambda](0;c) = \partial_k^3[\Lambda](0;c) = 0$.
And so we need to show $\partial_k^2[\Lambda](0;\cs)=0$, $\partial_k^4[\Lambda](0;\cs) \ne 0$, and $\partial_k^2[\Lambda](0;\cs) \ne 0$ for $c \ne \cs$.

We calculate
\[
\partial_k^2[\Lambda](0;c)
= 8\kappa{w}-2c^2(1+w)(1+\kappa)
\]
and
\[
\partial_k^4[\Lambda](0;c)
= 24c^4-32\kappa{w}.
\]
We see at once, then, that $\partial_k^2[\Lambda](0;c) = 0$ if and only if $|c|=\cs$.
In particular, 0 is a double root of $\Lambda(\cdot;c)$ when $|c| > \cs$.

Next, taking $c = \cs$, we have
\begin{equation}\label{eqn: Lambda-4th-deriv}
\partial_k^4[\Lambda](0;\cs)
= \frac{-32\kappa{w}\big((1+\kappa)^2w^2+(2\kappa^2-8\kappa+2)w+(1+\kappa)^2\big)}{(1+w)^2(1+\kappa)^2}
\end{equation}
Consider the factor
\begin{equation}\label{eqn: Lambda-c-star-4}
(1+\kappa)^2w^2+(2\kappa^2-8\kappa+2)w+(1+\kappa)^2
\end{equation}
as a quadratic polynomial in $w$.
Its discriminant is
\[
-48\kappa(\kappa^2-\kappa+1).
\]
In turn, the discriminant of the quadratic factor $\kappa^2-\kappa+1$ is $-3$.
It follows that 
\[
-48\kappa(\kappa^2-\kappa+1) 
< 0
\]
for all $\kappa > 0$.
That is, the discriminant of \eqref{eqn: Lambda-c-star-4} as a quadratic in $w$ is negative.
Hence the quantity in \eqref{eqn: Lambda-c-star-4} is nonzero for all $w \in \R$ and, in fact, positive.
Returning to \eqref{eqn: Lambda-4th-deriv}, we conclude $\partial_k^4[\Lambda](0;\cs) < 0$, and so 0 indeed has multiplicity 4 as a root of $\Lambda(\cdot;\cs)$.

\item
The function $\M(\cdot;\cs)$ is (real) analytic, and so its zeros are isolated.
\qedhere
\end{enumerate}
\end{proof}

For slightly greater notational simplicity, we will now write
\begin{equation}\label{eqn: star simplicity}
\L_0(\kappa,w)
:= \L(\kappa,w,\cs(\kappa,w))
\quadword{and}
\omega_*(\kappa,w)
:= \omega_{\cs(\kappa,w)}(\kappa,w).
\end{equation}

\begin{remark}\label{rem: omega-c and predecessors}
The value $\omega_c(1,w)$ from part \ref{part: Lambda-1} of Proposition \ref{prop: Lambda props} is called $k_c$ in the mass dimer paper \cite{faver-wright}, where it appears in part (vi) of Lemma 2.1.
The value $\omega_c(\kappa,1)$ is called $\Omega_c$ in the spring dimer paper \cite{faver-spring-dimer}, where it appears in part (v) of Proposition 2.1.
\end{remark}

Part \ref{part: Lambda-3} of Proposition \ref{prop: Lambda props} describes in part how the eigenvalue at 0 of $\L_0(\kappa,w)$ perturbs when $|c|$ is slightly greater than $\cs(\kappa,w)$: its multiplicity goes down by 2 from 4.
The next proposition completes that description: the multiplicity of the 0 eigenvalue goes down from 4 to 2 and two real eigenvalues of multiplicity 1 split off.
This behavior is sketched in Figure \ref{fig: eigenvalues} (inspired by graphics like \cite[Fig.\@ 1]{venney-zimmer}).

\begin{proposition}\label{prop: real det roots}
Let $\kappa$, $w > 0$ with $\max\{\kappa,w\} > 1$.
There exist $\tc_*(\kappa,w) > 0$ and $C_{\xsf} > 0$ such that if $\cs(\kappa,w) < |c| < \tc_*(\kappa,w)$, then the following hold.

\begin{enumerate}[label={\bf(\roman*)}]

\item
There exists a unique $\xsf_c(\kappa,w) > 0$ such that for $x \in \R$, $\det(\M(x;\kappa,w,c)) = 0$ if and only if $x = 0$ or $x = \pm\xsf_c(\kappa,w)$.

\item
The roots $\pm\xsf_c(\kappa,w)$ are simple roots of $\det(\M(\cdot;\kappa,w,c))$.

\item
The roots $\pm\xsf_c(\kappa,w)$ satisfy $|\xsf_c(\kappa,w)-\xsf_{\cs}| < C_{\xsf}|c^{-2}-\cs^{-2}|$.
\end{enumerate}
\end{proposition}

The proof of this proposition is a careful, highly quantitative application of the intermediate value theorem, which we omit, since we do not use these eigenvalues $\pm\xsf_c(\kappa,w)$ any further.
In principle we should also be able to calculate the leading order behavior in the small value $|c-\cs(\kappa,w)|$ of $\xsf_c(\kappa,w)$ using the sort of perturbation argument in \cite[Sec.\@ 6.2]{kapitula-promislow}.

\begin{figure}
\[
\begin{tabular}{cc}
\adjustbox{valign=t}{\begin{tikzpicture}
\draw[very thick, <->] (-2.5,0)--(2.5,0)node[right]{$\R$};
\draw[very thick, <->] (0,-2)node[below]{$c = \cs$}--(0,2)node[right]{$i\R$};

\fill[blue] (0,1)node[black,right]{\ $+i\omega_{\cs}$} circle(.2);
\fill[blue] (0,-1)node[black,right]{\ $-i\omega_{\cs}$} circle(.2);

\fill[blue] (-.25,.25) rectangle (.25,-.25);

\end{tikzpicture}}
\qquad
&
\qquad
\adjustbox{valign=t}{\begin{tikzpicture}
\draw[very thick, <->] (-2.5,0)--(2.5,0)node[right]{$\R$};
\draw[very thick, <->] (0,-2)node[below]{$|c| \gtrsim \cs$} --(0,2)node[right]{$i\R$};

\fill[blue] (0,1.25)node[black,right]{\ $+i\big(\omega_{\cs} + \O(\mu)\big)$} circle(.2);
\fill[blue] (0,-1.25)node[black,right]{\ $-i\big(\omega_{\cs} + \O(\mu)\big)$} circle(.2);

\fill[blue] (1.5,0) circle(.2);
\node[below=2pt] at (2,0){$\sqrt{\mu}\mathsf{x}_{\cs} + \O(\mu)$};

\fill[blue] (-1.5,0) circle(.2);
\node[below=2pt] at (-2,0){$-\sqrt{\mu}\mathsf{x}_{\cs} + \O(\mu)$};

\fill[red] (-.25,-.25)--(0,.25)--(.25,-.25)--cycle;

\end{tikzpicture}
}
\end{tabular}
\]

\[
\begin{tabular}{lll lll lll}
\begin{tikzpicture}
\fill[blue] (-.15,.15) rectangle (.15,-.15);
\end{tikzpicture}
& Multiplicity 4 
&\qquad
\begin{tikzpicture}
\fill[red] (-.15,-.15)--(0,.15)--(.15,-.15)--cycle;
\end{tikzpicture}
& Multiplicity 2 
&\qquad
\begin{tikzpicture}
\fill[blue] (0,0) circle(.15);
\end{tikzpicture}
& Multiplicity 1
\end{tabular}
\]

\caption{Spectral behavior of $\L(\kappa,w,c)$ for $c = \cs(\kappa,w)$ and $|c| \gtrsim \cs(\kappa,w)$ with $\mu = \cs(\kappa,w)^{-2}-c^{-2}$.
For $|c| \gtrsim \cs(\kappa,w)$, the eigenvalue at 0 retains multiplicity 2 and two real eigenvalues of multiplicity 1 split off.}

\label{fig: eigenvalues}
\end{figure}
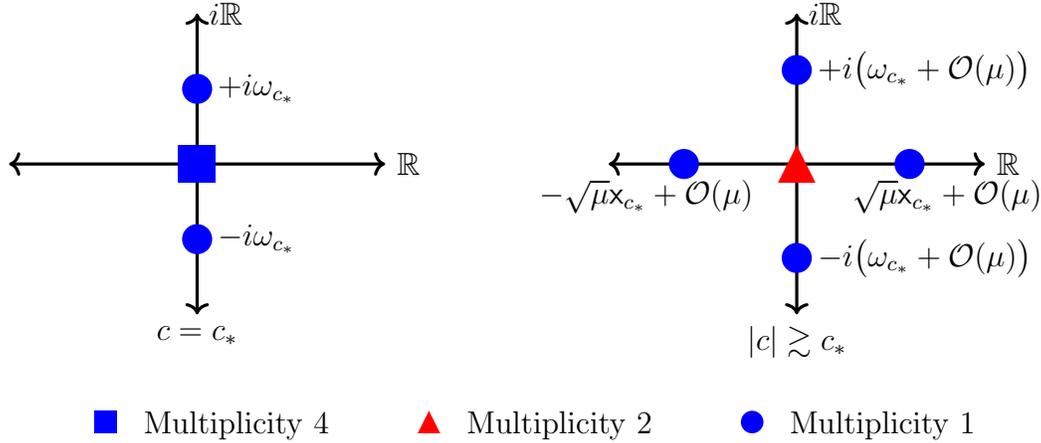

\subsubsection{The generalized eigenvectors of $\L_0(\kappa,w)$ corresponding to 0}\label{sec: gen ev}
Since the eigenvalue at 0 of $\L_0(\kappa,w)$ has algebraic multiplicity 4, we want four vectors $\chib_k(\kappa,w) \in \D$ such that 
\begin{equation}\label{eqn: gen ev prop1}
\L_0(\kappa,w)\chib_0(\kappa,w) = 0
\quadword{and}
\L_0(\kappa,w)\chib_{k+1}(\kappa,w) = \chib_k(\kappa,w), \ k = 0,1,2.
\end{equation}
Of course, such a Jordan chain always exists. 
However, in the case of a mass or spring dimer, we will also want our eigenvectors to enjoy a special interaction with the symmetries \eqref{eqn: md symm} and \eqref{eqn: sd symm}, namely
\begin{equation}\label{eqn: gen ev prop2}
\S_{\Mb}\chib_k(1,w) 
= (-1)^{k+1}\chib_k(1,w)
\quadword{and}
\S_{\Kb}\chib_k(\kappa,1)
= (-1)^{k+1}\chib_k(\kappa,1),
\end{equation}
and not every Jordan chain for 0 must satisfy \eqref{eqn: gen ev prop2}.

The components of these generalized eigenvectors are, unfortunately, rather overwhelming and depend in complicated ways on the parameters $\kappa$ and $w$. 
Recalling from \eqref{eqn: X defn} and \eqref{eqn: D defn} that the last two components of a vector in $\D$ are functions of the variable $v \in [-1,1]$, we specify what these functions are and then construct the other components out of them.
Broadly, these components resemble those of the generalized eigenvectors that our spatial dynamics predecessors have encountered, see, e.g., \cite[Eq.\@ (27)]{venney-zimmer}.
However, our generalized eigenvectors have quite a few more terms and depend wearingly on the parameters $\kappa$ and $w$.
They are as follows:
\begin{equation}\label{eqn: chib0}
\chib_0(\kappa,w)
:= (1,1, 0, 0, 1,1);
\end{equation}
\begin{equation}\label{eqn: chib1}
\chib_1(\kappa,w)
:= \big(\chi_{15}(0;\kappa,w), \chi_{16}(0;\kappa,w) , 1,1, \chi_{15}(\cdot;\kappa,w), \chi_{16}(\cdot;\kappa,w)\big),
\end{equation}
where
\[
\begin{cases}
\chi_{15}(v;\kappa,w)
:= v+\frac{1}{2}\left(\frac{1-\kappa}{1+\kappa}\right) \\
\\
\chi_{16}(v;\kappa,w)
:= v-\frac{1}{2}\left(\frac{1-\kappa}{1+\kappa}\right);
\end{cases}
\]
\begin{equation}\label{eqn: chib2}
\chib_2(\kappa,w)
:= \big(\chi_{25}(0;\kappa,w), \chi_{26}(0;\kappa,w), \chi_{15}(0;\kappa,w), \chi_{16}(0;\kappa,w), \chi_{25}(\cdot;\kappa,w), \chi_{35}(\cdot;\kappa,w)\big),
\end{equation}
where
\[
\begin{cases}
\chi_{25}(v;\kappa,w)
:= \frac{v^2}{2}+\frac{1}{2}\left(\frac{1-\kappa}{1+\kappa}\right)v + \frac{\kappa(1-w)}{(1+\kappa)^2(1+w)} \\
\\
\chi_{26}(v;\kappa,w)
:= \frac{v^2}{2} - \frac{1}{2}\left(\frac{1-\kappa}{1+\kappa}\right)v - \frac{\kappa(1-w)}{(1+\kappa)^2(1+w)};
\end{cases}
\]
and
\begin{equation}\label{eqn: chib3}
\chib_3(\kappa,w)
:= \big(\chi_{35}(0;\kappa,w), \chi_{36}(0;\kappa,w), \chi_{25}(0;\kappa,w), \chi_{26}(0;\kappa,w), \chi_{35}(\cdot;\kappa,w), \chi_{36}(\cdot;\kappa,w)\big),
\end{equation}
where
\[
\begin{cases}
\chi_{35}(v;\kappa,w)
= \frac{v^3}{6}+\frac{1}{4}\left(\frac{1-\kappa}{1+\kappa}\right)v^2 + \frac{\kappa(1-w)}{(1+\kappa)^2(1+w)}v + \frac{(\kappa-1)(\kappa^2+14\kappa+1)}{24(1+\kappa)^3} \\
\\
\chi_{36}(v;\kappa,w)
=  \frac{v^3}{6}-\frac{1}{4}\left(\frac{1-\kappa}{1+\kappa}\right)v^2 - \frac{\kappa(1-w)}{(1+\kappa)^2(1+w)}v - \left(\frac{(\kappa-1)(\kappa^2+14\kappa+1)}{24(1+\kappa)^3}\right).
\end{cases}
\]

Direct, and lengthy, calculations verify the identities \eqref{eqn: gen ev prop1} and \eqref{eqn: gen ev prop2}.
Similar, and also omitted, calculations establish the following lemma, which shows that $\big(\chib_0(\kappa,w),\chib_1(\kappa,w)\big)$ is always a Jordan chain for $\L(\kappa,w,c)$, regardless of the value of $c$, while the nonlinearity $\nl_0+\nl_1$ is invariant under translation by $\chib_0$.

\begin{lemma}\label{lemma: gen ev stuff}
Let $\kappa$, $w > 0$ and $\beta$, $c \in \R$.
Then

\begin{enumerate}[label={\bf(\roman*)}, ref={(\roman*)}]

\item
With $\L(\kappa,w,c)$ defined in \eqref{eqn: Lc defn}, we have
\[
\L(\kappa,w,c)\chib_0(\kappa,w) = 0
\quadword{and}
\L(\kappa,w,c)\chib_1(\kappa,w) = \chib_0(\kappa,w).
\]

\item
With $\nl_0$ defined in \eqref{eqn: nl0 defn} and $\nl_1$ in \eqref{eqn: nl1 defn}, we have
\[
\nl_0(\Ub+\gamma\chib_0(\kappa,w),\Ub+\gamma\chib_0(\kappa,w);\beta,w) = \nl_0(\Ub,\Ub;\beta,w)
\]
and
\[
\nl_1(\Ub+\gamma\chib_0(\kappa,w);\Vscr_1,\Vscr_2,w) = \nl_1(\Ub;\Vscr_1,\Vscr_2,w)
\]
for all $\Ub \in \X$ and $\gamma \in \R$.
\end{enumerate}
\end{lemma}

\begin{remark}\label{rem: pos trans invar}
We can view the high multiplicity of 0 as an eigenvalue and the structure of some of the generalized eigenvectors above as the inheritance of the system \eqref{eqn: IK system} of some properties of the original equations of motion \eqref{eqn: original equations of motion}.
Suppose that $\{u_j\}_{j \in \Z}$ is a solution set for the original equations of motion \eqref{eqn: original equations of motion}, not necessarily in the dimer or even polyatomic case.
Let $d_1$, $d_2 \in \R$ and $d_3 \in \Z$ and define $\tilde{u}_j(t) := u_{j+d_3}(t) + d_1t+d_2$.
Then $\{\tilde{u}_j\}_{j \in \Z}$ is also a solution set for \eqref{eqn: original equations of motion}.
In particular, the structure of $\chib_0$ and $\chib_1$ reflect the affine invariance of adding $d_1t+d_2$ here.
\end{remark}

\subsubsection{The projection onto the zero generalized eigenspace of $\L_0(\kappa,w)$}\label{sec: eigenproj}
Since 0 is an isolated eigenvalue of $\L_0(\kappa,w)$, the spectral projection $\Pi_0(\kappa,w)$ of $\L_0(\kappa,w)$ corresponding to 0 is defined on $\X$.
The image $\Pi_0(\kappa,w)(\X)$ is spanned by the four generalized eigenvectors $\chib_k(\kappa,w)$ from Section \ref{sec: gen ev}, so we can express this projection in the form
\begin{equation}\label{eqn: spectral proj FPUT}
\Pi_0(\kappa,w)\Ub 
= \sum_{k=0}^3 \chib_k^*[\Ub;\kappa,w]\chib_k(\kappa,w), 
\end{equation}
where $\chib_k^*[\cdot;\kappa,w] \colon \X \to \R$ are linear functionals satisfying 
\[
\chib_k^*[\chib_j(\kappa,w);\kappa,w]
= \begin{cases} 1, &j = k 
\\ 0 &j \ne k. 
\end{cases}
\]
We can find explicit, although complicated, formulas for the coefficient functionals $\chib_k^*$.

\begin{proposition}\label{prop: chik-star formulas}
Define 
\begin{equation}\label{eqn: a-4}
\asf_{-4}(\kappa,w)
:= \frac{4!}{\partial_k^4[\Lambda_{\cs(\kappa,w)}](0;\kappa,w)}
= -\frac{3(1+\kappa)^2(1+w)^2}{\kappa{w}\big((1+\kappa)^2w^2+2(\kappa^2-4\kappa+1)w + (1+\kappa)^2\big)}
\end{equation}
and
\begin{equation}\label{eqn: a-2}
\asf_{-2}(\kappa,w)
:= -\frac{4}{5}\left(\frac{\partial_k^6[\Lambda_{\cs(\kappa,w)}](0;\kappa,w)}{\big(\partial_k^4[\Lambda_{\cs(\kappa,w)}](0;\kappa,w)\big)^2}\right)
= -\frac{(1+\kappa)^4(1+w)^4}{10\kappa{w}((1+\kappa)^2w^2+2(\kappa^2-4\kappa+1)w+(1+\kappa)^2)}.
\end{equation}
For simplicity, write $\asf_{-4} = \asf_{-4}(\kappa,w)$, $\asf_{-2} = \asf_{-2}(\kappa,w)$, and $\cs = \cs(\kappa,w)$.

Then the functionals $\chib_k^*$ that control the zero eigenprojection \eqref{eqn: spectral proj FPUT} satisfy the following formulas for $\Ub = (p_1,p_2,\xi_1,\xi_2,P_1,P_2) \in \X$:
\begin{align}\label{eqn: chib0-star}
\chib_0^*[\Ub;\kappa,w]
&= \frac{\cs^2}{4}
\big(
(2\cs^2+w(1+\kappa))\asf_{-4}+4w(1+\kappa)\asf_{-2}
\big)
p_1 \\[5pt]\nonumber
&+ \frac{\cs^2}{4}\big(
(2\cs^2+1+\kappa)\asf_{-4}+4(1+\kappa)\asf_{-2}
\big)p_2 \\[5pt]\nonumber
&+ \frac{\cs^2w(\kappa-1)(\asf_{-4}+6\asf_{-2})}{12}\left(\xi_1-\frac{\xi_2}{w}\right) \\[5pt]\nonumber
&+ \frac{\asf_{-4}w}{2}\int_0^1 (1-s)^2\left[\frac{(1+\kappa)(1-s)}{3}\big(\kappa{P}_1(s)+P_2(s)\big) + \frac{\kappa-1}{2}\big(\kappa{P}_1(s)-P_2(s)\big)\right] \ds \\[5pt]\nonumber
&- \frac{\asf_{-4}w}{2}\int_{-1}^0 (1+s)^2\left[\frac{\kappa-1}{2}\big(P_1(s)-\kappa{P}_2(s)\big) + \frac{(1+\kappa)(1+s)}{3}\big(P_1(s)+\kappa{P}_2(s)\big)\right] \ds \\[5pt]\nonumber
&+ \frac{(\asf_{-4}+6\asf_{-2})(\kappa-1)w}{12}\left[\int_0^1 \big(\kappa{P}_1(s)-P_2(s)\big) \ds - \int_{-1}^0 \big(P_1(s)-\kappa{P}_2(s)\big) \ds\right] \\[5pt]\nonumber
&- \frac{w\big((2\cs^2+1+\kappa)\asf_{-4}+4(1+\kappa)\asf_{-2}\big)}{4}\left(\int_0^1 \kappa(1-s)P_1(s) \ds + \int_{-1}^0 (1+s)P_1(s) \ds \right) \\[5pt]\nonumber
&-\frac{\big((2\cs^2+w(1+\kappa))\asf_{-4}+4w(1+\kappa)\asf_{-2}\big)}{4}\left(\int_0^1 (1-s)P_2(s) \ds + \int_{-1}^0 \kappa(1+s)P_2(s) \ds\right);
\end{align}
\begin{align}\label{eqn: chib1-star}
\chib_1^*[\Ub;\kappa,w]
&= \frac{\asf_{-4}\cs^2(\kappa-1)}{2}\left(p_1-\frac{p_2}{w}\right) \\[5pt]\nonumber
&+ \frac{\cs^2\big((2\cs^2+w(1+\kappa))\asf_{-4} + 4w(1+\kappa)\asf_{-2}\big)}{4}\xi_1 \\[5pt]\nonumber
&+ \frac{\cs^2\big((2\cs^2+1+\kappa)\asf_{-4}+4(1+\kappa)\asf_{-2}\big)}{4}\xi_2 \\[5pt]\nonumber
&+ \frac{\asf_{-4}w}{2}\int_0^1 (1-s)\bigg[(\kappa-1)\big(\kappa{P}_1(s)-P_2(s)\big) + (1+\kappa)(s-1)\big(\kappa{P}_1(s)+P_2(s)\big)\bigg] \ds \\[5pt]\nonumber
&+ \frac{\asf_{-4}w}{2}\int_{-1}^0 (1+s)\bigg[(\kappa-1)\big(P_1(s)-\kappa{P}_2(s)\big) + (1+\kappa)(1+s)\big(P_1(s)+\kappa{P}_2(s)\big)\bigg] \ds \\[5pt]\nonumber
&-\frac{w(2\cs^2+1+\kappa)\asf_{-4} + 4(1+\kappa)\asf_{-2}}{4}\left(\int_0^1 \kappa{P}_1(s) \ds - \int_{-1}^0 P_1(s) \ds\right) \\[5pt]\nonumber
&- \frac{(2\cs^2+w(1+\kappa))\asf_{-4}+4w(1+\kappa)\asf_{-2}}{4}\left(\int_0^1 P_2(s) \ds - \int_{-1}^0 \kappa{P}_2(s) \ds \right);
\end{align}
\begin{align}\label{eqn: chib2-star}
\chib_2^*[\Ub;\kappa,w]
&= \asf_{-4}\cs^2w\left[(1+\kappa)\left(p_1+\frac{p_2}{w}\right) + \frac{\kappa-1}{2}\left(\xi_1-\frac{\xi_2}{w}\right)\right] \\[5pt]\nonumber
&+\asf_{-4}w\int_0^1 \left[(1+\kappa)(s-1)\big(\kappa{P}_1(s)+P_2(s)\big) + \frac{\kappa-1}{2}\big(\kappa{P}_1(s)-P_2(s)\big)\right] \ds \\[5pt]\nonumber
&-\asf_{-4}w\int_{-1}^0 \left[(1+\kappa)(1+s)\big(P_1(s)+\kappa{P}_2(s)\big) + \frac{\kappa-1}{2}\big(P_1(s)-\kappa{P}_2(s)\big)\right] \ds;
\end{align}
and
\begin{align}\label{eqn: chib3-star}
\chib_3^*[\Ub;\kappa,w]
&= \asf_{-4}w(1+\kappa)\cs^2\left(\xi_1 + \frac{\xi_2}{w}\right) \\[5pt]\nonumber
&+  \asf_{-4}w(1+\kappa)\left(\int_{-1}^0 \big(P_1(s)+\kappa{P}_2(s)\big) \ds - \int_0^1 \big(\kappa{P}_1(s)+P_2(s)\big) \ds\right).
\end{align}
\end{proposition}

\begin{proof}
As usual, we suppress most notational dependence on $\kappa$, $\beta$, and $w$ in this proof.
Since 0 has algebraic multiplicity 4 as an eigenvalue of $\L_0$, for $z \in \rho(\L_0)$ we have the Laurent series expansion at the pole $z=0$ \cite[eqn.\@ (6.32)]{kato} of the resolvent operator $\rhs$ for $\L_0$:
\begin{equation}\label{eqn: resolvent Laurent series}
\rhs(z)
= \frac{\P^3}{z^4} + \frac{\P^2}{z^3} + \frac{\P}{z^2} + \frac{\Pi_0}{z} + \sum_{s=0}^{\infty} z^s\tilde{\rhs}^s,
\end{equation}
where $\Pi_0$ is the spectral projection \eqref{eqn: spectral proj FPUT}, $\P := \L_0\Pi_0$, and $\tilde{\rhs} \in \b(\X)$.
The operators $\Pi_0$ and $\P$ commute.

Recall from \eqref{eqn: component convention} our convention of denoting the six components of $\Ub \in \X$ by $(\Ub)_s$ for $j=1,\ldots,6$.
The generalized eigenvectors $\chib_j$ from \eqref{eqn: chib0}, \eqref{eqn: chib1}, \eqref{eqn: chib2}, and \eqref{eqn: chib3} satisfy
\[
(\chib_0)_1
= (\chib_0)_2
= 1
\]
and
\[
(\chib_k)_1
= -(\chib_k)_2, \ k = 1,2,3.
\]
Consequently,
\[
(\Pi_0\Ub)_1
+ (\Pi_0\Ub)_2
= \sum_{k=0}^3 \chib_k^*[\Ub](\chib_k)_1+ \sum_{k=0}^3 \chib_k^*[\Ub](\chib_k)_2 
= 2\chib_0^*[\Ub],
\]
and thus
\begin{equation}\label{eqn: chib0-star id}
\chib_0^*[\Ub]
= \frac{(\Pi_0\Ub)_1 + (\Pi_0\Ub)_2}{2}.
\end{equation}

As in part \ref{part: chi-k+1-star on L0} of Lemma \ref{lem: odds and ends}, we have the identities
\begin{equation}\label{eqn: gen-ev+1}
\chib_k^*[\L_0\Ub] = \chib_{k+1}^*[\Ub], \ k = 0,1,2,
\quadword{and}
\chib_3^*[\L_0\Ub] = 0.
\end{equation}
Using the formula \eqref{eqn: chib0-star id} after iterating \eqref{eqn: gen-ev+1} several times, along with the commutativity of $\P$ and $\Pi_0$, yields
\begin{equation}\label{eqn: chibj-star formula}
\chib_k^*[\Ub]
= \chib_0^*[\P^k\Ub]
= \frac{(\P^k\Ub)_1 + (\P^k\Ub)_2}{2}, \ k = 1,2,3.
\end{equation}

For $k = 1$, $2$, $3$, the number $(\P^k\Ub)_1$ is the coefficient of $z^{-(k+1)}$ in the Laurent series \eqref{eqn: resolvent Laurent series} of $(\rhs(z)\Ub)_1$, while $(\Pi_0\Ub)_1$ is the coefficient of $z^{-1}$ in this series.
We can also calculate the Laurent series for $(\rhs(z)\Ub)_1$ directly from the formula in \eqref{eqn: resolv comp1}.
This formula is the sum of two $\Ub$-dependent terms, and each term is the quotient of an entire function of $z$ divided by $\det\big(\M(z;\cs)\big)$.
We then equate the explicit $\Ub$-dependent Laurent coefficients with $(\P^k\Ub)_1$ or $(\Pi_0\Ub)_1$, as appropriate, and use the formula \eqref{eqn: chibj-star formula} to conclude the values of $\chib_k^*$.
\end{proof}

\begin{remark}\label{rem: asf-4 is negative}
By the proof of part \ref{part: Lambda-4} of Proposition \ref{prop: Lambda props}, we have $\partial_k^4[\Lambda](0;\kappa,w,\cs(\kappa,w)) < 0$, and so, from its definition in \eqref{eqn: a-4}, we see that $\asf_{-4}(\kappa,w) < 0$.
We will need the negativity of $\asf_{-4}(\kappa,w)$ in Section \ref{sec: lin nondegen concrete}.
\end{remark}

\subsubsection{The eigenvectors of $\L_0(\kappa,w)$ corresponding to the eigenvalues $\pm{i}\omega_*(\kappa,w)$}\label{sec: omega ev}
These eigenvalues are simple by part \ref{part: Lambda-2} of Proposition \ref{prop: Lambda props} and \ref{part: ew geo simple} of Proposition \ref{prop: eigenprops}.
By part \ref{part: Lambda-1} of Proposition \ref{prop: Lambda props}, these are the only eigenvalues of $\L_0(\kappa,w)$ on the imaginary axis.
The one-dimensional eigenspaces corresponding to these eigenvalues are, by \eqref{eqn: eigenvector}, spanned by the vectors
\begin{equation}\label{eqn: chib-pm-omega}
\chib_{\pm}(\kappa,w)
:= \Eb(\pm{i}\omega_*(\kappa,w);\kappa,\cs(\kappa,w)),
\end{equation}
where $\Eb$ was defined in \eqref{eqn: Ub-kappa-lambda}.

The spectral projections of $\L_0(\kappa,w)$ corresponding to $\pm{i}\omega_*(\kappa,w)$ then have the form $\chib_{\pm}^*[\cdot;\kappa,w]\chib_{\pm}(\kappa,w)$ for certain functionals $\chib_{\pm}^*[\cdot;\kappa,w]$.
We could calculate formulas for these functionals like those in Proposition \ref{prop: chik-star formulas} using the Laurent series of the resolvent of $\L_0(\kappa,w)$ expanded at $\pm{i}\omega_*(\kappa,w)$ and the formulas for the resolvent in part \ref{part: resolvent} of Proposition \ref{prop: eigenprops}.
However, we will not explicitly employ the functionals $\chib_{\pm}^*[\cdot;\kappa,w]$ in our subsequent work, and so we omit further calculations.

\subsection{The near-sonic small parameter}\label{sec: near-sonic}
Now we are ready to introduce a small parameter $\mu$ into the differential equation \eqref{eqn: IK system}.
Because of the factor of $c^{-2}$ that appears in the linear and quadratic terms of our traveling wave system in Section \ref{sec: IK-cov}, we take $\mu$ to satisfy
\[
\frac{1}{c^2}
= \frac{1}{\cs(\kappa,w)^2} - \mu,
\qquad
0 \le \mu \le \frac{1}{2\cs(\kappa,w)^2} = \mu_0(\kappa,w).
\]
Equivalently, the wave speed is now
\begin{equation}\label{eqn: cnu}
\cmu(\kappa,w)^2
:= \frac{\cs(\kappa,w)^2}{1-\mu\cs(\kappa,w)^2}
= \cs(\kappa,w)^2 + \O(\mu). 
\end{equation}

From \eqref{eqn: Lc defn}, we can then express
\begin{equation}\label{eqn: Lsuper1}
\L(\kappa,w,\cmu(\kappa,w))
= \L_0(\kappa,w) + \mu\L_1(\kappa,w).
\end{equation}
The operator $\L_0(\kappa,w)$ was defined in \eqref{eqn: star simplicity}, while the new operator $\L_1(\kappa,w)$ is
\begin{equation}\label{eqn: L1}
\L_1(\kappa,w) := \begin{bmatrix*}
0 &0 &0 \\
(1+\kappa)\diag(1,w) &0 &-\Delta(\kappa,w) \\
0 &0 &0
\end{bmatrix*},
\end{equation}
The component operator $\Delta$ was defined in \eqref{eqn: Delta defn}.

For the nonlinear terms, first put
\begin{equation}\label{eqn: concrete nl0}
\nl_0(\Ub,\grave{\Ub};\kappa,\beta,w)
:= \cs(\kappa,w)^{-2}\nl_0(\Ub,\grave{\Ub};\beta,w),
\end{equation}
where $\nl_0(\Ub,\grave{\Ub};\beta,w)$ was defined in \eqref{eqn: nl0 defn}.
Then set
\begin{equation}\label{eqn: concrete nl}
\nl(\Ub,\mu;\kappa,\beta,\Vscr_1,\Vscr_2,w)
:= -\mu\nl_0(\Ub,\Ub;\beta,w) + \big(\cs(\kappa,w)^{-2}-\mu\big)\nl_1(\Ub;\Vscr_1,\Vscr_2,w),
\end{equation}
where $\nl_1$ was defined in \eqref{eqn: nl1 defn}.

With all of these definitions, our problem \eqref{eqn: IK system} is equivalent to
\begin{equation}\label{eqn: concrete final}
\Ub'(x) 
= \F(\Ub(x),\mu;\kappa,\beta,\Vscr_1,\Vscr_2,w),
\end{equation}
where 
\begin{equation}\label{eqn: F-ep-kappa-beta-w}
\F(\Ub,\mu;\kappa,\beta,\Vscr_1,\Vscr_2,w)
:=  \L_0(\kappa,w)\Ub + \mu\L_1(\kappa,w)\Ub + \nl(\Ub,\mu;\kappa,\beta,\Vscr_1,\Vscr_2,w).
\end{equation}

We will apply the abstract theory of Section \ref{sec: abstract} to the problem \eqref{eqn: F-ep-kappa-beta-w}.
To do so, we must check that Hypotheses \ref{hypo: F structure} through \ref{hypo: opt reg} hold.
By inspection, it is obvious that $\L_0(\kappa,w)$ and $\L_1(\kappa,w)$ are bounded operators from $\D$ to $\X$, where these Banach spaces were defined in \eqref{eqn: X defn} and \eqref{eqn: D defn}.
The analyticity of $\nl$ in $\Ub$ and $\mu$ is also obvious, as are the bilinearity of $\nl_0$ and the cubicity of $\nl-\nl_0$.
This verifies Hypothesis \ref{hypo: F structure} for the equation \eqref{eqn: concrete final}.

Recall that in Section \ref{sec: symmetries} we constructed symmetries $\S_{\Mb}$ and $\S_{\Kb}$ for the mass and spring dimer versions of \eqref{eqn: Fc L+Q0+Q1}, respectively.
The proofs that $\L_1(1,w)$ and $\nl(\cdot,\mu;1,1,\Vscr,\Vscr,w)$ anticommute with $\S_{\Mb}$ and that $\L_1(\kappa,1)$ and $\nl(\cdot,\mu;\kappa,\beta,\Vscr_1,\Vscr_2,1)$ anticommute with $\S_{\Kb}$ are direct calculations, which we omit.
This verifies Hypothesis \ref{hypo: symmetry} in the separate mass and spring dimer cases.

We checked Hypothesis \ref{hypo: spectrum} on the center spectrum of $\L_0(\kappa,w)$ in Section \ref{sec: spectral analysis}, specifically in Proposition \ref{prop: Lambda props}.
Finally, in Sections \ref{sec: gen ev} and \ref{sec: eigenproj} we proved Hypothesis \ref{hypo: eigenproj} on the properties of the projection onto the zero eigenspace for $\L_0(\kappa,w)$.

As for the remaining hypotheses, we will verify Hypothesis \ref{hypo: first int} on the first integral in Section \ref{sec: first int}, Hypothesis \ref{hypo: nondegen} on the nondegeneracies in Section \ref{sec: nondegen}, and Hypothesis \ref{hypo: opt reg} on optimal regularity in Section \ref{sec: opt reg}.

\begin{remark}\label{rem: mass dimer chi0-star}
The functional $\chib_0^*[\cdot;1,w]$, defined in \eqref{eqn: chib0-star}, vanishes on vectors $\Ub \in \X$ of the form $\Ub = (0,0,\xi_1,\xi_2,0,0)$.
The subspace of all vectors in $\X$ of this form, of course, contain the images of $\X$ under $\L_1(\kappa,w)$ and $\nl$.
This is exactly what happens in the corresponding situations in \cite{iooss-james, iooss-fpu, venney-zimmer}.
For the mass dimer, then, we might be tempted to add to Hypothesis \ref{hypo: eigenproj} the conditions (using the language of Section \ref{sec: abstract}) 
\begin{equation}\label{eqn: a hypo we do not add}
\chi_0^*[\L_1(\mu)U] 
= \chi_0^*[\nl(U,\mu)] 
= 0
\end{equation} 
for all $U$ and $\mu$.
Doing so would simplify the function $\Gamma_{\mu}$ in \eqref{eqn: Gamma-mu} considerably.
But the situation when $\kappa \ne 1$ is not as nice, since, from their definitions in \eqref{eqn: a-4} and \eqref{eqn: a-2}, it turns out that
\[
\asf_{-4}(\kappa,w) + 6\asf_{-2}(\kappa,w)
> 0.
\]
Thus $\chib_0^*[\Ub;\kappa,w]$ is not automatically 0 for $\Ub = (0,0,\xi_1,\xi_2,0,0)$ and $\kappa$, $w$ arbitrary, and so we do not add a condition like \eqref{eqn: a hypo we do not add} to Hypothesis \ref{hypo: eigenproj}.
\end{remark}

\subsection{First integrals: verification of Hypothesis \ref{hypo: first int}}\label{sec: first int}
The original equations of motion \eqref{eqn: original equations of motion} form a Hamiltonian system; this is a key property of FPUT lattices in many treatments, see, e.g., \cite[Ch.\@ 1]{pankov}, \cite[Sec.\@ 5.1]{gmwz}, and \cite{friesecke-pego2, friesecke-pego3, friesecke-pego4}.
Unsurprisingly, then, the problem \eqref{eqn: original equations of motion} possesses conserved quantities.

The precise first integral that we will need emerges at the level of the original position traveling wave problem \eqref{eqn: position tw eqns}.
To motivate how this first integral will behave in Iooss--Kirchg\"{a}ssner variables, we perform the following manipulations.
Divide the second equation in \eqref{eqn: position tw eqns} by $w$ and add this rescaled equation to the first to find that the position traveling wave profiles $p_1$ and $p_2$ must satisfy
\begin{equation}\label{eqn: first int initial}
c^2\left(p_1'' + \frac{p_2''}{w}\right)
- (1-S^{-1})\big[\V_1'(S^1p_2-p_1)+\V_2'(S^1p_1-p_2)\big].
\end{equation}
For $f \in \Cal^1(\R)$, define
\begin{equation}\label{eqn: int ops first int}
(\I_-f)(x) 
:= \int_{-1}^0 f(x+s) \ds.
\end{equation}
After differentiating under the integral and invoking the fundamental theorem of calculus, we have
\[
\partial_x\I_- 
= 1-S^{-1}.
\]
Then \eqref{eqn: first int initial} reads
\[
\partial_x\left[c^2\left(p_1' + \frac{p_2'}{w}\right)-\I_-\big[\V_1'(S^1p_2-p_1)+\V_2'(S^1p_1-p_2)\big]\right]
= 0,
\]
and so the map
\[
x
\mapsto
c^2\left(p_1'(x)+\frac{p_2'(x)}{w}\right) 
- \int_{-1}^0 \big[\V_1'(p_2(x+s+1)-p_1(x+s))+\V_2'(p_1(x+s+1)-p_2(x+s))\big] \ds
\]
is constant when $p_1$ and $p_2$ solve \eqref{eqn: position tw eqns}.

We can translate this into a first integral for the equation \eqref{eqn: IK system} using the Iooss--Kirchg\"{a}ssner variables from Section \ref{sec: IK-cov}.
The key observation is that $p_j(x+s+1) = P_j(x,s+1)$, per \eqref{eqn: Pj}.
We define this first integral precisely and present some of its fundamental properties.

\begin{proposition}\label{prop: first integral}
Given $\Ub = (p_1,p_2,\xi_1,\xi_2,P_1,P_2) \in \X$, set
\begin{equation}\label{eqn: first integral defn}
\J(\Ub;\V_1,\V_2,w,c)
:= c^2\left(\xi_1+\frac{\xi_2}{w}\right)
-\int_{-1}^0 \big[\V_1'(P_2(s+1)-P_1(s))+\V_2'(P_1(s+1)-P_2(s))\big] \ds.
\end{equation}

\begin{enumerate}[label={\bf(\roman*)}, ref={(\roman*)}]

\item\label{part: Jc deriv}
The map $\X \times \R \colon (\Ub,c) \mapsto \J(\Ub;\V_1,\V_2,w,c)$ is analytic and, with $\F(\Ub;\V_1,\V_2,w,c)$ defined in \eqref{eqn: Fc}, 
\[
D\J(\Ub;\V_1,\V_2,w,c)\F(\Ub;\V_1,\V_2,w,c)
= 0
\]
for all $\Ub \in \X$.

\item\label{part: Jc deriv 0}
$D\J(0;\V_1,\V_2,w,c)\Ub 
= c^2\xi_1
+ \frac{c^2}{w}\xi_2
+\int_{-1}^0 \big(P_1(s) + \kappa{P}_2(s)\big) \ds
-\int_0^1 \big(\kappa{P}_1(s)+P_2(s)\big) \ds$

\item\label{part: Jc and S}
In the special cases of the mass and spring dimers, $\J$ is invariant under the symmetries from Section \ref{sec: symmetries}:
\[
\J(\S_{\Mb}\Ub;\V,\V,w,c) = \J(\Ub;\V,\V,w,c)
\quadword{and}
\J(\S_{\Kb}\Ub;\V_1,\V_2,1,c) = \J(\Ub;\V_1,\V_2,1,c),
\]
where $\S_{\Mb}$ was defined in \eqref{eqn: md symm} and $\S_{\Kb}$ in \eqref{eqn: sd symm}.

\item\label{part: Jc on chib0}
With $\chib_0(\kappa,w)$ defined in \eqref{eqn: chib0}, we have
\[
\J(\Ub + \mu\chib_0(\kappa,w);\V_1,\V_2,w,c) 
= \J(\Ub;\V_1,\V_2,w,c)
\]
for all $\Ub \in \X$ and $\mu \in \R$.

\item\label{part: Jscr-w}
Let
\begin{equation}\label{eqn: Jscr-w defn}
\Jscr_w\Ub
:= \xi_1 + \frac{\xi_2}{w}.
\end{equation}
Then $\Jscr_w$ is a bounded linear functional on $\X$ and 
\begin{equation}\label{eqn: Jscr-w id}
\J(\Ub;\V_1,\V_2,w,c)-\J(\Ub;\V_1,\V_2,w,\grave{c})
= (c^2-\grave{c}^2)\Jscr_w\Ub.
\end{equation}
\end{enumerate}
\end{proposition}

\begin{proof}

\begin{enumerate}[label={\bf(\roman*)}]

\item
For $\Ub = (p_1,p_2,\xi_1,\xi_2,P_1,P_2)$, $\grave{\Ub} = (\grave{p}_1,\grave{p}_2,\grave{\xi}_1,\grave{\xi}_2,\grave{P}_1,\grave{P}_2) \in \X$, it is straightforward to calculate
\begin{multline}\label{eqn: DJ}
D\J(\Ub;\V_1,\V_2,w,c)\grave{\Ub}
= c^2\left(\grave{\xi}_1+\frac{\grave{\xi}_2}{w}\right)
- \int_{-1}^0 \V_1''(P_2(s+1)-P_1(s))(\grave{P}_2(s+1)-\grave{P}_1(s)) \ds \\
- \int_{-1}^0 \V_2''(P_1(s+1)-P_2(s))(\grave{P}_1(s+1)-\grave{P}_2(s)) \ds.
\end{multline}
Now we take $\grave{\Ub} = \F(\Ub;\V_1,\V_2,w,c)$.
Per the definition of $\F$ in \eqref{eqn: Fc}, set 
\[
\grave{\xi}_1 = c^{-2}\V_1(P_2(1)-p_1)-c^{-2}\V_2'(p_1-P_2(-1)),
\]
\[
\grave{\xi}_2 = c^{-2}w\V_2'(P_1(1)-p_2)-c^{-2}w\V_1'(p_2-P_1(-1)),
\]
and $\grave{P}_j = P_j'$.
Then the first integral in \eqref{eqn: DJ} becomes
\begin{multline*}
\int_{-1}^0 \V_1''(P_2(s+1)-P_1(s))(\grave{P}_2(s+1)-\grave{P}_1(s)) \ds
= \int_{-1}^0 \partial_s [\V_1'(S^1P_2-P_1)] \ds \\
= \V_1'(P_2(1)-P_1(0))-\V_1'(P_2(0)-P_1(-1))
= \V_1'(P_2(1)-p_1)-\V_1'(p_2-P_1(-1)),
\end{multline*}
where we have used the definition of the space $\X$ to find $P_j(0) = p_j$.
Consequently,
\[
c^2\grave{\xi}_1-\int_{-1}^0 \V_1''(P_2(s+1)-P_1(s))(\grave{P}_2(s+1)-\grave{P}_1(s)) \ds
= 0.
\]
A similar cancelation occurs when we add $c^2\grave{\xi}_2/w$ to the second integral in \eqref{eqn: DJ}.

\item
This is a direct calculation using \eqref{eqn: DJ} and the identities $\V_1''(0) = 1$ and $\V_2''(0) = \kappa$, per \eqref{eqn: dimer springs}.
Substitute $u=s+1$ in the second integral to obtain the desired formula.

\item
For the mass dimer, using the definition of $\S_{\Mb}$ in \eqref{eqn: md symm}, we have
\[
\J(\S_{\Mb}\Ub;\V,\V,w,c)
= c^2\left(\xi+\frac{\xi_2}{w}\right)
+ \int_{-1}^0 \big[\V'(-P_2(-s-1)+P_1(-s))+\V'(-P_1(-s-1)+P_2(-s))\big] \ds.
\]
Substitute $u = -s-1$ in the integral to conclude $\J(\S_{\Mb}\Ub;\V,\V,w,c) = \J(\Ub;\V,\V,w,c)$.
The proof for the spring dimer uses exactly the same substitution.

\item
This is a direct calculation.

\item
This is a direct calculation.
\qedhere
\end{enumerate}
\end{proof}

This first integral $\J$ is not quite the right one to use in conjunction with Hypothesis \ref{hypo: first int}, since we want our first integral to interact in a special way with the functional $\chib_3^*[\cdot;\kappa,w]$.
Instead, we simultaneously rescale $\J$ by $\asf_{-4}(\kappa,w)$, as defined in \eqref{eqn: a-4}, and select the wave speed to be $c=\cmu(\kappa,w)$.
Define
\begin{equation}\label{eqn: concrete first int}
\J_{\mu}(\Ub;\V_1,\V_2,\kappa,w)
:= w(1+\kappa)\asf_{-4}(\kappa,w)\J(\Ub;\V_1,\V_2,w,\cmu(\kappa,w)).
\end{equation}
It is then straightforward to calculate that the rescaled first integral $\J_{\mu}$ satisfies all the conditions of Hypothesis \ref{hypo: first int} for the near-sonic problem \eqref{eqn: concrete final}.
In particular, we may use the expression for $D\J(0;\V_1,\V_2,w,c)$ given in part \ref{part: Jc deriv 0} above and the formula for $\chib_3^*[\cdot;\kappa,w]$ given in \eqref{eqn: chib3-star} to find
\[
D\J_0(0;\V_1,\V_2,\kappa,w)\Ub
= w(1+\kappa)\asf_{-4}(\kappa,w)D\J(0;\V_1,\V_2,w,\cs(\kappa,w))\Ub
= \chib_3^*[\Ub;\kappa,w].
\]
Also, put
\begin{equation}\label{eqn: Jscr concrete}
\Jscr_*(\mu;\kappa,w)\Ub
:=w(1+\kappa)\asf_{-4}(\kappa,w)\cs(\kappa,w)^4\left(\frac{1}{1-\cs(\kappa,w)^2\mu}\right)\Jscr_w\Ub.
\end{equation}
Then the identity \eqref{eqn: Jscr-w id} reads
\[
\J_{\mu}(\Ub;\V_1,\V_2,\kappa,w) -\J_0(\Ub;\V_1,\V_2,\kappa,w)
= \mu\Jscr_*(\mu;\kappa,w)\Ub.
\]

\subsection{The nondegeneracy conditions: verification of Hypothesis \ref{hypo: nondegen}}\label{sec: nondegen}

\subsubsection{The linear nondegeneracy condition}\label{sec: lin nondegen concrete}
Put
\begin{equation}\label{eqn: lin nondegen concrete}
\Lfrak_0(\kappa,w)
:= \chib_2^*\big[\L_1(\kappa,w)\chib_1(\kappa,w);\kappa,w\big]
- \Jscr_*(0;\kappa,w)\chib_1(\kappa,w).
\end{equation}
We need to show $\Lfrak_0(\kappa,w) > 0$.
Using the definition of $\chib_1$ in \eqref{eqn: chib1} and $\L_1$ in \eqref{eqn: L1}, it is easy to show that 
\[
\L_1(\kappa,w)\chib_1(\kappa,w)
= 0.
\]
Next, the definition of $\Jscr_*$ in \eqref{eqn: Jscr concrete} implies
\[
\Jscr_*(0;\kappa,w)\chib_1(\kappa,w)
= w(1+\kappa)\asf_{-4}(\kappa,w)\cs(\kappa,w)^4\left(1+\frac{1}{w}\right)
< 0
\]
since $\asf_{-4}(\kappa,w) < 0$ by its definition in \eqref{eqn: a-4} and Remark \ref{rem: asf-4 is negative}.
Hence
\begin{equation}\label{eqn: Lfrak0-concrete}
\Lfrak_0(\kappa,w)
= -(1+\kappa)(1+w)\asf_{-4}(\kappa,w)\cs(\kappa,w)^4
> 0.
\end{equation}

\subsubsection{The quadratic nondegeneracy condition}
Put
\begin{equation}\label{eqn: quad nondegen concrete}
\Qfrak_0(\kappa,\beta,w)
:= 2\chib_2^*\big[\nl_0\big(\chib_1(\kappa,w),\chib_1(\kappa,w);\kappa,\beta,w\big);\kappa,w\big] -D^2\J_0(0;\V_1,\V_2,\kappa,w)\big[\chib_1(\kappa,w),\chib_1(\kappa,w)\big].
\end{equation}
We need to show $\Qfrak_0(\kappa,\beta,w) \ne 0$.

We recall that $\nl_0$ is defined in \eqref{eqn: concrete nl0}, $\J_0$ in \eqref{eqn: concrete first int}, and $\chib_1(\kappa,w)$ in \eqref{eqn: chib1}.
It is then straightforward to calculate
\[
\nl_0\big(\chib_1(\kappa,w),\chib_1(\kappa,w);\kappa,\beta,w\big)
= \frac{(\beta-\kappa^2)(1+w)}{(1+\kappa)\kappa{w}}
\begin{pmatrix*}
0 \\
0 \\
-1 \\
w \\
0 \\
0
\end{pmatrix*}.
\]
Next, the definition of $\chib_2^*[\cdot;\kappa,w]$ in \eqref{eqn: chib2-star} gives
\begin{equation}\label{eqn: chi2 on Q0}
2\chib_2^*\big[\nl_0\big(\chib_1(\kappa,w),\chib_1(\kappa,w);\kappa,\beta,w\big);\kappa,w\big]
= -\frac{8(\beta-\kappa^2)\asf_{-4}(\kappa,w)(\kappa-1)w}{(1+\kappa)^2},
\end{equation}
where $\asf_{-4}(\kappa,w) \ne 0$ was defined in \eqref{eqn: a-4}.

Finally, we can calculate $D^2\J_{\mu}$ using \eqref{eqn: concrete first int} and \eqref{eqn: DJ}.
For $\Ub = (p_1,p_2,\xi_1,\xi_2,P_1,P_2) \in \X$ in general, we find
\begin{multline*}
D^2\J(0;\V_1,\V_2,w,c)[\Ub,\Ub]
= 4\beta\int_{-1}^0 P_1(s+1)P_2(s) \ds
- 2\beta\int_{-1}^0 \big(P_1(s+1)^2 + P_2(s)^2\big) \ds \\
+ 4\int_0^1 P_1(s-1)P_2(s) \ds
- 2\int_0^1 \big(P_1(s-1)^2 + P_2(s)^2\big) \ds.
\end{multline*}
Then we evaluate this second derivative at the generalized eigenvector $\Ub = \chib_1(\kappa,w)$ defined in \eqref{eqn: chib1} and rescaling the result per \eqref{eqn: concrete first int} to find
\begin{equation}\label{eqn: D2J value}
D^2\J_0(0;\kappa,\beta,w)\chib_1(\kappa,w),\chib_1(\kappa,w)]
= -\frac{8(\beta+\kappa^2)w\asf_{-4}(\kappa,w)}{1+\kappa}.
\end{equation}
Subtracting the value in \ref{eqn: D2J value} from \eqref{eqn: chi2 on Q0}, we conclude
\begin{equation}\label{eqn: final quad nondegen}
\Qfrak_0(\kappa,\beta,w)
= \frac{16w\asf_{-4}(\kappa,w)(\beta+\kappa^3)}{(1+\kappa)^2}
\ne 0,
\end{equation}
provided that $\beta + \kappa^3 \ne 0$.
This, of course, leads to our requirement that $\beta\ne\kappa^3$ in Theorem~\ref{thm: main theorem spring dimer}.

\subsection{Optimal regularity: verification of Hypothesis \ref{hypo: opt reg}}\label{sec: opt reg}
We presume familiarity here with our nomenclature in Appendix \ref{app: opt reg}.
The central difficulty of verifying this hypotheses is that the operator $\L_0(\kappa,w)$ from \eqref{eqn: star simplicity} does not have the good resolvent estimate \eqref{eqn: Lombardi resolvent est} and so we cannot just apply Theorem \ref{thm: Lombardi opt reg} to conclude the hypothesis.
Indeed, for $k \in \Z$ such that $|k| > \omega_*(\kappa,w)$, define $\Ub^{(k)} = \big(p_1^{(k)},p_2^{(k)},\xi_1^{(k)},\xi_2^{(k)},P_1^{(k)},P_2^{(k)}\big) \in \X$ via $p_1^{(k)} = 1/2$, $p_2^{(k)} = \xi_1^{(k)} = \xi_2^{(k)} = 0$, $P_1^{(k)}(v) = 0$, and $P_2^{(k)}(v) = e^{ikv}/2$.
Then with the resolvent $\rhs(ik;\kappa,w,\cs(\kappa,w))$ defined in part \ref{part: resolvent} of Proposition \ref{prop: eigenprops}, we have
\[
\norm{\Ub^{(k)}}_{\X} = 1
\quadword{but}
\lim_{k \to \infty} \norm{\rhs(ik;\kappa,w,\cs(\kappa,w))\Ub^{(k)}}_{\X} \ne 0.
\]

Consequently, we must work much more explicitly with the (sub)optimal regularity problems.
Let
\[
\Y
:= \set{(0,0,\xi_1,\xi_2,0,0) \in \X}{\xi_1,\xi_2 \in \R},
\]
so $\Y \subseteq \D$ and also
\[
\L_1(\kappa,w)\Ub \in \Y
\quadword{and}
\nl(\Ub,\mu;\kappa,\beta,\Vscr_1,\Vscr_2,w) \in \Y
\]
for all $\Ub \in \X$, per the structure of these operators given in \eqref{eqn: L1} and \eqref{eqn: concrete nl}.
Thus if 
\begin{equation}\label{eqn: Y-hsf}
\Y_{\hsf}
:= (\ind_{\X}-\Pi)\Y,
\end{equation}
then
\[
(\ind_{\X}-\Pi)\L_1(\kappa,w)\Ub \in \Y_{\hsf}
\quadword{and}
(\ind_{\X}-\Pi)\nl(\Ub,\mu;\kappa,\beta,\Vscr_1,\Vscr_2,w) \in \Y_{\hsf}.
\]

The localized optimal, periodic optimal, and suboptimal regularity problems of Hypothesis \ref{hypo: opt reg} all have the following broad structure in their lattice incarnations.
Given a map $\Gb \colon \R \to \Y_{\hsf}$ and a scalar $\msf \in \R$, we need to solve a differential equation of the form
\begin{equation}\label{eqn: ur opt reg dimer}
\Ub'(x)
= \msf\L_0(\kappa,w)\Ub(x) + \Gb(x),
\end{equation}
where $\L_0(\kappa,w)$ was defined in \eqref{eqn: star simplicity}.
Of course, for each kind of problem we will have some more stringent requirements on $\Gb$, $\Ub$, and $\msf$.
In particular, for the localized problem we will need to extend the solution $\Ub$ from the real line to a small complex strip.
For now, we just sketch the common method of how one starts to solve \eqref{eqn: ur opt reg dimer}.

Write, as usual,
\[
\Ub(x) 
= \big(p_1(x),p_2(x),\xi_1(x),\xi_2(x),P_1(x,\cdot),P_2(x,\cdot)\big)
\]
and
\[
\Gb(x) 
= \big(g_1(x),g_2(x),h_1(x),h_2(x),G_1(x,\cdot),G_2(x,\cdot)\big).
\]
Then \eqref{eqn: ur opt reg dimer} requires $P_1$ and $P_2$ to satisfy the forced transport equations
\[
\begin{cases}
\partial_x[P_j]-\msf\partial_v[P_j] = G_j \\
P_j(x,0) = p_j(x).
\end{cases}
\]
The unique solution to this problem is
\begin{equation}\label{eqn: forced transport}
P_j(x,v)
= p_j(x+\msf^{-1}v)-\msf^{-1}\int_0^v G_j(x+\msf^{-1}(v-s),s) \ds.
\end{equation}

Abbreviate
\[
\pbit
:= (p_1,p_2,\xi_1,\xi_2),
\]
\[
\Deltab_{\msf}(\kappa,w)\pbit
:= \begin{pmatrix*}
\xi_1 \\
\xi_2 \\
-\cs(\kappa,w)^{-2}(1+\kappa)p_1+\cs(\kappa,w)^{-2}(S^{1/\msf}+\kappa{S}^{-1/\msf})p_2 \\
-\cs(\kappa,w)^{-2}w(1+\kappa)p_2 + \cs(\kappa,w)^{-2}w(\kappa{S}^{1/\msf}+S^{-1/\msf})p_1
\end{pmatrix*},
\]
\[
\Rfrak_3^{\msf}(\Gb;\kappa,w)(x)
:= -\frac{1}{\cs(\kappa,w)^2\msf}\left(\int_0^1 G_2(x+\msf^{-1}(1-s),s) \ds - \kappa\int_{-1}^0 G_2(x+\msf^{-1}(-1-s),s) \ds\right),
\]
\[
\Rfrak_4^{\msf}(\Gb;\kappa,w)(x)
:= -\frac{w}{\cs(\kappa,w)^2\msf}\left(\kappa\int_0^1 G_1(x+\msf^{-1}(1-s),s) \ds - \int_{-1}^0 G_1(x+\msf^{-1}(1-s),s) \ds\right),
\]
and
\[
\Rfrakb_{\msf}(\Gb;\kappa,w)
:= \big(g_1,g_2,\Rfrak_3^{\msf}(\Gb;\kappa,w)+h_1,\Rfrak_4^{\msf}(\Gb;\kappa,w)+h_2\big).
\]
Then our remaining unknowns $\pbit$ must satisfy
\begin{equation}\label{eqn: pbit}
\pbit'
= \msf\Deltab_{\msf}(\kappa,w)\pbit + \Rfrakb_{\msf}(\Gb;\kappa,w).
\end{equation}
This is really an advance-delay differential equation, due to the shift operators in $\Deltab_{\msf}(\kappa,w)$.

We sketch in the following sections how we solve \eqref{eqn: pbit} for each regularity problem.
We do not provide the full details here.
Our foremost hesitation to do so arises from the burdensome calculation that gives the precise form of $\Gb =(\ind_{\X}-\Pi)\tGb$, where $\tGb = (0,0,h_1,h_2,0,0)$.
A glance at the formulas for the coordinate functionals of $\Pi_0$ alone, as given in Proposition \ref{prop: chik-star formulas}, indicates just how detailed this enterprise is.

Given the precise information that we have presented elsewhere, however, a difficult calculation alone is not a reason for omission.
Rather, we contend that, thanks to the detailed work of our predecessors, we would not learn anything truly new by presenting all the details here.
We feel that our {\it{distillation}} and {\it{phrasing}} of the optimal regularity problems in Appendix \ref{app: opt reg} is one of our novel contributions to the spatial dynamics community; the ideas behind their {\it{verifications}} follow the established lattice literature.
Conversely, the prior detailed calculations of the generalized eigenvectors and nondegeneracy coefficients really are different in the context of our FPUT dimers from any other lattice results, and they are absolutely essential to determining both the precise leading order structure of our position fronts and to comparing our results with the Beale's method relative displacement nanopterons.

Finally, we believe that the optimal regularity properties of the system \eqref{eqn: pbit} could be generalized substantially to similar linear affine problems with $\Deltab_{\msf}(\kappa,w)$ replaced by a much more abstract operator of an advance-delay or Fourier multiplier type, along with a suitable substitute for the affine term $\Rfrakb_{\msf}(\Gb;\kappa,w)$.
Working out the right hypotheses to encompass our problem \eqref{eqn: pbit}, along with the analogues in the lattice papers \cite{iooss-kirchgassner, calleja-sire, james-sire, hilder-derijk-schneider, iooss-james, venney-zimmer} is far beyond the scope of this project, but we hope to treat it in detail in the future.
Such a treatment would shed further light on the underlying similarities among the generalized Fourier techniques of \cite{iooss-kirchgassner} and the classical Fourier methods of \cite{faver-wright, faver-spring-dimer}, between which we outline an interpolation in Section \ref{sec: opt reg soln1} to address the localized problem, and the Laplace transform techniques of \cite{hvl}, which we mention in Section \ref{sec: opt reg soln3} for tackling the suboptimal problem.

\subsubsection{A sketch of the solution to the localized problem}\label{sec: opt reg soln1}
We can show that the optimal regularity hypothesis holds whenever $q \in (0,\Lfrak_0(\kappa,w)^{1/2})$ and $b_0 \in (0,\pi)$.
Here we assume that $\pbit$ and $\Gb$ are exponentially localized, and so they have Fourier transforms.
Also, $\msf = \mu$.
On the Fourier side, \eqref{eqn: pbit} reads
\[
\big(i\mu{k}\ind_4-\tilde{\Deltab}(\mu{k}; \kappa,w)\big)\hat{\pbit}(k)
= \mu\ft[\Rfrakb_{\mu^{-1}}(\Gb;\kappa,w)](k),
\]
where
\[
\ind_4
:= \diag(1,1,1,1)
\]
and, for $K \in \R$,
\[
\tDeltab(K;\kappa,w)
:= \begin{bmatrix*}
0 &0 &1 &0 \\
0 &0 &0 &1 \\
-\cs(\kappa,w)^{-2}(1+\kappa) &\cs(\kappa,w)^{-2}(e^{iK}+\kappa{e}^{-iK}) &0 &0 \\
\cs(\kappa,w)^{-2}w(\kappa{e}^{iK}+e^{-iK}) &-\cs(\kappa,w)^{-2}w(1+\kappa) &0 &0
\end{bmatrix*}.
\]
We find
\begin{equation}\label{eqn: opt reg det}
\det\big(iK\ind_4-\tDeltab(K;\kappa,w)\big)
= \Lambda(K;\kappa,w,\cs(\kappa,w)),
\end{equation}
where $\Lambda$ was defined in \eqref{eqn: Lambdac defn}.
Part \ref{part: Lambda-1} of Proposition \ref{prop: Lambda props} thereby imposes ``solvability conditions'' on the problem \eqref{eqn: pbit}: we need
\begin{equation}\label{eqn: opt reg solv}
\ft[\Rfrakb_{\mu^{-1}}(\Gb;\kappa,w)](0)
= \ft[\Rfrakb_{\mu^{-1}}(\Gb;\kappa,w)](\pm\mu\omega_*(\kappa,w))
= 0.
\end{equation}
Here $\ft[f]$ is the Fourier transform of $f$, i.e.,
\[
\ft[f](k)
:= \frac{1}{\sqrt{2\pi}}\int_{-\infty}^{\infty} f(x)e^{-ikx} \dx.
\]
The condition \eqref{eqn: opt reg solv} can be met since $\Gb(x) \in \Y_{\hsf}$, with $\Y_{\hsf}$ from \eqref{eqn: Y-hsf}.
It is then possible to recover
\begin{equation}\label{eqn: opt reg pre-conv}
\hat{\pbit}(k)
= \mu\big(i\mu{k}\ind_4-\tDeltab(\mu{k};\kappa,w)\big)^{-1}\ft[\Rfrakb_{\mu^{-1}}(\Gb;\kappa,w)](k).
\end{equation}
Thereafter one defines $\pbit$ as a convolution; extends, for $\Gb$ defined on a complex strip, this convolution to the strip; and last obtains many uniform estimates in $\mu$.
The Lombardi-amenable map needed in the definition of localized optimal regularity will be $\Mfrak = \norm{\cdot}_{L^{\infty}}$.
Defining the solution as a convolution and then extending that convolution from the real line is essentially what Lombardi does for the water wave problem in \cite[Lem.\@ 8.2.1, 8.2.2]{Lombardi}, except his convolutions arise in a more straightforward way from a ``quasi-semigroup'' generated by the resolvent.

This is essentially the program carried out by Iooss and Kirchg\"{a}ssner in \cite[App.\@ 1]{iooss-kirchgassner}.
They, however, do not incorporate the small parameter $\mu$ and, due to their center manifold theory preferences, work in spaces of exponentially {\it{growing}} functions.
This motivates their masterful manipulation of Fourier theory in growing spaces, which retains all the familiar properties of the Fourier transform in $L^p$-spaces, e.g., with respect to derivatives and convolutions.
See the discussion in Appendix \ref{app: opt reg intro} for references to subsequent optimal regularity problems treated along the lines of \cite{iooss-kirchgassner}.

\subsubsection{A sketch of the solution to the periodic problem}
The set-up of this problem is largely the same as that in Section \ref{sec: opt reg soln1}, except now we are working with $2\pi$-periodic $\pbit$ and $\Gb$, and $\msf = (\omega_*(\kappa,w)/\mu + \grave{\omega})^{-1}$, where $\grave{\omega} = \O(\mu)$.
In lieu of the Fourier transform, we use Fourier coefficients with $k \in \Z$.
Again, the crux of the matter is that the determinant of the matrix that we would like to invert is \eqref{eqn: opt reg det}, with $K = (\omega/\mu+\grave{\omega})k$, and so we have solvability conditions like \eqref{eqn: opt reg pre-conv}.
Venney and Zimmer provide an illustrative overview of the problem in their discussion after \cite[Eq.\@ (52)]{venney-zimmer}.

\subsubsection{A sketch of the solution to the suboptimal problem}\label{sec: opt reg soln3}
Now we will allow $\pbit$ and $\Gb$ to grow exponentially, and we take $\msf=1$.
The system \eqref{eqn: pbit} is suited to the Laplace transform-based techniques of Hupkes and Verduyn Lunel for linear nonhomogeneous mixed-type functional differential equations (MFDE); see \cite[Sec.\@ 5, Prop.\@ 5.1]{hvl}.
Although the generalized Fourier methods of Iooss and Kirchg\"{a}ssner also apply to this growing problem, our prior success in adapting the Hupkes--Verduyn Lunel methods to the FPUT equal mass limit in \cite{faver-hupkes-equal-mass} leads us to favor the latter approach.
The MFDE approach is also mentioned as a alternative strategy in \cite[Sec.\@ II.B]{iooss-james}.

\subsection{Spatial dynamics in relative displacement coordinates}\label{sec: spat dyn in rel disp}
It is wholly possible to make the Iooss--Kirchg\"{a}ssner change of variables from Section \ref{sec: IK-cov} on the relative displacement traveling wave system \eqref{eqn: rel disp tw eqns}.
For simplicity, here we assume that the spring forces have no superquadratic terms, i.e., $\Vscr_1=\Vscr_2=0$ in \eqref{eqn: dimer springs}.
Put $\xi_j := \varrho_j'$ and $\Rho_j(x,v) := \varrho_j(x+v)$.
Doing so yields a problem like \eqref{eqn: IK system}, which we write in the intentionally similar notation
\begin{equation}\label{eqn: rel disp IK syst}
\Ub'
= \Lscr(\kappa,w,c)\Ub + c^{-2}\Qscr(\Ub;\beta,w),
\end{equation}
where $\Ub = (\varrho_1,\varrho_2,\xi_1,\xi_2,\Rho_1,\Rho_2)$, and the operators above are
\[
\Lscr(\kappa,w,c)
:= \begin{bmatrix*}
0 &\ind &0 \\
-(1+w)c^{-2}\diag(1,\kappa) &0 &c^{-2}\Delta^{\star}(\kappa,w) \\
0 &0 &\partial_v\ind
\end{bmatrix*},
\]
\[
\Delta^{\star}(\kappa,w)
:= \begin{bmatrix*}
0 &\kappa(w\delta^1+\delta^{-1}) \\
(\delta^1+w\delta^{-1}) &0
\end{bmatrix*},
\]
and
\[
\Qscr(\Ub;\beta,w)
:= \begin{pmatrix*}
0 \\
0 \\
-(1+w)\varrho_1^2 + \beta(w\delta^1+\delta^{-1})\Rho_2^2 \\
-\beta(1+w)\varrho_2^2 + (\delta^1+w\delta^{-1})\Rho_1^2 \\
0 \\
0
\end{pmatrix*}
\]

This problem is reversible in the mass and spring dimer cases with the respective symmetries
\[
\Sscr_{\Mb} := \begin{bmatrix*}
\Jbb &0 &0 \\
0 &-\Jbb &0 \\
0 &0 &R\Jbb
\end{bmatrix*}
\quadword{and}
\Sscr_{\Kb} := \begin{bmatrix*}
\ind &0 &0 \\
0 &-\ind &0 \\
0 &0 &R\ind
\end{bmatrix*}
\]
with $R$ defined in \eqref{eqn: R refl}, $\Jbb$ defined in \eqref{eqn: sd symm}, and $\ind$ defined in \eqref{eqn: Lc defn}.
It is also possible to construct a first integral very much like the one in \eqref{eqn: first integral defn}.

An analysis exactly along the lines of Section \ref{sec: spectrum} shows that $\Lscr(\kappa,w,c)$ has the same spectrum as $\L(\kappa,w,c)$, the linearization of the position problem.
That is,
\[
\sigma(\Lscr(\kappa,w,c)) 
= \set{z \in \C}{\det\big(\M(z;\kappa,w,c)\big) = 0},
\]
with the determinant calculated in \eqref{eqn: Mc det}.
Moreover, when $c=\cs(\kappa,w)$ from \eqref{eqn: cs defn}, the center spectrum of $\Lscr(\kappa,w,\cs(\kappa,w))$ is the same; it consists of the eigenvalues 0 and $\pm{i}\omega_c(\kappa,w)$, as in Proposition \ref{prop: Lambda props}.
These eigenvalues have the same algebraic and geometric multiplicities as in that proposition, too.

This should not be surprising.
At the {\it{linear}} level, the position and relative displacement traveling wave problems are closely related.
The linear operators governing \eqref{eqn: position tw eqns} and \eqref{eqn: rel disp tw eqns} are, respectively,
\[
\begin{bmatrix*}
-(1+\kappa) &(S^1+\kappa{S}^{-1}) \\
w(\kappa{S}^1+S^{-1}) &-w(1+\kappa)
\end{bmatrix*}
\quadword{and}
\begin{bmatrix*}
-(1+w) &\kappa(wS^1+S^{-1}) \\
(S^1+wS^{-1}) &-\kappa(1+w)
\end{bmatrix*}.
\]
We can view these operators as formal adjoints of each other, provided that the parameters $\kappa$ and $w$ are swapped.
In particular, if we consider them as Fourier multipliers, their symbols are the conjugate transposes of each other, albeit with $\kappa$ and $w$ interchanged.
This ``formal adjoint'' structure is certainly replicated in the components of $\Lscr(\kappa,w,c)$ as compared to those of $\L(\kappa,w,c)$ in \eqref{eqn: Lc defn}.

Because 0 has algebraic multiplicity 4 as an eigenvalue of $\Lscr(\kappa,w,\cs(\kappa,w))$, we cannot, as in position coordinates, directly apply Lombardi's results to \eqref{eqn: rel disp IK syst}.
The obstacle to a reduction procedure like that of Section \ref{sec: abstract} appears to be that \eqref{eqn: rel disp IK syst} is no longer translation invariant; this is a divergence of \eqref{eqn: rel disp IK syst} and the position problem \eqref{eqn: IK system} at the {\it{nonlinear}} level.
More precisely, while the one-dimensional eigenspace of $\Lscr(\kappa,w,c)$ corresponding to 0 is spanned by the eigenvector
\[
\Eb_0(\kappa,w) 
:= (\kappa,1,0,0,\kappa,1),
\]
in general we have
\[
\Qscr(\Ub+\mu\Eb_0(\kappa,w);\beta,w)
\ne \Qscr(\Ub;\beta,w)
\]
Physically, the loss of translation invariance is unsurprising, as the original relative displacement problem \eqref{eqn: rel disp eqns} does not retain the affine invariance that the position problem does, as we outlined in Remark \ref{rem: pos trans invar}.

Additionally, in the separate mass and spring dimer cases, we have
\begin{equation}\label{eqn: syms and ev for rel disp}
\Sscr_{\Mb}\Eb_0(1,w) = \Eb_0(1,w)
\quadword{and}
\Sscr_{\Kb}\Eb_0(\kappa,1) = \Eb_0(\kappa,1),
\end{equation}
and so \eqref{eqn: rel disp IK syst} has a ``$0^{4+}i\omega$'' bifurcation at the origin, which contrasts with the $0^{4-}i\omega$ bifurcation in position coordinates; see Remark \ref{rem: 04+iomega}.
Given these subtle, but substantial, distinctions between the position and relative displacement problems, and considering our success with running spatial dynamics in position coordinates, we do not pursue further an attempt at spatial dynamics in relative displacement coordinates.
Neither, however, are we willing to rule out the possibility that Lombardi's methods could apply to some reduced version of \eqref{eqn: rel disp IK syst} as they do for position.
\section{The Proofs of Theorems \ref{thm: main theorem mass dimer} and \ref{thm: main theorem spring dimer}}\label{sec: proofs of main lattice theorems}

We prove Theorem \ref{thm: main theorem mass dimer} in detail in Section \ref{sec: proof of main thm mass dimer}.
The proof of Theorem \ref{thm: main theorem spring dimer}, which we give in Section \ref{sec: proof of main thm spring dimer}, will then be quite similar, so we only discuss the notable points of departure.
These chiefly hinge on the very different structures of the generalized eigenvectors $\chib_1(1,w)$ and $\chib_1(\kappa,1)$, which can be seen from their definitions in \eqref{eqn: chib1}.

Since the mass and spring dimers separately satisfy Hypotheses \ref{hypo: F structure} through \ref{hypo: opt reg}, Theorem \ref{thm: main abstract} provides solutions to the problem \eqref{eqn: concrete final} when $\kappa=\beta=1$ or when $w=1$.
These solutions are given in the Iooss--Kirchg\"{a}ssner variables $\Ub$, as defined in Section \ref{sec: IK-cov}, and so our main task is to convert them back to position and relative displacement coordinates.
The recovery of position coordinates is fairly straightforward, but the relative displacement construction requires rather more detail.
In particular, we put significant effort into establishing that the periodic profiles for relative displacement cannot vanish identically and into extracting just the right explicit and abstract exponentially localized terms in the nanopteron profile, so that our results match those from Beale's method in \cite{faver-wright, faver-spring-dimer}.

\subsection{The proof of Theorem \ref{thm: main theorem mass dimer}}\label{sec: proof of main thm mass dimer}

\subsubsection{Solutions to the near-sonic traveling wave problem \eqref{eqn: concrete final} in Iooss--Kirchg\"{a}ssner variables}
The work in Section \ref{sec: tw prob redux} shows that the near-sonic traveling wave problem \eqref{eqn: concrete final} satisfies Hypotheses \ref{hypo: F structure} through \ref{hypo: opt reg}, and therefore by Theorem \ref{thm: main abstract} this problem has solutions $\Ub = \Usf_{\mu}^{\alpha}$ and $\Ub = \Psf_{\mu}^{\alpha}$, where $\Usf_{\mu}^{\alpha}$ has the form \eqref{eqn: abstract drifting nanopteron} and $\Psf_{\mu}^{\alpha}$ has the form \eqref{eqn: Psf-mu-alpha}.
Here $0 < \mu < \mu_*$.
To bring the small parameter in line with the classical long wave scaling, we take $\mu = \ep^2$ and restrict $0 < \ep < \mu_*^{1/2}$.

With $\mu = \ep^2$, our solutions are
\begin{multline}\label{eqn: Usf mass dimer}
\Usf_{\ep^2}^{\alpha}(x)
= \ep\left[
\left(-\frac{3\Lfrak_0(1,w)^{1/2}}{\Qfrak_0(1,1,w)}+\ep\Lup_{\ep^2}^{\alpha}\right)
\tanh\left(\frac{\Lfrak_0(1,w)^{1/2}\ep{x}}{2}\right)
+ \ep\Upsilon_{\ep^2}^{\alpha,0}(\ep{x})
\right]\chib_0(1,w) \\
- \frac{3\Lfrak_0(1,w)}{2\Qfrak_0(1,1,w)}\ep^2\sech^2\left(\frac{\Lfrak_0(1,w)^{1/2}\ep{x}}{2}\right)\chib_1(1,w) 
+ \ep^3\Upsilon_{\ep^2}^{\alpha,*}(\ep{x}) \\
+\alpha\ep^2\Phi_{\ep^2}^{\alpha}(\Tup_{\ep^2}(\ep{x}))
+ \alpha\ep(\alpha+\ep)\left(\int_0^{\Tup_{\ep^2}(\ep{x})} \Phi_{\ep^2}^{\alpha,\int}(s) \ds\right)\chib_0(1,w)
\end{multline}
and
\begin{equation}\label{eqn: Psf mass dimer}
\Psf_{\ep^2}^{\alpha}(x)
=\alpha\ep^2\Phi_{\ep^2}^{\alpha}(\ep{x})
+ \alpha\ep(\alpha+\ep^2)\left(\int_0^{\ep{x}} \Phi_{\ep^2}^{\alpha,\int}(s)\ds\right)\chib_0(1,w)
\end{equation}
with
\[
\Tup_{\ep^2}^{\alpha}(X)
:= X + \ep^2\vartheta_{\ep^2}^{\alpha}\tanh\left(\frac{\Lfrak_0(1,w)^{1/2}X}{2}\right).
\]
The eigenvectors $\chib_0(1,w)$ and $\chib_1(1,w)$ were defined in \eqref{eqn: chib0} and \eqref{eqn: chib1}, respectively.
The maps $\Upsilon_{\ep^2}^{\alpha,0}$, $\Upsilon_{\ep^2}^{\alpha,*}$, $\Phi_{\ep^2}^{\alpha}$, and $\Phi_{\ep^2}^{\alpha,\int}$ and the scalars $\Lup_{\ep^2}^{\alpha}$ and $\vartheta_{\ep^2}^{\alpha}$ have all the same properties as their counterparts in Theorem \ref{thm: main abstract}, using the identity $\mu=\ep^2$ throughout.
The scalar $\Qfrak_0(1,1,w)$ was computed in \eqref{eqn: final quad nondegen}.
To eliminate unnecessary instances of $\ep^2$, put
\begin{equation}\label{eqn: ep2 to ep}
\theta_{\ep}^{\alpha} := \vartheta_{\ep^2}^{\alpha},
\qquad
\Lsf_{\ep}^{\alpha} := \Lup_{\ep^2}^{\alpha},
\quadword{and}
\Tsf_{\ep}^{\alpha} := \Tup_{\ep^2}^{\alpha}.
\end{equation}

We conclude this introduction by selecting the values of $b_0$ and $q$ in the statement of Theorem \ref{thm: main abstract}.
We fix an arbitrary $q \in (0,\Lfrak_0(1,w)^{1/2})$, where $\Lfrak_0(1,w)$ was defined in \eqref{eqn: lin nondegen concrete}.
Purely for simplicity, we take $b_0 = \pi/2$ and set
\begin{equation}\label{eqn: what b really is}
\Alpha_{\infty}
:= \left(\frac{\pi}{2}\right)\omega_*(1,w)\Lfrak_0(1,w)^{1/2},
\end{equation}
where $\omega_*(1,w)$ was defined in \eqref{eqn: star simplicity}.
Theorem \ref{thm: main abstract} then tells us that the solutions $\Usf_{\ep^2}^{\alpha}$ and $\Psf_{\ep^2}^{\alpha}$ exist over the ranges 
\begin{equation}\label{eqn: initial intervals of existence}
0 < \ep < \tep_*
\quadword{and}
\Alpha_0\ep{e}^{-\Alpha_{\infty}/\ep} \le \alpha \le \tAlpha_1
\end{equation}
for some $\Alpha_0$, $\tAlpha_1$, $\tep_* > 0$.

\subsubsection{	The proof of part \ref{part: main theorem mass dimer position per} of Theorem \ref{thm: main theorem mass dimer}}\label{sec: mass dimer position per construction}
With $\Psf_{\ep^2}^{\alpha}$ defined in \eqref{eqn: Psf mass dimer}, put
\begin{equation}\label{eqn: pper}
p_{\per,j,\ep}^{\alpha}(x)
:= (\Psf_{\ep^2}^{\alpha}(x))_j, \ j = 1,2,
\end{equation}
so that by the Iooss--Kirchg\"{a}ssner change of variables in Section \ref{sec: IK-cov}, the pair $(p_{\per,1,\ep}^{\alpha},p_{\per,2,\ep}^{\alpha})$ solves the position traveling wave problem \eqref{eqn: position tw eqns}.
Define
\begin{equation}\label{eqn: varphi-j-ep-alpha}
\varphi_{j,\ep}^{\alpha}(X)
:= \begin{cases}
(\Phi_{\ep^2}^{\alpha}(X))_1, \ j \text{ is odd} \\
(\Phi_{\ep^2}^{\alpha}(X))_2, \ j \text{ is even}
\end{cases}
\end{equation}
and
\begin{equation}\label{eqn: Gsf-ep-alpha}
\Gsf_{\ep}^{\alpha}(X) 
:= \int_0^X \Phi_{\ep^2}^{\alpha,\int}(s) \ds.
\end{equation}
The estimate \ref{eqn: main periodic est} from Theorem \ref{thm: main abstract} guarantees that these function satisfy the estimate \eqref{eqn: varphi G est md} in Theorem \ref{thm: main theorem mass dimer}.
By \eqref{eqn: Lombardi periodic frequency}, the leading order term of the frequency of $\varphi_{j,\ep}^{\alpha}$ is  $\omega_w := \omega_*(1,w)$, which was defined in \eqref{eqn: star simplicity}.

We then have
\begin{equation}\label{eqn: p-per-j-ep-alpha}
p_{\per,j,\ep}^{\alpha}(x)
= \alpha\ep^2\varphi_{j,\ep}^{\alpha}(\ep{x}) + \alpha\ep(\alpha+\ep^2)\Gsf_{\ep}^{\alpha}(\ep{x}),
\end{equation}
and so taking
\[
u_j(t)
= \begin{cases}
p_{\per,1,\ep}^{\alpha}(j-\cep{t}), \ j \text{ is odd} \\
p_{\per, 2,\ep}^{\alpha}(j-\cep{t}), \ j \text{ is even}
\end{cases}
\]
produces the solutions \eqref{eqn: md pos per soln} to the position system \eqref{eqn: original equations of motion}.

\subsubsection{	The proof of part \ref{part: main theorem mass dimer position} of Theorem \ref{thm: main theorem mass dimer}}\label{sec: mass dimer position full construction}
With $\Usf_{\ep^2}^{\alpha}$ defined in \eqref{eqn: Usf mass dimer}, put
\begin{equation}\label{eqn: p-j-ep-alpha}
p_{j,\ep}^{\alpha}(x)
= (\Usf_{\ep^2}^{\alpha}(x))_j, \ j = 1,2.
\end{equation}
By the Iooss--Kirchg\"{a}ssner change of variables, this solves the position traveling wave profile system \eqref{eqn: position tw eqns}, and thus taking
\[
u_j(t)
= \begin{cases}
p_{1,\ep}^{\alpha}(j-\csf_{\ep^2}(1,w){t}), \ j \text{ is odd} \\
p_{2,\ep}^{\alpha}(j-\csf_{\ep^2}(1,w){t}), \ j \text{ is even}
\end{cases}
\]
solves the original equations of motion \eqref{eqn: original equations of motion}.
The wave speed $\csf_{\ep^2}(1,w)$ was defined in \eqref{eqn: cnu}.

Now we rewrite these solutions into the form promised by part \ref{part: main theorem mass dimer position} of Theorem \ref{thm: main theorem mass dimer}.
First, we will need the identities
\begin{equation}\label{eqn: chib0-chib1-1-2}
(\chib_0(1,w))_1 = (\chib_0(1,w))_2 = 1
\quadword{and}
(\chib_1(1,w))_1 = (\chib_1(1,w))_2 = 0.
\end{equation}
Here we are using the componentwise notation from \eqref{eqn: component convention}.
Next, abbreviate
\begin{equation}\label{eqn: eta-j-ep mass dimer}
\eta_{j,\ep}^{\alpha}(X)
:= \begin{cases}
\Upsilon_{\ep^2}^{\alpha,0}(X) 
+ \ep(\Upsilon_{\ep^2}^{\alpha,*}(X))_1, \ j \text{ is odd} \\
\Upsilon_{\ep^2}^{\alpha,0}(X) 
+ \ep(\Upsilon_{\ep^2}^{\alpha,*}(X))_2, \ j \text{ is even}.
\end{cases} 
\end{equation}
The estimates \ref{eqn: main expn loc} from Theorem \ref{thm: main abstract} then guarantee that these three functions satisfy the estimate \eqref{eqn: md pos est} of Theorem \ref{thm: main theorem mass dimer}.

Last, define
\begin{multline}\label{eqn: Fsf-j-ep-alpha}
\Fsf_{j,\ep}^{\alpha}(X)
:= \left(-\frac{3\Lfrak_0(1,w)^{1/2}}{\Qfrak_0(1,1,w)}+\ep\Lsf_{\ep}^{\alpha}\right)
\tanh\left(\frac{\Lfrak_0(1,w)^{1/2}X}{2}\right)
+ \ep\eta_{j,\ep}^{\alpha}(X)
+ \alpha\ep\varphi_{j,\ep}^{\alpha}(\Tsf_{\ep}^{\alpha}(X)) \\
+ \alpha(\alpha+\ep^2)\Gsf_{\ep}^{\alpha}(\Tsf_{\ep}^{\alpha}(X)),
\end{multline}
where $\varphi_{j,\ep}^{\alpha}$ was defined in \eqref{eqn: varphi-j-ep-alpha} and $\Gsf_{\ep}^{\alpha}$ in \eqref{eqn: Gsf-ep-alpha}.
We conclude that $p_{j,\ep}^{\alpha}(x) = \ep\Fsf_{j,\ep}^{\alpha}(\ep{x})$ and thus $u_j(t) = \ep\Fsf_{j,\ep}^{\alpha}(\ep(j-\cep{t}))$.
Using the definition of $\Lfrak_0(1,w)$ in \eqref{eqn: Lfrak0-concrete} and the definition of $\Qfrak_0(1,1,w)$ in \eqref{eqn: final quad nondegen}, we calculate
\[
-\frac{3\Lfrak_0(1,w)^{1/2}}{\Qfrak_0(1,1,w)} = \left(\frac{6w(w^2-w+1)}{(1+w)^3}\right)^{1/2}
\quadword{and}
\frac{\Lfrak_0(1,w)^{1/2}}{2} = \frac{1}{2}\left(\frac{6w(1+w)}{w^2-w+1}\right)^{1/2}.
\]
This leads to the identities \eqref{eqn: md pos F} and \eqref{eqn: md pos tw final}.

\subsubsection{The proof of part \ref{part: main theorem mass dimer relative disp per} of Theorem \ref{thm: main theorem mass dimer}}\label{sec: mass dimer periodic construction}
With $p_{\per,j,\ep}^{\alpha}$ defined in \eqref{eqn: p-per-j-ep-alpha}, the identities \eqref{eqn: tw profile rels} tell us that the pair $(\varrho_{\per,1,\ep}^{\alpha},\varrho_{\per,2,\ep}^{\alpha})$ defined by
\begin{equation}\label{eqn: rhoper}
\varrho_{\per,1,\ep}^{\alpha}(x) := p_{\per,2,\ep}^{\alpha}(x+1)-p_{\per,1,\ep}^{\alpha}(x)
\quadword{and}
\varrho_{\per,2,\ep}^{\alpha}(x) := p_{\per,1,\ep}^{\alpha}(x+1)-p_{\per,2,\ep}^{\alpha}(x)
\end{equation}
solves the relative displacement traveling wave profile system \eqref{eqn: rel disp tw eqns}.
Consequently, taking
\[
r_j(t)
= \begin{cases}
\varrho_{\per,1,\ep}^{\alpha}(j-\csf_{\ep^2}(1,w){t}), \ j \text{ is odd} \\
\varrho_{\per,2,\ep}^{\alpha}(j-\csf_{\ep^2}(1,w){t}), \ j \text{ is even}
\end{cases}
\]
solves the original traveling wave problem.

Now we rewrite these solutions into the form \eqref{eqn: md rel disp per soln} stated in Theorem \ref{thm: main theorem mass dimer}.
The definitions of $\varrho_{\per,j,\ep}^{\alpha}$ and $p_{\per,j,\ep}^{\alpha}$ motivate us to set
\begin{equation}\label{eqn: tvarphi}
\tvarphi_{j,\ep}^{\alpha}(X)
:= \begin{cases}
[\varphi_{2,\ep}^{\alpha}(X+\ep) - \varphi_{1,\ep}^{\alpha}(X)] + \ep^{-1}(\alpha+\ep^2)\int_X^{X+\ep} \Phi_{\ep^2}^{\alpha,\int}(s) \ds, \ j \text{ is odd} \\[10pt]
[\varphi_{1,\ep}^{\alpha}(X+\ep) - \varphi_{2,\ep}^{\alpha}(X)] +  \ep^{-1}(\alpha+\ep^2)\int_X^{X+\ep} \Phi_{\ep^2}^{\alpha,\int}(s) \ds, \ j \text{ is even},
\end{cases}
\end{equation}
where $\varphi_{j,\ep}^{\alpha}$ was defined in \eqref{eqn: varphi-j-ep-alpha}.
Then
\[
\varrho_{\per,j,\ep}^{\alpha}(x)
= \alpha\ep^2\tvarphi_{j,\ep}^{\alpha}(\ep{x}), \ j = 1,2.
\]
and so putting
\[
r_j(t)
= \alpha\ep^2\tvarphi_{j,\ep}^{\alpha}(\ep(j-\cep{t})),
\]
solves the relative displacement problem \eqref{eqn: rel disp eqns}.

It follows immediately from \eqref{eqn: tvarphi} that $\tvarphi_{j,\ep}^{\alpha}$ is periodic in $X$ with periodic independent of $j$ and uniformly bounded in $\alpha$, $\ep$, and $X$.
Last, we verify that $\tvarphi_{1,\ep}^{\alpha}$ is not identically zero; the proof for $\tvarphi_{2,\ep}^{\alpha}$ is the same.
Our strategy is to write $\varphi_{2,\ep}^{\alpha}(\cdot+\ep) - \varphi_{1,\ep}^{\alpha}$ as the sum of a leading order function whose value at 0 is nonzero and uniformly bounded away from 0 and a ``higher-order'' function.
Then we use the fact that the integral terms in $\tvarphi_{1,\ep}^{\alpha}$ come with prefactors of $\alpha$ and $\ep$ that force them to be uniformly small, at least if the ranges of $\alpha$ and $\ep$ are suitably controlled.

By \eqref{eqn: Lombardi periodic1}, \eqref{eqn: Lombardi periodic frequency}, and \eqref{eqn: Lombardi periodic2}, we can expand
\begin{multline}\label{eqn: varphi-j-ep-alpha expansion}
\frac{\varphi_{j,\ep}^{\alpha}(X)}{\Lfrak_0(1,w)}
= \cos(\Omega_{\ep}^{\alpha}X)\left(\frac{(\chib_{\omega}(1,w))_j + (\chib_{-\omega}(1,w))_j}{2}\right) \\
+ \sin(\Omega_{\ep}^{\alpha}X)\left(\frac{(\chib_{\omega}(1,w))_j - (\chib_{-\omega}(1,w))_j}{2i}\right)
+ \alpha\psi_{j,\ep}^{\alpha}(\Omega_{\ep}^{\alpha}X),
\end{multline}
where 
\begin{equation}\label{eqn: Xi-bound}
\sup_{\substack{0 < \ep < \tep_* \\ 0 \le \alpha \le \tAlpha_1 \\ X \in \R}} 
|\psi_{j,\ep}^{\alpha}(X)|
< \infty
\end{equation}
and
\begin{equation}\label{eqn: tOmega bound}
\Omega_{\ep}^{\alpha}
= \frac{\omega_*(1,w)}{\ep} + \ep\tOmega_{\ep}^{\alpha},
\quadword{with}
\sup_{\substack{0 < \ep < \tep_* \\ 0 \le \alpha \le \tAlpha_1 \\ X \in \R}}
|\tOmega_{\ep}^{\alpha}| 
< \infty.
\end{equation}
The vectors $\chib_{\pm\omega}(1,w)$ were defined in \eqref{eqn: chib-pm-omega}, while the scalar $\omega_*(1,w)$ was defined in \eqref{eqn: star simplicity}.

Put
\begin{equation}\label{eqn: Esf-w}
\Esf_w
:= \frac{\cos(\omega_*(1,w))}{1-\dfrac{w}{1+w}(\omega_*(1,w))^2}.
\end{equation}
Then the definition of $\chib_{\pm\omega}(1,w)$ in \eqref{eqn: chib-pm-omega}, we find
\[
\frac{\chib_{\omega}(1,w))_j + (\chib_{-\omega}(1,w))_j}{2}
= \begin{cases}
\Esf_w, \ j = 1 \\
1, \ j = 2
\end{cases}
\]
and
\[
\frac{\chib_{\omega}(1,w))_j - (\chib_{-\omega}(1,w))_j}{2i} = 0, \ j = 1,2.
\]
Abbreviate
\[
f_{\ep}^{\alpha}(X)
:= \frac{\varphi_{2,\ep}^{\alpha}(X+\ep)-\varphi_{1,\ep}^{\alpha}(X)}{\Lfrak_0(1,w)^{1/2}}
\]
to find
\begin{equation}\label{eqn: f-ep-alpha aux}
f_{\ep}^{\alpha}(X)
= \cos(\Omega_{\ep}^{\alpha}(X+\ep))-\Esf_w\cos(\Omega_{\ep}^{\alpha}X)
+ \alpha(\psi_{2,\ep}^{\alpha}(\Omega_{\ep}^{\alpha}(X+\ep))-\psi_{1,\ep}^{\alpha}(\Omega_{\ep}^{\alpha}X)).
\end{equation}

We consider two cases.

\begin{enumerate}[label={\it{Case \arabic*.}}]

\item
$\Esf_w \ne 0$.
Then, per \eqref{eqn: Esf-w}, $\cos(\omega_*(1,w)) \ne 0$.
We claim for the moment that $\cos(\omega_*(1,w))-\Esf_w \ne 0$ and we calculate
\begin{multline*}
f_{\ep}^{\alpha}(0)
= \cos(\Omega_{\ep}^{\alpha}\ep)-\Esf_w + \alpha(\psi_{2,\ep}^{\alpha}(\Omega_{\ep}^{\alpha}\ep) - \psi_{1,\ep}^{\alpha}(0)) \\
= \cos(\omega_*(1,w)+\ep^2\tOmega_{\ep}^{\alpha}) - \Esf_w + \alpha(\psi_{2,\ep}^{\alpha}(\omega_*(1,w)+\ep^2\tOmega_{\ep}^{\alpha})-\psi_{1,\ep}^{\alpha}(0)).
\end{multline*}
We use the uniform bounds on $\tOmega_{\ep}^{\alpha}$ and $\psi_{j,\ep}^{\alpha}$ to find $\bep_* \in (0,\tep_*)$ and $\bAlpha_1 \in (0,\tAlpha_1)$ such that if $0 < \ep < \bep_*$, then $\Alpha_0\ep{e}^{-\Alpha_{\infty}/\ep} < \bAlpha_1$ and if $\Alpha_0\ep{e}^{-\Alpha_{\infty}/\ep} \le \alpha \le \bAlpha_1$, then
\[
|f_{\ep}^{\alpha}(0)| 
\ge \frac{|\cos(\omega_*(1,w))-\Esf_w|}{2}
> 0.
\]
Then we use the uniform bound \eqref{eqn: main periodic est} on $\Phi_{\ep^2}^{\int,\alpha}$ to find $\ep_* \in (0,\bep_*)$ and $\Alpha_1 \in (0,\bAlpha_1)$ such that if $0 < \ep < \ep_*$, then $\Alpha_0\ep{e}^{-\Alpha_{\infty}/\ep} < \Alpha_1$ and
\[
\inf_{\substack{0 < \ep < \ep_* \\ 0 \alpha \le \Alpha_1}}
|\tvarphi_{1,\ep}^{\alpha}(0)|
\ge \frac{|\cos(\omega_*(1,w))-\Esf_w|}{4}
> 0.
\]

Finally, we check that $\cos(\omega_*(1,w))-\Esf_w \ne 0$ as claimed above.
Otherwise, if $\cos(\omega_*(1,w))-\Esf_w = 0$, then the definition of $\Esf_w$ in \eqref{eqn: Esf-w} gives
\[
1 - \frac{1}{1-\dfrac{w}{1+w}(\omega_*(1,w))^2}
= 0.
\]
But this rearranges to $\omega_*(1,w) = 0$, which is impossible by part \ref{part: Lambda-1} of Proposition \ref{prop: Lambda props}, so $\cos(\omega_*(1,w))-\Esf_w \ne 0$.

\item
$\Esf_w = 0$.
Then from \eqref{eqn: f-ep-alpha aux} we calculate
\[
f_{\ep}^{\alpha}(-\ep)
= 1 + \alpha(\psi_{2,\ep}^{\alpha}(0)-\psi_{1,\ep}^{\alpha}(-\Omega_{\ep}^{\alpha}\ep)),
\]
and exactly the same type of restrictions on $\ep$ and $\alpha$ as above guarantee that for some $\ep_* \in (0,\ep_0)$ and $\Alpha_1 \in (0,\tAlpha_1)$ we have
\[
\inf_{\substack{0 < \ep < \ep_* \\ 0 \le \alpha \le \Alpha_1}}
|\tvarphi_{1,\ep}^{\alpha}(-\ep)|
\ge \frac{1}{2}.
\]
\end{enumerate}

\subsubsection{The proof of part \ref{part: main theorem mass dimer relative disp} of Theorem \ref{thm: main theorem mass dimer}}\label{sec: md rel disp nano constr}
With $p_{j,\ep}^{\alpha}$ defined in \eqref{eqn: p-j-ep-alpha}, put
\begin{equation}\label{eqn: rho defn}
\varrho_{1,\ep}^{\alpha}(x) := p_{2,\ep}^{\alpha}(x+1)-p_{1,\ep}^{\alpha}(x)
\quadword{and}
\varrho_{2,\ep}^{\alpha}(x) := p_{1,\ep}^{\alpha}(x+1)-p_{2,\ep}^{\alpha}(x)
\end{equation}
to see that the pair $(\varrho_{1,\ep}^{\alpha},\varrho_{2,\ep}^{\alpha})$ solves the relative displacement traveling wave profile system \eqref{eqn: rel disp tw eqns}.
Thus taking
\begin{equation}\label{eqn: final rj defn}
r_j(t) 
= \begin{cases}
\varrho_{1,\ep}^{\alpha}(j-\cep{t}), \ j \text{ is odd} \\
\varrho_{2,\ep}^{\alpha}(j-\cep{t}), \ j \text{ is even}
\end{cases}
\end{equation}
solves the original equations of motion \eqref{eqn: rel disp eqns} in relative displacement coordinates.

Now we rewrite these solutions into the nanopteron form of part \ref{part: main theorem mass dimer relative disp} in Theorem \ref{thm: main theorem mass dimer}.
Since $p_{j,\ep}^{\alpha}(x) = \ep\Fsf_{j,\ep}^{\alpha}(\ep{x})$, we are led to put
\begin{equation}\label{eqn: Nsf defn}
\Nsf_{j,\ep}^{\alpha}(X)
:= \begin{cases}
\ep^{-1}[\Fsf_{2,\ep}^{\alpha}(X+\ep)-\Fsf_{1,\ep}^{\alpha}(X)], \ j \text{ is odd} \\
\ep^{-1}[\Fsf_{1,\ep}^{\alpha}(X+\ep)-\Fsf_{2,\ep}^{\alpha}(X)], \ j \text{ is even}
\end{cases}
\end{equation}
and unravel the definitions \eqref{eqn: rho defn} and \eqref{eqn: final rj defn} to see that defining $r_j$ by \eqref{eqn: final rj defn} implies 
\begin{equation}\label{eqn: final final rj defn}
r_j(t)
= \ep^2\Nsf_{j,\ep}^{\alpha}(\ep(j-\cep{t})).
\end{equation}
We just need to show that $\Nsf_{j,\ep}^{\alpha}$ has the nanopteron form \eqref{eqn: md rel disp N}.
We only do this for $j$ odd, as the work for $j$ even is identical.

We use the definitions of $\Nsf_{1,\ep}^{\alpha}$ in \eqref{eqn: Nsf defn} and $\Fsf_{j,\ep}^{\alpha}$ in \eqref{eqn: Fsf-j-ep-alpha}  to write
\begin{equation}\label{eqn: Nsf1-ep-alpha}
\Nsf_{1,\ep}^{\alpha}(X)
= -\frac{3\Lfrak_0(1,w)}{2\Qfrak_0(1,1,w)}\sech^2\left(\frac{\Lfrak_0(1,w)^{1/2}X}{2}\right)
+ \ep\teta_{1,\ep}^{\alpha}(X)
+ \alpha\tvarphi_{1,\ep}^{\alpha}(\Tsf_{\ep}^{\alpha}(X)),
\end{equation}
where $\tvarphi_{1,\ep}^{\alpha}$ was defined in \eqref{eqn: tvarphi} and where we put
\begin{equation}\label{eqn: teta mass dimer}
\teta_{1,\ep}^{\alpha}
:= \sum_{k=1}^6 \teta_{1,k,\ep}^{\alpha},
\end{equation}
with
\begin{multline*}
\teta_{1,1,\ep}^{\alpha}(X)
:= -\ep^{-2}\frac{3\Lfrak_0(1,w)^{1/2}}{\Qfrak_0(1,1,w)}
\left[
\tanh\left(\frac{\Lfrak_0(1,w)^{1/2}(X+\ep)}{2}\right)-\tanh\left(\frac{\Lfrak_0(1,w)^{1/2}X}{2}\right)
\right] \\
+ \frac{3\Lfrak_0(1,w)}{2\Qfrak_0(1,1,w)}\sech^2\left(\frac{\Lfrak_0(1,w)^{1/2}X}{2}\right),
\end{multline*}
\[
\teta_{1,2,\ep}^{\alpha}(X)
:= \ep^{-1}\Lsf_{\ep}^{\alpha}
\left[
\tanh\left(\frac{\Lfrak_0(1,w)^{1/2}(X+\ep)}{2}\right)-\tanh\left(\frac{\Lfrak_0(1,w)^{1/2}X}{2}\right)
\right],
\]
\[
\teta_{1,3,\ep}^{\alpha}(X)
:= \ep^{-1}[\Upsilon_{\ep^2}^{\alpha,0}(X+\ep)-\Upsilon_{\ep^2}^{\alpha,0}(X)],
\]
\[
\teta_{1,4,\ep}^{\alpha}(X)
:= (\Upsilon_{\ep^2}^{\alpha,*}(X+\ep))_2-(\Upsilon_{\ep^2}^{\alpha,*}(X))_1,
\]
\[
\teta_{1,5,\ep}^{\alpha}(X)
:= \alpha(\alpha+\ep^2)\ep^{-2}
\left(\int_{\Tsf_{\ep}^{\alpha}(X)}^{\Tsf_{\ep}^{\alpha}(X+\ep)} \Phi_{\ep^2}^{\alpha,\int}(s) \ds
- \int_{\Tsf_{\ep}^{\alpha}(X)}^{\Tsf_{\ep}^{\alpha}(X)+\ep} \Phi_{\ep^2}^{\alpha,\int}(s) \ds\right),
\]
and
\[
\teta_{1,6,\ep}^{\alpha}(X)
:= \alpha\ep^{-1}[\varphi_{2,\ep}^{\alpha}(\Tsf_{\ep}^{\alpha}(X+\ep))-\varphi_{2,\ep}^{\alpha}(\Tsf_{\ep}^{\alpha}(X)+\ep)].
\]

We emphasize that the leading-order $\sech^2$-term in \eqref{eqn: Nsf1-ep-alpha} is the same as the coefficient on $\chib_1(1,w)$ in \eqref{eqn: Usf mass dimer}, except for the now missing factor of $\ep^2$.
While the zero components of $\chib_1(1,w)$ eliminate this $\sech^2$-term in the position profiles \eqref{eqn: Fsf-j-ep-alpha}, it reappears here due to the subtraction of the $\ep$-shifted $\tanh$-terms.

We recall that the terms $\Upsilon_{\ep^2}^{\alpha,0}$ and $\Upsilon_{\ep^2}^{\alpha,*}$ appeared in the expansion \eqref{eqn: Usf mass dimer} of our original solution $\Usf_{\ep^2}^{\alpha}$ and satisfy the same properties as their counterparts in Theorem \ref{thm: main abstract}.
We need to show
\[
\sup_{\substack{0 < \ep < \ep_* \\ \Alpha_0\ep{e}^{-\Alpha_{\infty}/\ep} \le \alpha \le \Alpha_1}}
e^{q|X|}|\teta_{1,k,\ep}^{\alpha}(X)|
< \infty, \ k=1,\ldots,6.
\]
For $k=1$, $2$, this follows from Taylor's theorem.
For $k=3$, we use the estimate \eqref{eqn: main expn loc} on the derivative of $\Upsilon_{\ep^2}^{\alpha,0}$ to bound
\[
|\teta_{1,3,\ep}^{\alpha}(X)|
\le \ep^{-1}\int_X^{X+\ep} |\partial_s[\Upsilon_{\ep^2}^{\alpha,0}](s)| \ds
\le C\ep^{-1}\int_X^{X+\ep} e^{-q|s|} \ds
\le Ce^{-q|X|}.
\]
The estimate for $k=4$ also follows from \eqref{eqn: main expn loc} applied to the components of $\Upsilon_{\ep^2}^{\alpha,*}$.
The estimate for $k=5$ follows from part \ref{part: tanh1} of Lemma \ref{lem: the big tanh lemma} and for $k=6$ from part \ref{part: tanh4} of Lemma~\ref{lem: the other big tanh lemma}.

Last, the definitions of $\Lfrak_0(1,w)$ in \eqref{eqn: Lfrak0-concrete} and $\Qfrak_0(1,1,w)$ in  \eqref{eqn: final quad nondegen} give
\[
-\frac{3\Lfrak_0(1,w)}{2\Qfrak_0(1,1,w)} 
= \frac{3w}{1+w}.
\]

\subsection{The proof of Theorem \ref{thm: main theorem spring dimer}}\label{sec: proof of main thm spring dimer}
As we indicated earlier, we just point out the few ways in which the spring dimer proof differs from the mass dimer proof given above.
First, instead of defining the solution $\Usf_{\ep^2}^{\alpha}$ via \eqref{eqn: Usf mass dimer}, we have
\begin{multline}\label{eqn: Usf spring dimer}
\Usf_{\ep^2}^{\alpha}(x)
= \ep\left[
\left(-\frac{3\Lfrak_0(\kappa,1)^{1/2}}{\Qfrak_0(\kappa,\beta,1)}+\ep\Lup_{\ep^2}^{\alpha}\right)
\tanh\left(\frac{\Lfrak_0(\kappa,1)^{1/2}\ep{x}}{2}\right)
+ \ep\Upsilon_{\ep^2}^{\alpha,0}(\ep{x})
\right]\chib_0(\kappa,1) \\
- \frac{3\Lfrak_0(\kappa,1)}{2\Qfrak_0(\kappa,\beta,1)}\ep^2\sech^2\left(\frac{\Lfrak_0(\kappa,1)^{1/2}\ep{x}}{2}\right)\chib_1(\kappa,1) 
+ \ep^3\Upsilon_{\ep^2}^{\alpha,*}(\ep{x})
+ \Psf_{\ep^2}^{\alpha}(\ep^{-1}\Tup_{\ep^2}^{\alpha}(\ep{x})).
\end{multline}
Here $\Psf_{\ep^2}^{\alpha}$ has the same form as \eqref{eqn: Psf mass dimer}.
The key difference now is that instead of generalized eigenvector component identities \eqref{eqn: chib0-chib1-1-2}, we have
\[
(\chib_1(\kappa,1))_1 = \frac{1}{2}\left(\frac{1-\kappa}{1+\kappa}\right)
\quadword{and} 
(\chib_0(\kappa,1))_2 = -\frac{1}{2}\left(\frac{1-\kappa}{1+\kappa}\right).
\]

This only requires one adjustment to the construction of position solutions: instead of defining the localized term $\eta_{j,\ep}^{\alpha}$ via \eqref{eqn: eta-j-ep mass dimer}, include an additional $\sech^2$-term:
\begin{equation}\label{eqn: eta-j-ep-alpha sd pos}
\eta_{j,\ep}^{\alpha}(X)
= (-1)^j\frac{3\Lfrak_0(\kappa,1)(\kappa-1)}{2\Qfrak_0(\kappa,\beta,1)(1+\kappa)}
\sech^2\left(\frac{\Lfrak_0(\kappa,1)^{1/2}X}{2}\right)
+ \Upsilon_{\ep^2}^{\alpha,0}(X)
+ \ep(\Upsilon_{\ep^2}^{\alpha,*}(X))_j.
\end{equation}
With $\varphi_{j,\ep}^{\alpha}$ defined as in \eqref{eqn: varphi-j-ep-alpha} and $\Gsf_{\ep}^{\alpha}$ as in \eqref{eqn: Gsf-ep-alpha}, and using the notation of \eqref{eqn: ep2 to ep}, we find that the spring dimer position solutions $u_j$ satisfy $u_j(t) = \ep\Fsf_{j,\ep}^{\alpha}(\ep(j-\cep{t}))$, where
\begin{multline*}
\Fsf_{j,\ep}^{\alpha}(X)
= \left(-\frac{3\Lfrak_0(\kappa,1)^{1/2}}{\Qfrak_0(\kappa,\beta,1)}+\ep\Lsf_{\ep}^{\alpha}\right)
\tanh\left(\frac{\Lfrak_0(\kappa,1)^{1/2}\ep{x}}{2}\right)
+ \ep\eta_{j,\ep}^{\alpha}(X)
+ \alpha\ep\varphi_{j,\ep}^{\alpha}(\Tsf_{\ep}^{\alpha}(X)) \\
+ \alpha(\alpha+\ep^2)\Gsf_{\ep}^{\alpha}(\Tsf_{\ep}^{\alpha}(X)).
\end{multline*}
The leading order frequency term of $\varphi_{j,\ep}^{\alpha}$ is $\omega_{\kappa} := \omega_*(\kappa,1)$ from \eqref{eqn: star simplicity}.
We use the formula \eqref{eqn: Lfrak0-concrete} for $\Lfrak_0(\kappa,1)$ and \eqref{eqn: final quad nondegen} for $\Qfrak_0(\kappa,\beta,1)$ to calculate
\[
-\frac{3\Lfrak_0(\kappa,1)^{1/2}}{\Qfrak_0(\kappa,\beta,1)} = \frac{[6\kappa^3(1+\kappa)(\kappa^2-\kappa+1)]^{1/2}}{2(\beta+\kappa^3)}
\quadword{and}
\frac{\Lfrak_0(\kappa,1)^{1/2}}{2} = \frac{1}{2}\left(\frac{6\kappa(1+\kappa)}{\kappa^2-\kappa+1}\right)^{1/2}.
\]

Next, the verification in Section \ref{sec: mass dimer periodic construction} that $\tvarphi_{1,\ep}^{\alpha}$ is nonvanishing must proceed with some different constants and functions.
Instead of $\Esf_w$ as in \eqref{eqn: Esf-w}, put
\begin{equation}\label{eqn: Esf-kappa}
\Esf_{\kappa}
:= \frac{e^{i\omega} + \kappa{e}^{-i\omega}}{1+\kappa-\dfrac{2\kappa}{1+\kappa}\omega^2}.
\end{equation}
Then in  \eqref{eqn: varphi-j-ep-alpha expansion}, replace $\chib_{\pm\omega}(1,w)$ with $\chib_{\pm\omega}(\kappa,1)$, again defined via \eqref{eqn: chib-pm-omega}, to find
\[
\frac{(\chib_{\omega}(\kappa,1))_j + (\chib_{-\omega}(\kappa,1))_j}{2}
= \begin{cases}
2\re(\Esf_{\kappa}), &j = 1\\ 
1, &j = 2
\end{cases}
\]
and
\[
\frac{(\chib_{\omega}(\kappa,1))_j - (\chib_{-\omega}(\kappa,1))_j}{2}
= \begin{cases}
2\im(\Esf_{\kappa}), &j = 1\\ 
0, &j = 2.
\end{cases}
\]
Then the function $f_{\ep}^{\alpha}$ defined in \eqref{eqn: f-ep-alpha aux} reads
\[
f_{\ep}^{\alpha}(X)
= \breve{f}_{\ep}^{\alpha}(X)
+ \alpha(\psi_{2,\ep}^{\alpha}(\Omega_{\ep}^{\alpha}(X+\ep))-\psi_{1,\ep}^{\alpha}(\Omega_{\ep}^{\alpha}X)),
\]
where 
\[
\breve{f}_{\ep}^{\alpha}(X)
:= \cos(\Omega_{\ep}^{\alpha}(X+\ep)) 
- \re(\Esf_{\kappa})\cos(\Omega_{\ep}^{\alpha}X)
+ i\im(\Esf_{\kappa})\sin(\Omega_{\ep}^{\alpha}X)
\]
and where $\psi_{j,\ep}^{\alpha}$ satisfies the estimate of \eqref{eqn: Xi-bound} and $\Omega_{\ep}^{\alpha}$ satisfies \eqref{eqn: tOmega bound}.

We just show that the leading order part $\breve{f}_{\ep}^{\alpha}$ is always uniformly bounded away from 0, as then the remainder of the work in Section \ref{sec: mass dimer periodic construction} proceeds identically.
We calculate
\[
\breve{f}_{\ep}^{\alpha}(0)
= \cos(\Omega_{\ep}^{\alpha}\ep) - \re(\Esf_{\kappa})
= \cos(\omega_*(\kappa,1) + \ep^2\tOmega_{\ep}^{\alpha}) - \re(\Esf_{\kappa}),
\]
where, from \eqref{eqn: Esf-kappa}
\[
\re(\Esf_{\kappa})
= \frac{\cos(\omega_*(\kappa,1))(1+\kappa)}{1+\kappa-\dfrac{2\kappa}{1+\kappa}\omega_*(\kappa,1)^2}.
\]

First suppose $\cos(\omega_*(\kappa,1)) \ne 0$.
If $\cos(\omega_*(\kappa,1)) - \re(\Esf_{\kappa}) = 0$, then after some algebraic rearrangements, it follows that $\omega_*(\kappa,1) = 0$, a contradiction.
So, we must have $\cos(\omega_*(\kappa,1)) - \re(\Esf_{\kappa}) \ne 0$, and then taking $\ep$ and $\alpha$ sufficiently small, we obtain
\[
|\breve{f}_{\ep}^{\alpha}(0)|
\ge \frac{|\cos(\omega_*(\kappa,1)) - \re(\Esf_{\kappa})|}{2}.
\]
In the case that $\cos(\omega_*(\kappa,1)) = 0$, we find $\re(\Esf_{\kappa}) = 0$, too, and so $\breve{f}_{\ep}^{\alpha}(-\ep) = 1$ for all $\ep$ and $\alpha$.
This is a sufficient uniform lower bound on $\breve{f}_{\ep}^{\alpha}$ and thus $f_{\ep}^{\alpha}$.

Last, in constructing the relative displacement nanopterons as in Section \ref{sec: md rel disp nano constr}, we need to take into account the alternating $\sech^2$-term in the spring dimer's localized $\eta_{j,\ep}^{\alpha}$ term defined in \eqref{eqn: eta-j-ep-alpha sd pos}.
This leads us to express
\[
\Nsf_{1,\ep}^{\alpha}(X)
= -\frac{6\kappa\Lfrak_0(\kappa,1)}{\Qfrak_0(\kappa,\beta,1)(1+\kappa)}\sech^2\left(\frac{\Lfrak_0(\kappa,1)^{1/2}X}{2}\right)
+ \ep\teta_{1,\ep}^{\alpha}(X)
+ \alpha\tvarphi_{1,\ep}^{\alpha}(\Tsf_{\ep}^{\alpha}(X)),
\]
instead of \eqref{eqn: Nsf1-ep-alpha}, where now
\begin{equation}\label{eqn: teta-ep-alpha1 sd}
\teta_{1,\ep}^{\alpha}
:= \sum_{k=1}^7 \teta_{1,k,\ep}^{\alpha},
\end{equation}
with the first six terms defined analogously to the above in \eqref{eqn: teta mass dimer} and the new seventh term defined by
\[
\teta_{1,7,\ep}^{\alpha}(X)
:= 
\frac{3\Lfrak_0(\kappa,1)(\kappa-1)}{2\Qfrak_0(\kappa,\beta,1)(1+\kappa)}
\left[\sech^2\left(\frac{\Lfrak_0(\kappa,1)^{1/2}(X+\ep)}{2}\right)
- \sech^2\left(\frac{\Lfrak_0(\kappa,1)^{1/2}X}{2}\right)\right].
\]
We also find
\[
\Nsf_{2,\ep}^{\alpha}(X)
= -\frac{3\Lfrak_0(\kappa,1)}{\Qfrak_0(\kappa,\beta,1)(1+\kappa)}\sech^2\left(\frac{\Lfrak_0(\kappa,1)^{1/2}X}{2}\right)
+ \ep\teta_{2,\ep}^{\alpha}(X)
+ \alpha\tvarphi_{2,\ep}^{\alpha}(\Tsf_{\ep}^{\alpha}(X)),
\]
with $\teta_{2,\ep}^{\alpha}$ also defined as a sum of seven terms like \eqref{eqn: teta-ep-alpha1 sd}, where in particular $\teta_{2,7,\ep}^{\alpha} := -\teta_{1,7,\ep}^{\alpha}$.
The important distinction is that $\Nsf_{1,\ep}^{\alpha}$ has an extra factor of $\kappa$ on its leading-order $\sech^2$-term, while $\Nsf_{2,\ep}^{\alpha}$ does not.

Finally, we calculate 
\[
-\frac{3\Lfrak_0(\kappa,1)}{\Qfrak_0(\kappa,\beta,1)(1+\kappa)}
= \frac{3\kappa^2}{\beta+\kappa^3}
\quadword{and}
\frac{\Lfrak_0(\kappa,1)^{1/2}}{2}
= \frac{1}{2}\left(\frac{6\kappa(1+\kappa)}{\kappa^2-\kappa+1}\right)^{1/2}.
\]
\section{Future Directions}\label{sec: future directions}

We suggest several directions of future research that are natural successors to our work and the related results in \cite{faver-wright, faver-spring-dimer}.

\begin{enumerate}[label={\bf\arabic*.}]

\item
{\it{Traveling waves in lattices with more complicated material heterogeneities.}}
We are forced to specialize, ultimately, to mass and spring dimers so that we can employ symmetries.
We did not disprove the existence of a symmetry for the general dimer, but nor were we able to construct one, as we lamented in Section \ref{sec: symmetries}.
However, the KdV approximation result \eqref{eqn: gmwz est} of \cite{gmwz} is valid for ``most'' polyatomic lattices, not just mass and spring dimers.
What, then, does this KdV approximation mean for exact solutions in the long wave limit?
Do they exist, and, if so, are they nanopterons?

To extend our results to more complicated polyatomic lattices, we will either need to rethink seriously the role of symmetry or work with specialized lattices that do retain some sort of symmetry.
One such candidate could be the ``$1:n$ dimer,'' in which $n$ particles of mass 1 alternate with a particle of mass $m > 1$; such a lattice is $(n+1)$-polyatomic in the sense of \eqref{eqn: N-periodic data}.
There already exists a body of results concerning solitary waves under a Hertzian potential \cite{jvs-pulse, jvs-solitary0, jvs-solitary1}.

Assuming that we are able to meet the challenge of symmetry, another large difference from the dimers looms.
We expect in $N$-polyatomic lattices there will be $N-1$ critical frequencies, and so we would be in the case of a $0^{4-}(i\omega_1)\cdots(i\omega_{N-1})$ bifurcation.
(We allow that the precise multiplicity of the 0 eigenvalue, and its interaction with the putative symmetry, may not be ``$0^{4-}$.'')
Such a center spectrum will require a substantial generalization of Lombardi's results if we wish to purse the spatial dynamics approach.
See \cite[Sec.\@ 5.3]{james-sire}, \cite[Sec.\@ 6.3]{sire} for similar discussions of these complications in problems with analogously more complicated ``heterogeneities,'' as well as \cite{iooss-lombardi-nf-expn-small-rem} for a conjecture on an extension of Lombardi-type results to higher-order resonances.

\item
{\it{Existence of solitary waves in dimers.}}
We have constructed decidedly {\it{nonlocal}} or {\it{generalized}} solitary wave solutions in position and relative displacement coordinates.
In each case, the periodic ripples have nonzero amplitudes; in position coordinates, we must also contend with the possibly growing term.
We still do not yet know if there exists a purely exponentially localized solution for either position or relative displacement, despite the abundance of evidence, from different perspectives, as we discussed in Question \ref{ques: per amp is 0?}.
While Lombardi's results (see Appendix \ref{app: Lombardi solitary nonexistence}) encourage us to conjecture that for almost all mass ratios $w$ or linear spring constants $\kappa$, there are no long wave-scaled solitary traveling waves, much more rigorous analysis is needed.

\item
{\it{Additional applications of our abstract theory.}}
It would be an interesting exercise to check and rederive the results of \cite{venney-zimmer, iooss-james} in the language of our hypotheses and Theorem \ref{thm: main abstract}.
Despite the value of their inspiration to us, these papers used substantially different language and calculations, particularly when ``factoring out'' their first integrals and deriving the analogues of their linear and quadratic nondegeneracies.
We are also curious if our spatial dynamics method can provide another perspective on the Whitham problem for which long wave nanopterons and micropterons exist \cite{johnson-wright, johnson-truong-wheeler}.
A challenge here is that in these Whitham papers, the derivative is integrated out of the traveling wave equation, leaving a strictly nonlocal equation that does not appear to have the form of our general problem \eqref{eqn: abstract ode}.

Somewhat contemporaneously with our research, Hilder, de Rijk, and Schneider \cite{hilder-derijk-schneider} used a lucid modernization of these spatial dynamics techniques to construct supersonic fronts in a monatomic lattice with nearest+next-to-nearest-neighbor interactions.
Their precise choice of the lattice potentials leads the nonlinear Schr\"{o}dinger equation to be the natural continuum limit, unlike KdV as in the dimers.
The center spectrum of their version of $\L_0$ from the abstract system \eqref{eqn: sd intro} consists of $0$ and $\pm{i}\omega$ with each of these eigenvalues being geometrically simple and and algebraically double.
Taking into account their precise symmetry, this is a $0^{2-}(i\omega)^2$ bifurcation.
Due to these physical and theoretical differences, the authors did not elect to apply Lombardi's nanopteron theory but instead worked more directly with a normal form system and center manifold reduction.
For a related {\it{subsonic}} traveling wave problem in this lattice, they conjectured \cite[Conj. 6.1]{hilder-derijk-schneider} the existence of nanopterons.
We expect that a modification of our procedure in Section \ref{sec: abstract} from the $0^{4-}i\omega$ bifurcation to the $0^{2-}(i\omega)^2$ case could allow their conjecture to be solved by Lombardi's methods.

\item
{\it{Stability of nanopterons.}}
There are several stability results known for solitary waves in monatomic lattices.
Friesecke and Pego \cite{friesecke-pego2, friesecke-pego3, friesecke-pego4} proved that their long wave solutions are stable in very precise detail.
Herrmann and Matthies \cite{herrmann-matthies-asymptotic, herrmann-matthies-stability, herrmann-matthies-uniqueness} have studied solitary waves in a ``high-energy'' limit in monatomic lattices, where for certain convex potentials with singularities (unlike our smooth potentials), the energy of solutions becomes infinite in an appropriate limit.
Numerical evidence \cite{co-ops, faver-hupkes-numerics} indicates that the lattice nanopterons are metastable, but the precise details require much further exploration.
\end{enumerate}

\appendix

\section{Spectral Theory}\label{app: spectral theory}

\subsection{Elementary spectral theory}\label{app: elementary spectral theory}
This standard material is discussed at length in Section III.6 of \cite{kato} and nicely summarized in Appendix A.2 of \cite{haragus-iooss}; we include it for completeness.

Let $\X$ be a Banach space and $\L \colon \D \subseteq \X \to \X$ be a linear operator in $\X$.
We denote the spectrum of $\L$ by $\sigma(\L)$, the point spectrum or eigenvalues of $\L$ by $\sigma_{\pt}(\L)$, and the resolvent set of $\L$ by $\rho(\L)$.
If $z \in \rho(\L)$, we denote the resolvent of $\L$ at $\lambda$ by $\rhs(z) := (z\ind_{\D}-\L)^{-1}$.

Suppose that $\lambda \in \sigma_{\pt}(\L)$ is an isolated point of $\sigma(\L)$.
Choose $\ep > 0$ so small that if $0 < |z-\lambda| < 2\ep$, then $z \not\in \sigma(\L)$.
The spectral projection of $\L$ corresponding to $\lambda$ is the operator
\begin{equation}\label{eqn: spectral proj}
\Pi_{\lambda}
:= \frac{1}{2\pi{i}}\int_{|z-\lambda| = \ep} \rhs(z) \dz.
\end{equation}
The generalized eigenspace of $\L$ corresponding to $\lambda$ is the space $\Pi_{\lambda}(\X)$, and the dimension of $\Pi_{\lambda}(\X)$ is the algebraic multiplicity of $\lambda$ as an eigenvalue of $\L$.

The operator $\Pi_{\lambda}$ is a projection, and $\Pi_{\lambda}$ and $\L$ commute on $\D$.
Moreover,
\[
\sigma(\restr{\L}{\Pi_{\lambda}(\X)}) = \{\lambda\},
\quadword{and}
\sigma(\restr{\L}{(\ind_{\D}-\Pi_{\lambda})(\X)}) = \sigma(\L)\setminus\{\lambda\},
\]
and $\Pi_{\lambda}$ is the unique projection $\X \to \Pi_{\lambda}(\X)$ that commutes with $\L$.

More generally, if $\lambda_1,\ldots,\lambda_n \in \sigma_{\pt}(\L)$ are isolated points in $\sigma(\L)$, then the operator $\Pi := \medsum_{k=1}^n \Pi_{\lambda_k}$, with $\Pi_{\lambda_k}$ defined in \eqref{eqn: spectral proj}, is the spectral projection of $\L$ corresponding to $\{\lambda_1,\ldots,\lambda_n\}$.
This operator $\Pi$ satisfies
\[
\sigma(\restr{\L}{\Pi(\X)}) = \{\lambda_1,\ldots,\lambda_n\}
\quadword{and}
\sigma(\restr{\L}{(\ind_{\D}-\Pi)(\X)}) = \sigma(\L)\setminus\{\lambda_1,\ldots,\lambda_n\},
\]
and $\Pi$ is the unique projection $\X \to \Pi(\X)$ that commutes with $\L$.

\subsection{Jordan chains and algebraic multiplicities of eigenvalues}
If $\lambda$ is an isolated eigenvalue of $\L \colon \D \subseteq \X \to \X$ with algebraic multiplicity $\mfrak \ge 2$, a Jordan chain for $\lambda$ is an $\mfrak$-tuple $(\chi_0,\chi_1,\ldots,\chi_{\mfrak-1}) \in \D^{\mfrak}$ such that $(\lambda\ind_{\D}-\L)\chi_0 = 0$ and $(\lambda\ind_{\D}-\L)\chi_{k+1} = \chi_k$ for $0 \le k \le \mfrak-2$.
If $\chi_0 \ne 0$, the set $\{\chi_k\}_{k=0}^{\mfrak-1}$ is necessarily linearly independent.

For completeness, we state several results in detail but do not prove them.

\begin{lemma}\label{lem: alg mult}
Let $\X$ be a Banach space and $\L \colon \D \subseteq \X \to \X$ be a linear operator whose spectrum satisfies
\begin{equation}\label{abstr spec}
\sigma(\L)
= \set{\lambda \in \C}{f(\lambda) = 0}
\end{equation}
for some entire function $f \colon \C \to \C$.
Suppose as well that the resolvent of $\L$ has the form
\begin{equation}\label{abstr resolv}
\rhs(z)
= (z\ind_{\D} - \L)^{-1}
= \frac{1}{f(z)}\rhs_1(z) + \rhs_2(z), \ z \in \rho(\L),
\end{equation}
where $\rhs_1$, $\rhs_2 \colon \rho(\L) \to \b(\X)$ are analytic.
Then 

\begin{enumerate}[label={\bf(\roman*)}]

\item
$\sigma(\L)$ consists entirely of isolated points;

\item
If $\lambda \in \sigma(\L)$ is a geometrically simple eigenvalue of $\L$, then the algebraic multiplicity of $\lambda$ equals the multiplicity of $\lambda$ as a root of $f$.
\end{enumerate}
\end{lemma}

\begin{lemma}\label{lem: main spectral theory lemma}
Let $\X$ be a Banach space and let $\L \colon \D \subseteq \X \to \X$ be a linear operator whose spectrum consists entirely of eigenvalues: $\sigma(\L) = \sigma_{\pt}(\L)$.
Suppose that 0 is an isolated eigenvalue of geometric multiplicity 1 and algebraic multiplicity $\mfrak \ge 2$ for $\L$.
Write the spectral projection for $\L$ onto the generalized eigenspace corresponding to 0 as
\[
\Pi{U}
= \sum_{k=0}^{\mfrak-1} \chi_k^*(U)\chi_k,
\]
where $\chi_k^* \in \X^*$, $\L\chi_0 = 0$, $\chi_0 \ne 0$, $\L\chi_{k+1} = \chi_k$ for $0 \le k \le m-2$, and $\chi_k^*(\chi_j) = \delta_{j,k}$.

Now set
\[
\tX
:= \set{W \in \X}{\chi_0^*(W) = \chi_{m-1}^*(W) = 0}
\quadword{and}
\tD := \D \cap \tX.
\]
Define the operator $\tL$ in $\tX$ with domain $\tD$ by
\[
\tL{W}
:= \L{W} -\chi_1^*(W)\chi_0,
\ W \in \tD.
\]
Then 

\begin{enumerate}[label={\bf(\roman*)}]

\item
$\sigma(\tL) = \sigma(\L)$, and every point in $\sigma(\tL)$ is an eigenvalue of $\tL$.

\item\label{part: alg mult for tL}
If $\lambda \in \sigma(\tL)\setminus\{0\}$ is geometrically simple, then the algebraic multiplicity of $\lambda$ as an eigenvalue for $\tL$ equals its algebraic multiplicity as an eigenvalue of $\L$.

\item
If $\mfrak \ge 3$, then 0 is an eigenvalue of algebraic multiplicity $\mfrak-2$ for $\tL$ and the operator
\begin{equation}\label{eqn: abstract 0 spectral proj reduced}
\tPi{W}
:= \sum_{k=1}^{\mfrak-2} \chi_k^*(W)\chi_k
\end{equation}
is the spectral projection for $\tL$ onto the generalized eigenspace corresponding to 0.
\end{enumerate}
\end{lemma}
\section{Properties of the Hyperbolic Tangent}

The lemmas in this appendix are all variations on the common, and familiar, theme that if a function asymptotes exponentially fast at $\pm\infty$ to constant values, then that function is an exponentially localized perturbation of a suitably scaled hyperbolic tangent, plus a constant.
We just need, unsurprisingly, certain sharp estimates.

\begin{lemma}\label{lem: tanh approx}
Suppose that $f \colon \R \to \R$ is a function that satisfies the estimates
\[
\sup_{X \ge X_0} e^{q_*X}|f(X)-L_+| \le C_+ ,
\quad
\sup_{-X_0 \le X \le X_0} |f(X)| \le C_0,
\quad
\text{and}
\quad
\sup_{X \le -X_0} e^{-q_*X}|f(X)-L_-| \le C_-
\]
for some $q_*$, $X_0$, $C_{\pm}$, $C_0 > 0$ and $L_{\pm} \in \R$.
Then
\begin{multline*}
\sup_{X \in \R} e^{q_*|X|}\left|f(X) - \left[\left(\frac{L_+-L_-}{2}\right)\tanh(q_*X) + \frac{L_++L_-}{2}\right]\right| \\
< \max\left\{C_+ + \frac{|L_--L_+|}{2}, C_- + \frac{|L_--L_+|}{2}, e^{q_*|X_0|}\big(C_0 + \max\{|L_+|,|L_-|\}\big)\right\}.
\end{multline*}
\end{lemma}

\begin{proof}
Rewrite
\[
\left(\frac{L_+-L_-}{2}\right)\tanh(q_*X) + \frac{L_++L_-}{2}
= \frac{L_+}{2}\big(\tanh(q_*X)+1\big)-\frac{L_-}{2}\big(\tanh(q_*X)-1\big).
\]
Then
\[
f(X) - \left[\left(\frac{L_+-L_-}{2}\right)\tanh(q_*X) + \frac{L_++L_-}{2}\right]
=  \big(f(X) - L_+\big) + \left(\frac{L_--L_+}{2}\right)\big(\tanh(q_*X)-1\big),
\]
and so
\[
\sup_{X \ge X_0} e^{q_*X}\left|f(X) - \left[\left(\frac{L_+-L_-}{2}\right)\tanh(q_*X) + \frac{L_++L_-}{2}\right]\right|
\le C_+ + \frac{|L_--L_+|}{2}.
\]
by the hypothesis on $f$ and the estimate $|\tanh(q_*X)-1| \le 1$, valid for $X \ge 0$.

Similarly, we may rewrite
\[
\left(\frac{L_+-L_-}{2}\right)\tanh(q_*X) + \frac{L_++L_-}{2}
= (f(X)-L_-) + \left(\frac{L_--L_+}{2}\right)\big(\tanh(q_*X)+1\big)
\]
to see that 
\[
\sup_{X \le -X_0} e^{-q_*X}\left|f(X) - \left[\left(\frac{L_+-L_-}{2}\right)\tanh(q_*X) + \frac{L_++L_-}{2}\right]\right|
\le C_- + \frac{|L_--L_+|}{2}
\]
by the hypothesis on $f$ and the estimate $|\tanh(q_*X)+1| \le 1$, valid for $X \le 0$.

Finally, a naive estimate with the triangle inequality gives
\begin{multline*}
\sup_{-X_0 \le X \le X_0} e^{q_*|X|}\left|f(X) - \left[\left(\frac{L_+-L_-}{2}\right)\tanh(q_*X) + \frac{L_++L_-}{2}\right]\right| \\
\le e^{q_*|X_0|}\big(C_0 + \max\{|L_+|,|L_-|\}\big). 
\qedhere
\end{multline*}
\end{proof}

\begin{lemma}\label{lem: tanh diff}
Let $0 < q_1 < q_2$.
Then
\[
\sup_{X \in \R} e^{2q_1|X|}|\tanh(q_1X)-\tanh(q_2X)| 
< \infty.
\]
\end{lemma}

\begin{proof}
Without loss of generality, let $0 < q_1 < q_2$.
Suppose $X > 0$.
Then
\[
|\tanh(q_1X)-\tanh(q_2X)|
= \left|\int_{q_1X}^{q_2X} \sech^2(s) \ds\right|
\le 4\int_{q_1X}^{q_2X} e^{-2s} \ds
= 2e^{-2q_1X}\left(\frac{1-e^{2(q_1-q_2)X}}{2}\right),
\]
and so
\[
\sup_{X > 0} e^{2q_1|X|}|\tanh(q_1X)-\tanh(q_2X)|
< \infty.
\]
The same estimate over $X < 0$ follows from the oddness of the hyperbolic tangent.
\end{proof}

\begin{lemma}\label{lem: the big tanh lemma}
Suppose that $f_{\nu}^{\alpha} \colon \R \to \R$ is a family of even, contimuous functions defined for $\nu \in \mathscr{I}_1 \subseteq (0,1)$ and $\alpha \in \mathscr{I}_2 \subseteq \R$.
Let $\{\vartheta_{\nu}^{\alpha}\}_{\nu \in \mathscr{I}_1, \alpha \in \mathscr{I}_2} \subseteq [0,\infty)$, and suppose there is $C > 0$ such that 
\[
\norm{f_{\nu}^{\alpha}}_{L^{\infty}} \le C,
\quadword{and}
\vartheta_{\nu}^{\alpha} \le C\nu^2
\]
for all $\nu \in \mathscr{I}_1$ and $\alpha \in \mathscr{I}_2$.
Fix $q_* > 0$.

\begin{enumerate}[label={\bf(\roman*)}, ref={(\roman*)}]

\item\label{part: tanh1}
Define
\[
F_{\nu}^{\alpha}(X)
:= \int_{X+\vartheta_{\nu}^{\alpha}\tanh(q_*X)}^{X+\nu+\vartheta_{\nu}^{\alpha}\tanh(q_*X+q_*\nu)} f_{\nu}^{\alpha}(s) \ds
- \int_{X+\vartheta_{\nu}^{\alpha}\tanh(q_*X)}^{X+\nu+\vartheta_{\nu}^{\alpha}\tanh(q_*X)} f_{\nu}^{\alpha}(s) \ds.
\]
Then
\begin{equation}\label{eqn: tanh int est2}
\sup_{\substack{\nu \in \mathscr{I}_1 \\ \alpha \in \mathscr{I}_2 \\ X \in \R}} \nu^{-3}e^{2q_*|X|}|F_{\nu}^{\alpha}(X)|
< \infty.
\end{equation}

\item\label{part: tanh3}
Suppose that $0 < q < q_*$ and
\[
\sup_{\substack{\nu \in \mathscr{I}_1 \\ \alpha \in \mathscr{I}_2 \\ X \in \R}} e^{q|X|}|f_{\nu}^{\alpha}(X)|
< \infty.
\]
Then there exists $L_{\nu}^{\alpha,\int} \in \R$ such that 
\[
\sup_{\substack{\nu \in \mathscr{I}_1 \\ \alpha \in \mathscr{I}_2}}
\left(\sup_{X \in \R}
e^{q|X|}\left|\int_0^X f_{\nu}^{\alpha}(s) \ds - L_{\nu}^{\alpha,\int}\tanh\left(\frac{q_*X}{2}\right)\right|\right)
+ |L_{\nu}^{\alpha,\int}|
< \infty.
\]
\end{enumerate}
\end{lemma}

\begin{proof}

\begin{enumerate}[label={\bf(\roman*)}]

\item
We estimate
\[
|F_{\nu}^{\alpha}(X)|
= \left|\int_{X+\nu+\vartheta_{\nu}^{\alpha}\tanh(q_*X)}^{X+\nu+\vartheta_{\nu}^{\alpha}\tanh(q_*X+q_*\nu)} f_{\nu}^{\alpha}(s) \ds\right|
\le \norm{f_{\nu}^{\alpha}}_{L^{\infty}}\vartheta_{\nu}^{\alpha}|\tanh(q_*X+q_*\nu)-\tanh(q_*X)|.
\]
Then \eqref{eqn: tanh int est2} follows from a Lipschitz estimate on the hyperbolic tangent.

\item
Let
\[
\I_{\nu}^{\alpha}(X) := \int_0^X f_{\nu}^{\alpha}(s) \ds
\quadword{and}
L_{\nu}^{\alpha,\int} := \int_0^{\infty} f_{\nu}^{\alpha}(s) \ds.
\]
Then $\I_{\nu}^{\alpha}$ satisfies the hypotheses of Lemma \ref{lem: tanh approx} with $L_+ = L_{\nu}^{\alpha,\int}$ and $L_- = -L_{\nu}^{\alpha,\int}$, thanks to the evenness of $f_{\nu}^{\alpha}$.
This allows us to write
\[
\I_{\nu}^{\alpha}(X)
= L_{\nu}^{\alpha,\int}\tanh(qX) + R_{\nu,1}^{\alpha}(X),
\]
where there is $C > 0$ such that $|R_{\nu,1}^{\alpha}(X)| \le Ce^{-q|X|}$ for all $\alpha$, $\nu$, and $X$.

Now set
\[
R_{\nu,2}^{\alpha}(X)
:= L_{\nu}^{\alpha,\int}\left[\tanh(qX)-\tanh\left(\frac{q_*X}{2}\right)\right],
\]
so that 
\[
\I_{\nu}^{\alpha}(X)
= L_{\nu}^{\alpha,\int}\tanh\left(\frac{q_*X}{2}\right) + R_{\nu,1}^{\alpha}(X) + R_{\nu,2}^{\alpha}(X).
\]
Lemma \ref{lem: tanh diff} implies
\[ 
|R_{\nu,2}^{\alpha}(X)| 
\le C\exp\left(-2\min\left\{q,\frac{q_*}{2}\right\}|X|\right)
\] 
for all $\alpha$, $\nu$, and $X$.
We know $q < q_*$ and certainly $q < 2q$.
Thus
\[
2\min\left\{q,\frac{q_*}{2}\right\}
= \min\{2q,q_*\}
> q,
\]
and so $|R_{\nu,2}^{\alpha}(X)| \le C\exp(-2\min\{q,q_*/2\}|X|) < e^{-q|X|}$.
\qedhere
\end{enumerate}
\end{proof}

\begin{lemma}\label{lem: the other big tanh lemma}
Assume the notation and hypotheses of Lemma \ref{lem: the big tanh lemma}.
Suppose that, as well, there is $C > 0$ such that $\Lip(f_{\nu}^{\alpha}) \le C\nu^{-1}$ for all $\nu \in \mathscr{I}_1$ and $\alpha \in \mathscr{I}_2$.

\begin{enumerate}[label={\bf(\roman*)},ref={(\roman*)}]

\item\label{part: tanh2}
Define
\[
G_{\nu}^{\alpha}(X)
:= \int_0^X f_{\nu}^{\alpha}(s+\vartheta_{\nu}^{\alpha}\tanh(q_*s)) \ds
- \int_0^{X+\vartheta_{\nu}^{\alpha}\tanh(q_*X)} f_{\nu}^{\alpha}(s) \ds.
\]
Then there exists $L_{\nu}^{\alpha,\infty} \in \R$ such that 
\begin{equation}\label{eqn: tanh int est}
\sup_{\substack{\nu \in \mathscr{I}_1 \\ \alpha \in \mathscr{I}_2 \\}} \left[\nu^{-1}\left(\sup_{X \in \R} e^{q_*|X|}|G_{\nu}^{\alpha}(X) - \nu{L}_{\nu}^{\alpha,\infty}\tanh(q_*X)|\right) + |L_{\nu}^{\alpha,\infty}|\right]
< \infty.
\end{equation}

\item\label{part: tanh4}
Define
\[
H_{\nu}^{\alpha}(X)
:= f_{\nu}^{\alpha}(X+\nu+\vartheta_{\nu}^{\alpha}\tanh(q_*X+q_*\nu))
- f_{\nu}^{\alpha}(X+\nu+\vartheta_{\nu}^{\alpha}\tanh(q_*X)).
\]
Then
\[
\sup_{\substack{\nu \in \mathscr{I}_1 \\ \alpha \in \mathscr{I}_2 \\ X \in \R}}
\nu^{-2}e^{-2q_*|X|}|H_{\nu}^{\alpha}(X)|
< \infty.
\]
\end{enumerate}
\end{lemma}

\begin{proof}

\begin{enumerate}[label={\bf(\roman*)}]

\item
Let $\Xi \in \Cal^{\infty}$ be a function such that $\Xi(X) = \sgn(X)$ for $|X| \ge 1$.
Define
\[
\I_{1,\nu}^{\alpha}(X)
:= \int_0^X \big[f_{\nu}^{\alpha}(s+\vartheta_{\nu}^{\alpha}\tanh(q_*s))-f_{\nu}^{\alpha}(s+\Xi(X)\vartheta_{\nu}^{\alpha})\big] \ds
\]
and
\[
\I_{1,\nu}^{\alpha}(X)
:= \int_0^X f_{\nu}^{\alpha}(s+\Xi(X)\vartheta_{\nu}^{\alpha}) \ds - \int_0^{X+\vartheta_{\nu}^{\alpha}\tanh(q_*X)} f_{\nu}^{\alpha}(s) \ds.
\]
Then $G_{\nu}^{\alpha}(X) = \I_{1,\nu}^{\alpha}(X) + \I_{1,\nu}^{\alpha}(X)$.
We will estimate $\I_{1,\nu}^{\alpha}$ and $\I_{1,\nu}^{\alpha}$ separately using Lemma \ref{lem: tanh approx}.

\begin{enumerate}[label={\it{Estimates on}}]

\item
$\I_{1,\nu}^{\alpha}$.
First suppose $X \le -1$.
For $X \le s \le 0$, we have
\[
\big|f_{\nu}^{\alpha}(s+\vartheta_{\nu}^{\alpha}\tanh(q_*s))-f_{\nu}^{\alpha}(s+\Xi(X)\vartheta_{\nu}^{\alpha})\big| 
\le \Lip(f_{\nu}^{\alpha})\vartheta_{\nu}^{\alpha}|1+\tanh(q_*s)|
\le \Lip(f_{\nu}^{\alpha})\vartheta_{\nu}^{\alpha}e^{q_*s}.
\]
It follows that the improper integral
\[
L_{\nu,-}^{\alpha}
:= \lim_{X \to -\infty} \I_{1,\nu}^{\alpha}(X)
\]
exists, with 
\[
|L_{\nu,-}^{\alpha}| 
\le \Lip(f_{\nu}^{\alpha})\vartheta_{\nu}^{\alpha}q_*^{-1}
\le C\nu, 
\]
and also that
\[
|\I_{1,\nu}^{\alpha}(X)-L_{\nu,-}^{\alpha}|
\le \frac{\Lip(f_{\nu}^{\alpha})\vartheta_{\nu}^{\alpha}}{q_*}e^{q_*X}
\le C\nu{e}^{q_*X}.
\]

Now suppose $-1 \le X \le 1$.
Then
\[
|\I_{1,\nu}^{\alpha}(X)|
\le \Lip(f_{\nu}^{\alpha})\vartheta_{\nu}^{\alpha}\int_{-1}^1 |\Xi(X)-\tanh(q_*s)| \ds
\le C\nu,
\]
since the integral is of course finite.

Last, suppose $X \ge 1$.
For $0 \le s \le X$, we have
\[
|f_{\nu}^{\alpha}(s+\vartheta_{\nu}^{\alpha}\tanh(q_*s))-f_{\nu}^{\alpha}(s+\Xi(X)\vartheta_{\nu}^{\alpha})\big| 
\le \Lip(f_{\nu}^{\alpha})\vartheta_{\nu}^{\alpha}|1-\tanh(q_*s)| 
\le \Lip(f_{\nu}^{\alpha})\vartheta_{\nu}^{\alpha}e^{-q_*s}.
\]
It follows that the improper integral 
\[
L_{\nu,+}^{\alpha} := \lim_{X \to \infty} \I_{1,\nu}^{\alpha}(X)
\]
exists, with $|L_{\nu,+}^{\alpha}| \le \Lip(f_{\nu}^{\alpha})\vartheta_{\nu}^{\alpha}q_*^{-1}$, and also that
\[
|\I_{1,\nu}^{\alpha}(X)-L_{\nu,+}^{\alpha}|
\le \frac{\Lip(f_{\nu}^{\alpha})\vartheta_{\nu}^{\alpha}}{q_*}e^{-q_*X}
\le C\nu{e}^{-q_*X}.
\]

Furthermore, since $f_{\nu}^{\alpha}$ is even, a lengthy, but straightforward, calculation shows that $\I_{1,\nu}^{\alpha}$ is odd, and so $L_{\nu,+}^{\alpha} = -L_{\nu,-}^{\alpha}$.
Lemma \ref{lem: tanh approx} then gives the estimate
\begin{equation}\label{eqn: tanhI est}
|\I_{1,\nu}^{\alpha}(X)- L_{\nu,+}^{\alpha}\tanh(q_*X)|
\le C\nu{e}^{-q_*|X|}.
\end{equation}

\item
$\I_{1,\nu}^{\alpha}$.
First suppose $X \le -1$.
Then
\[
\I_{1,\nu}^{\alpha}(X)
= \int_{-\vartheta_{\nu}^{\alpha}}^0 f_{\nu}^{\alpha}(s) \ds
- \int_{X-\vartheta_{\nu}^{\alpha}}^{X+\vartheta_{\nu}^{\alpha}\tanh(q_*X)}
f_{\nu}^{\alpha}(s) \ds,
\]
and so
\[
\left|\I_{1,\nu}^{\alpha}(X)-\int_{-\vartheta_{\nu}^{\alpha}}^0 f_{\nu}^{\alpha}(s) \ds\right|
\le \norm{f_{\nu}^{\alpha}}_{L^{\infty}}\vartheta_{\nu}^{\alpha}|\tanh(q_*X)+1|
\le C\nu^2e^{q_*X}.
\]
Also,
\[
\left|\int_{-\vartheta_{\nu}^{\alpha}}^0 f_{\nu}^{\alpha}(s)\ds\right|
\le \norm{f_{\nu}^{\alpha}}_{L^{\infty}}\vartheta_{\nu}^{\alpha}
\le C\nu^2.
\]

Next, if $-1 \le X \le 1$, then
\[
\I_{1,\nu}^{\alpha}(X)
= \int_0^X f_{\nu}^{\alpha}(s+\Xi(X)\vartheta_{\nu}^{\alpha})-f_{\nu}^{\alpha}(s) \ds
+ \int_{X+\vartheta_{\nu}^{\alpha}\tanh(q_*X)}^X f_{\nu}^{\alpha}(s) \ds.
\]
We estimate
\[
\left|\int_0^X f_{\nu}^{\alpha}(s+\Xi(X)\vartheta_{\nu}^{\alpha})-f_{\nu}^{\alpha}(s) \ds\right|
\le \Lip(f_{\nu}^{\alpha})\vartheta_{\nu}^{\alpha}|\Xi(X)|\int_{-1}^1 \ds
= 2\Lip(f_{\nu}^{\alpha})\vartheta_{\nu}^{\alpha}
\le C\nu.
\]
We bound the second integral by 
\[
\left|\int_{X+\vartheta_{\nu}^{\alpha}\tanh(q_*X)}^X f_{\nu}^{\alpha}(s) \ds\right|
\le \vartheta_{\nu}^{\alpha}|\tanh(q_*X)|\norm{f_{\nu}^{\alpha}}_{L^{\infty}}
\le \vartheta_{\nu}^{\alpha}\norm{f_{\nu}^{\alpha}}_{L^{\infty}}
\le C\nu^2.
\]
Thus
\[
\sup_{-1 \le X \le 1} |\I_{1,\nu}^{\alpha}(X)|
\le C\nu^2+C\nu
\le C\nu.
\]

Finally suppose $X \ge 1$.
Then
\[
\I_{1,\nu}^{\alpha}(X)
= -\int_0^{\vartheta_{\nu}^{\alpha}} f_{\nu}^{\alpha}(s) \ds
+ \int_{X+\vartheta_{\nu}^{\alpha}\tanh(q_*X)}^{X+\vartheta_{\nu}^{\alpha}} f_{\nu}^{\alpha}(s) \ds,
\]
and so
\[
\left|\I_{1,\nu}^{\alpha}(X)
- \left(-\int_0^{\vartheta_{\nu}^{\alpha}} f_{\nu}^{\alpha}(s) \ds\right)\right|
\le \norm{f_{\nu}^{\alpha}}_{L^{\infty}}\vartheta_{\nu}^{\alpha}|1-\tanh(q_*X)|
\le \norm{f_{\nu}^{\alpha}}_{L^{\infty}}\vartheta_{\nu}^{\alpha}e^{-q_*X}
\le C\nu^2e^{-q_*X}
\]
with
\[
\left|\int_0^{\vartheta_{\nu}^{\alpha}} f_{\nu}^{\alpha}(s) \ds\right|
\le \norm{f_{\nu}^{\alpha}}_{L^{\infty}}\vartheta_{\nu}^{\alpha}
\le C\nu^2.
\]
Lemma \ref{lem: tanh approx} then gives the estimate
\begin{equation}\label{eqn: tanhII est}
\left|\I_{1,\nu}^{\alpha}(X) - \left(-\int_0^{\vartheta_{\nu}^{\alpha}} f_{\nu}^{\alpha}(s) \ds\right)\tanh(q_*X)\right|
\le C\nu{e}^{-q_*|X|}.
\end{equation}
\end{enumerate}

We put 
\[
L_{\nu}^{\alpha,\infty} 
:= \nu^{-1}\left(L_{\nu,+}^{\alpha}-\int_0^{\vartheta_{\nu}^{\alpha}} f_{\nu}^{\alpha}(s) \ds\right)
\]
and use the decomposition $G_{\nu}^{\alpha} = \I_{1,\nu}^{\alpha} + \I_{1,\nu}^{\alpha}$ and the estimates \eqref{eqn: tanhI est} and \eqref{eqn: tanhII est} to conclude the desired estimate \eqref{eqn: tanh int est}.

\item
This follows from Lipschitz estimateson $f_{\nu}^{\alpha}$ and the hyperbolic tangent, and the estimate on $\vartheta_{\nu}^{\alpha}$.
\qedhere
\end{enumerate}
\end{proof}
\section{Optimal Regularity}\label{app: opt reg}

\subsection{Overview of the optimal regularity problems and prior results}\label{app: opt reg intro}
An essential part of Lombardi's nanopteron program involves solving equations of the form
\begin{equation}\label{eqn: opt reg ur-eqn}
f'(z)
= \frac{\A}{\mu}f(z) + g(z), \qquad z \in \U_b := \set{z \in \C}{|\im(z) < b}, \ b > 0
\end{equation}
and
\begin{equation}\label{eqn: subopt reg ur-eqn}
f'(x)
= \A{f}(x) + g(x), \qquad x \in \R,
\end{equation}
with different conditions on the affine term $g$ and the desired solution $f$ in each case.
In both equations, $\A \colon \D \subseteq \X \to \X$ is a linear operator in the Banach space $\X$, and a solution to the equation is a map $f \in \Cal(\R,\D) \cap \Cal^1(\R,\X)$ satisfying the requisite equality pointwise.
For \eqref{eqn: opt reg ur-eqn}, one must contend with the small singular perturbation parameter $\mu \gtrsim 0$ and work with functions $f$ and $g$ that are holomorphic and exponentially localized on the strip $\U_b$.
One seeks an ``optimal'' regularity result for \eqref{eqn: opt reg ur-eqn} in the sense that if $g$ and its first $r$ derivatives belong to a certain function space $\W$, then $f$ and its first $r+1$ derivatives also belong to $\W$.
In contrast, for \eqref{eqn: subopt reg ur-eqn} $f$ and $g$ are merely $r$-times continuously differentiable on $\R$ and may grow at $\pm\infty$, and here one is content with a ``suboptimal'' regularity result where the solution $f$ need only be as differentiable as $g$.

The optimal regularity problem rears its head in the development of Lombardi's fixed point problem for the nanopteron's exponentially localized tails \cite[Sec.\@ 8.2, 8.4.1]{Lombardi}.
The suboptimal regularity problem is buried in the proof of Lombardi's infinite-dimensional normal form transformation \cite[Thm.\@ 8.1.10, App.\@ 8.A]{Lombardi}, which provides the essential change of variables that transforms his original problem into the system on which he actually runs his nanopteron program.
A third, somewhat tamer, optimal regularity problem in periodic Sobolev spaces must be solved for the construction of the exact periodic solutions \cite[App.\@ 8.B]{Lombardi}; we do not dwell on this more straightforward problem in this introduction.

Both regularity problems have the flavor of a fundamental concern for invocations of the center manifold theorem, see \cite[(H)--(ii), p.\@ 127]{VI} and \cite[Hypo.\@ 2.7]{haragus-iooss}.
In those situations, one wants to solve \eqref{eqn: subopt reg ur-eqn} in a space of growing functions with an optimal regularity result, i.e., $g$ should satisfy a condition like
\[
\sup_{x \in \R} \norm{e^{-q|x|}\partial_x^j[g](x)}_{\X} < \infty, \ j = 0,\ldots,r
\]
with $q > 0$, and $f$ should satisfy this estimate for $j=0,\ldots,r+1$.
Moreover, the mapping $g \mapsto f$ is bounded in these norms.
One common way to obtain this result is to show it explicitly for $q = 0$ and then use a more abstract perturbation-theoretic argument in $q$ to extend it to small $q > 0$.
See \cite[Lem.\@ 4.4]{james-sire} and \cite[Lem.\@ 2.3]{mielke-uber} for such perturbations involving essentially arbitrary operators and Banach spaces.
This has the flavor of the ``operator conjugation'' techniques used in such diverse contexts as, for example, \cite[Eqn.\@ (1.15)]{pego-weinstein}, \cite[Prop.\@ 3.4]{hvl}, and \cite[App.\@ D.4]{faver-dissertation}.
In the context of lattice problems, this ``optimal regularity'' step has been treated in \cite[Sec.\@ 4]{iooss-kirchgassner}, the progenitor of these techniques, as well as \cite[Sec.\@ 4]{iooss-fpu}, \cite[Sec.\@ 5]{calleja-sire}, \cite[Sec.\@ 4]{james-sire}, \cite[Sec.\@ 4]{sire}, and \cite[Sec.\@ 3]{hilder-derijk-schneider}; it has been mentioned, but not considered in full detail, in \cite[Sec.\@ 2.B]{iooss-james} and \cite[Sec.\@ 3.4]{venney-zimmer}.

However, a perturbation-in-the-decay-parameter approach to \eqref{eqn: opt reg ur-eqn} could be fraught with difficulty in a singularly perburbed problem like \eqref{eqn: opt reg ur-eqn}, since it is not clear how nicely the perturbations would depend on the delicate small parameter $\mu$.
Furthermore, our reading of Lombardi leads us to desire not an optimal regularity result in a space of {\it{growing}} functions but rather an optimal regularity result in a space of {\it{decaying}} functions.
We will be content with a ``suboptimal'' regularity result in a space of growing functions, i.e., one where the solution $f$ is not necessarily smoother than the forcing function $g$.
Finally, we have no intention of applying the center manifold theorem in the first place, since we are adopting and adapting Lombardi's methods, and he explicitly avoids the center manifold theorem so that he can work in spaces of analytic functions, for the purposes of invoking his highly precise integral estimates.

For these reasons, we pose separate hypotheses for the successful resolution of the optimal and suboptimal regularity problems in the general Banach space framework of Section \ref{sec: abstract}, and the language of these hypotheses are quite different from the optimal regularity discussions of our predecessors.
Our sketched verification of these hypotheses for our concrete lattice problems in Section \ref{sec: opt reg}, however, ultimately follows, at least formally, the general strategies of these antecedents, and we are indebted to them for their insights and labors.

\subsection{Function spaces for optimal regularity}\label{app: opt reg function spaces}
To achieve a healthy level of generality that encompasses both the lattice and water wave problems, we need a somewhat intricate family of function spaces.

Let $\X$ be a Banach space, let $q \in \R$, and let $b > 0$.
Let $\U_b$ be the horizontal complex strip
\[
\U_b
:= \set{z \in \C}{|\im(z)| < b}.
\]
Take $\H_{q,b}(\X)$ to be the space of holomorphic functions $f \colon \U_b \to \X$ such that 
\[
\norm{f}_{\H_{q,b}(\X)}
:= \sup_{z \in \U_b} e^{q|\re(z)|}\norm{f(z)}_{\X}
< \infty.
\]
If $f \in \H_{q,b}(\X)$, we write
\[
\restr{f}{y}(x)
:= f(x+iy)
\]
for $x \in \R$ and $y \in (-b,b)$.
Throughout, $q$ will remain fixed, and we will not be making any assumptions on the $q$-dependence of various estimates and existence regimes; in general, this dependence could be quite bad as $q$ approaches some extreme limit.

Nanopteron constructions using spatial dynamics tend to involve one of two kinds of subspaces of $\H_{q,b}(\X)$.
First, for an integer $r \ge 0$, define $\Cal_{q,b}^r(\X)$ to be the space of all $f \in \H_{q,b}(\X)$ such that 
\[
\norm{f}_{\Cal_{q,b}^r(\X)}
:= \max_{0 \le j \le r} \sup_{z \in \U_b} e^{q|\re(z)|}\norm{\partial_z^j[f](z)}_{\X} 
< \infty.
\]
Versions of this space have appeared in all of the lattice spatial dynamics papers \cite{iooss-kirchgassner, iooss-fpu, calleja-sire, james-sire, hilder-derijk-schneider,iooss-james,venney-zimmer} due to their close relation to the natural function spaces used in center manifold theory \cite[p.\@ 127]{VI}, \cite[p.\@ 29]{haragus-iooss}.

Lombardi, however, uses the following vector-valued ``localized Sobolev space'' in his abstract treatment of the water wave problem \cite[p.\@ 341]{Lombardi}.
Let $\E_{q,b}^r(\X)$ be the space of all holomorphic functions $f \colon \U_b \to \X$ such that 
\[
\norm{f}_{\E_{q,b}^r(\X)}
:= \max_{0 \le j \le r} \sup_{-b < y < b} \left(\int_{-\infty}^{\infty} e^{2q|x|}\norm{\partial_z^j[f](x+iy)}_{\X}^2\dx\right)^{1/2}
< \infty.
\]

Both $\Cal_{q,b}^r(\X)$ and $\E_{q,b}^r(\X)$ are Banach spaces, and, more than that, they really have the same underlying structure.
We claim that the properties in the following definition and the subsequent Lemma \ref{lem: Lombardi amenable} are the features that Lombardi really needs in his proofs; otherwise, the algebraic structure of $\E_{q,b}^r(\X)$ is incidental.

\begin{definition}\label{defn: Lombardi amenable}
Let $\Cal^+(\R)$ denote the cone of all continuous functions $f \colon \R \to [0,\infty)$ and let $\Mfrak \colon \Cal^+(\R) \to [0,\infty]$ satisfy 

\begin{enumerate}[label={\bf(M\arabic*)}, ref={M\arabic*}]

\item\label{part: M1}
$\Mfrak(f+g) \le \Mfrak(f) + \Mfrak(g)$ and $\Mfrak(\alpha{f}) = \alpha\Mfrak(f)$ for all $f$, $g \in \Cal^+(\R)$ and $\alpha \ge 0$;

\item\label{part: M2}
If $\Mfrak(f) = 0$ for some $f \in \Cal^+(\R)$, then $f(x) = 0$ for all $x \in \R$;

\item\label{part: M3}
If $f$, $g \in \Cal^+(\R)$ with $f(x) \le g(x)$ for all $x \in \R$, then $\Mfrak(f) \le \Mfrak(g)$.
\end{enumerate}

Denote by $\W_{q,b}^r(\X;\Mfrak)$ the space of all holomorphic functions $f \colon \U_b \to \X$ such that 
\[
\norm{f}_{\W_{q,b}^r(\X;\Mfrak)}
:= \max_{0 \le j \le r} \sup_{-b < y < b} \Mfrak(e^{q|\cdot|}\norm{\partial_z^j[\restr{f}{y}]}_{\X})
< \infty.
\]
Suppose that the following hold for all $q \in \R$, all integers $r \ge 0$, and all Banach spaces $\X$.

\begin{enumerate}[label={\bf(W\arabic*)}]

\item
Each $\W_{q,b}^r(\X;\Mfrak)$ is a Banach space.

\item
At least one of $\W_{q,b}^0(\X;\Mfrak)$ or $\W_{q,b}^1(\X;\Mfrak)$ embeds continuously into $\H_{q,b}^0(\X)$.
\end{enumerate}
Then we say that the map $\Mfrak$ is \defn{Lombardi-amenable}.
\end{definition}

Certainly each of the $L^p$-norms, $1 \le p \le \infty$, is Lombardi-amenable.
Since $\Cal_{q,b}^r(\X) = \W_{q,b}^r(\X;\norm{\cdot}_{L^{\infty}})$ and $\E_{q,b}^r(\X) = \W_{q,b}^r(\X;\norm{\cdot}_{L^2})$, both $\norm{\cdot}_{L^{\infty}}$ and $\norm{\cdot}_{L^2}$ are Lombardi-amenable.
Since $\Cal_{q,b}^0(\X) = \H_{q,b}^0(\X)$, we have the embedding for $\Cal_{q,b}^r(\X)$ with $r = 0$, while for $\E_{q,b}^r(\X)$, we need to take $r \ge 1$ to invoke the Sobolev embedding.

Our notation elides the fact that the definition of $\W_{q,b}^r(\X;\Mfrak)$ really depends on the norm chosen for $\X$.
A more precise notation would be something like $\W_{q,b}^r(\X;\norm{\cdot}_{\X},\Mfrak)$, but this is too baroque, even for us.

Here are two useful, and unsurprising, properties of the spaces $\W_{q,b}^r(\X;\Mfrak)$.
The first relates membership of a function in $\W_{q,b}^r(\X;\Mfrak)$ to membership of its derivatives in $\W_{q,b}^0(\X;\Mfrak)$.
The second gives a comparison test for the purpose of establishing membership in $\W_{q,b}^r(\X;\Mfrak)$ and inheriting estimates in $\W_{q,b}^r(\X;\Mfrak)$ from ``pointwise'' estimates that in particular relates results involving $\norm{\cdot}_{\X}$ to conclusions involving $\norm{\cdot}_{\W_{q,b}^r(\X;\Mfrak)}$.

\begin{lemma}\label{lem: Lombardi amenable}
Let $\Mfrak \colon \Cal^+(\R) \to [0,\infty]$ be Lombardi-amenable; let $q \in \R$ and $r \ge 0$ be an integer; and let $\X$ be a Banach space.

\begin{enumerate}[label={\bf(\roman*)}]

\item
Let $f \in \H_{q,b}(\X)$.
Then $f \in \W_{q,b}^r(\X;\Mfrak)$ if and only if $\partial_z^j[f] \in \W_{q,b}^0(\X;\Mfrak)$ for $j=0,\ldots,r$.

\item
Let $\X_1,\ldots,\X_n$ also be Banach spaces.
Suppose that $f \colon \U_b \to \X$ is holomorphic and there are functions $g_k \in \W_{q,b}^r(\X_k;\Mfrak)$, $k=1,\ldots,n$, and a continuous map $M_n \colon [0,\infty)^n \to [0,\infty)$ such that $M_n(0) = 0$ and
\[
\norm{f(z)}_{\X}
\le M_n(\norm{g_1(z)}_{\X_1},\ldots,\norm{g_n(z)}_{\X_n}).
\]
Then $f \in \W_{q,b}^0(\X;\Mfrak)$ and 
\[
\norm{f}_{\W_{q,b}^0(\X)}
\le M_n(\norm{g_1}_{\W_{q,b}^r(\X_1\;\Mfrak)},\ldots,\norm{g_n}_{\W_{q,b}^r(\X_n;\Mfrak)}).
\]
\end{enumerate}
\end{lemma}

\subsection{Optimal regularity for localized spaces on complex strips}\label{app: loc opt reg}
We have constructed the elaborate $\W_{q,b}^r(\X;\Mfrak)$-spaces of Definition \ref{defn: Lombardi amenable} in order to pose the following two optimal regularity properties on them.
The first is the most fundamental for Lombardi's constructions.
We use the convention that if $\X_1$ and $\X_2$ are normed spaces, then 
\[
\norm{U}_{\X_1\cap\X_2}
:= \norm{U}_{\X_1} + \norm{U}_{\X_2}.
\]

\begin{definition}\label{defn: opt reg}
Let $\Mfrak \colon \Cal^+(\R) \to [0,\infty]$ be Lombardi-amenable.
Let $\D$, $\Y$, and $\X$ be Banach spaces with $\D$ and $\Y$ continuously embedded in $\X$.
A linear operator $\A \colon \D \to \X$ has the \defn{localized optimal regularity property} on $(\D,\Y,\X)$ with decay rate $q > 0$ and strip width $b_0$ if there exist $\mu_0 > 0$ and $r_0 \in \{0,1\}$ such that the following hold.
For all $\mu \in (0,\mu_0)$ and $b \in (0,b_0)$, there is a bounded linear operator 
\[
\K_b(\mu) 
\colon \W_{q,b}^{r_0}(\Y;\Mfrak) \to \W_{q,b}^{r_0}(\D;\Mfrak) \cap \W_{q,b}^{r_0+1}(\X;\Mfrak)
\]
such that for all $g \in \W_{q,b}^{r_0}(\Y;\Mfrak)$, the unique solution in $\W_{q,b}^{r_0}(\D;\Mfrak) \cap \W_{q,b}^{r_0+1}(\X;\Mfrak)$ to
\begin{equation}\label{eqn: opt reg 1}
f'(z)
= \frac{\A}{\mu}f(z)+g(z), \ z \in \U_b,
\end{equation}
is $f = \K_b(\mu)$, and, moreover,
\begin{equation}\label{eqn: unif opt reg}
\sup_{\substack{0 < \mu < \mu_0 \\ 0 < b < b_0}} \norm{\K_b(\mu)}_{\b(\W_{q,b}^{r_0}(\Y;\Mfrak), \W_{q,b}^{r_0}(\D;\Mfrak) \cap \W_{q,b}^{r_0+1}(\X;\Mfrak))}
< \infty.
\end{equation}
\end{definition}

\subsection{Optimal regularity for periodic Sobolev spaces}\label{app: per opt reg}
A similar optimal regularity result must hold for periodic functions, but here we can work in comparatively more pedestrian vector-valued Sobolev spaces of periodic functions.
Following \cite[Sec.\@ 8.1]{kress}, \cite[Sec.\@ 8.3.1]{Lombardi}, we define $L_{\per}^2(\X)$ to be the completion of the space
\[
\Cal_{\per}^{\infty}(\X)
:= \set{f \in \Cal^{\infty}([0,2\pi],\X)}{f(0) = f(2\pi)}
\]
under the norm
\[
\norm{f}_{L_{\per}^2(\X)}
:= \left(\int_0^{2\pi} \norm{f(x)}_{\X}^2 \dx\right)^{1/2}.
\]
Then we let $H_{\per}^r(\X)$ denote the space of all $f \in L_{\per}^2(\X)$ such that 
\[
\norm{f}_{H_{\per}^r(\X)}
:= \left(\sum_{k = -\infty}^{\infty} (1+k^2)^r\norm{\hat{f}(k)}_{\X}^2\right)^{1/2}
< \infty,
\]
where 
\[
\hat{f}(k)
:= \frac{1}{\sqrt{2\pi}}\int_0^{2\pi} e^{-ikx}f(x) \dx.
\]

\begin{definition}\label{defn: per opt reg}
Let $\D$, $\Y$, and $\X$ be Banach spaces with $\D$ and $\Y$ continuously embedded in $\X$.
A linear operator operator $\A \colon \D \to \X$ has the \defn{periodic optimal regularity property} on $(\D,\Y,\X)$ with base frequency $\omega \in \R$ if there exist $\mu_0$, $\mu_{\omega} > 0$ and $r_0 \in \{0,1\}$ such that for each $\mu \in (0,\mu_0)$, there is a bounded linear operator
\[
\K_{\per}(\mu)
\colon H_{\per}^{r_0}(\Y) \to H_{\per}^{r_0}(\D) \cap H_{\per}^{r_0+1}(\X)
\]
such that for all $g \in H_{\per}^{r_0}(\Y)$ and all $\grave{\omega} \in (0,\mu_{\omega})$, the unique solution in $H_{\per}^{r_0}(\D) \cap H_{\per}^{r_0+1}(\X)$ to
\[
\left(\frac{\omega}{\mu} + \grave{\omega}\right)f'(x)
= \frac{\A}{\mu} {f}(x) + g(x), \ 0 \le x \le 2\pi,
\]
is $f = \K_{\per}(\mu)g$ and, moreover,
\[
\sup_{0 < \mu < \mu_{\per}} \norm{\K_{\per}(\mu)}_{\b(H_{\per}^{r_0}(\Y), H_{\per}^{r_0}(\D) \cap H_{\per}^{r_0+1}(\X))}
< \infty.
\]
\end{definition}

\subsection{Suboptimal regularity}\label{app: subopt reg}
This final regularity result is, perhaps, the least particular of all.
Here we just need sufficiently differentiable functions on $\R$.
Denote by $\Cal_q^r(\X)$ the space of all $r$-times continuously differentiable functions $f \colon \R \to \X$ such that 
\[
\norm{f}_{\Cal_q^r(\X)}
:= \max_{0 \le j \le r} e^{q|x|}\sup_{x \in \R} \norm{\partial_x^j[f](x)}_{\X}
< \infty.
\]

\begin{definition}\label{defn: subopt reg}
The operator $\A \colon \D \to \X$ has the \defn{suboptimal regularity property on $(\D,\Y,\X)$ with growth rate $q$} if there is a bounded operator $\K \colon \Cal_q^1(\Y) \to \Cal_q^1(\D)$ such that for all $g \in \Cal_q^1(\Y)$, the unique solution in $\Cal_q^1(\D)$ of 
\[
f'(x)
= \A{f}(x) + g(x), \ x \in \R,
\]
is $f = \K{g}$.
\end{definition}

Lombardi summons up such a result in \cite[App.\@ 8.A]{Lombardi} as he proves his infinite-dimensional normal form change of variables.

\subsection{A sufficient condition for optimal regularity}
At this point the reader may wonder if, given a particular operator, there is any easy way to check the (sub)optimal regularity conditions without working through all our intricate definitions.
Lombardi \cite[Sec.\@ 8.2.1, 8.2.2, App.\@ 8.B]{Lombardi} provides a sufficient condition that implies all the optimal regularity results.

\begin{theorem}[Lombardi's (sub)optimal regularity result]\label{thm: Lombardi opt reg}
Let $\D$, $\Y$, and $\X$ be Banach spaces with $\D$ and $\Y$ continuously embedded in $\X$.
Suppose that $\A \colon \D \to \X$ is a linear operator with the property that for some $C$, $k_0 > 0$, if $k \in \R$ with $|k| > k_0$, then $ik \in \rho(\A)$ and
\begin{equation}\label{eqn: Lombardi resolvent est}
\norm{(ik\ind_{\D}-\A)^{-1}}_{\b(\X)}
\le \frac{C}{|k|}.
\end{equation}
Then $\A$ has the optimal, suboptimal, and periodic optimal regularity properties.
For the optimal and suboptimal regularity properties, take $\Y = \X$, $r_0 = 1$, and $\Mfrak = \norm{\cdot}_{L^2}$; the precise details of the other parameters can be deduced from \cite[Lem.\@ 8.1.7]{Lombardi}.
\end{theorem}

Without the resolvent estimate \eqref{eqn: Lombardi resolvent est}, establishing the regularity conditions is more difficult.
This is the situation in the lattice spatial dynamics papers, and we outline in Section \ref{sec: opt reg} a method for proceeding in these situations, distilled from \cite{iooss-kirchgassner, iooss-james, hvl}.
\section{Lombardi's Nanopteron Method}\label{app: Lombardi}

\subsection{Lombardi's main nanopteron theorem}
We state below our version of Lombardi's nanopteron result, which he proves in \cite[Ch.\@ 8]{Lombardi}.
The hypotheses \ref{part: Lombardi quadratic}--\ref{part: Lombardi opt reg} here quite resemble our Hypotheses \ref{hypo: F structure}--\ref{hypo: opt reg} in Section \ref{sec: abstract}.
While there is a certain amount of repetition between here and Section \ref{sec: abstract}, in the interests of clarity we do not try to compress or elide anything with references to our problem.

\begin{theorem}[Lombardi's nanopteron theorem]\label{thm: Lombardi}
Let $\Zcal_0$ and $\Zcal$ be Banach spaces with $\Zcal_0$ continuously embedded in $\Zcal$.
Let $\A_0 \in \b(\Zcal_0,\Zcal)$ and let $\A_1 \colon [0,\mu_0] \to \b(\Zcal_0,\Zcal)$ and $\Ncal \colon \Zcal \times [0,\mu_0] \to \Zcal$ be analytic.
Define
\begin{equation}\label{eqn: Lombardi G}
\G(W,\mu)
:= \A_0W + \mu\A_1(\mu)W + \Ncal(W,\mu).
\end{equation}
Assume the following.

\begin{enumerate}[label={\bf(L\arabic*)}, ref={(L\arabic*)}]

\item\label{part: Lombardi quadratic}
The map $\Ncal$ is quadratic in the sense that 
\[
\Ncal(0,\mu) = 0
\quadword{and}
D_W\Ncal(0,\mu) = 0
\]
for all $\mu$.

\item\label{part: Lombardi reversible}
There exists $\S \in \b(\Zcal)$ such that $\S^2 = \ind_{\Zcal}$ and
\[
\A_0\S = -\S\A_0,
\qquad
\A_1(\mu)\S = -\S\A_1(\mu),
\quadword{and}
\Ncal(\S{W},\mu) = -\S\Ncal(W,\mu).
\]

\item
There exist $\omega$, $\lambda_0 > 0$ such that $\sigma(\A_0) \cap i\R = \{0,\pm\omega\}$, where $0$ and $\pm{i}\omega$ are geometrically simple eigenvalues of algebraic multiplicity 1.
If $\lambda \in \sigma(\A_0)\setminus{i}\R$, then $|\lambda| \ge \lambda_0$.
Let $\chi_{\pm\omega}$ be eigenvectors of $\A_0$ corresponding to $\pm{i}\omega$.

\item
Let $(\chi_1,\chi_2)$ be a Jordan chain corresponding to the eigenvalue 0 of $\A_0$ and suppose that the spectral projection for $\A_0$ corresponding to $0$ has the form
\[
\Pi_0W
= \chi_1^*(W)\chi_1 + \chi_2^*(W)\chi_2,
\]
where
\begin{equation}\label{eqn: Lombardi frak}
\S\chi_1 = \chi_1,
\qquad
\Lfrak_0
:= \chi_2^*\big(\A_1(0)\chi_1) > 0,
\quadword{and}
\Qfrak_0
:= \frac{\chi_2^*\big(D_{WW}^2\G(0,0)[\chi_1,\chi_1]\big)}{2}
\ne 0.
\end{equation}

\item\label{part: Lombardi opt reg}
Let $\Pi$ be the spectral projection for $\A_0$ corresponding to $\{0,\pm{i}\omega\}$ and set
\[
\Zcal_{\hsf}
:= (\ind_{\Zcal}-\Pi)(\Zcal),
\quadword{and}
\Zcal_{0,\hsf} := \Zcal_0 \cap \Zcal_{\hsf}.
\]
There is a subspace $\Zcal_{1,\hsf}$ of $\Zcal_{\hsf}$ such that 
\[
(\ind_{\Zcal}-\Pi)\A_1(\mu)W \in \Zcal_{1,\hsf}
\quadword{and}
(\ind_{\Zcal}-\Pi)\Ncal(W,\mu) \in \Zcal_{1,\hsf}
\]
for all $W \in \Zcal_{0,\hsf}$.
Norm $\Zcal_{\hsf}$ and $\Zcal_{1,\hsf}$ by $\norm{\cdot}_{\Zcal}$ and $\Zcal_{0,\hsf}$ by $\norm{\cdot}_{\Zcal_0}$. 
There exist $b_0 \in (0,\pi)$, $q \in (0,\Lfrak_0^{1/2})$, and $\grave{q} < 0$ such that on the triple $(\Zcal_{0,\hsf},\Zcal_{1,\hsf},\Zcal_{\hsf})$, the operator $\restr{\A_0}{\Zcal_{0,\hsf}}$ has the optimal regularity property with decay rate $q\Lfrak_0^{-1/2}$ and strip width $b_0$; the periodic optimal regularity property with base frequency $\omega$; and the suboptimal regularity property with growth rate $\grave{q}$.
\end{enumerate}

Then there exist $\Alpha_0$, $\Alpha_1$, $\mu_* > 0$ such that if
\[
\mu \in (0,\mu_*)
\quadword{and}
\alpha \in \left[\Alpha_0\mu\exp\left(-\frac{b\omega}{\Lfrak_0^{1/2}\mu^{1/2}}\right), \Alpha_1\right] =: \Ascr_{\mu},
\]
there is a real analytic, $\S$-reversible solution $W = \Wsf_{\mu}^{\alpha}$ to 
\begin{equation}\label{eqn: original Lombardi syst}
W'(x) 
= \G(W(x),\mu)
\end{equation}
of the form
\begin{multline}\label{eqn: Lombardi nanopteron}
\Wsf_{\mu}^{\alpha}(x)
= -\frac{3\Lfrak_0}{2\Qfrak_0}\mu\sech^2\left(\frac{\Lfrak_0^{1/2}\mu^{1/2}x}{2}\right)\chi_1
+ \mu^{3/2}\Upsilon_{\mu}^{\alpha}(\mu^{1/2}x) \\
+ \alpha\mu\Phi_{\mu}^{\alpha}\left(\mu^{1/2}x + \mu\vartheta_{\mu}^{\alpha}\tanh\left(\frac{\Lfrak_0^{1/2}\mu^{1/2}x}{2}\right)\right).
\end{multline}

The components of $\Wsf_{\mu}^{\alpha}$ have the following additional properties.

\begin{enumerate}[label={\bf(\roman*)}, ref={(\roman*)}]

\item
The map $\Upsilon_{\mu}^{\alpha} \colon \R \to \Zcal_0$ is exponentially localized and real analytic with
\begin{equation}\label{eqn: Lombardi expn loc}
\sup_{\substack{0 < \mu < \mu_* \\ \alpha \in \Ascr_{\mu} }} e^{q|X|}\norm{\Upsilon_{\mu}^{\alpha}(X)}_{\Zcal_0}
< \infty.
\end{equation}

\item\label{part: Lombardi periodic}
The map $\Phi_{\mu}^{\alpha} \colon \R \to \Zcal_0$ is periodic and real analytic and has the form
\begin{equation}\label{eqn: Lombardi periodic1}
\Phi_{\mu}^{\alpha}(X)
= \Psi_{\mu}^{\alpha}(\omega_{\mu}^{\alpha}X),
\end{equation}
where
\begin{equation}\label{eqn: Lombardi periodic frequency}
\omega_{\mu}^{\alpha} 
= \frac{\omega}{\mu^{1/2}} + \mu^{1/2}\tomega_{\mu}^{\alpha},
\qquad \sup_{\substack{0 < \mu < \mu_* \\ \alpha \in \Ascr_{\mu}}} |\tomega_{\mu}^{\alpha}| < \infty,
\qquad \sup_{\substack{0 < \mu < \mu_* \\ \alpha \in \Ascr_{\mu}}} \Lip(\Psi_{\mu}^{\alpha}) < \infty,
\end{equation}
\begin{equation}\label{eqn: Lombardi periodic2}
\Psi_{\mu}^{\alpha}(t)
= \Lfrak_0\cos(t)\left(\frac{\chi_++\chi_-}{2}\right) + \Lfrak_0\sin(t)\left(\frac{\chi_+-\chi_-}{2i}\right) + \alpha\tPsi_{\mu}^{\alpha}(\Lfrak_0^{1/2}t),
\end{equation}
and $\tPsi_{\mu}^{\alpha}$ is $2\pi$-periodic with
\[
\sup_{\substack{0 < \mu < \mu_* \\ \alpha \in \Ascr_{\mu} \\ -\pi \le s \le \pi}} \norm{\tPsi_{\mu}^{\alpha}(s)}_{\Zcal_0}
< \infty.
\]
The map
\[
\Psf_{\mu}^{\alpha}(x)
:= \alpha\mu\Phi_{\mu}^{\alpha}(\mu^{1/2}x)
\]
is a solution to \eqref{eqn: original Lombardi syst}.

\item
There are constants $0 < C_1 < C_2$ such that the phase shift $\vartheta_{\mu}^{\alpha}$ satisfies
\begin{equation}\label{eqn: Lombardi phase shift est}
0
< C_1
\le \vartheta_{\mu}^{\alpha} 
\le C_2
\end{equation}
for all $\mu \in (0,\mu_*)$ and $\alpha \in \Ascr_{\mu}$.
\end{enumerate}

\end{theorem}

We mention that particularly clear and streamlined examples of Lombardi's results in action are given in \cite[Sec.\@ 4]{bona-dougalis-mitsotakis}, which constructs nanopterons for a coupled KdV--KdV system, and \cite[App.\@ A]{haragus-wahlen}, which constructs nanopteron for a fifth-order Kadomtsev-Peviashvili equation.
In both cases, one can take $\Zcal = \C^4$, and so no consideration of Hypothesis \ref{part: Lombardi opt reg} is needed.

Before continuing with our outline of Lombardi's nanopteron construction, we make two notational remarks.
First, for typographical convenience, we will denote certain vectors as row or column vectors interchangeably.
Second, as in Appendix \ref{app: opt reg}, we will abbreviate complex strips by
\[
\U_b
:= \set{z \in \C}{|\im(z)| < b},
\qquad b > 0.
\]

\subsection{The normal form transformation and long wave rescaling}
The following lemma, which we extract from Lombardi's phrasings in \cite[Thm.\@ 8.1.10, Sec.\@ 8.1.2.2]{Lombardi}, allows us both to change variables and rescale the small parameter in the original problem \eqref{eqn: original Lombardi syst} in useful ways.
Lombardi's proof of this lemma in \cite[App.\@ 8.A]{Lombardi} requires $\restr{\A_0}{\Zcal_{0,\hsf}}$ to satisfy the suboptimal regularity property from Appendix \ref{app: subopt reg}.
Moreover, the $\S$-reversibility from Hypothesis \ref{part: Lombardi reversible} is necessary here, see \cite[Ex.\@ 3.2.9]{Lombardi} and \cite[Sec.\@ 4.3, pp.\@ 205--206]{haragus-iooss}.
We remark that Lombardi's scalings do not contain the factor $1/2$ that we have put on $\Qfrak_0$ in \eqref{eqn: Lombardi frak}; this factor is indeed necessary to achieve the leading order principal part \eqref{eqn: principal part} below.

\begin{lemma}[Lombardi's normal form transformation and rescaling]\label{lem: Lombardi nf}
Fix a norm $|\cdot|$ on $\C^4$ and, for $r > 0$, let
\[
\Bfrak(r)
:= \set{(\yb,Y) \in \R^4 \times \Zcal_{0,\hsf}}{|\yb| + \norm{Y}_{\Zcal_0} < r}.
\]
Given $\yb = (y_1,y_2,y_3,y_4) \in \R^4$ and $\nu \in \R$, let
\[
\iota_{\nu}[\yb]
:= \nu^2\left[-\frac{3}{2\Qfrak_0}(y_1\chi_1+\nu{y}_2\chi_2) + y_3\left(\frac{\chi_++\chi_-}{2}\right) + y_4\left(\frac{\chi_+-\chi_-}{2i}\right)\right].
\]
Then there exist $r_0 > 0$ and analytic maps 
\begin{multline}\label{eqn: Lombardi nf aux maps}
\Nscr \colon \R^4 \times [0,\nu_0] \to \R^4,
\qquad
\Rscr_{\csf} \colon \Bfrak(r_0) \times [0,\nu_0] \to \R^4,
\qquad
\Rscr_{\hsf} \colon \Bfrak(r_0) \times [0,\nu_0] \to \Zcal_{\hsf}, \\
\quadword{and}
\Tscr \colon \R^4 \times [0,\mu_0] \to \Zcal_0
\end{multline}
with the following properties.

\begin{enumerate}[label={\bf(\roman*)}, ref={(\roman*)}]

\item
Suppose that the maps
\[
\yb_{\nu} \colon \R \to \R^4
\quadword{and}
Y_{\nu} \colon \R \to \Zcal_{0,\hsf}
\]
solve the system
\begin{equation}\label{eqn: syst after NF and rescale}
\begin{cases}
\yb_{\nu}'(t) = \Nscr(\yb_{\nu}(t),\nu) + \Rscr_{\csf}(\yb_{\nu}(t),Y_{\nu}(t),\nu) \\[5pt]
Y_{\nu}'(t) = \frac{\A_0}{\nu}Y_{\nu}(t) + \Rscr_{\hsf}(\yb_{\nu}(t),Y_{\nu}(t),\nu).
\end{cases}
\end{equation}

Define
\begin{equation}\label{eqn: Lombardi nu}
\tilde{W}_{\nu}(t)
:= \iota_{\nu}[\yb_{\nu}(t)]
+ \nu^3Y_{\nu}(t)
+\Tscr(\iota_{\nu}[\yb_{\nu}(t)],\nu).
\end{equation}
and
\begin{equation}\label{eqn: Lombardi mu}
W_{\mu}(x)
:= \tilde{W}_{\Lfrak_0^{1/2}\mu^{1/2}}(\Lfrak_0^{1/2}\mu^{1/2}x).
\end{equation}
Then $W_{\mu}'(x) = \G(W_{\mu}(x),\mu)$ for all $x$; that is, $W=W_{\mu}$ solves the original problem \eqref{eqn: original Lombardi syst}.

\item
The principal part $\Nscr$ has the explicit formula
\begin{equation}\label{eqn: principal part}
\Nscr(\yb,\nu)
= \begin{bmatrix*}
0 &1 &0 &0 \\
1 &0 &0 &0 \\
0 &0 &0 &-\omega/\nu \\
0 &0 &\omega/\nu &0
\end{bmatrix*}\yb
+ \begin{pmatrix*}
0 \\
-3y_1^2/2  - \nsf_1\Qfrak_0(y_3^2+y_4^2) \\
\nsf_2\nu{y}_4/\Lfrak_0 +\nsf_3\nu{y}_1y_4/\Lfrak_0 \\
\nsf_2\nu{y}_3/\Lfrak_0 +\nsf_3\nu{y}_1y_3/\Lfrak_0
\end{pmatrix*},
\end{equation}
for some $\nsf_1$, $\nsf_2$, $\nsf_3 \in \R$ that are independent of $\nu$.

\item
There is a map $\Bscr \in \Cal(\C^4 \times \Zcal_{0,\hsf},\R_+)$ such that $\Bscr(0,0) = 0$ and the higher-order terms  $\Rscr_{\csf} = (\Rscr_{\csf,1},\Rscr_{\csf,2},\Rscr_{\csf,3},\Rscr_{\csf,4})$ and $\Rscr_{\hsf}$ satisfy the estimates
\begin{equation}\label{eqn: Rscr est1}
|\Rscr_{\csf,1}(\yb,Y,\nu)| + |\Rscr_{\csf,3}(\yb,Y,\nu)| + |\Rscr_{\csf,4}(\yb,Y,\nu)|
\le \Bscr(\yb,Y)\nu^2\big(\nu|\yb| + \norm{Y}_{\Zcal_0}\big)
\end{equation}
and
\begin{equation}\label{eqn: Rscr est2}
|\Rscr_{\csf,2}(\yb,Y,\nu)| + \norm{\Rscr_{\hsf}(\yb,Y,\nu)}_{\Zcal}
\le \Bscr(\yb,Y)\nu\big(\nu|\yb| +\norm{Y}_{\Zcal_0}\big).
\end{equation}
These estimates also hold if the $\Rscr$ operators are replaced by any of their derivatives, where the appropriate operator norms are used.

\item\label{part: Tscr}
The remainder term $\Tscr$ can be written as
\begin{equation}\label{eqn: Tscr}
\Tscr(\yb,\mu)
= \mu\Tscr_1(\yb,\mu) + \Tscr_2(\yb,\yb,\mu),
\end{equation}
where $\Tscr_1(\cdot,\mu) \colon \C^4 \to \Zcal_0$ is linear and $\Tscr_2(\cdot,\cdot,\mu) \colon \C^4 \times \C^4 \to \Zcal_0$ is bilinear.
The maps $\mu \mapsto \Tscr_1(\cdot,\mu)$ and $\mu \mapsto \Tscr_2(\cdot,\cdot,\mu)$ are analytic (into the appropriate spaces of bounded linear/bilinear operators from $\C^4$ or $\C^4 \times \C^4$ to $\Zcal_0$).

\item\label{part: Ssf reversible}
Put $\S_{\csf} := \diag(1,-1,1,-1)$ and $\S_{\hsf} := \restr{\S}{\Zcal_{\hsf}}$.
Then $\S_{\csf}$ anticommutes with $\Nscr$ and $\Rscr_{\csf}$ from \eqref{eqn: Lombardi nf aux maps} and commutes with $\Tscr$ from \eqref{eqn: Tscr}, while $\S_{\hsf}$ anticommutes with $\Rscr_{\hsf}$.
\end{enumerate}
\end{lemma}

Building on part \ref{part: Ssf reversible} above, we say that a function $(\yb,Y) \colon \R \to \C^4 \times \Zcal_{\hsf}$ is \defn{$(\S_{\csf},\S_{\hsf})$-reversible}, hereafter, for simplicity, just \defn{reversible}, if $\yb$ is $\S_{\csf}$-reversible and $Y$ is $\S_{\hsf}$-reversible (both in the sense of Definition \ref{defn: reversible}).

\subsection{The truncated problem}
Lombardi presents compelling motivation \cite[Sec.\@ 7.1.2, 8.1.3]{Lombardi} for the subsequent nanopteron ansatz by studying a ``truncated'' version of \eqref{eqn: syst after NF and rescale}:
\begin{equation}\label{eqn: Lombardi truncated}
\begin{cases}
\yb'(t) = \Nscr(\yb(t),\nu) \\
Y'(t) = \frac{\A_0}{\nu}Y(t).
\end{cases}
\end{equation}
First, taking $(\yb,Y) = \Sigmab$, where
\begin{equation}\label{eqn: Lombardi Sigma}
\Sigmab
:= \begin{pmatrix*}
\varsigmab \\
0
\end{pmatrix*},
\qquad
\varsigmab(t)
:= \begin{pmatrix*}
\sech^2\left(\frac{t}{2}\right), -\sech^2\left(\frac{t}{2}\right)\tanh\left(\frac{t}{2}\right), 0, 0
\end{pmatrix*},
\end{equation}
solves \eqref{eqn: Lombardi truncated}.

From the definition of $\Nscr$ in \eqref{eqn: principal part}, it is straightforward to check that $\varsigmab'(z) = \Nscr(\varsigmab(z),\nu)$ for all $\nu$.
Next, the explicit structure of the truncated problem leads one to develop reversible periodic solutions $(\yb,Y) = (\phib_{\star,\nu}^{\Alpha},0)$ to \eqref{eqn: Lombardi truncated}, valid for $\Alpha$ and $\nu$ suitably small.
In particular, 
\begin{equation}\label{eqn: phib-0}
\phib_{\star,\nu}^{\Alpha}(t) = A\phib_0(\Omega_{\star,\nu}^{\Alpha}t) + \O(\Alpha^2),
\qquad
\Omega_{\star,\nu}^{\Alpha} = \frac{\omega}{\nu} + \O(\nu),
\qquad
\phib_0(s) := \big(0,0,\cos(s),\sin(s)\big).
\end{equation}
Finally, there are reversible asymptotically periodic solutions $(\yb,Y) = (\yb_{\star,\nu}^{\Alpha},0)$ to \eqref{eqn: Lombardi truncated} that satisfy
\[
\lim_{t \to \infty} \big|\yb_{\star,\nu}^{\Alpha}(t) - \phib_{\star,\nu}^{\Alpha}(t + \nu\vartheta_{\star,\nu})\big|
= 0
\]
exponentially fast.
In a further limiting sense, $\yb_{\star,0}^0 = \varsigmab$.
In the rest of this appendix, we summarize how Lombardi extends these truncated results to the full problem \eqref{eqn: syst after NF and rescale}.

\subsection{Periodic solutions}\label{app: Lombardi periodic}
Lombardi first constructs periodic solutions of the following form.
This construction \cite[Thm.\@ 8.1.12, Sec.\@ 8.3, App.\@ 8.B]{Lombardi} requires $\restr{\A_0}{\Zcal_{0,\hsf}}$ to satisfy the periodic optimal regularity property as discussed in Appendix \ref{app: per opt reg}.
The construction also uses the symmetry in several key steps, specifically to prevent an overdetermined system and for a dimension-counting argument involved in inverting an operator between finite-dimensional spaces \cite[App.\@ 4.A, eq.\@ (4.13); Lem.\@ 4.A.9]{Lombardi}.
The analytic continuation in part \ref{part: Lombardi periodic analytic continuation} below is proved in \cite[Rem.\@ 7.2.2, Sec.\@ 8.3.2]{Lombardi}.

\begin{lemma}[Lombardi's periodic solutions]\label{lem: Lombardi periodics}
There exist $\Alpha_{\per}$, $\nu_{\per} > 0$ such that if $0 < \nu < \nu_{\per}$ and $0 \le \Alpha \le \Alpha_{\per}$, there are $2\pi$-periodic, reversible, real analytic maps $\phib_{\nu}^{\Alpha} \colon \R \to \R^4$ and $\Phiit_{\nu}^{\Alpha} \colon \R \to \Zcal_{0,\hsf}$ and a scalar $\Omega_{\nu}^{\Alpha}$ with the following properties.

\begin{enumerate}[label={\bf(\roman*)}, ref={(\roman*)}]

\item
Taking
\begin{equation}\label{eqn: Lombardi periodic full}
\begin{pmatrix*}
\yb(t) \\
Y(t)
\end{pmatrix*}
= \Alpha\!\Phib_{\nu}^{\Alpha}(t)
:= \Alpha\begin{pmatrix*}
\phib_0(\Omega_{\nu}^{\Alpha}t) + \Alpha\phib_{\nu}^{\Alpha}(\Omega_{\nu}^{\Alpha}t) \\
\Alpha\Phiit_{\nu}^{\Alpha}(\Omega_{\nu}^{\Alpha}t)
\end{pmatrix*}
\end{equation}
solves \eqref{eqn: syst after NF and rescale}.

\item
The maps $\phib_{\nu}^{\Alpha}$ and $\Phiit_{\nu}^{\Alpha}$ and the ``frequency'' $\Omega_{\nu}^{\Alpha}$ satisfy
\[
\sup_{\substack{0 < \nu < \nu_{\per} \\ 0 \le \Alpha \le \Alpha_{\per}}}
\left(\sup_{0 \le s \le 2\pi} |\phib_{\nu}^{\Alpha}(s)| + \norm{\Phiit_{\nu}^{\Alpha}(s)}_{\Zcal_0}\right) 
+ \Lip(\phib_{\nu}^{\Alpha}) + \Lip(\Phiit_{\nu}^{\Alpha})
+ \nu^{-1}\left|\Omega_{\nu}^{\Alpha} - \frac{\omega}{\nu}\right|
< \infty.
\]

\item\label{part: Lombardi periodic analytic continuation}
If $0 < b < b_0$ and $\Alpha = \alpha{e}^{-b\omega/\nu}$ for $\alpha \in [0,\Alpha_{\per}]$, then the maps $\phib_{\nu}^{\Alpha}$ and $\Phiit_{\nu}^{\Alpha}$ have analytic continuations on the strip $\U_b$.
\end{enumerate}
\end{lemma}

\subsection{Lombardi's nanopteron ansatz}
With the periodic solutions $\Alpha\!\Phib_{\nu}^{\Alpha}$ from \eqref{eqn: Lombardi periodic full} in hand, Lombardi can make his nanopteron ansatz for \eqref{eqn: syst after NF and rescale}:
\begin{equation}\label{eqn: Lombardi nanopteron ansatz}
\begin{pmatrix*}
\yb(z) \\
Y(z)
\end{pmatrix*}
= \Sigmab(z)
+ \Alpha\Phib_{\nu}^{\Alpha}\left(z+\vartheta\tanh\left(\frac{z}{2}\right)\right)
+ \Upsilonb(z).
\end{equation}
This ansatz is posed for $z \in \U_b$ with $b \in (0,b_0)$.
The leading order term $\Sigmab$ in the ansatz \eqref{eqn: Lombardi nanopteron ansatz} was defined above in \eqref{eqn: Lombardi Sigma}, and $\Sigmab$ extends to be analytic on any such $\U_b$.
The oddness of the $\tanh$-coefficient helps to preserve the $(\S_{\csf},\S_{\hsf})$-reversibility of the ansatz if $\Upsilonb$ is $(\S_{\csf},\S_{\hsf})$-reversible.

Inserting this ansatz into the normal form change of variables \eqref{eqn: Lombardi nu} and \eqref{eqn: Lombardi mu} and using the properties of $\Tscr$ in part \ref{part: Tscr} of Lemma \ref{lem: Lombardi nf} leads to the nanopteron \eqref{eqn: Lombardi nanopteron}.
Replacing $\nu$ in these results with $\Lfrak_0^{1/2}\mu^{1/2}$ and putting $\omega_{\mu}^{\alpha} := \Lfrak_0^{1/2}\Omega_{\Lfrak_0^{1/2}\mu^{1/2}}^{\alpha}$, per Lemma \ref{lem: Lombardi periodics}, leads to the other conclusions of Theorem \ref{thm: Lombardi}; the details do require some algebraic care, but there are no surprises.

To ensure that $\Phi_{\nu}^{\Alpha}(z+\vartheta\tanh(z/2))$ is defined for $z \in \U_b$, Lombardi takes $\Alpha = \alpha{e}^{-b\omega/\nu}$ with $\alpha \in [0,\Alpha_{\per}]$, as in part \ref{part: Lombardi periodic analytic continuation} of Lemma \ref{lem: Lombardi periodics} and assumes $\vartheta=\O(\nu)$, with a more precise bound obtained later.
Lombardi eventually pushes the upper bound on $\alpha$ to an $\O(1)$ constant independent of $\nu$ using certain $b$-independent estimates that underly the following work \cite[Sec.\@ 7.3.7]{Lombardi}; among these estimates is the bound \eqref{eqn: unif opt reg} from the localized optimal regularity hypothesis.

The unknown remainder is
\[
\Upsilonb
= \begin{pmatrix*}
\upsilonb \\
\Upsilonit
\end{pmatrix*},
\qquad \upsilonb \colon \U_b \to \C^4,
\qquad \Upsilonit \colon \U_b \to \Zcal_{0,\hsf}.
\]
We will specify the precise function spaces for $\upsilonb$ and $\Upsilonit$ momentarily and for now just mention that both functions should vanish exponentially fast at $\pm\infty$.

Under the ansatz \eqref{eqn: Lombardi nanopteron ansatz}, Lombardi then converts the problem \eqref{eqn: syst after NF and rescale} into the system
\begin{equation}\label{eqn: upsilon system}
\begin{cases}
\upsilonb'(z) - D_{\yb}\mathscr{N}(\varsigmab,\nu)\upsilonb(z) = \Gscr_{\csf}[\upsilonb(z),\Upsilonit(z),\alpha,\vartheta,\nu] \\[5pt]
\Upsilonit'(z) - \frac{\A_0}{\nu}\Upsilonit(z) = \Gscr_{\hsf}[\upsilonb(z),\Upsilonit(z),\alpha,\vartheta,\nu]
\end{cases}
\end{equation}
for appropriate operators $\Gscr_{\csf}$ and $\Gscr_{\hsf}$, which are $\C^4$- and $\Zcal_{1,\hsf}$-valued, respectively.

\subsection{Phase shift selection and error term construction}\label{app: Lombardi final fp syst}
The equation for $\upsilonb$ above is really just an ordinary differential equation in $\C^4$, and Lombardi solves this with variation of parameters \cite[Sec.\@ 7.3.4, 8.4.2]{Lombardi} to write
\begin{equation}\label{eqn: upsilonb fp eqn}
\upsilonb
= \Fscr[\upsilonb,\Upsilonit,\alpha,\vartheta,\nu],
\end{equation}
where $\Fscr$ is a superposition of integral operators.
By requiring $\upsilonb$ to be $\S_{\csf}$-reversible and $\Upsilonit$ to be $\S_{\hsf}$-reversible and using the assumption that $\upsilonb$ and $\Upsilonit$ vanish at $\pm\infty$, Lombardi eliminates three of the four ``free constants'' that naturally arise from variation of parameters.
However, in the process a ``solvability condition'' arises: given reversible $\upsilonb$ and $\Upsilonit$, the solution $\Fscr[\upsilonb,\Upsilonit,\alpha,\vartheta,\nu]$ is reversible if and only if 
\begin{equation}\label{eqn: Lombardi osc int}
\int_0^{\infty} e^{i\omega{t}/\nu}\mathcal{I}[\upsilonb,\Upsilonit,\alpha,\vartheta,\nu](t)\dt
= 0
\end{equation}
for a certain operator $\I$.

The vanishing of this integral can be forced \cite[Sec.\@ 7.3.5, 8.4.3, p.\@ 226]{Lombardi} by taking $\vartheta$ to have a special value of the form
\begin{equation}\label{eqn: Theta fp}
\vartheta
= \Theta(\upsilonb,\Upsilonit,\alpha,\nu)
\end{equation}
for a certain operator $\vartheta$.
This also eliminates $\vartheta$ as an unknown in the system \eqref{eqn: upsilon system}.
The heart of the construction of $\vartheta$ is an intermediate value theorem argument \cite[Prop.\@ 7.3.20, Thm\@ 8.4.2]{Lombardi} that requires $\alpha$ to be $\O(\nu^2)$ and nonzero; this is one technical reason why Lombardi's periodic amplitudes are never 0.

\begin{remark}\label{rem: why Lombardi real analytic}
Lombardi achieves very precise control over the oscillatory integral \eqref{eqn: Lombardi osc int} via estimates that depend on the analyticity of the integrand \cite[Lem.\@ 2.1.1]{Lombardi}.
This is why Lombardi demands that the original maps $\A_1$ and $\Ncal$ in \eqref{eqn: Lombardi G} be analytic and why, in our application of spatial dynamics to FPUT lattices, we require that the spring potentials be real analytic.
\end{remark}

Lombardi then turns to the equation for $\Upsilonit$ in \eqref{eqn: upsilon system}.
For fixed $\alpha$, $\vartheta$, and $\nu$, the operator $\Gscr_{\hsf}[\cdot,\cdot,\alpha,\vartheta,\nu]$ maps $\C^4 \times \Zcal_{0,\hsf}$ into $\Zcal_{1,\hsf}$; in order to stay within the same function space and to solve the $\Upsilonit$ equation, Lombardi deploys optimal regularity.
(A more detailed motivation for the use of optimal regularity here, with references to the exact formula for $\Gscr_{\hsf}$, appears in the first two paragraphs of \cite[Sec.\@ 8.4.1]{Lombardi}.)
Using the localized optimal regularity language of Appendix \ref{app: loc opt reg}, he takes
\[
\upsilonb \in \W_{q\Lfrak_0^{-1/2},b}^{r_0}(\C^4;\Mfrak)
\quadword{and}
\Upsilonit \in \W_{q\Lfrak_0^{-1/2},b}^{r_0}(\Zcal_{0,\hsf};\Mfrak) \cap \W_{q\Lfrak_0^{-1/2},b}^{r_0+1}(\Zcal_{\hsf};\Mfrak)
\]
and restricts these functions to a ball of radius $O(\nu)$.
Lombardi then applies the optimal regularity property that $\restr{\A_0}{\Zcal_{0,\hsf}}$ satisfies to solve for $\Upsilonit$ as
\begin{equation}\label{eqn: Upsilonit fp eqn}
\Upsilonit 
= \K_b(\nu)\Gscr_{\hsf}[\upsilonb,\Upsilonit,\alpha,\vartheta,\nu].
\end{equation}
This equation and \eqref{eqn: upsilonb fp eqn} constitute a fixed point system for $(\upsilonb,\Upsilonit)$ in the auxiliary small parameters $\alpha$ and $\nu$, once we have replaced $\vartheta$ by \eqref{eqn: Theta fp}.

In order to obtain good contraction mapping estimates \cite[Sec.\@ 8.4.4]{Lombardi} on the system \eqref{eqn: upsilonb fp eqn}--\eqref{eqn: Upsilonit fp eqn}, it is important to take $r_0 \in \{0,1\}$; while the system is well-defined on spaces with $r_0 > 1$, certain higher derivatives on the periodic functions from Lemma \ref{lem: Lombardi periodics} introduce pernicious powers of $\nu^{-1}$ into the estimates.
Otherwise, now that all the players are at last on the stage, the analysis of the fixed point problem given by \eqref{eqn: upsilonb fp eqn} and \eqref{eqn: Upsilonit fp eqn} reduces to a fairly routine quantitative contraction mapping argument.

\subsection{Toward the nonexistence of solitary waves}\label{app: Lombardi solitary nonexistence}
Suppose that instead of making the nanopteron ansatz \eqref{eqn: Lombardi nanopteron ansatz} for the problem \eqref{eqn: syst after NF and rescale}, we look for purely localized perturbations of $\Sigmab$ from \eqref{eqn: Lombardi Sigma} and posit instead
\[
(\yb,Y)
= \Sigmab + \Upsilonb
\] 
with $\Upsilonb$ exponentially localized, as before.
Lombardi shows \cite[Sec.\@ 7.4.1.1]{Lombardi} that if such an ansatz solves \eqref{eqn: syst after NF and rescale}, then 
\[
\int_{-\infty}^{\infty} e^{i\omega{t}/\nu}\tilde{\I}_{\nu}[\Upsilonb](t) \dt
= 0
\]
for a certain operator $\tilde{\I}_{\nu}$.
Just as in \eqref{eqn: Lombardi osc int}, another oscillatory integral must vanish.
Lombardi achieves the detailed, delicate expansion of this integral as
\[
\int_{-\infty}^{\infty} e^{i\omega{t}/\nu}\tilde{\I}_{\nu}[\Upsilonb](t) \dt
= \nu^{-2}e^{-\omega\pi/\nu}\big(\Ifrak_1 + \O(\nu^{1/4}\big)
\]
for a certain coefficient $\Ifrak_1$.
So, if $\Ifrak_1 \ne 0$, then, for $\nu$ suitably small, the integral cannot vanish, and there are no purely localized solutions.
Although Lombardi does obtain a formula for $\Ifrak_1$ in terms of numerous other quantities associated with the normal form transformation above \cite[Eq. (7.17), Prop.\@ 7.4.8]{Lombardi}, whether $\Ifrak_1$ vanishes is not at all obvious.
As Lombardi himself writes, ``we have a generic non existence [{\it{sic}}] result, but from a theoretical point of view, we are not able in general to determine whether [the coefficient $\Ifrak_1$] vanishes or not'' \cite[Rem.\@ 7.1.17]{Lombardi}.

Lombardi expects, but does not prove, that ``generically'' $\Ifrak_1 \ne 0$.
The adverb ``generically'' refers to the original nonlinear operator $\G$ from \eqref{eqn: Lombardi G} governing the whole problem.
It is not precisely clear how to interpret ``generically'' when referring to a fairly arbitrary operator.
One possibility, suggested by the language in \cite[Rem.\@ 8.5.1]{Lombardi}, \cite[Sec.\@ 5.3.1]{james-sire}, \cite[Sec.\@ 6.3.1]{sire} is that if $\G$ depends on a fixed parameter --- for example, the Bond number in the water wave problem, or the mass ratio $w$ in the mass dimer long wave problem --- then for almost all values of that parameter, there are no localized solutions, at least when the problems natural  small parameter is sufficiently small.
\section{Beale's Nanopteron Method}\label{app: Beale}
As we discussed in Section \ref{sec: prior nanopteron results}, the papers \cite{faver-wright, faver-spring-dimer} use a functional-analytic method originally due to Beale \cite{beale} to construct the mass and spring dimer nanopterons in relative displacement coordinates.
Here we sketch this method and point out how certain particular steps relate to our spatial dynamics program.
For simplicity, as in Section \ref{sec: spat dyn in rel disp}, we assume that the spring forces contain linear and quadratic terms only, so $\V_1(r) = r+r^2$ and $\V_2(r) = \kappa{r}+\beta{r}^2$.
Then the original relative displacement traveling wave problem \eqref{eqn: rel disp tw eqns} becomes
\begin{equation}\label{eqn: rel disp tw eqns Beale}
\begin{cases}
c^2\varrho_1'' = -(1+w)\varrho_1 + \kappa(wS^1+S^{-1})\varrho_2 - (1+w)\varrho_1^2 + \beta(wS^1+S^{-1})\varrho_2^2 \\
c^2\varrho_2'' = (S^1+wS^{-1})\varrho_1 - \kappa(1+w)\varrho_2 + (S^1+wS^{-1})\varrho_1^2 - \beta(1+w)\varrho_2^2.
\end{cases}
\end{equation}

\subsection{Diagonalization}
The problem \eqref{eqn: rel disp tw eqns Beale} is not stated in the ideal coordinates for Beale's method.
First consider the linear operator 
\begin{equation}\label{eqn: rel disp fourier mult}
L
:= \begin{bmatrix*}
-(1+w) &\kappa(wS^1+S^{-1}) \\
(S^1+wS^{-1}) &-\kappa(1+w)
\end{bmatrix*}
\end{equation}
as a Fourier multiplier with a matrix-valued symbol.
That is, for each $k \in \R$, there is a matrix $\tilde{L}(k) \in \C^{2\times2}$ such that 
\[
L[e^{ik\cdot}\vb](x)
= e^{ikx}\tilde{L}(k)\vb, \ x \in \R, \vb \in \C^2.
\]
The eigenvalues of $\tilde{L}(k)$ are the distinct numbers $\tlambda_{\pm}(k)$ from \eqref{eqn: tlambda}, so $\tilde{L}(k)$ is diagonalizable.
More precisely, we can write
\begin{equation}\label{eqn: diagonalized}
\tilde{L}(k)
= \tilde{J}(k)\begin{bmatrix*}
\tlambda_-(k) &0 \\
0 &\tlambda_+(k)
\end{bmatrix*}
\tilde{J}(k)^{-1},
\end{equation}
where $\tilde{J} \colon \R \to \C^{2\times2}$ is $2\pi$-periodic and real analytic.
The columns of $\tilde{J}(k)$ consist of eigenvectors of $\tilde{L}(k)$, and so there are many such $\tilde{J}(k)$ that we might use; in the mass and spring dimer cases, we choose particular scalings of the eigenvectors to preserve the parities of the traveling wave profiles later.
This is one instance of the use of (the different) symmetries for mass and spring dimers.

A careful examination of the formula for $\tlambda_-$ in \eqref{eqn: tlambda} allows us to factor
\[
\tlambda_-(k)
= 2(1-\cos(k))\tLambda_-(k),
\]
where $\tLambda_-$ is $2\pi$-periodic and real analytic.
Likewise, $\tlambda_+$ is $2\pi$-periodic and real analytic, and so we have the Fourier series expansions
\[
\tilde{J}(k) = \sum_{n=-\infty}^{\infty} e^{ikn}J_n,
\qquad
\tLambda_-(k) = \sum_{n=-\infty}^{\infty} c_n^-e^{ikn},
\quadword{and}
\tlambda_+(k) = \sum_{n=-\infty}^{\infty} c_n^+e^{ikn}.
\]
Here $J_n \in \C^{2 \times 2}$ and $c_n^{\pm} \in \C$.
Since $k \mapsto e^{ikn}$ is the symbol of the shift operator $S^n$, we are motivated to define
\[
J := \sum_{n=-\infty}^{\infty} S^nJ_n,
\quadword{and}
\Lambda_{\pm} := \sum_{n=-\infty}^{\infty} c_n^{\pm}S^n
\]
as bounded operators on $L^{\infty} \times L^{\infty}$ and $L^{\infty}$, respectively.
Since $k \mapsto 2(1-\cos(k))$ is the symbol of $-(S^1-2+S^{-1})$, it follows that
\[
L
= J\begin{bmatrix*}
-(S^1-2+S^{-1})\Lambda_- &0 \\
0 &\Lambda_+
\end{bmatrix*}
J^{-1}.
\]

Now we are ready to diagonalize.
Put 
\[
\begin{pmatrix*}
\pfrak_1 \\
\pfrak_2
\end{pmatrix*}
= J\begin{pmatrix*}
\varrho_1 \\
\varrho_2
\end{pmatrix*}
\]
to see that, if the profiles $\varrho_1$ and $\varrho_2$ are assumed to be differentiable and bounded, then \eqref{eqn: rel disp tw eqns} is equivalent to 
\begin{equation}\label{eqn: beale diag}
\begin{cases}
c^2\pfrak_1'' -(S^1-2+S^{-1})\Lambda_-[\pfrak_1+\Qsf_1(\pfrak_1,\pfrak_2)] = 0 \\
c^2\pfrak_2'' + \Lambda_+[\pfrak_2+\Qsf_2(\pfrak_1,\pfrak_2)] = 0.
\end{cases}
\end{equation}

For $f \in \Cal^1(\R)$, put 
\[
(\I_+f)(x) := \int_0^1 f(x+s) \ds
\quadword{and}
(\I_-f)(x) := \int_{-1}^0 f(x+s);
\]
we met $\I_-$ in \eqref{eqn: int ops first int}.
Then
\[
S^1-2+S^{-1}
= \partial_x^2\I_+\I_-.
\]
Consequently, we may integrate the first equation in \eqref{eqn: beale diag} twice to find that 
\[
c^2\pfrak_1-\I_+\I_-\Lambda_-[\pfrak_1+\Qsf_1(\pfrak_1,\pfrak_2)]
\]
is constant for any solution $(\pfrak_1,\pfrak_2)$ of \eqref{eqn: beale diag}.
Thus this quantity is a first integral for the relative displacement traveling wave problem.
We took pains above to diagonalize using shift operators, not Fourier multipliers, so that this first integral is valid for as broad a class of profiles as possible, like the first integral in Iooss--Kirchg\"{a}ssner coordinates from Section \ref{sec: first int}.

Suppose we choose the value of this first integral to be zero, very much along the lines of the remarks preceding Lemma \ref{lem: chi3 reduction consistency}.
It turns out that the operator $c^2-\I_+\I_-\Lambda_-$ is invertible on the range of $\I_+\I_-$.
Put
\[
\Msf_-(c)
:= (c^2-\I_+\I_-\Lambda_-)^{-1}\I_+\I_-\Lambda_-
\]
to convert the first equation in \eqref{eqn: beale diag} to
\begin{equation}\label{eqn: beale diag first comp}
\pfrak_1 + \Msf_-(c)\Qsf_1(\pfrak_1,\pfrak_2)
= 0.
\end{equation}
This sort of ``cancelation'' is precisely the method of Friesecke and Pego \cite[Sec.\@ 2, Eq.\@ (2.3)]{friesecke-pego1} for rearranging the monatomic traveling wave problem into a friendly form.
In the dimer, the more complicated structure of the equations occludes the first integral until after diagonalization.

\subsection{Long wave coordinates}
Now introduce the long wave ansatz
\begin{equation}\label{eqn: beale lw ansatz}
\pfrak_1(x) = \nu^2\Rho_1(\nu{x}),
\qquad
\pfrak_2(x) = \nu^2\Rho_2(\nu{x}),
\quadword{and}
c^2 = c_*^2 + \nu^2,
\end{equation}
where $c_*$ is the dimer's speed of sound, defined precisely in \eqref{eqn: cs defn}.
Following our notation in Section \ref{sec: faver wright comparisons}, we use $\nu$ for the long wave parameter here, not $\ep$.
After rewriting, relabeling, and condensing the terms and operators in \eqref{eqn: beale diag} and \eqref{eqn: beale diag first comp}, we arrive at
\begin{equation}\label{eqn: beale1}
\begin{cases}
\Rho_1 + \Qsf_1^{\nu}(\Rho_1,\Rho_2) = 0 \\
\nu^2(\cs^2+\nu^2)\Rho_2'' + \Msf_{\nu}\Rho_2 +  \Qsf_2^{\nu}(\Rho_1,\Rho_2) = 0.
\end{cases}
\end{equation}
The $\nu$-dependencies are all delicate, and we cheerfully omit the details here.

It turns out that \eqref{eqn: beale1} has the solution
\begin{equation}\label{eqn: Beale ep=0 soln}
\Rho_1(X) = \Csf_1\sech^2(\Csf_2X)
\quadword{and}
\Rho_2(X) = 0
\end{equation}
at $\nu = 0$.
The constants $\Csf_1$ and $\Csf_2$ depend on $\kappa$, $\beta$, and $w$; moreover, we need $\beta+\kappa^3 \ne 0$.
Due to that $\nu$-dependent linear change of variables, $\Csf_1$ and $\Csf_2$ are not quite the constants that appear in the $\sech^2$-type leading order solutions discussed in Section \ref{sec: faver wright comparisons}.

\subsection{The naive perturbation problem}
The natural instinct is to perturb from the $\nu=0$ solution in \eqref{eqn: Beale ep=0 soln} by setting
\begin{equation}\label{eqn: beale naive ansatz}
\Rho_1(X) = \Csf_1\sech^2(\Csf_2X) + \zeta_1(X)
\quadword{and}
\Rho_2(X) = \zeta_2(X),
\end{equation}
where $\zeta_1$ and $\zeta_2$ are ``small'' and exponentially localized.
Evaluating \eqref{eqn: beale1} at this ansatz leads to a system of two equations for $\zeta_1$ and $\zeta_2$.

It is fairly easy to convert the first of these equations into a fixed point equation for $\zeta_1$ as a nonlinear function of itself, $\zeta_2$, and $\nu$, but the second equation in \eqref{eqn: beale1} of course requires $\zeta_2$ to satisfy
\begin{equation}\label{eqn: beale2}
\nu^2(\cs^2+\nu^2)\zeta_2'' + \Msf_{\nu}\zeta_2
= -\Qsf_2^{\nu}(\Csf_1\sech^2(\Csf_2\cdot)+\zeta_1,\zeta_2).
\end{equation}
A careful analysis of the roots of the symbol of the Fourier multiplier $\nu^2(\cs^2+\nu^2)\partial_X^2+\Msf_{\nu}$ shows that if 
\begin{equation}\label{eqn: Omega-nu}
\Omegait_{\nu} 
:= \frac{\omega_{\sqrt{\cs(\kappa,w)^2+\nu^2}}}{\nu},
\end{equation}
where $\omega_{\sqrt{\cs(\kappa,w)^2+\nu^2}}$ was defined in Proposition \ref{prop: Lambda props}, then whenever $f$ and $g$ are any exponentially localized functions meeting
\begin{equation}\label{eqn: beale2.5}
\nu^2(\cs^2+\nu^2)f''+\Msf_{\nu}f
= g,
\end{equation}
the Fourier transform $\hat{g} = \ft[g]$ of $g$ must also satisfy
\begin{equation}\label{eqn: beale solv}
\hat{g}(\pm\Omegait_{\nu})
= 0.
\end{equation}
This critical frequency $\Omegait_{\nu} = \O(\nu^{-1})$ is closely tied to the ``acoustic'' and ``optical'' band structure of the dispersion relation for \eqref{eqn: rel disp tw eqns}; we do not enter into a discussion of these terms here but refer the reader to \cite[Rem.\@ 2.2]{faver-wright} and \cite[Ch.\@ IV, Sec.\@ 13, 15]{brill}.

Thus it appears that the {\it{two}} unknown perturbation terms $\zeta_1$ and $\zeta_2$ must really meet {\it{four}} equations: the fixed point equation for $\zeta_1$, the equation \eqref{eqn: beale2} for $\zeta_2$, and the Fourier transform conditions
\begin{equation}\label{eqn: beale3}
\ft[\Qsf_2^{\nu}(\Csf_1\sech^2(\Csf_2\cdot) + \zeta_1,\zeta_2)](\pm\Omegait_{\nu})
= 0.
\end{equation}
Consequently, the long wave problem now seems quite overdetermined.
However, by first restricting our analysis to mass or spring dimers, not the general dimer, we can reduce the number of equations by one.
Symmetry properties, which really are inherited from the precise choice of the eigenvectors in $\tilde{J}(k)$ back in \eqref{eqn: diagonalized}, allow us to assume that $\zeta_1$ is even and $\zeta_2$ is odd in the case of the mass dimer and that both perturbation terms are even for the spring dimer.
This allows us to replace the two equations in \eqref{eqn: beale3} by just the Fourier transform at $+\Omegait_{\nu}$.

\subsection{Beale's nanopteron ansatz}
We emphasize that the following work holds equally well for the mass dimer or the spring dimer, but not (to date) the general dimer.

Despite the reduction from symmetry, we still have three equations and two unknowns, and here we turn to Beale's insights.
One does not need to view the system \eqref{eqn: beale1} strictly within a universe of exponentially localized functions.
Rather, the same symbol properties that gave the solvability condition \eqref{eqn: beale solv} for \eqref{eqn: beale2.5} mean that the exact sinusoidal functions
\[
\Rho_1(X) = 0
\quadword{and}
\Rho_2(X) = \psi(\Omegait_{\nu}X)
\]
solve the linearization (at $\Rho_1=\Rho_2=0$) of \eqref{eqn: beale1}.
Here $\psi(s) = \sin(s)$ for the mass dimer and $\psi(s) = \sin(s)$ for the spring dimer.
It is possible to use a quantitative version of the classical Crandall--Rabinowitz--Zeidler ``bifurcation from a simple eigenvalue'' technique \cite{crandall-rabinowitz, zeidler} to extend these solutions into a family of periodic solutions to the full nonlinear problem \eqref{eqn: beale1} of the form
\begin{equation}\label{eqn: beale periodic}
\Rho_1(X) = a\phi_{1,\nu}^a(X)
\quadword{and}
\Rho_2(X) = a\phi_{2,\nu}^a(X) = a\psi(\Omegait_{\nu}^aX) + \O(a^2)
\end{equation}
for $a \in \R$ sufficiently small.
The profiles $\phi_{k,\nu}^a$ are $2\pi$-periodic with frequency $\Omegait_{\nu}^a = \Omegait_{\nu} + \O(a^2)$; it then follows from \eqref{eqn: Omega-nu} and \eqref{eqn: Lombardi periodic frequency} that $\Omegait_{\nu}^a$ agrees to leading order with the frequency of the spatial dynamics-derived periodic profiles.
In a subtle way, symmetry nicely reduces the number of equations that must be solved in the Crandall--Rabinowitz--Zeidler-type bifurcation that constructs the periodic solutions, very much along the lines of what occurs in Lombardi's periodic construction (Appendix \ref{app: Lombardi periodic}).

Beale's success in overcoming the analogous difficulties in the water wave problem \cite{beale} hinged on incorporating the periodic solutions \eqref{eqn: beale periodic} into a revision of the original exponentially localized ansatz \eqref{eqn: beale naive ansatz} by positing instead
\[
\Rho_1(X) = \Csf_1\sech^2(\Csf_2X) + \zeta_1(X) + a\phi_{1,\nu}^a(X)
\quadword{and}
\Rho_2(X) = \zeta_2(X)+ a\phi_{2,\nu}^a(X).
\]
Now there are three unknowns: the two perturbation terms $\zeta_1$ and $\zeta_2$ and the ``amplitude'' coefficient $a$.
Careful manipulations of the new versions of \eqref{eqn: beale2} and \eqref{eqn: beale3} that result from this ansatz lead to fixed point equations for $\zeta_2$ and $a$, whereas previously we only had a good equation for $\zeta_1$.
A quantitative contraction mapping argument, taking into account an unusually bad Lipschitz constant, produces, for $\nu$ sufficiently small, the solutions that ultimately lead to \eqref{eqn: intro dimer nano}.

\subsection{The periodic amplitude coefficient}
In particular, if $\zeta_1^{\nu}$, $\zeta_2^{\nu}$, and $a_{\nu}$ are the solutions, then $a_{\nu}$ satisfies
\begin{equation}\label{eqn: beale method a eqn}
a_{\nu}
= \int_{-\infty}^{\infty} \I_{\nu}[\eta_1^{\nu},\eta_2^{\nu},a_{\nu}](X)\sin(\Omegait_{\nu}X) \dX.
\end{equation}

The function $\I_{\nu}[\eta_1,\eta_2,a]$ is as smooth as the dimer's potentials $\V_1$ and $\V_2$ and the input functions $\eta_1$ and $\eta_2$ are; a bootstrapping argument shows that the solutions $\eta_1^{\nu}$ and $\eta_2^{\nu}$ are as smooth as $\V_1$ and $\V_2$.
When $\V_1$ and $\V_2$ are $\Cal^{\infty}$, so is $\I_{\nu}[\eta_1^{\nu},\eta_2^{\nu},a_{\nu}]$, and so a Riemann--Lebesgue estimate on the oscillatory integral \eqref{eqn: beale method a eqn} shows that $a_{\nu}$ is small beyond all algebraic orders of $\nu$, which is the estimate \eqref{eqn: small BAAO} above.

\begin{remark}\label{rem: why Cinfty potentials}
More precisely, this Riemann--Lebesgue estimate states that if $f$ and its first $r$ derivatives vanish exponentially fast and if $\Omegait_{\nu} = \O(\nu^{-1})$, then 
\[
\left|\int_{-\infty}^{\infty} f(X)e^{-i\Omegait_{\nu}X} \dX\right|
\le C_r\nu^r.
\]
See Lemma A.5 in \cite{faver-wright}.
The fixed point construction of $\eta_1^{\nu}$ and $\eta_2^{\nu}$ above does not require the spring potentials to be $\Cal^{\infty}$, merely $\Cal^r$ for a sufficiently large $r$.
And so if the potentials are $\Cal^r$ but not $\Cal^{r+1}$, then $\eta_1^{\nu}$, $\eta_2^{\nu}$ are only $\Cal^r$, and thus $a_{\nu}$ is not necessarily small beyond all orders.
\end{remark}

\subsection{Further dialogue with the Lombardi--spatial dynamics method}\label{app: Beale sd comp}
The principal contrast between Lombardi's nanopteron construction and Beale's is that Lombardi fixes the amplitude of his periodic ripple and then runs a fixed point argument for the exponentially localized error alone, whereas Beale runs a fixed point argument for both the error and the amplitude.
Secondarily, Lombardi is required to accept a small phase shift in the periodic terms, whereas the phase shift is optional for Beale and, to all appearances, not necessarily small.
Based on the Amick--Toland method \cite{amick-toland} and the adapted outline by Faver and Wright \cite[Sec.\@ 7]{faver-wright}, it should be possible to select first an arbitrary phase shift and then run Beale's method.
See \cite[Rem.\@ 7.3.31]{Lombardi} for a comparison along these lines in Lombardi's own words.

Another contrast is that the original relative displacement problem \eqref{eqn: rel disp eqns} is not translation invariant in the sense that the position problem \eqref{eqn: original equations of motion} is.
That is, $\{r_j\}_{j \in \Z}$ solves \eqref{eqn: rel disp eqns}, it need not be the case that $\{\tilde{r}_j\}_{j \in \Z}$ does, too, where $\tilde{r}_j(t) := r_j(t)+d_1t+d_2$, $d_1$, $d_2 \in \R$.
This contrasts with the position invariances described in Remark \ref{rem: pos trans invar}.

Diagonalization has not been used in the other applications of Beale's method to FPUT and MiM lattices \cite{hoffman-wright, faver-hupkes-equal-mass, faver-mim-nanopteron}.
Instead, those traveling wave problems possess a ``mean-zero'' symmetry in one of their components; roughly, the Fourier mode at zero of one component of the problem is always zero.
This does occur in the first equation in \eqref{eqn: beale diag}, but we had to exploit it further with the ``Friesecke--Pego cancelation.''
In those other ``material limits,'' the mean-zero symmetry has a host of technical applications; in particular, it reduces by one the dimension of the kernel of the linearization of the periodic problem, thereby facilitating, in part, the modified bifurcation-from-a-simple-eigenvalue approach.

Finally, we note two possible disadvantages of the spatial dynamics/Lombardi method relative to Beale's.
First, the nature of Lombardi's rescaling (Lemma \ref{lem: Lombardi nf}) requires, in the spring dimer case, $\beta+\kappa^3\ne0$.
However, this is necessary to apply Lombardi's periodic construction.
In \cite{faver-spring-dimer}, it was possible to produce the periodics even with $\beta+\kappa^3=0$.

Second, as we noted in Section \ref{sec: prior nanopteron results}, most other applications of Beale's ansatz \cite{hoffman-wright,faver-mim-nanopteron, faver-hupkes-equal-mass} have produced nanopterons whose localized core is $\O(1)$ in the problem's small parameter.
However, Lombardi's long wave scaling ensures that his cores are indeed small in the relevant small parameter; this is apparent in \eqref{eqn: Lombardi nanopteron}.
Although the amplitude estimates of Beale's method have been, heretofore, less precise than those of Lombardi, it seems that Beale's method affords more flexibility as to the precise structure of the nanopteron.

\bibliographystyle{siam}
{ \footnotesize
\bibliography{spatial_dynamics_bib}{}
}

\end{document}